\documentclass[11pt,reqno]{amsart}

\usepackage[letterpaper, margin=1in]{geometry}
\usepackage{amsfonts}
\usepackage{amssymb}
\usepackage{amsmath}
\usepackage{mathrsfs}
\usepackage{color}
\usepackage{cite}
\usepackage{verbatim}
\usepackage{dsfont}
\usepackage{bbm}
\usepackage{enumerate}
\usepackage{graphicx}
\usepackage{xcolor} 
\usepackage{slashed}
\usepackage{marginnote}

\numberwithin{equation}{section}

\newtheorem{theorem}{Theorem}[section]
\newtheorem{lemma}[theorem]{Lemma}
\newtheorem{proposition}[theorem]{Proposition}
\newtheorem{remark}[theorem]{Remark}

\definecolor{deepgreen}{cmyk}{1,0,1,0.5}

\newcommand{\jap}[1]{\langle #1\rangle}

\newcommand{\barg}{{\overline g}}

\newcommand{\barv}{{\overline v}}

\renewcommand{\hbar}{{\underline h}}

\newcommand{\bbC}{\mathbb C}

\newcommand{\bbN}{\mathbb N}

\newcommand{\bbR}{\mathbb R}

\newcommand{\bbZ}{\mathbb Z}

\newcommand{\calA}{\mathcal A}
\newcommand{\calB}{\mathcal B}
\newcommand{\calC}{\mathcal C}
\newcommand{\calD}{\mathcal D}

\newcommand{\calF}{\mathcal F}

\newcommand{\calI}{\mathcal I}
\newcommand{\calJ}{\mathcal J}
\newcommand{\calK}{\mathcal K}
\newcommand{\calL}{\mathcal L}
\newcommand{\calM}{\mathcal M}
\newcommand{\calN}{\mathcal N}
\newcommand{\calO}{\mathcal O}

\newcommand{\calQ}{\mathcal Q}
\newcommand{\calR}{\mathcal R}
\newcommand{\calS}{\mathcal S}
\newcommand{\calT}{\mathcal T}

\newcommand{\frakm}{\mathfrak m}
\newcommand{\frakn}{\mathfrak n}

\newcommand{\tila}{{\tilde{a}}}

\newcommand{\tild}{{\tilde{d}}}

\newcommand{\hatf}{{\hat{f}}}

\newcommand{\hatg}{{\hat{g}}}

\newcommand{\hath}{{\hat{h}}}

\newcommand{\hatq}{{\hat{q}}}

\newcommand{\hatv}{{\hat{v}}}

\newcommand{\whatw}{{\widehat{w}}}

\newcommand{\whatK}{{\widehat{K}}}

\newcommand{\ud}{\mathrm{d}}

\newcommand{\px}{\partial_x}

\newcommand{\pt}{\partial_t}
\newcommand{\ps}{\partial_s}
\newcommand{\pxi}{\partial_\xi}

\newcommand{\jt}{\jap{t}}
\newcommand{\js}{\jap{s}}
\newcommand{\jx}{\jap{x}}
\newcommand{\jxi}{\jap{\xi}}

\newcommand{\hf}{\frac{1}{2}}
\newcommand{\thf}{\frac{3}{2}}

\newcommand{\jxione}{\jap{\xi_1}}
\newcommand{\jxitwo}{\jap{\xi_2}}
\newcommand{\jxithree}{\jap{\xi_3}}
\newcommand{\jxifour}{\jap{\xi_4}}

\newcommand{\pxione}{\partial_{\xi_1}}
\newcommand{\pxitwo}{\partial_{\xi_2}}
\newcommand{\pxithree}{\partial_{\xi_3}}
\newcommand{\pxifour}{\partial_{\xi_4}}

\newcommand{\bv}{\bar{v}}

\newcommand{\hatbarf}{\hat{\bar{f}}}
\newcommand{\hatbarg}{\hat{\bar{g}}}

\newcommand{\whatalpha}{\widehat{\alpha}}
\newcommand{\whatbeta}{\widehat{\beta}}

\newcommand{\wtilcalI}{\widetilde{\mathcal{I}}}
\newcommand{\wtilcalJ}{\widetilde{\mathcal{J}}}

\newcommand{\wtilvarphi}{\widetilde{\varphi}}

\newcommand{\linop}{L}

\newcommand{\pvdots}{\mathrm{p.v.}}

\newcommand{\supp}{\mathrm{supp}}

\renewcommand{\Im}{\mathrm{Im}}
\renewcommand{\Re}{\mathrm{Re}}
\DeclareMathOperator{\sech}{sech}
\DeclareMathOperator{\cosech}{cosech}

\def\wh{\widehat}
\def\R{\mathbb{R}}

\def\les{\lesssim}

\def\les{\lesssim} 
\def\calL{\mathcal{L}}
\def\calS{\mathcal{S}}
\def\pih{\frac{\pi}{2}}

\def\calI{\mathcal{I}}

\def\jD{\jap{D}}
\def\jxi{\jap{\xi}}

\def\calC{\mathcal{C}}

\def\calT{\mathcal{T}}
\def\calJ{\mathcal{J}}
\def\calF{\mathcal{F}}

\newcommand{\EQ}[1]{\begin{equation}\begin{split} #1 \end{split}\end{equation}} 

\begin{document}

\title[On codimension one stability of the soliton for the 1D focusing cubic KG equation]{On codimension one stability of the soliton for the \\ 1D focusing cubic Klein-Gordon equation}

\author[J. L\"uhrmann]{Jonas L\"uhrmann}
\address{Department of Mathematics \\ Texas A\&M University \\ College Station, TX 77843, USA}
\email{luhrmann@math.tamu.edu}

\author[W. Schlag]{Wilhelm Schlag}
\address{Department of Mathematics \\ Yale University \\ New Haven, CT 06511, USA}
\email{wilhelm.schlag@yale.edu}

\thanks{
J. L\"uhrmann was partially supported by NSF grant DMS-1954707. 
W. Schlag was partially supported by NSF grant DMS-1902691.
}

\begin{abstract}
We consider the codimension one asymptotic stability problem for the soliton of the focusing cubic Klein-Gordon equation on the line under even perturbations.
The main obstruction to full asymptotic stability on the center-stable manifold is a small divisor in a    quadratic source term of the perturbation equation.
This singularity is due to the threshold resonance of the linearized operator and the absence of null structure in the nonlinearity. The threshold resonance of the linearized operator produces a one-dimensional space of slowly decaying Klein-Gordon waves, relative to local norms.
In contrast, the closely related perturbation equation for the sine-Gordon kink does exhibit null structure, which makes the corresponding quadratic source term amenable to normal forms \cite{LS1}.



The main result of this work establishes decay estimates up to exponential time scales for small ``codimension one type'' perturbations of the soliton of the focusing cubic Klein-Gordon equation.
The proof is based upon a super-symmetric approach to the study of modified scattering for 1D nonlinear Klein-Gordon equations with P\"oschl-Teller potentials from~\cite{LS1}, and an implementation of a version of an adapted functional framework introduced in~\cite{GP20}.
\end{abstract}

\maketitle 

\tableofcontents
 
\section{Introduction}


\subsection{Main result}

We consider the focusing cubic Klein-Gordon equation in one space dimension
\begin{equation} \label{equ:focusing_cubic_KG}
 (\pt^2 - \px^2 +1) \phi = \phi^3, \quad (t,x) \in \bbR \times \bbR.
\end{equation}
Its solutions formally conserve the energy 
\begin{equation*}
 E = \int_\bbR \Bigl( \frac12 (\pt \phi)^2 + \frac12 (\px \phi)^2 + \frac12 \phi^2 - \frac14 \phi^4 \Bigr) \, \ud x.
\end{equation*}
Local well-posedness of~\eqref{equ:focusing_cubic_KG} for $H^1_x \times L^2_x$ initial data is a consequence of a standard fixed point argument, and the global existence of solutions with small $H^1_x \times L^2_x$ initial data can be inferred from the conservation of energy. 
For large initial data, solutions to~\eqref{equ:focusing_cubic_KG} may form singularities in finite time.

This work is concerned with the long-time dynamics of even solutions to~\eqref{equ:focusing_cubic_KG} in the vicinity of the soliton solution
\begin{equation} \label{equ:soliton}
 Q(x) = \sqrt{2} \sech(x), \quad x \in \bbR.
\end{equation}
Note that the flow of~\eqref{equ:focusing_cubic_KG} preserves the even parity.
The evolution equation for a perturbation
\begin{equation*}
 \varphi(t,x) := \phi(t,x) - Q(x)
\end{equation*}
of the soliton is given by
\begin{equation} \label{equ:intro_perturbation_equation}
 \bigl(\pt^2 - \px^2 - 3 Q^2 + 1\bigr) \varphi = 3 Q \varphi^2 + \varphi^3.
\end{equation}
The linearized operator 
\begin{equation} \label{equ:linearized_operator}
 \linop = -\px^2 -3 Q^2 + 1 = - \partial_x^2 - 6 \sech^2(x) + 1
\end{equation}
features the Schr\"odinger operator $- \partial_x^2 - 6 \sech^2(x)$, which is the second member in the hierarchy of Schr\"odinger operators $-\px^2 - \ell (\ell+1) \sech^2(x)$, $\ell \in \bbN$, with reflectionless P\"oschl-Teller potentials~\cite{PoschlTeller}. Their spectra can be computed explicitly \cite[Chapter 4.19]{Titchmarsh_Part1}.
It turns out that the linearized operator $\linop$ has essential spectrum $[1,\infty)$ and that it exhibits the even ($L^2_x$-normalized) eigenfunction $Y_0$ with negative eigenvalue $-3$, 
\begin{equation*}
 Y_0(x) = c_0 \sech^2(x), \qquad L Y_0 = - \nu^2 Y_0, \qquad c_0 := \sqrt{\frac{3}{4}}, \qquad \nu := \sqrt{3},
\end{equation*}
the odd ($L^2_x$-normalized) eigenfunction $Y_1$ with zero eigenvalue
\begin{equation*}
 Y_1(x) = c_1 \sech(x) \tanh(x), \qquad L Y_1 = 0, \qquad c_1 := \sqrt{\frac{3}{2}},
\end{equation*}
and the even threshold resonance
\begin{equation*}
 Y_2(x) = 1 - \frac32 \sech^2(x), \qquad L Y_2 = Y_2.
\end{equation*}
The odd eigenfunction $Y_1 \simeq Q'$ with zero eigenvalue is related to the invariance under spatial translations and is referred to as the translational mode. Since we only consider even perturbations of the soliton, the odd translational mode is not relevant for our analysis.
In contrast, the even eigenfunction $Y_0$ associated with the negative eigenvalue $-\nu^2$ and the even threshold resonance $Y_2$ decisively affect the dynamics of (even) solutions to \eqref{equ:focusing_cubic_KG} in the vicinity of the soliton $Q$.

The negative eigenvalue of the linearized operator gives rise to the exponentially growing solution $\varphi(t,x) = e^{\nu t} Y_0(x)$ to the linearized equation $(\pt^2 + L) \varphi = 0$, which is thus an obstruction to the stability of the soliton $Q$ under small perturbations.
However, Kowalczyk-Martel-Mu\~{n}oz~\cite[Theorem 2]{KMM19} showed\footnote{While the orbital stability on the center-stable manifold result \cite[Theorem 2]{KMM19} is formulated for the family of focusing Klein-Gordon equations $(\pt^2 - \px^2 + 1) \phi = |\phi|^{p-1} \phi$ with powers $p > 3$, the proof carries over verbatim to the cubic case $p = 3$.} that near the soliton there exists a codimension one manifold of even initial data in the energy space, for which the solutions to~\eqref{equ:focusing_cubic_KG} exist for all times $t \geq 0$ and stay close to the soliton in the energy norm.

\begin{theorem}[{\protect Kowalczyk-Martel-Mu\~{n}oz \cite[Theorem 2]{KMM19}}] \label{thm:KMM_orbital}
There exist constants $C, \delta_0 > 0$ and a Lipschitz function $h \colon \calA_0 \to \bbR$ with
\begin{equation*}
 \calA_0 := \bigl\{ (\varphi_0, \varphi_1) \in H^1_x(\bbR) \times L^2_x(\bbR) \text{ even } \, \big| \, \|(\varphi_0, \varphi_1)\|_{H^1_x \times L^2_x} < \delta_0 \text{ and } \langle Y_0, \nu \varphi_0 + \varphi_1 \rangle = 0 \bigr\},
\end{equation*}
and
\begin{equation*}
 h(0,0) = 0, \quad |h(\varphi_0,\varphi_1)| \leq C \|(\varphi_0, \varphi_1)\|_{H^1_x \times L^2_x}^{\frac32},
\end{equation*}
such that denoting
\begin{equation*}
 \calM := \bigl\{ (Q,0) + (\varphi_0, \varphi_1) + h(\varphi_0, \varphi_1) (Y_0, \nu Y_0) \, \big| \, (\varphi_0, \varphi_1) \in \calA_0 \bigr\},
\end{equation*}
the following holds:
\begin{itemize}
 \item[(1)] If $(\phi_0, \phi_1) \in \calM$, then the solution $(\phi, \pt \phi)$ to \eqref{equ:focusing_cubic_KG} exists for all times $t \geq 0$ and satisfies 
 \begin{equation*}
  \sup_{t \geq 0} \, \bigl\| \bigl(\phi(t), \pt \phi(t)\bigr) - (Q,0) \bigr\|_{H^1_x \times L^2_x} \leq C \bigl\| (\phi_0, \phi_1) - (Q,0) \bigr\|_{H^1_x \times L^2_x}.
 \end{equation*}
 
 \item[(2)] If an even solution $(\phi, \pt \phi)$ to \eqref{equ:focusing_cubic_KG} satisfies
 \begin{equation*}
  \sup_{t \geq 0} \, \bigl\| \bigl(\phi(t), \pt \phi(t)\bigr) - (Q,0) \bigr\|_{H^1_x \times L^2_x} \leq \frac{1}{10} \delta_0,
 \end{equation*}
 then $(\phi(t), \pt \phi(t)) \in \calM$ for all $t \geq 0$.
\end{itemize}
\end{theorem}

We emphasize that the statement of Theorem~\ref{thm:KMM_orbital} is by far not the main result from \cite{KMM19}, see the discussion of the related literature further below.
It is thus natural to ask if the soliton $Q$ enjoys stronger codimension one asymptotic stability properties in the sense that the solutions (or a subset of the solutions) on the center-stable manifold~$\calM$ asymptotically converge to $Q$.
In one space dimension it is customary to distinguish the notion of local asymptotic stability in the sense of convergence in a local energy norm, and the notion of full asymptotic stability in the sense of explicit decay estimates (and usually asymptotics).
In this work we make partial progress on the codimension one full asymptotic stability question for the soliton $Q$ of the focusing cubic Klein-Gordon equation~\eqref{equ:focusing_cubic_KG}.
We prove for a subset of initial conditions in $\calA_0$, which are of size $0 < \varepsilon \ll 1$ measured in a weighted Sobolev norm, that upon correcting for the exponentially growing mode caused by the negative eigenvalue of the linearized operator, the corresponding solution to~\eqref{equ:focusing_cubic_KG} decays back to the soliton $Q$ in $L^\infty_x$ at the rate $\jt^{-\frac12} \log(2+t) \varepsilon$ up to times $\exp(c\varepsilon^{-\frac13})$.
As we will explain in more detail, the logarithmic slow-down of the decay rate in comparison to the ordinary $\jt^{-\frac12}$ decay in $L^\infty_x$ of free Klein-Gordon waves in one space dimension and the limitation to times up to $\exp(c\varepsilon^{-\frac13})$ are intimately tied to the effects of the threshold resonance $Y_2$ on the dynamics of perturbations of the soliton.

We are now in the position to state our main result.

\begin{theorem} \label{thm:main}
 There exist absolute constants $0 < \varepsilon_0 \ll 1$, $0 < c \ll 1$, and $C \geq 1$ with the following property:
 For every even $(\varphi_0, \varphi_1) \in H^4_x \times H^3_x$ satisfying
 \begin{equation*}
  \varepsilon := \|\jx (\varphi_0, \varphi_1)\|_{H^4_x \times H^3_x} \leq \varepsilon_0
 \end{equation*}
 and
 \begin{equation*}
  \langle Y_0, \nu \varphi_0 + \varphi_1 \rangle = 0,
 \end{equation*}
 there exists $d \in \bbR$ with $|d| \leq C \varepsilon^{\frac32}$ such that the solution to~\eqref{equ:focusing_cubic_KG} with initial data
 \begin{equation*}
  (\phi_0, \phi_1) = (Q,0) + (\varphi_0, \varphi_1) + d (Y_0, \nu Y_0)
 \end{equation*}
 satisfies 
 \begin{equation} \label{equ:thm_asserted_decay}
  \| \phi(t) - Q \|_{L^\infty_x} \leq \frac{\log(2+t)}{\jt^\hf} C \varepsilon \quad \text{for all } \, 0 \leq t \leq \exp\bigl(c \varepsilon^{-\frac13}\bigr).
 \end{equation}
\end{theorem}

We continue with a few comments on Theorem~\ref{thm:main}.

\begin{remark}
 The choice of the parameter $d$ in the statement of Theorem~\ref{thm:main} is not unique, because up to finite times $0 \leq t \leq \exp(c \varepsilon^{-\frac13})$ the exponential growth of double-exponentially small variations of the parameter $d$, say of type $\ll \exp(-C \exp(c \varepsilon^{-\frac13}))$, would not destroy the asserted overall decay~\eqref{equ:thm_asserted_decay}.
\end{remark}

\begin{remark}
 We conjecture that there exists a unique value for the parameter $d$ such that \eqref{equ:thm_asserted_decay} holds for all times $t \geq 0$, possibly under stronger assumptions on the initial conditions. The logarithmic slow-down of the decay rate in \eqref{equ:thm_asserted_decay} should be essentially optimal as $t \to \infty$ in view of such sharp decay estimates with asymptotics established in \cite{LLS2, LLSS} for simplified model problems related to the perturbation equation for the soliton $Q$. See Subsection~\ref{subsubsec:main_difficulties_localized_quadratic} below for a more detailed discussion of this point.
\end{remark}

This work is part of the broader goal to understand the long-time dynamics of solutions to the family of 1D focusing Klein-Gordon equations 
\begin{equation} \label{equ:pfocusingKG}
 (\pt^2 - \px^2 + 1) \phi = |\phi|^{p-1} \phi, \quad (t,x) \in \bbR \times \bbR, \quad p > 1,
\end{equation}
in the vicinity of their static even soliton solutions 
\begin{equation*}
 Q_p(x) = \bigl( {\textstyle \frac{p+1}{2}} \bigr)^{\frac{1}{p-1}} \sech^{\frac{2}{p-1}}\bigl( {\textstyle \frac{p-1}{2}} x \bigr), \quad p > 1.
\end{equation*}
The associated linearized operators are given by 
\begin{equation*}
 L_p = -\px^2 - {\textstyle \frac{p(p+1)}{2}} \sech^2\bigl( {\textstyle \frac{p-1}{2}} x \bigr) + 1, \quad p > 1.
\end{equation*}
For $p > 3$, the operator $L_p$ only exhibits a negative eigenvalue and a zero eigenvalue (translational mode). The cubic case $p = 3$ additionally features the above mentioned threshold resonance $Y_2$. For $1 < p < 3$ as $p \to 1^+$, the operators $L_p$ have more and more positive gap eigenvalues and sometimes threshold resonances. For instance, the quadratic case $p=2$ features one positive gap eigenvalue and a threshold resonance along with the negative eigenvalue and the zero eigenvalue. We refer to Chang-Gustafson-Nakanishi-Tsai~\cite[Section 3]{CGNT07} for a detailed description of the spectra of the linearized operators $L_p$.

For $p > 1$, Payne-Sattinger~\cite{PayneSattinger75} proved that for energies below that of the soliton $Q$ one either has global existence or blowup in both time directions. Ibrahim-Masmoudi-Nakanishi~\cite{IMN11} then established that for $p > 5$ for energies below that of the soliton, global existence implies scattering.
A numerical investigation of the convergence rate back to $Q_p$ of trapped perturbations of $Q_p$ (for arbitrary $p>1$) was undertaken by Bizo\'{n}-Chmaj-Szpak \cite{BizTadSzp11}.
For $p > 5$ a complete classification of the dynamics of even solutions to \eqref{equ:pfocusingKG} with energies slightly above that of $Q_p$ was achieved by Krieger, Nakanishi, and the second author~\cite{KNS12}.
The latter includes the construction of a $C^1$ center-stable manifold near the soliton $Q_p$ for even solutions that scatter (linearly) back to $Q_p$ in the energy space. 
For $p > 3$, Kowalczyk-Martel-Mu\~{n}oz~\cite{KMM19} established the conditional asymptotic stability of $Q_p$ under even perturbations locally in the energy space. In the case $p=2$, the first author and Li \cite{LL1} extended the latter result to a quadratic\footnote{The results in \cite{LL1} pertain to the quadratic nonlinearity $\phi^2$, but not necessarily to the quadratic nonlinearity $|\phi| \phi$. In the latter case it is not clear how to exploit the oscillations of the solutions to infer decay of the internal mode in that setting.} Klein-Gordon equation.
Kairzhan-Pusateri~\cite{KairzhanPusateri22} obtained the full asymptotic stability of $Q_4$ under even perturbations in the case of a Klein-Gordon equation with quartic nonlinearity $\phi^4$.

Closely related to the codimension one asymptotic stability problem for the solitons $Q_p$ of the family of focusing Klein-Gordon equations~\eqref{equ:pfocusingKG} in one space dimension is the asymptotic stability problem for kinks that arise in scalar field theories on the line
\begin{equation} \label{equ:scalar_field_theory_EL}
 (\pt^2 - \px^2) \phi + W'(\phi) = 0, \quad (t,x) \in \bbR \times \bbR, 
\end{equation}
where $W \colon \bbR \to [0,\infty)$ is a scalar double-well potential, i.e., $W$ has at least two consecutive global minima $\phi_- < \phi_+$ with $W(\phi_\pm) = 0$ and $W''(\phi_\pm) > 0$. Prime examples include the $\phi^4$ model with $W(\phi) = \frac14 (1-\phi^2)^2$, the sine-Gordon model with $W(\phi) = 1-\cos(\phi)$, double-sine Gordon theories, and the more general $P(\phi)_2$ theories. 

The main result of this work is closely related to the recent work of Kairzhan-Pusateri~\cite{KairzhanPusateri22} on codimension one full asymptotic stability of the soliton for \eqref{equ:pfocusingKG} in the quartic case and to the recent works concerning the full asymptotic stability of kinks under odd perturbations by Delort-Masmoudi~\cite{DelMas20} for the $\phi^4$ model up to times $\varepsilon^{-4+c}$ with $0 < c \ll 1$, by Germain-Pusateri~\cite{GP20} on double sine-Gordon models, see also Germain-Pusateri-Zhang~\cite{GermPusZhang22}, and by the authors~\cite{LS1} on the sine-Gordon model. See also Chen-Liu-Lu~\cite{CLL20}, Chen-Pusateri~\cite{ChenPus19, ChenPus22}, Chen~\cite{Chen21}, and L\'eger-Pusateri \cite{LegPus21, LegPus22}.

At the heart of the full codimension one asymptotic stability problems for solitons for the focusing Klein-Gordon equations~\eqref{equ:pfocusingKG} under even perturbations and of the full asymptotic stability problems for kink solutions to~\eqref{equ:scalar_field_theory_EL} under odd perturbations is the analysis of the long-time behavior of small symmetric solutions to 1D nonlinear Klein-Gordon equations
\begin{equation} \label{equ:intro_schematic_perturbation_equ}
 \bigl(\pt^2 - \px^2 + V(x) + m^2\bigr) u = \alpha(x) u^2 + \beta_0 u^3 + \cdots, \quad (t,x) \in \bbR \times \bbR,
\end{equation}
where $V(x)$ is a smooth localized potential, $m > 0$ is a mass parameter, $\alpha(x)$ is a (possibly localized) variable coefficient and $\beta_0 \in \bbR$ is a constant coefficient.

We view this work on the codimension one full asymptotic stability of the soliton for the focusing cubic Klein-Gordon equation \eqref{equ:focusing_cubic_KG} under even perturbations as a natural continuation of our previous full asymptotic stability result \cite{LS1} for the sine-Gordon kink under odd perturbations.
One of the main difficulties for both problems is to deal with a singular quadratic source term that stems from the localized quadratic nonlinearity and the slow local decay of the corresponding Klein-Gordon waves due to the threshold resonance of the respective linearized operators. In the sine-Gordon case this source term exhibits a remarkable null structure. In contrast, in the case of the focusing cubic Klein-Gordon equation such favorable structure is not present and one has to face the full force of the corresponding quadratic source term.
We point out that while the sine-Gordon equation is completely integrable, the focusing cubic Klein-Gordon equation is not.

\medskip

\noindent {\it Acknowledgements:}
The authors are grateful to Gong Chen, Pierre Germain, Yongming Li, Claudio Mu\~{n}oz, Beno\^{i}t Pausader, and Fabio Pusateri for valuable comments on the manuscript.
The first author would like to thank Beno\^{i}t Pausader for helpful discussions at an early stage of the investigation.
Part of this work was done while the first author participated in the ICERM semester program ``Hamiltonian Methods in Dispersive and Wave Evolution Equations''. He is grateful to ICERM for the hospitality and the support.
The authors thank the referees for their careful proof-reading of the manuscript and for many helpful comments.

\subsection{References}


In this subsection we collect references to works that are closely related to the codimension one asymptotic stability problem for the soliton of the focusing cubic Klein-Gordon equation~\eqref{equ:focusing_cubic_KG}.
In view of the rich and vast literature on soliton stability and modified scattering, the following references are by far not exhaustive.

For the study of the long-time dynamics in the vicinity of the solitons of the family of 1D focusing Klein-Gordon equations, we refer to \cite{PayneSattinger75, IMN11, BizTadSzp11, KNS12, KMM19, KairzhanPusateri22, LL1}. See also \cite{Schlag09, NakSchlag_Book}.
The asymptotic stability of kinks in scalar field theories on the line has been investigated in
\cite{HPW82, KMM17, KMMV20, KM22, CuccMaeda22, AMP20, KK11_1, KK11_2, GP20, GermPusZhang22, CLL20, LS1, DelMas20}. We also refer to \cite{MunPal19, ChenJendrej22, JKL19, Moutinho22_1, Moutinho22_2, Moutinho22_3} for results on the dynamics of multi-kink solutions.
Modified scattering of small solutions to the 1D cubic Schr\"odinger equation without potential has been studied in \cite{HN98, LS06, KatPus11, IT15}, and with potential in \cite{GermPusRou18, ChenPus19, ChenPus22,
Del16, Naum16, Naum18, MasMurphSeg19, NaumWed22}.
Similarly, modified scattering of small solutions to 1D Klein-Gordon equations with low power constant coefficient nonlinearities, but without a potential, has been investigated in \cite{Del01, LS05_1, LS05_2, HN08, HN10, HN12, Del16_KG, Stingo18, CL18}. Klein-Gordon models in one space dimension with variable coefficient nonlinearities or potentials have been considered in \cite{Sterb16, LS15, LLS1, LLS2, LLSS, GP20, GermPusZhang22}.
For works on radiation damping in the presence of internal modes we refer to \cite{SofWein99, Sigal93, BamCucc11, TsaiYau02, KK11_1, CuccMaeda21_AnnPDE, DelMas20, LegPus21, LegPus22} and references therein.
See \cite{BusPerel92, BusSul03, KS06, Mizumachi08, CuccagnaPelinovsky14, Chen21, CuccMaeda22_CMP, Martel21, MasMurphSeg20} and references therein for a sample of works on the asymptotic stability of solitary wave solutions to nonlinear Schr\"odinger equations in one space dimension.

Finally, we point the reader to the monographs
\cite{DauxPey10, Lamb80, MantSut04, SG_Series, phi4_Series}
for more background on solitons, and we refer to the survey articles~\cite{Tao09, KMM17_Survey, CuccMaeda20_Survey, Martel_ICM} on asymptotic stability of solitons and soliton interactions for more references.

\subsection{Main difficulties}

In this subsection, we discuss some of the main difficulties of the codimension one full asymptotic stability problem for the soliton of the 1D focusing cubic Klein-Gordon equation under even perturbations, and how these are relevant in the proof of the decay estimate~\eqref{equ:thm_asserted_decay} in Theorem~\ref{thm:main} up to exponential time scales.
Since it is a relatively standard step to take into account the exponential instability caused by the negative eigenvalue of the linearized operator around the soliton for the focusing cubic Klein-Gordon equation~\eqref{equ:focusing_cubic_KG}, we ignore this aspect of the problem in the discussion in this subsection.
The main difficulties can be described in the context of the analysis of the long-time behavior of small symmetric solutions to 1D Klein-Gordon equations of the form
\begin{equation} \label{equ:intro_main_difficulties_schematic_perturbation_equ}
 \bigl( \pt^2 - \px^2 + V(x) + 1 \bigr) u = \alpha(x) u^2 + \beta_0 u^3 + \cdots, \quad (t,x) \in \bbR \times \bbR,
\end{equation}
where the linear operator $-\px^2 + V(x) + 1$ has a threshold resonance, but no bound states, and where $\alpha(x)$ is smooth and spatially localized, and $\beta_0 \in \bbR$. In particular, we assume that the parity of the solution $u$ does not avoid the threshold resonance.


\subsubsection{Vector field method for 1D Klein-Gordon equations with a potential}

In order to derive explicit decay estimates and asymptotics for small (symmetric) solutions to 1D Klein-Gordon equations \eqref{equ:intro_main_difficulties_schematic_perturbation_equ} with a potential and low power nonlinearities,
one basically has to take a vector field based approach that also allows to capture the oscillations of the solutions.
While a number of methods have been developed over the years for the flat (zero potential) case, only more recently several vector field methods have been introduced that apply to 1D Klein-Gordon equations with a potential in various degrees of generality:
\begin{itemize}
 \item[(i)] Using the distorted Fourier transform adapted to the Schr\"odinger operator $-\px^2 + V$, see for instance~\cite{GHW15, GPR16, ChenPus19, GP20, GermPusZhang22, LegPus21, ChenPus22, PusSof20, Naum16, Naum18}.

 \item[(ii)] Applying the wave operator associated with the Schr\"odinger operator $-\px^2 + V$ to conjugate to the flat case, see for instance~\cite{Del16, DelMas20}.
 \item[(iii)] Exploiting specific super-symmetry factorization properties of the hierarchy of Schr\"odinger operators $-\px^2 - \ell(\ell+1) \sech^2(x)$, $\ell \in \bbN$, with P\"oschl-Teller potentials to transform to the flat case through the conjugation identity
 \begin{equation} \label{equ:conjugation_identity}
  \calD_1 \cdots \calD_\ell \bigl(- \px^2 - \ell(\ell+1) \sech^2(x) + m^2 \bigr) = (-\px^2 + m^2) \calD_1 \cdots \calD_\ell, \quad \ell \in \bbN,
 \end{equation}
 where
 \begin{equation*} 
  \calD_k = \px + k \tanh(x), \quad 1 \leq k \leq \ell.
 \end{equation*}
 This approach was employed by the authors~\cite{LS1} in the analysis of the modified scattering behavior of odd perturbations of the sine-Gordon kink.
 Such super-symmetry ideas go back to the 19th century work of Darboux~\cite{Darboux1882},
 and have been used before in various forms in the rigorous study of nonlinear dispersive equations.
\end{itemize}

In this work we pursue the super-symmetry approach (iii) to transform the evolution equation~\eqref{equ:intro_perturbation_equation} (with $\ell = 2$) for (the dispersive part of) perturbations of the soliton to a nonlinear Klein-Gordon equation with zero potential. We then study the latter in the spirit of the space-time resonances method~\cite{GMS12_Ann, GMS12_JMPA, GMS09, GNT09}.

We also point out a vector field method on the distorted Fourier side for wave equations with a potential introduced in \cite{DonKrie1, DKSW} and the $r^p$ method \cite{DR1} for wave equations on asymptotically flat backgrounds.

\subsubsection{Long-range nature of the constant coefficient cubic nonlinearity}

In view of the slow $t^{-\frac12}$ dispersive $L^\infty_x$ decay of free Klein-Gordon waves in one space dimension, the constant coefficient cubic nonlinearity $\beta_0 u^3$ in \eqref{equ:intro_main_difficulties_schematic_perturbation_equ} can typically be thought to have the schematic form $\beta_0 u^3 \sim \varepsilon^2 t^{-1} u$, where $0 < \varepsilon \ll 1$ is a measure of the smallness of the solution. It is thus critical in the sense that $t^{-1}$ barely fails to be integrable as $t\to\infty$.
For this reason one cannot hope to obtain energy estimates for vector fields of the solution that are uniformly bounded in time, but one rather has to reluctantly work with energy bounds that are slowly growing like $t^{C \varepsilon^2}$.
This precludes to recover the free $L^\infty_x$ decay rate $t^{-\frac12}$ for small solutions to \eqref{equ:intro_main_difficulties_schematic_perturbation_equ} just through Klainerman-Sobolev type estimates.

However, at least {\it in the absence of quadratic terms $\alpha(x) u^2$} in \eqref{equ:intro_main_difficulties_schematic_perturbation_equ}, upon taking into account logarithmic phase corrections in the asymptotic behavior of the solution induced by the critical constant coefficient cubic nonlinearity, such slowly growing bounds suffice to recover the free $L^\infty_x$ decay rate $t^{-\frac12}$.

Since in this work we derive decay estimates for perturbations of the soliton of the focusing cubic Klein-Gordon equation up to exponential time scales, there is not yet any need to take into account potential logarithmic phase corrections in the asymptotic behavior.
The limitation up to exponential time scales stems from the effects of the localized quadratic nonlinearity $\alpha(x) u^2$, as we explain next.

\subsubsection{Localized quadratic nonlinearity and the threshold resonance of the linearized operator} \label{subsubsec:main_difficulties_localized_quadratic}

Since the variable coefficient $\alpha(x)$ of the quadratic nonlinearity in~\eqref{equ:intro_main_difficulties_schematic_perturbation_equ} is spatially localized, the local decay of the solution $u(t,x)$ determines the leading order behavior of $\alpha(x) u(t,x)^2$.
Due to the threshold resonance of the linear operator, this local decay is slow. We recall from \cite[Corollary 2.17]{LLSS} the following local decay estimate for linear Klein-Gordon waves
\begin{equation} \label{equ:intro_main_difficulties_local_decay_subtract_off}
 \biggl\| \jx^{-\sigma} \biggl( e^{it\sqrt{-\px^2 + V + 1}} P_c f - c_0 \frac{e^{i\frac{\pi}{4}} e^{it}}{t^\hf} \langle \psi, f \rangle \psi \biggr) \biggr\|_{L^2_x} \lesssim \frac{1}{t^\thf} \| \jx^{\sigma} f \|_{L^2_x}, \quad t \geq 1.
\end{equation}
Here, $\psi \in L^\infty_x \backslash L^2_x$ is the threshold resonance of the linear operator $-\px^2 + V(x) + 1$ with normalization $\psi(x) \to 1$ as $x \to \infty$, $\sigma > \frac92$, $P_c$ denotes the projection to the continuos spectral subspace, and $c_0$ only depends on the scattering matrix of the potential $V(x)$ at zero energy.
This suggests that to leading order the quadratic nonlinearity should exhibit source terms of the form
\begin{equation} \label{equ:intro_main_difficulties_source_term}
 \alpha(x) \psi(x)^2 e^{i2t} \frac{\varepsilon^2}{t}, \quad t \gg 1,
\end{equation}
where $0 < \varepsilon \ll 1$ is a measure of the smallness of the solution.
Passing to a first-order formulation of the problem for the variable $v(t) := \frac12 (u(t) - i \jap{\widetilde{D}}^{-1} \pt u(t))$ with $\jap{\widetilde{D}} = (-\px^2+V+1)^{\frac12}$, the contribution of~\eqref{equ:intro_main_difficulties_source_term} to the distorted Fourier transform of the profile $f(t) := e^{-it\jap{\widetilde{D}}} v(t)$ is of the schematic form
\begin{equation} \label{equ:intro_main_difficulties_duhamel_source_term}
 \int_1^t e^{-is(1+\xi^2)^{\frac12}} \widetilde{\calF}\bigl[\alpha \psi^2](\xi) e^{i2s} \frac{\varepsilon^2}{s} \, \ud s.
\end{equation}
The latter has a time resonance when $-(1+\xi^2)^{\frac12} + 2 = 0$, which occurs for the frequencies $\xi = \pm \sqrt{3}$.
These are the problematic frequencies of the problem. Unless $\widetilde{\calF}[\alpha \psi^2](\pm \sqrt{3}) = 0$, the integral~\eqref{equ:intro_main_difficulties_duhamel_source_term} grows logarithmically at $\xi = \pm \sqrt{3}$.
Correspondingly, if $\widetilde{\calF}[\alpha \psi^2](\pm \sqrt{3}) \neq 0$, we may expect that the solution $v(t,x)$ only decays like $t^{-\frac12} \log(t)$ along the associated rays $x = \pm \frac{\sqrt{3}}{2} t$, i.e., that it may feature a logarithmic slow-down with respect to the linear $L^\infty_x$ decay rate $t^{-\frac12}$.

These heuristics were made rigorous in \cite{LLSS}, see also \cite{LLS2}, for the model problem
\begin{equation} \label{equ:intro_main_difficulties_model_problem}
 \bigl( \pt^2 - \px^2 + V(x) + 1 \bigr) u = \alpha(x) u^2.
\end{equation}
Since \eqref{equ:intro_main_difficulties_model_problem} only features localized nonlinearities, it sufficed in \cite{LLSS} to work with weighted $L^2_x$-based norms for the solution to analyze the asymptotics of small solutions to~\eqref{equ:intro_main_difficulties_model_problem}. This is not possible for the full problem when there is in particular an additional constant coefficient cubic nonlinearity $\beta_0 u^3$ on the right-hand side of~\eqref{equ:intro_main_difficulties_model_problem}. Instead one then has to resort to vector field methods for Klein-Gordon equations with a potential.

Remarkably, the Klein-Gordon equation for odd perturbations of the sine-Gordon kink exhibits the null structure $\widetilde{\calF}[\alpha \psi^2](\pm \sqrt{3}) = 0$, see \cite{LLSS, LS1}. However, this is not the case for perturbations of the soliton of the focusing cubic Klein-Gordon equation~\eqref{equ:focusing_cubic_KG}, as evidenced in Lemma~\ref{lem:resonance_condition_focusing_cubic_source_term} below.

The likely logarithmic slow-down of the decay rate of solutions to~\eqref{equ:intro_main_difficulties_schematic_perturbation_equ} has severe consequences for the analysis of the effects of the critical constant coefficient cubic nonlinearity $\beta_0 u^3$, discussed above. If we now crudely have to think of it as $\beta_0 u^3 \sim t^{-1} \log^2(t) \varepsilon^2 u$ for $t \gg 1$, we would be facing disastrous energy bounds that grow super-polynomially like $t^{C\varepsilon^2 \log^2(t)}$.

While the latter become truly problematic at exponential time scales, the localized quadratic nonlinearity already becomes problematic at much earlier time scales for the weighted energy estimates in a vector field method approach.
To illustrate this point, we consider the simplified (first-order) Klein-Gordon equation
\begin{equation} \label{equ:intro_main_difficulties_first_order_simplified}
 (\pt - i \jD) v = (2i\jD)^{-1} \bigl( q(x) v^2 + \beta_0 |v|^2 v \bigr)
\end{equation}
with $\jD = (1-\px^2)^{\frac12}$ and a localized coefficient $q(x)$ satisfying $\hatq(\pm \sqrt{3}) \neq 0$.
In fact, through the super-symmetry approach in this work we largely reduce the proof of Theorem~\ref{thm:main} to the analysis of the long-time behavior of small solutions to the equation~\eqref{equ:intro_main_difficulties_first_order_simplified}.
To this end one would typically try to propagate slowly growing $H^2_x$ energy bounds for the vector field $x - it\jD^{-1} \px = e^{it\jD} x e^{-it\jD}$ applied to the solution $v(t)$ to \eqref{equ:intro_main_difficulties_first_order_simplified}, or equivalently, to propagate slowly growing bounds for the weighted energies $\|\jx f(t)\|_{H^2_x} \simeq \|\jxi^2 \jap{\pxi} \hatf(t,\xi)\|_{L^2_\xi}$ for the profile $f(t) := e^{-it\jD} v(t)$ of the solution $v(t)$.
In view of the local decay estimate~\eqref{equ:intro_main_difficulties_local_decay_subtract_off} and the threshold resonance $\psi(x) = 1$ for the flat Klein-Gordon operator, it is reasonable to think of the quadratic nonlinearity in \eqref{equ:intro_main_difficulties_first_order_simplified} to leading order as
\begin{equation*}
 q(x) v(t,x)^2 \simeq q(x) e^{i2t} \frac{\varepsilon^2}{t}, \quad t \gg 1.
\end{equation*}
For its contribution to the weighted energy estimates for the profile, we compute that
\begin{equation} \label{equ:intro_main_difficulties_schematic_weighted_energy}
 \jxi^2 \pxi \biggl( (2i\jxi)^{-1} \int_1^t e^{-is\jxi} \hatq(\xi) e^{i2s} \frac{\varepsilon^2}{s} \, \ud s \biggr) \simeq \int_1^t e^{is(2-\jxi)} \xi \hatq(\xi) \cdot s \cdot \frac{\varepsilon^2}{s} \, \ud s + \bigl\{ \text{better} \bigr\}.
\end{equation}
Clearly, the integrand of the leading order term on the right-hand side of~\eqref{equ:intro_main_difficulties_schematic_weighted_energy} has no decay in $s$ and is essentially monotone (no oscillations) for $|2-\jxi| \lesssim t^{-1}$, i.e., for $||\xi|-\sqrt{3}| \lesssim t^{-1}$. Thus, we can at best expect to obtain a badly growing weighted energy estimate of the form
\begin{equation} \label{equ:intro_main_difficulties_badly_growing}
 \bigl\| \jxi^2 \pxi \hatf(t,\xi) \bigr\|_{L^2_\xi} \simeq t \cdot t^{-\frac12} \varepsilon^2 \simeq t^{\frac12} \varepsilon^2,
\end{equation}
where the $t^{-\frac12}$ gain stems from the smallness of the $L^2_\xi$ norm of the frequency region $||\xi|-\sqrt{3}| \lesssim t^{-1}$.

To avoid the disastrous growth of the weighted energy~\eqref{equ:intro_main_difficulties_badly_growing}, we implement a version of an adapted functional framework introduced in the remarkable work of Germain-Pusateri~\cite{GP20}, which takes into account the singularities at the problematic frequencies $\xi = \pm \sqrt{3}$. Heuristically, the idea is to propagate only ``half a $\pxi$ derivative'' near these problematic frequencies. Then the resulting adapted weighted energies only grow slowly, which can be compensated up to exponential time scales by the additional smallness of the nonlinear terms, leading to the asserted decay estimate~\eqref{equ:thm_asserted_decay} in Theorem~\ref{thm:main}.
Going beyond these exponential time scales is very delicate.

Finally, we emphasize that the occurrence of a slow-down of the decay rate due to the presence of a space-time resonance was pointed out in \cite{BerGerm13} in the setting of proving bilinear dispersive estimates for quadratic interactions of 1D free dispersive waves.
Moreover, for instances where the linear decay rate cannot be propagated by the nonlinear flow, we refer to \cite{DIP17} on global solutions to the Euler–Maxwell system for electrons in two space dimensions, and to \cite{DIPP17} on global solutions of the gravity-capillary water-wave system in three space dimensions.

\medskip

We conclude the discussion of the main difficulties by verifying that the singular quadratic source term for the perturbation equation of the soliton of the focusing cubic Klein-Gordon equation satisfies the resonance condition $\widetilde{\calF}[\alpha \psi^2](\pm \sqrt{3}) \neq 0$ described above.
For definitions related to the spectral theory of Schr\"odinger operators on the line and for conventions used in the next lemma, we refer to \cite[Section 2]{LLSS}.

\begin{lemma} \label{lem:resonance_condition_focusing_cubic_source_term}
Denote by $\widetilde{\calF}$ the distorted Fourier transform associated with the Schr\"odinger operator $H = -\px^2-6\sech^2(x)$ featured in the linearized operator $L$ defined in \eqref{equ:linearized_operator}.
Recall that $3Q(x)$ is the variable coefficient of the quadratic nonlinearity in the perturbation equation~\eqref{equ:intro_perturbation_equation} of the soliton and that $Y_2(x)$ is the threshold resonance of the linearized operator $L$.
Then we have
\begin{equation} \label{equ:intro_resonance_condition}
 \widetilde{\calF}\bigl[3QY_2^2\bigr](\pm \sqrt{3}) = \frac{3}{28} (1-3i\sqrt{3}) \sqrt{\pi} \sech \Bigl( \frac{\pi \sqrt{3}}{2} \Bigr) \neq 0.
\end{equation}
\end{lemma}
\begin{proof}
We first compute the Jost solutions for the Schr\"odinger operator $H = -\px^2 - 6 \sech^2(x)$.
Using the adjoint of the conjugation identity~\eqref{equ:conjugation_identity} (with $\ell=2$ and $m^2=0$), we see that
\begin{equation*}
 H \calD_2^\ast \calD_1^\ast \bigl( e^{ix\xi} \bigr) = \xi^2 \calD_2^\ast \calD_1^\ast \bigl( e^{ix\xi} \bigr).
\end{equation*}
Thus, we obtain that the Jost solutions for $H$ with the normalization $e^{\mp ix\xi} f_\pm(x,\xi) \to 1$ as $x \to \pm \infty$ are given by
\begin{align*}
 f_+(x,\xi) = c(\xi) \calD_2^\ast \calD_1^\ast \bigl( e^{ix\xi} \bigr), \quad
 f_-(x,\xi) = c(\xi) \calD_2^\ast \calD_1^\ast \bigl( e^{-ix\xi} \bigr),
\end{align*}
with
\begin{equation*}
 c(\xi) := \frac{1}{2-\xi^2-3i\xi}.
\end{equation*}
From the relation $T(\xi) W(f_+(\cdot,\xi), f_-(\cdot,\xi)) = - 2i \xi$ we infer by direct computation that the transmission coefficient is given by
\begin{equation*}
 T(\xi) = \frac{\xi^2-2 + 3i\xi}{\xi^2-2-3i\xi}, \quad \xi \in \bbR.
\end{equation*}
Note that $|T(\xi)| = 1$ for all $\xi \in \bbR$, and therefore $R_\pm(\xi) = 0$ for all $\xi \in \bbR$.
Recall that the distorted Fourier basis associated with $H$ is given by
\begin{equation*}
 e(x,\xi) := \frac{1}{\sqrt{2\pi}} \left \{ \begin{array}{ll} T(\xi) f_+(x,\xi), &\xi \geq 0, \\
  T(-\xi) f_-(x,-\xi), & \xi<0. \end{array} \right.
\end{equation*}
Hence, we find for $\xi \geq 0$ that
\begin{equation*}
 \begin{aligned}
  \widetilde{\calF}\bigl[3 Q Y_2^2\bigr](\xi) &= \int_\bbR \overline{e(x,\xi)} 3 Q(x) Y_2(x)^2 \, \ud x \\
  &= \overline{T(\xi)} \overline{c(\xi)} \frac{1}{\sqrt{2\pi}} \int_\bbR \overline{ \calD_2^\ast \calD_1^\ast \bigl( e^{ix\xi} \bigr) } 3 Q(x) Y_2(x)^2 \, \ud x \\
  &= \overline{T(\xi)} \overline{c(\xi)} \frac{1}{\sqrt{2\pi}} \int_\bbR e^{-ix\xi} \calD_1 \calD_2 \bigl( 3 Q(x) Y_2(x)^2 \bigr) \, \ud x \\
  &= \overline{T(\xi)} \overline{c(\xi)} \calF\bigl[ \calD_1 \calD_2 \bigl( 3 Q Y_2^2 \bigr)\bigr](\xi),
 \end{aligned}
\end{equation*}
while for $\xi < 0$ we find similarly that
\begin{equation*}
 \begin{aligned}
  \widetilde{\calF}\bigl[3 Q Y_2^2\bigr](\xi) 
  &= \overline{T(-\xi)} \overline{c(-\xi)} \calF\bigl[ \calD_1 \calD_2 \bigl( 3 Q Y_2^2 \bigr)\bigr](\xi).
 \end{aligned}
\end{equation*}
Then by direct computation,
\begin{align*}
  \calD_1 \calD_2 \bigl( 3 Q Y_2^2 \bigr)(x) 
  &= -\frac{3}{4 \sqrt{2}} \bigl( 270 - 288 \cosh^2(x) + 40 \cosh^4(x) \bigr) \sech^7(x) \\
  &= -\frac{3}{32 \sqrt{2}} \Bigl( \bigl(-29 + 23 \px^2 + 9 \px^4 - 3 \px^6 \bigr) \sech(x) \Bigr).
\end{align*}
Using that
\begin{equation*}
 \calF[\sech(\cdot)](\xi) = \sqrt{\frac{\pi}{2}} \sech\Bigl( \frac{\pi}{2} \xi \Bigr), \quad \xi \in \bbR,
\end{equation*}
we conclude
\begin{equation} \label{equ:intro_resonance_condition_calFD1D2}
 \calF\bigl[ \calD_1 \calD_2 \bigl( 3 Q Y_2^2 \bigr) \bigr](\xi) 
 =  - \frac{3 \sqrt{\pi}}{64} \bigl( -29 - 23 \xi^2 + 9 \xi^4 + 3 \xi^6 \bigr) \sech \Bigl( \frac{\pi \xi}{2} \Bigr), \quad \xi \in \bbR.
\end{equation}
In particular,
\begin{equation*}
 \calF\bigl[ \calD_1 \calD_2 \bigl( 3 Q Y_2^2 \bigr) \bigr](\pm \sqrt{3}) = - 3 \sqrt{\pi} \sech \Bigl( \frac{\pi \sqrt{3}}{2} \Bigr).
\end{equation*}
Computing that $T(\sqrt{3}) c(\sqrt{3}) = \frac{1}{28}(-1-3i\sqrt{3})$, we arrive at the assertion~\eqref{equ:intro_resonance_condition}.
\end{proof}

\begin{remark}
We used the Wolfram Mathematica software system for the computation of some identities in the preceding proof of Lemma~\ref{lem:resonance_condition_focusing_cubic_source_term}.
\end{remark}

\subsection{Overview of the proof}

In this subsection we summarize the main steps in the proof of Theorem~\ref{thm:main}.

\subsubsection{Spectral decomposition}

We begin by enacting a spectral decomposition of the even perturbation
$\varphi(t,x) = \phi(t,x) - Q(x)$ of the soliton into
\begin{equation} \label{equ:overview_spectral_decomposition_varphi}
 \varphi(t,x) = (P_c \varphi)(t,x) + a(t) Y_0(x), \quad a(t) := \langle Y_0, \varphi(t) \rangle,
\end{equation}
where $P_c$ denotes the projection to the continuous spectral subspace of $L^2_x$ relative to the linearized operator.
Recall that the odd translational mode $Y_1$ is not relevant for even perturbations of the soliton.
We further decompose the coefficient $a(t)$ into an unstable mode and a stable mode
\begin{equation*}
 a(t) = a_+(t) + a_-(t),
\end{equation*}
where
\begin{equation*}
 a_+ := \frac12 \bigl(a + \nu^{-1} \pt a \bigr), \quad a_- := \frac12 \bigl(a-\nu^{-1} \pt a \bigr).
\end{equation*}
From the perturbation equation~\eqref{equ:intro_perturbation_equation} we obtain a coupled PDE/ODE system for the variables $(P_c \varphi, a_+, a_-)$ given by
\begin{equation} \label{equ:intro_overview_of_proof_pde_ode_system_for_varphi_a}
 \left\{ \begin{aligned}
          (\pt^2 + L) P_c \varphi &= P_c \bigl( 3 Q (P_c \varphi + a Y_0)^2 + (P_c \varphi + a Y_0)^3 \bigr), \\
          (\pt - \nu) a_+ &= (2\nu)^{-1} \bigl\langle Y_0, 3 Q (P_c \varphi + a Y_0)^2 + (P_c \varphi + a Y_0)^3 \bigr\rangle, \\
          (\pt + \nu) a_- &= -(2\nu)^{-1} \bigl\langle Y_0, 3 Q (P_c \varphi + a Y_0)^2 + (P_c \varphi + a Y_0)^3 \bigr\rangle.
         \end{aligned} \right.
\end{equation}

\subsubsection{Iterated Darboux transformations}

The linearized operator $L = -\px^2 - 6 \sech^2(x) + 1$ in \eqref{equ:intro_overview_of_proof_pde_ode_system_for_varphi_a} features the second member $\ell = 2$ in the hierarchy of P\"oschl-Teller potentials.
We can therefore use the conjugation identity \eqref{equ:conjugation_identity} to transform the nonlinear Klein-Gordon equation (with potential) for the dispersive part $(P_c \varphi)(t)$ into a nonlinear Klein-Gordon equation without potential.
To this end we apply the iterated Darboux transformation $\calD_1 \calD_2$ to the equation for $(P_c \varphi)(t)$ in~\eqref{equ:intro_overview_of_proof_pde_ode_system_for_varphi_a} and pass to the new variable
\begin{equation*}
 w := \calD_1 \calD_2 P_c \varphi. 
\end{equation*}
We show in Section~\ref{sec:darboux} that
\begin{equation} \label{equ:intro_overview_Pcvarphi_equal_PcJ}
 P_c \varphi = P_c \calJ[w],
\end{equation}
where $\calJ[w]$ is a right-inverse operator for $\calD_1 \calD_2$ given by
\begin{equation*}
 \calJ[w(t)](x) = \sech^2(x) \int_0^x \cosh(y) \int_0^y \cosh(z) w(t,z) \, \ud z \, \ud y.
\end{equation*}
Moreover, the kernel of the iterated Darboux transformation is spanned by the eigenfunctions $Y_0$ and $Y_1$ of the linearized operator.
We thus arrive at the following coupled PDE/ODE system for the variables $(w, a_+, a_-)$,
\begin{equation} \label{equ:intro_overview_pde_ode_system_w_a}
 \left\{ \begin{aligned}
          (\pt^2 -\px^2 + 1) w &= \calD_1 \calD_2 \bigl( 3 Q (P_c \calJ[w] + a Y_0)^2 \bigr) + \calD_1 \calD_2 \bigl( (P_c \calJ[w] + a Y_0)^3 \bigr), \\
          (\pt - \nu) a_+ &= (2\nu)^{-1} \bigl\langle Y_0, 3 Q (P_c \calJ[w] + a Y_0)^2 + (P_c \calJ[w] + a Y_0)^3 \bigr\rangle, \\
          (\pt + \nu) a_- &= -(2\nu)^{-1} \bigl\langle Y_0, 3 Q (P_c \calJ[w] + a Y_0)^2 + (P_c \calJ[w] + a Y_0)^3 \bigr\rangle.
         \end{aligned} \right.
\end{equation}
It now suffices to derive decay estimates for the variables $(w, a_+, a_-)$ for suitable initial data, which can then be transferred back to the original perturbation $\varphi$ of the soliton via \eqref{equ:intro_overview_Pcvarphi_equal_PcJ} and \eqref{equ:overview_spectral_decomposition_varphi}.

\subsubsection{Structure of the transformed equation}

To analyze the long-time behavior of the variable $w(t)$, it is necessary to unveil the fine structure of the nonlinearities in the nonlinear Klein-Gordon equation for $w$ in \eqref{equ:intro_overview_pde_ode_system_w_a}.
This step is carried out in detail in Section~\ref{sec:transformed_equation}.
We find that the transformed equation for the variable $w(t)$ has the (relatively accurate) schematic form
\begin{equation} \label{equ:intro_overview_schematic_w_equation}
 \begin{aligned}
  (\pt^2 - \px^2 + 1) w &= q(x) w^2 + \beta_0 w^3 + \tanh(x) (\jD^{-1} \px w) w^2 \\
  &\quad \quad + \beta(x) w^3 + \gamma(x) w a + \bigl\{ \text{terms with more decay} \bigr\},
 \end{aligned}
\end{equation}
where $q(x)$ with $\hatq(\pm \sqrt{3}) \neq 0$, $\beta(x)$, and $\gamma(x)$ are Schwartz functions, and where $\beta_0 \in \bbR$.

We will propagate stronger decay estimates $t^{-1} \log^2(2+t)$ for the unstable and stable coefficients $a_+(t)$ and $a_-(t)$ for suitable initial data, while we will only propagate $L^\infty_x$ decay for $w(t)$ at the rate $t^{-\frac12} \log(2+t)$. For this reason, almost all nonlinear terms in the equation for $w(t)$ that involve at least one input $a(t)$ are subsumed into the unspecified terms with better decay in \eqref{equ:intro_overview_schematic_w_equation}.

It turns out that the fine structure of the non-localized cubic terms in $w$ only really becomes transparent on the Fourier side, see Subsection~\ref{subsec:transformed_cubic}. This step appears reminiscent of the analysis of the structure of the nonlinear spectral distribution for cubic terms in a vector field method approach based on the distorted Fourier transform, see for instance \cite[Section~5]{GP20} and \cite[Section~4]{ChenPus22}.

Finally, we pass to a first-order formulation, which is more convenient for the analysis of the long-time behavior of the solutions.
To this end we introduce the variable
\begin{equation*}
 v(t) := \frac12 \bigl( w(t) - i\jD^{-1} \pt w(t) \bigr).
\end{equation*}
From \eqref{equ:intro_overview_schematic_w_equation} we obtain using that $w = v + \bv$ the schematic first-order nonlinear Klein-Gordon equation for $v(t)$ given by
\begin{equation} \label{equ:intro_overview_schematic_v_equation}
 \begin{aligned}
  (\pt - i\jD) v &= (2i\jD)^{-1} \Bigl( q(x) (v+\bv)^2 + \beta_0 (v+\bv)^3 + \tanh(x) (\jD^{-1} \px (v+\bv)) (v+\bv)^2 \\
  &\qquad \qquad \qquad \quad + \beta(x) (v+\bv)^3 + \gamma(x) (v+\bv) a + \bigl\{ \text{terms with more decay} \bigr\} \Bigr),
 \end{aligned}
\end{equation}
which is coupled to the ODEs for the unstable and stable coefficients
\begin{equation*}
 \begin{aligned}
  (\pt - \nu) a_+ &= (2\nu)^{-1} \bigl\langle Y_0, 3 Q (P_c \calJ[v+\bv] + a Y_0)^2 + (P_c \calJ[v+\bv] + a Y_0)^3 \bigr\rangle, \\
  (\pt + \nu) a_- &= -(2\nu)^{-1} \bigl\langle Y_0, 3 Q (P_c \calJ[v+\bv] + a Y_0)^2 + (P_c \calJ[v+\bv] + a Y_0)^3 \bigr\rangle.
 \end{aligned}
\end{equation*}
We also define the profile of the solution $v(t)$ to \eqref{equ:intro_overview_schematic_v_equation} by
\begin{equation*}
 f(t) := e^{-it\jD} v(t).
\end{equation*}

\subsubsection{Adapted functional framework and weighted energy estimates}

At this point we are in the position to set up a bootstrap argument together with a topological shooting argument to prove decay for the variables $(v, a_+, a_-)$, see Section~\ref{sec:bootstrap_setup}.
In order to take into account the degeneracy described in Subsection~\ref{subsubsec:main_difficulties_localized_quadratic}  around the problematic frequencies $\pm \sqrt{3}$ of the schematic quadratic nonlinearity $q(x) (v+\bv)^2$ in \eqref{equ:intro_overview_schematic_v_equation}, we implement a version of an adapted functional framework introduced by Germain-Pusateri~\cite{GP20}.
It is reflected in the following dispersive decay estimate
\begin{equation} \label{equ:intro_overview_dispersive_decay_est}
\begin{aligned}
 &\bigl\| e^{it\jD} f(t) \bigr\|_{L^\infty_x} \\
 &\qquad \lesssim \frac{\log(2+t)}{\jt^{\frac12}} \biggl( \bigl\| \jD^2 f(t) \bigr\|_{L^2_x} + \sup_{n \geq 1} \, \sup_{0 \leq \ell \leq n} \, 2^{-\frac12 \ell} \tau_n(t) \bigl\| \varphi_\ell^{(n)}(\xi) \jxi^2 \partial_\xi \hatf(t, \xi)\bigr\|_{L^2_\xi} \biggr),
\end{aligned}
\end{equation}
which we establish in Proposition~\ref{prop:dispersive_decay_estimate} and Lemma~\ref{lem:Linfty_decay_v}.
Here, $\{ \tau_n(t) \}_{n=1}^\infty$ is a smooth partition of unity of the positive time axis $[0,\infty)$ with $\tau_n(t)$ supported around $t \simeq 2^n$ for $n \geq 2$, and where $\varphi_\ell^{(n)}(\xi)$ for $0 \leq \ell \leq n$ are smooth cut-offs to small frequency annuli around the problematic frequencies $\pm \sqrt{3}$, smoothly localizing to the regions $||\xi| - \sqrt{3}| \lesssim 2^{-n}$ (for $\ell=n$), $||\xi|-\sqrt{3}| \simeq 2^{-\ell}$ for $1 \leq \ell \leq n-1$, and $||\xi|-\sqrt{3}| \gtrsim 1$ for $\ell = 0$.

The decay estimate~\eqref{equ:intro_overview_dispersive_decay_est} allows us to propagate $t^{-\frac12} \log(2+t)$ decay in $L^\infty_x$ for the solution $v(t)$ in our bootstrap argument as long as the energy norms on the right-hand side of~\eqref{equ:intro_overview_dispersive_decay_est} remain uniformly bounded in time.
We propagate the stronger $t^{-1} \log^2(2+t)$ decay for the unstable and the stable coefficients $a_+(t)$ and $a_-(t)$ for appropriate initial data.

The majority of the paper consists in establishing logarithmically growing bounds for the contributions to the second weighted energy term on the right-hand side of~\eqref{equ:intro_overview_dispersive_decay_est} of the quadratic terms of the schematic type $q(x) (v+\bv)^2$ and of the non-localized cubic terms of the schematic types $\beta_0 (v+\bv)^3$ and $\tanh(x) (\jD^{-1} \px (v+\bv)) (v+\bv)^2$.
Up to exponential time scales the logarithmic growth of the contributions of these nonlinear terms can be compensated by their additional smallness, thus closing the bootstrap.

The restriction to times $0 \leq t \leq \exp(c \varepsilon^{-\frac13})$ in the statement of Theorem~\ref{thm:main} is a consequence of several growth bounds on the energy norms on the right-hand side of \eqref{equ:intro_overview_dispersive_decay_est}, see \eqref{equ:proof_of_boostrap_bounds_NT_bound}. The simple growth bound $\|\jD^2 f(t)\|_{L^2_x} \lesssim \varepsilon + (\log(2+t))^3 \varepsilon^2$ for the $H^2_x$ norm of the profile established in Proposition~\ref{prop:H2_energy_estimate} already enforces the restriction to times $0 \leq t \leq \exp(c \varepsilon^{-\frac13})$ to close the bootstrap.

We comment on several aspects of the derivation of the weighted energy bounds.
\begin{itemize}
 \item[1.] We establish stronger weighted energy estimates for all spatially localized terms with at least cubic-type decay $t^{-\frac32}$ (up to logarithmic factors) using a streamlined version \cite[Proposition 4.9]{LS1} of an argument introduced in \cite{LLS1} based on exploiting improved local decay estimates, see the proof of Proposition~\ref{prop:prop49}.
 This allows us to efficiently reduce the derivation of the weighted energy estimates to the analysis of the contributions of the quadratic terms and of the non-localized cubic terms.

 \item[2.] In the derivation of the weighted energy estimates for the quadratic terms of the form $q(x) (v+\bv)^2$ in Section~\ref{sec:weighted_main_quadratic}, we exploit improved local decay estimates to peel off better behaved parts and to reduce to dealing with the most problematic heuristic source term $q(x) e^{i2t} t^{-1} \varepsilon^2$.
 At the heart of the treatment of the latter is a double integration by parts in time argument, see the proof of Proposition~\ref{prop:weighted_energy_est_calB1_bad}.

 \item[3.] The derivation of the weighted energy estimates for the non-localized cubic terms in Section~\ref{sec:weighted_main_cubic} in parts parallels the arguments in \cite[Section 9]{GP20} and in \cite[Section 11.4]{GP20}.
 We emphasize that the low-frequency improvement $(\jD^{-1} \px (v+\bv))$ of at least one input of the schematic cubic terms $\tanh(x) (\jD^{-1} \px (v+\bv)) (v+\bv)^2$ gives access to improved local decay for that input.
 The latter is crucial to establish acceptable weighted energy bounds for those cubic terms, see Step~3 of the proof of Proposition~\ref{prop:weighted_energy_est_pv_T2}.
 A related observation about the improved structure of similar cubic terms was made in \cite[Theorem 4.1]{ChenPus22}.
\end{itemize}

\subsubsection{Shooting argument and conclusion of the proof of Theorem~\ref{thm:main}}

In Section~\ref{sec:conclusion_of_proof} we conclude the proof of Theorem~\ref{thm:main} by using a standard topological shooting argument as in \cite[Lemma 6]{CMM11} to select initial data, which avoids the exponential growth of the unstable coefficient $a_+(t)$ and ensures its decay.

\subsection{Further remarks}

We end the introduction with a few more comments.

\begin{itemize}
 \item[1.] One major difficulty for several asymptotic stability problems (under symmetric perturbations) for solitons in 1D nonlinear Klein-Gordon type equations is the emergence of singular, spatially localized quadratic source terms in the perturbation equations.
 This phenomenon was discussed in detail in Subsection~\ref{subsubsec:main_difficulties_localized_quadratic} for the case of even perturbations of the soliton of the focusing cubic Klein-Gordon equation~\eqref{equ:focusing_cubic_KG}, and how the limitation to times $0 \leq t \leq \exp(c \varepsilon^{-\frac13})$ in Theorem~\ref{thm:main} is related to it.
 Here we attempt to provide a brief overview of the types of singular, spatially localized quadratic source terms that arise in several classical problems. This attempt is at the risk of over-simplifying or slightly misrepresenting some settings, which is not the intention of the authors.

 We recall from Subsection~\ref{subsubsec:main_difficulties_localized_quadratic} that the contribution of the spatially localized quadratic source term to the profile of small perturbations of the soliton of the focusing cubic Klein-Gordon equation~\eqref{equ:focusing_cubic_KG} is of the schematic type
 \begin{equation} \label{equ:intro_source_term_rem_KG3}
  \int_1^t \Bigl( e^{i s \sqrt{-\px^2+V+1}} \bigl( 3 Q Y_2^2 \bigr) \Bigr) e^{i2s} \frac{\varepsilon^2}{s} \, \ud s,
 \end{equation}
 where $V(x) = - 6 \sech^2(x)$, $3 Q(x)$ is the quadratic coefficient in the perturbation equation~\eqref{equ:intro_perturbation_equation}, and $Y_2(x)$ is the threshold resonance of the linearized operator defined in \eqref{equ:linearized_operator}.

 For odd perturbations of the sine-Gordon kink, one faces the same type of quadratic source term at first. But thanks to the null structure in the sine-Gordon case \cite{LLSS, LS1}, one can use a variable coefficient quadratic normal form to turn it into a localized quadratic source term of the schematic type
 \begin{equation}
  \int_1^t \Bigl( e^{is \sqrt{-\px^2+V+1}} q \Bigr) e^{i2s} \frac{\varepsilon^2}{s^{\frac32-\delta}} \, \ud s.
 \end{equation}
 Here, $V(x) = -2\sech^2(x)$, $q(x)$ is a smooth localized coefficient, and the parameter $0 < \delta \ll 1$ is related to quantifying slow growth estimates for suitable weighted energies of the profiles of the perturbations in that setting.

 Germain-Pusateri~\cite{GP20} consider general 1D nonlinear Klein-Gordon equations of the form
 \begin{equation} \label{equ:intro_source_term_rem_GP}
  \bigl( \pt^2 - \px^2 + V(x) + 1 \bigr) u = \alpha(x) u^2 + \beta_0 u^3,
 \end{equation}
 where $V(x)$ is a smooth, localized potential with no bound state, but possibly exhibiting a zero-energy resonance, and where the smooth coefficient $\alpha(x)$ may not be spatially localized, but has finite limits at spatial infinity $\lim_{x \to \pm \infty} \alpha(x) = \ell_\pm \in \bbR$. Among several results, \cite[Theorem 1.1]{GP20} establishes sharp decay estimates in $L^\infty_x$ at the rate $t^{-\frac12}$ and asymptotics for small solutions to \eqref{equ:intro_source_term_rem_GP} under the assumption that the distorted Fourier transform of the solution $u(t)$ vanishes at zero frequency for all times. The latter assumption gives rise to improved local decay of the solutions of the type $\|\jx^{-1} u(t)\|_{L^2_x} \lesssim \jt^{-\frac34 + \delta} \varepsilon$, $0 < \delta \ll 1$. Then \emph{after} performing a normal form transformation to deal with the non-vanishing ends of the coefficient $\alpha(x)$ at spatial infinity, one is left with a quadratic source term of the schematic type
 \begin{equation}
  \int_1^t \Bigl( e^{is \sqrt{-\px^2+V+1}} q \Bigr) e^{i2s} \frac{\varepsilon^2}{s^{\frac32-2\delta}} \, \ud s,
 \end{equation}
 with $q(x)$ spatially localized.
 It appears that the adapted functional framework introduced in \cite{GP20} could handle contributions of such singular quadratic source terms down to $q(x) e^{i2s} \frac{\varepsilon^2}{s^{1+\mu}}$, $\mu > 0$, and obtain sharp $L^\infty_x$ decay at the rate $t^{-\frac12}$ and asymptotics.

 For odd perturbations of the kink in double sine-Gordon theories (in an appropriate range of the deformation parameter), see for instance \cite{GP20, GermPusZhang22}, the associated linearized operator does not exhibit a threshold resonance, and one can propagate improved local decay of the solutions of type $\|\jx^{-1} u(t)\|_{L^2_x} \lesssim \jt^{-1+\delta} \varepsilon$, $0 < \delta \ll 1$. Correspondingly, a quadratic source term of the following schematic form arises
 \begin{equation}
  \int_1^t \Bigl( e^{is \sqrt{-\px^2+V+1}} q \Bigr) e^{i2s} \frac{\varepsilon^2}{s^{2-2\delta}} \, \ud s.
 \end{equation}

 Finally, in the case of odd perturbations of the kink of the $\phi^4$ model, a localized quadratic source term arises from the feedback of the slowly decaying internal mode into the nonlinear Klein-Gordon equation for the dispersive part of the perturbations.
 Its contribution to the profile is of the schematic form
 \begin{equation} \label{equ:intro_source_term_rem_phi4}
  \int_1^t \Bigl( e^{is \sqrt{-\px^2+V+2}} \bigl( H Y^2 \bigr) \Bigr) e^{i2 \mu s} \frac{\varepsilon^2}{1+\varepsilon^2 s} \, \ud s
 \end{equation}
 with $V(x) = -3\sech^2(\frac{x}{\sqrt{2}})$, $H(x) = \tanh(\frac{x}{\sqrt{2}})$, and $Y(x)$ the eigenfunction with eigenvalue~$\mu^2$, $\mu = \sqrt{\frac{3}{2}}$, of the internal mode of the $\phi^4$ model, see Delort-Masmoudi~\cite{DelMas20} and also L\'eger-Pusateri~\cite{LegPus21, LegPus22}.
 Observe that in contrast to all of the preceding quadratic source terms, the source term in~\eqref{equ:intro_source_term_rem_phi4} additionally loses smallness for large times.

 For related discussions we point the reader to the recent works of Delort-Masmoudi~\cite{DelMas20}, Germain-Pusateri~\cite{GP20}, Chen~\cite{Chen21}, and L\'eger-Pusateri~\cite{LegPus21, LegPus22}.

 \item[2.] It would be very interesting to try to establish a codimension one local asymptotic stability result for the soliton of the focusing cubic Klein-Gordon equation~\eqref{equ:focusing_cubic_KG} in the spirit of the series of works by Kowalczyk-Martel-Mu\~{n}oz~\cite{KMM17, KMM17_short, KMM19}, Kowalczyk-Martel-Mu\~{n}oz-Van den Bosch~\cite{KMMV20}, and Kowalczyk-Martel~\cite{KM22}.
 A key part of the beautiful approach in these works is to establish integrated local energy decay for the dispersive part of the perturbations of the respective solitons.
 However, the contribution of the threshold resonance in the local decay estimate~\eqref{equ:intro_main_difficulties_local_decay_subtract_off} logarithmically fails to be $L^2_t$-integrable.
 For this reason it is unclear to the authors how such a (global-in-time) local asymptotic stability result could be achieved in this spirit in the case of the focusing cubic Klein-Gordon equation, where the effect of the threshold resonance of the linearized operator cannot be avoided through parity restrictions on the perturbations.
 We point out though that it would be very interesting to try to establish such a local asymptotic stability result up to say exponential time scales building on the framework from \cite{KMM17, KMM17_short, KMM19, KMMV20, KM22}. Such a result was obtained by Palacios-Pusateri~\cite{PalaciosPusateri24} after completion of this work.

 \item[3.] It is also possible to approach the problem of proving long-time dispersive decay estimates such as~\eqref{equ:thm_asserted_decay} in Theorem~\ref{thm:main} using a more micro-local adapted functional framework in the spirit of the designer norms in Deng-Ionescu-Pausader~\cite{DIP17} and Deng-Ionescu-Pausader-Pusateri~\cite{DIPP17}. Instead of controlling the right-hand side of \eqref{equ:intro_overview_dispersive_decay_est}, one can for instance seek to obtain uniform-in-time bounds on the following weighted energy for the profile
 \begin{equation} \label{equ:further_remarks_Qjk}
  \sup_{\substack{j+k \geq 0 \\ j \geq 0}} \, \sup_{0 \leq \ell \leq j} \, 2^{2k} 2^j 2^{-(\frac12-\delta) \ell} \bigl\| \varphi_\ell^{(j)}(D) Q_{jk} f(t) \bigr\|_{L^2_x}, \quad 0 < \delta \ll 1,
 \end{equation}
 where $Q_{jk} = \varphi_j P_k$ with $\varphi_j(x)$ smooth cutoffs to the spatial regions $|x| \simeq 2^j$ and $P_k$ the usual Littlewood-Paley projections. Moreover, $\varphi_\ell^{(j)}(D)$ are smooth frequency cut-offs to small annuli around the problematic frequencies $\pm \sqrt{3}$, localizing to the regions $||\xi| - \sqrt{3}| \lesssim 2^{-j}$ for $\ell=j$, $||\xi|-\sqrt{3}| \simeq 2^{-\ell}$ for $1 \leq \ell \leq j-1$, and $||\xi|-\sqrt{3}| \gtrsim 1$ for $\ell = 0$.
 Working with the adapted weighted energies \eqref{equ:further_remarks_Qjk} instead of the right-hand side of \eqref{equ:intro_overview_dispersive_decay_est}, one can obtain long-time $L^\infty_x$ decay estimates $C_M \varepsilon t^{-\frac12}$ up to times $0 \leq t \leq \varepsilon^{-M}$ for arbitrary $M \in \bbN$ in place of \eqref{equ:thm_asserted_decay}.

 \item[4.]
 Thanks to the conjugation identity~\eqref{equ:conjugation_identity}, in this work we can use vector field techniques for flat Klein-Gordon equations to study the long-time behavior of perturbations of the soliton of the focusing cubic Klein-Gordon equation.
 This super-symmetric approach only applies to linearized operators with P\"oschl-Teller potential $-\ell(\ell+1) \sech^2(x)$, $\ell \in \bbN$. However, the latter arise in several classical asymptotic stability problems, namely for the kinks of the sine-Gordon model ($\ell=1$) and of the $\phi^4$ model ($\ell=2$), as well as for the solitons of the focusing cubic ($\ell=2$) and of the focusing quadratic ($\ell=3$) Klein-Gordon equations.

 We emphasize that beyond the intrinsic interest in the codimension one full asymptotic stability problem for the soliton of the focusing cubic Klein-Gordon equation under even perturbations, its resolution is also relevant for the study of odd perturbations of the kink of the $\phi^4$ model in view of the close resemblance of the singular quadratic source terms \eqref{equ:intro_source_term_rem_KG3} and \eqref{equ:intro_source_term_rem_phi4} in the respective perturbation equations.

 \item[5.]
 In order to study the full asymptotic stability of the soliton of the focusing cubic Klein-Gordon equation \eqref{equ:focusing_cubic_KG} for non-symmetric perturbations one has to use modulation to take into account that the soliton may start to move due to translation invariance. In this more general case one has to analyze the perturbation equation for the dispersive part, featuring a non self-adjoint matrix linearized operator, coupled to a system of first-order differential equations for the modulation parameters. This leads to several additional challenges beyond the techniques used in this paper. We refer to the recent works \cite{Chen21, CollotGermain23} that incorporate modulation techniques in the context of capturing modified scattering in related settings.

\end{itemize}

\section{Preliminaries} \label{sec:preliminaries}

\subsection{Notation}

We denote by $C > 0$ an absolute constant whose value may change from line to line.
For non-negative $X,Y$ we use the notation $X \lesssim Y$ if $X \leq C Y$, and we write $X \ll Y$ to indicate that the implicit constant should be regarded as small.
Moreover, for non-negative $X$ and arbitrary $Y$, we use the short-hand notation $Y = \calO(X)$ if $|Y| \leq C X$.
Throughout, we use the Japanese bracket notation
\begin{equation*}
 \jt := (1 + t^2)^{\frac12}, \quad \jx := (1+x^2)^{\frac12}, \quad \jxi := (1+\xi^2)^{\frac12}.
\end{equation*}
We write $D = -i\px$ and we use the standard notations for the Lebesgue spaces $L^p_x$ as well as for the Sobolev spaces $H^k_x$ and $W^{k,p}_x$. In what follows, it will be useful to have a short-hand notation for the function $\tanh(x)$. We set
\begin{equation*}
 K(x) := \tanh(x).
\end{equation*}

Our conventions for the Fourier transform of a Schwartz function on $\bbR^n$ are
\begin{equation} \label{eq:FT}
 \begin{aligned}
  \calF[f](\xi) &= \hatf(\xi) = \frac{1}{(2\pi)^{\frac{n}{2}}} \int_{\bbR^n} e^{-ix\cdot\xi} f(x) \, \ud x, \\
  \calF^{-1}[f](x) &= \check{f}(x) = \frac{1}{(2\pi)^{\frac{n}{2}}}  \int_{\bbR^n}  e^{ix\cdot\xi} f(\xi) \, \ud \xi.
 \end{aligned}
\end{equation}
Then the convolution laws are given by
\begin{equation*}
 \calF[f \ast g] = (2\pi)^{\frac{n}{2}} \hatf \hatg, \quad \calF[f g] = \frac{1}{(2\pi)^{\frac{n}{2}}} \hatf \ast \hatg
\end{equation*}
for $f, g \in \calS(\bbR^n)$. Usually, $\calF$ and $\calF^{-1}$ will refer to the Fourier transform on $\bbR$, but in the context of applying Lemma~\ref{lem:frakm_for_delta_three_inputs} or Lemma~\ref{lem:frakn_for_pv} we will use the same notation for the Fourier transform on~$\bbR^n$ with $n = 3$ or $n = 4$.

We denote by $P_c$ the spectral projection to the continuous subspace of $L^2_x$ relative to the linearized operator $\linop$ defined in~\eqref{equ:linearized_operator},
\begin{equation*}
 P_c f = f - \langle Y_0, f \rangle Y_0 - \langle Y_1, f \rangle Y_1.
\end{equation*}
Note that for even functions we have $P_c f = f - \langle Y_0, f \rangle Y_0$.

\subsection{Projection operators} \label{subsec:projection_operators}

Let $\varphi \in C_c^\infty(\bbR)$ be a smooth even non-negative bump function such that $\varphi(x) = 1$ for $|x| \leq 1$ and $\varphi(x) = 0$ for $|x| \geq 2$. 
Set $\psi(x) := \varphi(x) - \varphi(2x).$
We define
\begin{equation*}
 \varphi_k(x) := \psi\Bigl(\frac{x}{2^k}\Bigr), \quad k \in \bbZ,
\end{equation*}
and 
\begin{equation*}
 \varphi_{\leq k}(x) := \varphi\Bigl(\frac{x}{2^k}\Bigr), \quad \varphi_{\geq k}(x) := 1 - \varphi\Bigl(\frac{x}{2^k}\Bigr), \quad k \in \bbZ.
\end{equation*}
For integers $k_1 < k_2$ we set
\begin{equation*}
 \varphi_{[k_1,k_2]}(x) := \varphi_{\leq k_2}(x) - \varphi_{\leq k_1}(x).
\end{equation*}
We denote by $P_k$, $k \in \bbZ$, the usual Littlewood-Paley projection operator defined by
\begin{equation*}
 \widehat{P_k f}(\xi) := \varphi_k(\xi) \hatf(\xi).
\end{equation*}
Moreover, we define frequency cut-offs that localize to small annuli close to the problematic frequencies $\pm \sqrt{3}$. For any integer $n \geq 1$, we set
\begin{equation*}
 \begin{aligned}
  \varphi_0^{(n)}(\xi) &:= 1 - \varphi_{\leq -1}\bigl( 2^{100} ( |\xi| - \sqrt{3} ) \bigr), \\
  \varphi_\ell^{(n)}(\xi) &:= \varphi_{-\ell}\bigl( 2^{100} ( |\xi| - \sqrt{3} ) \bigr), \quad 1 \leq \ell \leq n-1, \\
  \varphi_n^{(n)}(\xi) &:= \varphi_{\leq -n}\bigl( 2^{100} ( |\xi| - \sqrt{3} ) \bigr).
 \end{aligned}
\end{equation*}
Then we have for any $n \geq 1$ and every $\xi \in \bbR$ that
\begin{equation*}
 \sum_{0 \leq \ell \leq n} \varphi_\ell^{(n)}(\xi) = 1.
\end{equation*}
We also set for any integer $n \geq 1$ and any $1 \leq \ell \leq n$,
\begin{equation*}
 \varphi_{\geq \ell}^{(n)}(\xi) := \varphi_{\leq -\ell}\bigl( 2^{100} \big| |\xi| - \sqrt{3} \big| \bigr).
\end{equation*}
Sometimes, it will be necessary to localize around just one of the two bad frequencies $\pm \sqrt{3}$. To this end, we introduce the notation 
\begin{equation*}
 \begin{aligned}
  \varphi_0^{(n),\pm}(\xi) &:= 1 - \varphi_{\leq -1}\bigl( 2^{100} (\xi \mp \sqrt{3}) \bigr), \\
  \varphi_\ell^{(n),\pm}(\xi) &:= \varphi_{-\ell}\bigl( 2^{100} (\xi \mp \sqrt{3}) \bigr), \quad 1 \leq \ell \leq n-1, \\
  \varphi_n^{(n),\pm}(\xi) &:= \varphi_{\leq -n}\bigl( 2^{100} ( \xi \mp \sqrt{3} ) \bigr), \\
  \varphi_{\geq \ell}^{(n),\pm}(\xi) &:= \varphi_{\leq -\ell}\bigl( 2^{100} ( \xi \mp \sqrt{3} ) \bigr).
 \end{aligned}
\end{equation*}
Occasionally, we will use slight fattenings of the cut-offs defined above. For instance, in the case of the cut-off $\varphi_\ell^{(n)}(\xi)$, we will use the notation $\widetilde{\varphi}_\ell^{(n)}(\xi)$ for a fattened cut-off with the standard property that $\widetilde{\varphi}_\ell^{(n)}(\xi) \varphi_\ell^{(n)}(\xi) = \varphi_\ell^{(n)}(\xi)$ for all $\xi \in \bbR$.

Finally, we define a smooth partition of unity for the positive time axis $[0, \infty)$. We set 
\begin{equation*}
 \tau_1(t) := \varphi\Bigl(\frac{t}{2}\Bigr), \quad t \geq 0,
\end{equation*}
and for every integer $n \geq 2$, we define 
\begin{equation*}
 \tau_n(t) := \varphi\Bigl(\frac{t}{2^n}\Bigr) - \varphi\Bigl( \frac{t}{2^{n-1}} \Bigr), \quad t \geq 0.
\end{equation*}
Then we have for all $t \geq 0$,
\begin{equation*}
 \sum_{n=1}^\infty \tau_n(t) = 1.
\end{equation*}

\subsection{Decay estimates for the linear Klein-Gordon evolution}

In this subsection we establish decay estimates for the linear Klein-Gordon evolution in terms of weighted energies that are adapted to the problematic frequencies $\pm \sqrt{3}$.
The latter enter the definition of the adapted functional framework in~\eqref{equ:definition_NT_norm}.
We begin with a dispersive decay estimate.

\begin{proposition} \label{prop:dispersive_decay_estimate}
There exists an absolute constant $C \geq 1$ such that for any integer $n \geq 1$, we have for all $t \geq 0$ that
\begin{equation} \label{eq:5p1}
  \bigl\| e^{it\jD} f \bigr\|_{L^\infty_x} \leq C \frac{n}{\jt^{\frac12}} \Bigl( \|\jD^2 f\|_{L^2_x} + \sup_{0 \leq \ell \leq n} \, 2^{- \frac12 \ell} \bigl\| \varphi_\ell^{(n)}(\xi) \jxi^2 \partial_\xi \hatf(\xi)\bigr\|_{L^2_\xi} \Bigr).
\end{equation}
\end{proposition}

Note that in the statement of Proposition~\ref{prop:dispersive_decay_estimate} the size of $t \geq 0$ and the size of $n \geq 1$ are independent of each other. Their sizes will be coupled in the definition of the adapted functional framework in~\eqref{equ:definition_NT_norm} via the smooth partition of unity $\{\tau_n\}_{n \geq 1}$ of the positive time axis.

\begin{proof}[Proof of Proposition~\ref{prop:dispersive_decay_estimate}]
We have
\EQ{ \label{eq:osc intrep}
\big( e^{i t \jD} f \big)(x) 
&   = \frac{1}{\sqrt{2\pi}} \int_\bbR e^{it\phi(\xi,u)} \hat{f}(\xi)\, \ud\xi 
}
with phase
\begin{equation*}
 \phi(\xi,u) := \jap{\xi} + u\xi, \quad u := \frac{x}{t}. 
\end{equation*}
Since 
\begin{equation*}
 \bigl\| e^{i t \jD} f \bigr\|_{L^\infty_x} \lesssim \|\hat{f}\|_{L^1_\xi} \lesssim \|\jD^2 f\|_{L^2_x},
\end{equation*}
it suffices to assume that $t \ge 100$. 
Note that
\begin{equation*}
 \pxi \phi(\xi, u) = \xi \jxi^{-1} + u, \quad \pxi^2 \phi(\xi, u) = \jxi^{-3}.
\end{equation*}
If $|x|\ge t\ge1$, then we have 
\EQ{ \label{eq:phase der}
|\partial_\xi \phi  (\xi,u )| &= |\xi\jap{\xi}^{-1} + u | \ge 1 - |\xi|\jap{\xi}^{-1}\ge  \jap{\xi}^{-2}/2.
}
Introduce a smooth partition of unity $1 = \chi_1(\xi^2/t) + \chi_0(\xi^2/t)$ in~\eqref{eq:osc intrep}, with $\chi_0(\cdot)$ being a smooth cutoff to the interval $[-1,1]$. Then integration by parts and~\eqref{eq:phase der} imply 
\begin{equation} \label{eq:E+}
\begin{aligned}
 \bigl| \big( e^{i t \jD} f \big)(x) \bigr| &\les \int_\bbR \chi_1(\xi^2/t) |\hat{f}(\xi)|\, \ud \xi + t^{-1} \int_\bbR \frac{|\partial_\xi^2\phi (\xi,u)|}{ |\partial_\xi \phi  (\xi,u )|^2}  \chi_0(\xi^2/t)   |\hat{f}(\xi)| \, \ud \xi  \\
 &\quad + t^{-1} \sum_{0 \leq \ell \leq n} \int_\bbR \varphi_\ell^{(n)}(\xi) |\partial_\xi \phi  (\xi,u )|^{-1}  \big| \partial_\xi \big( \chi_0(\xi^2/t) \hat{f}(\xi)\big) \big| \, \ud \xi  \\
 &\les t^{-\frac34} \cdot n \cdot \Bigl( \| \jap{\xi}^2 \hat{f}(\xi)\|_{L^2_\xi} + \sup_{0 \leq \ell \leq n} \, 2^{-\frac12 \ell} \bigl\| \varphi_\ell^{(n)}(\xi) \jxi^2 \partial_\xi \hatf(\xi)\bigr\|_{L^2_\xi} \Bigr). 
\end{aligned}
\end{equation}
If $|x|< t$, then $\phi (\xi,u)$ has a unique critical point at
\begin{equation*}
 \xi_0=-\jap{\xi_0} u, \quad \text{or equivalently} \quad  \xi_0=-\frac{u}{\sqrt{1-u^2}}.
\end{equation*}
One has $\phi (\xi_0, u)=\sqrt{1-u^2}$, which implies $t \phi (\xi_0, u)=\sqrt{t^2-x^2}=\rho$. 
It was shown in the proof of~\cite[Lemma 2.1]{LS1} that~\eqref{eq:phase der} holds for all $\xi\in\R\setminus I(\xi_0)$, where 
\[
I(\xi_0) := \bigl[ \xi_0 - \jap{\xi_0}/100, \xi_0 + \jap{\xi_0}/100 \bigr].
\] 
Hence, proceeding as above yields the same bound for 
\[
 e^{i t \jD} (1- \omega_u(D)) f,
\]
where $$\omega_u(\xi):= \chi\big(C_0(\xi-\xi_0)\jap{\xi_0}^{-1}\big)$$ with some large constant $C_0 \gg 1$.  
It remains to establish the lemma for the term 
\[
 \bigl( e^{i t \jD} \omega_u(D) f \bigr)(x) =  \frac{1}{\sqrt{2\pi}} \int_\bbR  e^{it\phi(\xi,u)}\omega_u(\xi) \hat{f}(\xi)\, \ud\xi.  
\]
On the support of $\omega_u(\xi)$, 
\[
\pxi^2 \phi(\xi, u) \simeq \jap{\xi_0}^{-3} ,\qquad  |\partial_\xi\phi(\xi,u)|\simeq |\xi-\xi_0|\jap{\xi_0}^{-3}.
\]
Let $M := t^{-\frac12} \jap{\xi_0}^{\frac32}$.
Introducing the partition of unity $1=\chi_{[|\xi-\xi_0|< M]} + \chi_{[|\xi-\xi_0|> M]}$ one has
\begin{equation*}
 \begin{aligned}
  &\bigl|(e^{i t \jD} \omega_u(D) f )(x)\bigr| \\
  &\lesssim M \bigl\| \omega_u \hat{f}\bigr\|_{L^\infty_\xi} + t^{-1} \int_\bbR \bigl| \partial_\xi\big[\partial_\xi\phi(\xi,u)^{-1}\chi_{[|\xi-\xi_0|>M]} \omega_u(\xi) \hat{f}(\xi)\big] \bigr| \, \ud \xi \\
  &\les t^{-\frac12} \bigl\| \jxi^{\frac32} \omega_u \hat{f} \bigr\|_{L^\infty_\xi} + t^{-1} \int_{\{|\xi-\xi_0|> M\}} \bigl|(\partial_\xi^2\phi)(\xi,u) \bigl( \partial_\xi \phi(\xi,u) \bigr)^{-2} \omega_u(\xi) \hat{f}(\xi)\big] \bigr| \, \ud\xi \\
  &\quad + (t M)^{-1} \int_{\{|\xi-\xi_0|\simeq M\}} \bigl| \partial_\xi\phi(\xi,u)^{-1} \omega_u(\xi) \hat{f}(\xi)\bigr| \, \ud\xi + t^{-1} \int_{\{|\xi-\xi_0|> M\}}  \bigl| \partial_\xi \phi(\xi,u)^{-1} \partial_\xi \bigl[ \omega_u(\xi) \hat{f}(\xi)\bigr] \bigr| \, \ud\xi.
 \end{aligned}
\end{equation*}
Then we estimate 
\begin{equation*}
 \begin{aligned}
  t^{-1} \int_{\{|\xi-\xi_0|> M\}}  \bigl| \partial_\xi^2\phi(\xi,u) \partial_\xi\phi(\xi,u)^{-2} \omega_u(\xi) \hat{f}(\xi)\bigr| \, \ud\xi  &\lesssim t^{-1} \jap{\xi_0}^3 \bigl\| \omega_u \hat{f} \bigr\|_{L^\infty_\xi} \int_{\{|\xi-\xi_0|> M\}}  (\xi-\xi_0)^{-2}\,\ud\xi \\
&\les t^{-\frac12} \bigl\| \jxi^{\frac32}\omega_u\hat{f} \bigr\|_{L^\infty_\xi},
 \end{aligned}
\end{equation*}
and, 
\begin{equation*}
 \begin{aligned}
  (tM)^{-1} \int_{\{|\xi-\xi_0|\simeq M\}}  \bigl| \partial_\xi \phi(\xi,u)^{-1} \omega_u(\xi) \hat{f}(\xi)\bigr| \, \ud\xi  \les (t M)^{-1} \jap{\xi_0}^{3} \bigl\| \omega_u \hat{f} \bigr\|_{L^\infty_\xi} \les t^{-\frac12} \bigl\|\jxi^{\frac32} \omega_u \hat{f} \bigr\|_{L^\infty_\xi},
 \end{aligned}
\end{equation*}
and, finally, 
\begin{equation*}
 \begin{aligned}
  & t^{-1} \int_{\{ |\xi-\xi_0|> M \}} \bigl| \partial_\xi\phi(\xi,u)^{-1} \partial_\xi \bigl[\omega_u(\xi) \hat{f}(\xi)\big]\bigr| \, \ud\xi \les t^{-\frac12} \int_\bbR \jxi^{\frac32} \bigl| \partial_\xi \bigl[ \omega_u(\xi) \hat{f}(\xi)\big]\bigr| \, \ud\xi.
 \end{aligned}
\end{equation*}
Now by Sobolev embedding
\begin{equation*}
 \begin{aligned}
  \bigl\| \jxi^{\frac32} \omega_u \hat{f} \bigr\|_{L^\infty_\xi} &\lesssim \bigl\| \jxi^{\frac32} \omega_u \hat{f} \bigr\|_{L^1_\xi} + \bigl\| \pxi \bigl( \jxi^{\frac32}\omega_u \hat{f} \bigr) \|_{L^1_\xi} \\
  &\lesssim \bigl\|\jxi^{2} \hat{f}\bigr\|_{L^2_\xi} + \sum_{\ell=0}^n \, \bigl\| \jxi^{\frac32}\omega_u \varphi_\ell^{(n)}(\xi) \partial_\xi \hat{f} \bigr\|_{L^1_\xi} \\
  &\lesssim n \cdot \Bigl( \bigl\|\jxi^{2} \hat{f}\bigr\|_{L^2_\xi} + \sup_{0 \leq \ell \leq n} 2^{-\frac{1}{2}\ell} \bigl\| \varphi_\ell^{(n)}(\xi) \jxi^{2} \partial_\xi \hat{f}\|_{L^2_\xi} \Bigr),
 \end{aligned}
\end{equation*}
as desired.
\end{proof}

We will also need a dispersive decay estimate with a gain for high frequencies.

\begin{lemma} \label{lem:decwithgain}
 There exists an absolute constant $C \geq 1$ such that for any integer $n \geq 1$, any integer $k \geq 0$, and any $t \geq 1$,
 \begin{equation}
  \bigl\| e^{it\jD} P_k f \bigr\|_{L^\infty_x} \leq C \frac{2^{-\frac12 k} \cdot n}{t^{\frac12}} \Bigl( \|\jD^2 f\|_{L^2_x} + \sup_{0 \leq \ell \leq n} \, 2^{-\frac12 \ell} \bigl\| \varphi_\ell^{(n)}(\xi) \jxi^2 \partial_\xi \hatf(\xi)\bigr\|_{L^2_\xi} \Bigr).
 \end{equation}
\end{lemma}
\begin{proof}
By Proposition~\ref{prop:dispersive_decay_estimate}, we can assume that $k\ge10$. 
From \cite[Lemma 2.2]{LS1}, 
\begin{equation*}
 \begin{aligned}
  \bigl\| e^{it\jD} P_k f \bigr\|_{L^\infty_x} \lesssim t^{-\frac12} 2^{\frac{3}{2} k} \|\jap{x} P_kf\|_{L^2_x} &\lesssim t^{-\frac12} 2^{\frac{3}{2} k} \bigl( \|P_k f\|_{L^2_x} + \| \partial_\xi \widehat{P_k f}\|_{L^2_\xi} \bigr)\\
  &\les t^{-\frac12} 2^{-\frac{1}{2} k} \Bigl( \|\jD^2 f\|_{L^2_x} +    \bigl\|   \jxi^2 \partial_\xi \hatf(\xi)\bigr\|_{L^2_{[|\xi|>100]}} \Bigr).
 \end{aligned}
\end{equation*}
This proves the lemma since only $\ell=0$ contributes. 
\end{proof}

Next, we establish local decay estimates for the linear Klein-Gordon evolution.

\begin{proposition} \label{prop:local_decay_estimates}
There exists an absolute constant $C \geq 1$ such that for any integer $n \geq 1$, we have for all $t \geq 0$ that
\begin{equation} \label{eq:locdec1}
\begin{aligned}
 &\bigl\| \jx^{-1} e^{it\jD} \px f \bigr\|_{H^1_x} + \bigl\| \jx^{-1} e^{it\jD} (\jD-1) f \bigr\|_{H^1_x} \\
 &\quad \leq C \frac{n}{\jt} \Bigl( \|\jD^2 f\|_{L^2_x} + \sup_{0 \leq \ell \leq n} \, 2^{-\frac12 \ell} \bigl\| \varphi_\ell^{(n)}(\xi) \jxi^2 \partial_\xi \hatf(\xi)\bigr\|_{L^2_\xi} \Bigr).
\end{aligned}
\end{equation}
Moreover, for any integer $k \leq 0$ and for any integer $n \geq 1$ we have for all $t \geq 0$ that
\begin{equation} \label{eq:locdec2}
\begin{aligned}
 &\bigl\| \jx^{-1} e^{it\jD} \px P_{\leq k} f \bigr\|_{L^2_x} + \bigl\| \jx^{-1} e^{it\jD} \px P_{>k} f \bigr\|_{L^2_x} \\
 &\quad \quad \leq C \frac{n}{\jt} \Bigl( \|\jD^2 f\|_{L^2_x} + \sup_{0 \leq \ell \leq n} \, 2^{-\frac12 \ell} \bigl\| \varphi_\ell^{(n)}(\xi) \jxi^2 \partial_\xi \hatf(\xi)\bigr\|_{L^2_\xi} \Bigr)
\end{aligned}
\end{equation}
as well as
\begin{equation} \label{eq:locdec3}
\begin{aligned}
 &\bigl\| \jx^{-1} e^{it\jD} P_k f \bigr\|_{L^2_x} + \bigl\| \jx^{-1} e^{it\jD} P_{>k} f \bigr\|_{L^2_x} \\
 &\quad \quad \leq C \frac{2^{|k|} \cdot n}{\jt} \Bigl( \|\jD^2 f\|_{L^2_x} + \sup_{0 \leq \ell \leq n} \, 2^{-\frac12 \ell} \bigl\| \varphi_\ell^{(n)}(\xi) \jxi^2 \partial_\xi \hatf(\xi)\bigr\|_{L^2_\xi} \Bigr).
\end{aligned}
\end{equation}
\end{proposition}
\begin{proof}
By Plancherel it suffices to assume that $t \geq 1$. From 
\begin{equation*}
 \begin{aligned}
  \jap{x}^{-1} \bigl( e^{it\jD} \partial_x f \bigr)(x) &= \frac{i}{\sqrt{2\pi}} \jap{x}^{-1} \int_\bbR e^{ix\xi} e^{it\jxi} \xi \wh{f}(\xi) \, \ud \xi \\
  &= - \frac{1}{\sqrt{2\pi}} (t\jap{x})^{-1} \int_\bbR e^{it\jxi} \partial_\xi \bigl( e^{ix\xi} \jxi \wh{f}(\xi) \bigr) \, \ud \xi,
 \end{aligned}
\end{equation*}
we obtain the estimate
\begin{equation} \label{eq:pxfj}
 \begin{aligned}
  \bigl\| \jap{x}^{-1} e^{it\jD} \partial_x f \bigr\|_{L^2_x} &\lesssim t^{-1} \bigl\| \jxi \hatf(\xi) \bigr\|_{L^2_\xi} + t^{-1} \bigl\| \partial_\xi \bigl( \jxi \hatf(\xi) \bigr) \bigr\|_{L^1_\xi} \\
  &\lesssim t^{-1} \Bigl( \bigl\| \jxi \hatf(\xi) \bigr\|_{L^2_\xi}  + \sum_{\ell=0}^n 2^{-\frac{1}{2}\ell} \bigl\| \varphi_\ell^{(n)}(\xi) \jxi^{2} \partial_\xi \hat{f}(\xi) \|_{L^2_\xi} \Bigr),
 \end{aligned}
\end{equation}
which yields the claimed bound on the first term on the left-hand side of~\eqref{eq:locdec1}.
The proof of the claimed bound for the the second term is analogous since $\jxi-1 = \xi^2 (\jxi+1)^{-1}$. The bounds~\eqref{eq:locdec2} and \eqref{eq:locdec3} involving Littlewood-Paley projections are standard corollaries of~\eqref{eq:locdec1}.
\end{proof}

We also recall from \cite[Lemma 2.3]{LS1} the following improved local decay estimate for the linear Klein-Gordon evolution, which requires stronger spatial weights. This estimate is only used in the proof of Proposition~\ref{prop:prop49}.

\begin{lemma} 
 Let $a > \frac32$. We have uniformly for all $t \in \bbR$ that
 \begin{equation} \label{equ:improved_local_decay_stronger_weights}
  \bigl\| \jx^{-a} \px \jD^{-1} e^{it\jD} f\bigr\|_{L^2_x} \lesssim \frac{1}{\jt^{\frac32}} \| \jx^a f\|_{L^2_x}.
 \end{equation}
\end{lemma}

\section{Darboux Transformations} \label{sec:darboux}

In this section we introduce the iterated Darboux transformations associated with the linearized operator $L$ defined in~\eqref{equ:linearized_operator}. Moreover, we determine corresponding right-inverse operators and their Fourier transforms.
For more background on Darboux transformations, see, e.g., \cite{DeiTru, InfeldHull51, CGNT07, MatveevSalle, RogersSchief02}.

\subsection{Basic definitions}

We define the differential operators
\begin{equation*}
 \calD_2 := Y_0 \cdot \px \cdot Y_0^{-1} = \px + 2 \tanh(x), \quad \calD_2^\ast := - Y_0^{-1} \cdot \px \cdot Y_0 = -\px + 2\tanh(x),
\end{equation*}
and observe the factorization identity
\begin{equation*}
 \calD_2^\ast \calD_2 = L+3.
\end{equation*}
Since $L Y_1 = 0$, we must have that
\begin{equation*}
 Z := \calD_2 Y_1 = c_1 \sech(x)
\end{equation*}
satisfies
\begin{equation*}
 \bigl( \calD_2^\ast \calD_2 - 3 \bigr) Z = 0.
\end{equation*}
Introducing the differential operators
\begin{equation*}
 \calD_1 := Z \cdot \px \cdot Z^{-1} = \px + \tanh(x), \quad \calD_1^\ast := - Z^{-1} \cdot \px \cdot Z = -\px + \tanh(x),
\end{equation*}
we find
\begin{equation*}
 \calD_2^\ast \calD_2 - 3 = \calD_1^\ast \calD_1,
\end{equation*}
and we observe
\begin{equation*}
 \calD_1 \calD_1^\ast = -\px^2 + 1.
\end{equation*}
In particular, we conclude the key conjugation identity that is used in this work
\begin{equation} \label{equ:conjugation_identity_sec_darboux}
 \calD_1 \calD_2 L = (-\px^2 + 1) \calD_1 \calD_2.
\end{equation}

Next, we determine right-inverse operators for the Darboux transformations $\calD_1$ and $\calD_2$.
Since $\calD_1 Z = 0$ by definition of $\calD_1$, the integral operator
\begin{equation} \label{equ:calI1_definition}
 \calI_1[g](x) := Z(x) \int_0^x \bigl( Z(y) \bigr)^{-1} g(y) \, \ud y = \sech(x) \int_0^x \cosh(y) g(y) \, \ud y
\end{equation}
satisfies
\begin{equation*}
 \calD_1 \bigl( \calI_1[g] \bigr) = g.
\end{equation*}
Integrating by parts, we also find that
\begin{equation*}
\calI_1[\calD_1 f](x) = f(x) - c_1^{-1} f(0) Z(x).
\end{equation*}
Similarly, since $\calD_2 Y_0 = 0$ by definition of $\calD_2$, the integral operator
\begin{equation} \label{equ:calI2_definition}
 \calI_2[g](x) := Y_0(x) \int_0^x \bigl( Y_0(y) \bigr)^{-1} g(y) \, \ud y = \sech^2(x) \int_0^x \cosh^2(y) g(y) \, \ud y
\end{equation}
satisfies
\begin{equation*}
 \calD_2 \bigl( \calI_2[g] \bigr) = g,
\end{equation*}
and we have
\begin{equation*}
\calI_2[\calD_2 f](x) = f(x) - c_0^{-1} f(0) Y_0(x).
\end{equation*}
Thus, we obtain that
\begin{equation} \label{equ:calJ_definition}
 \calJ[g](x) := \calI_2\bigl[ \calI_1[g] \bigr](x) = \sech^2(x) \int_0^x \cosh(y) \int_0^y \cosh(z) g(z) \, \ud z \, \ud y
\end{equation}
is a right-inverse operator for the iterated Darboux transformation $\calD_1 \calD_2$, i.e.,
\begin{equation*}
 \calD_1 \calD_2 \bigl( \calJ[g] \bigr) = g.
\end{equation*}
Using that $\calI_2[Z] = Y_1$, we arrive at the following representation formula
\begin{equation} \label{equ:relation_f_JD1D2f}
 f(x) = \calJ\bigl[ \calD_1 \calD_2 f \bigr](x) + c_0^{-1} f(0) Y_0(x) + c_1^{-1} f'(0) Y_1(x).
\end{equation}
In particular, it follows that
\begin{equation} \label{equ:relation_Pcf_PcJDDf}
 P_c f = P_c \calJ\bigl[ \calD_1 \calD_2 f \bigr].
\end{equation}

Moreover, integrating by parts in the definitions of the integral operators $\calI_1[g]$ and $\calJ[g]$, we obtain the following identities
\begin{align}
 \calI_1[g] &= K g + \wtilcalI_1[\px g], \label{equ:calI1_in_terms_of_wtilcalI1} \\
 \calJ[g] &= \frac12 K^2 g + \wtilcalJ[\px g] \label{equ:calJ_in_terms_of_wtilcalJ}
\end{align}
with $K(x) = \tanh(x)$ and
\begin{equation*}
 \begin{aligned}
  \wtilcalI_1[h](x) &:= - \sech(x) \int_0^x \sinh(y) h(y) \, \ud y, \\
  \wtilcalJ[h](x) &:= -\frac12 \sech^2(x) \int_0^x \sinh^2(y) h(y) \, \ud y.
 \end{aligned}
\end{equation*}

Throughout the remainder of this work, we will repeatedly use the following simple bounds for the integral operators $\calI_1$, $\widetilde{\calI}_1$, $\calJ$, and $\widetilde{\calJ}$ introduced above.
\begin{lemma} \label{lem:I1_and_J_bounds}
We have 
 \begin{align}
  \bigl\| \px^j \calI_1[v] \bigr\|_{L^\infty_x} &\lesssim \|v\|_{L^\infty_x}, \quad \quad \quad \quad j = 0, 1, \\
  \bigl\| \px^j \calJ[v] \bigr\|_{L^\infty_x} &\lesssim \|v\|_{L^\infty_x}, \quad \quad \quad \quad j = 0, 1, \\
  \bigl\| \px^j P_c \calJ[v] \bigr\|_{L^\infty_x} &\lesssim \|v\|_{L^\infty_x}, \quad \quad \quad \quad j = 0, 1, \\
  \bigl\| \jx \calI_1[v] \bigr\|_{L^2_x} &\lesssim \|\jx v\|_{L^2_x}, \\
  \bigl\| \jx \calJ[v] \bigr\|_{L^2_x} &\lesssim \|\jx v\|_{L^2_x}, \\
  \bigl\| \jx^{-1} \px^j \widetilde{\calI}_1[\px v] \bigr\|_{L^2_x} &\lesssim \|\jx^{-1} \px v\|_{L^2_x}, \quad j = 0, 1, \\
  \bigl\| \jx^{-1} \px^j \widetilde{\calJ}[\px v] \bigr\|_{L^2_x} &\lesssim \|\jx^{-1} \px v\|_{L^2_x}, \quad j = 0, 1, \\
  \bigl\| \jx^{-1} \widetilde{\calI}_1[\px v] \bigr\|_{L^\infty_x} &\lesssim \|\jx^{-1} \px v\|_{L^2_x}, \\
  \bigl\| \jx^{-1} \widetilde{\calJ}[\px v] \bigr\|_{L^\infty_x} &\lesssim \|\jx^{-1} \px v\|_{L^2_x}.
 \end{align}
\end{lemma}
\begin{proof}
 The asserted bounds follow in a straightforward manner from the definition of the integral operators and from the exponential localization of their kernels.
\end{proof}

\subsection{Fourier transforms for the integral operators $\calI_1$ and $\calJ$} \label{subsec:FT_calI1_calJ}

In this subsection we compute the Fourier transforms of the integral operators $\calI_1$ and $\calJ$.
We begin with the integral operator $\calI_1$ defined in~\eqref{equ:calI1_definition}.

\begin{lemma} \label{lem:FT_calI1}
 For $f \in \calS(\bbR)$ we have 
 \begin{equation} \label{equ:FT_calI1}
  \widehat{\calI_1[f]}(\xi) = \int_\bbR \calK_1(\xi,\eta) \hatf(\eta) \, \ud \eta 
 \end{equation}
 with 
 \begin{equation}
  \calK_1(\xi, \eta) = - i \, \pvdots \cosech\Bigl(\frac{\pi}{2}(\xi-\eta)\Bigr) \frac{1}{2 \jap{\eta}^2} + i \sech\Bigl(\frac{\pi}{2} \xi\Bigr) \frac{\eta}{2 \jap{\eta}^2}.
 \end{equation}
\end{lemma}
\begin{proof}
Let $f \in \calS(\bbR)$. We compute
\begin{equation*}
 \begin{aligned}
  \widehat{\calI_1[f]}(\xi) &= \frac{1}{\sqrt{2\pi}} \int_\bbR \calI_1[f](x) e^{-ix\xi} \, \ud x \\
  &= \frac{1}{2\pi}\lim_{\tau\to 1^+} \int_\R \sech(\tau x) \int_0^x \cosh(y) \int_{\R} \hatf(\eta) e^{iy\eta} \, \ud \eta \,  e^{-ix\xi} \, \ud y \, \ud x \\
  &= \frac{1}{4\pi}\lim_{\tau\to 1^+}\int_{\R}  \int_\R \sech(\tau x) \int_0^x \big(e^{y(1+i\eta)} + e^{y(-1+i\eta)}\big) \, \ud y \, e^{-ix\xi}  \, \ud x \, \hatf(\eta) \, \ud \eta \\
  &= \frac{1}{4\pi}\lim_{\tau\to 1^+}\int_{\R}  \int_\R \sech(\tau x) \biggl( \frac{e^{x(1+i\eta)}-1}{1+i\eta} + \frac{ e^{x(-1+i\eta)}-1}{-1+i\eta} \biggr) e^{-ix\xi} \, \ud x \, \hatf(\eta) \, \ud \eta  \\
  &=: \int_\bbR \calK_1(\xi,\eta)\, \hatf(\eta) \, \ud \eta  
 \end{aligned}
\end{equation*}
with
\begin{equation*}
 \begin{aligned}
  \calK_1(\xi,\eta) &= - \frac{1}{4\pi}\lim_{\tau\to1^+} \int_{\R} \sech(\tau x) \biggl( \frac{1-e^{x(1+i\eta)}}{1+i\eta} - \frac{1- e^{-x(1-i\eta)}}{1-i\eta} \biggr)  e^{-ix\xi} \, \ud x \\
  &= -\frac{i}{2\pi} \lim_{\tau\to1^+} \Im \int_{\R} \sech(\tau x) \frac{1-e^{x(1+i\eta)}}{1+i\eta}  e^{-ix\xi} \, \ud x.  
 \end{aligned}
\end{equation*}
To pass to the second line, we substituted $x \mapsto -x$ in the second term inside the parentheses on the first line. Hence, using that
\begin{equation*}
 \widehat{\sech}(\xi) = \sqrt{\frac{\pi}{2}} \sech\Bigl(\frac{\pi}{2} \xi\Bigr),
\end{equation*}
we find
\begin{equation*}
 \begin{aligned}
  \calK_1(\xi,\eta) &= -\frac{i}{2\pi}\lim_{\tau\to1^+} \Im \int_{\R} \frac{ \sech(\tau x) }{1+i\eta}  \big(  e^{-ix\xi} -e^{-ix(i+(\xi-\eta))}\big) \,  \ud x \\
  &= -\frac{i}{2} \Im \biggl( \frac{\sech(\pih \xi)}{1+i\eta} - \lim_{\tau\to1^+}  \frac{ \sech\bigl( \tau^{-1}\pih (\xi-\eta +i) \bigr)  }{\tau(1+i\eta)} \biggr) \\
  &= -\frac{i}{2} \Im \biggr( \frac{\sech(\pih \xi)}{1+i\eta} + i \, \pvdots \frac{ \cosech\bigl( \pih (\xi-\eta)\bigr)}{1+i\eta} \biggr) \\
  &= -\frac{i}{2(1+\eta^2)}  \biggl( - \eta\sech\Bigl( \frac{\pi}{2} \xi \Bigr) + \pvdots \cosech\Bigl( \pih (\xi-\eta)\Bigr) \biggr),
 \end{aligned}
\end{equation*}
as claimed.
\end{proof}

Next, we compute the Fourier transform for the integral operator $\calJ$ defined in~\eqref{equ:calJ_definition}.

\begin{lemma}
 For $f \in \calS(\bbR)$ we have 
 \begin{equation} \label{equ:FT_calJ}
  \widehat{\calJ[f]}(\xi) = \int_\bbR \calK(\xi,\eta) \hatf(\eta) \, \ud \eta 
 \end{equation}
 with 
 \begin{equation} \label{equ:FT_calJ_kernel}
  \begin{aligned}
   \calK(\xi, \eta) &= \frac{2-\eta^2}{(1+\eta^2)(4+\eta^2)} \delta_0(\xi-\eta) - \frac{3\eta}{2(1+\eta^2)(4+\eta^2)} \, \pvdots \cosech\Bigl(\frac{\pi}{2}(\xi-\eta)\Bigr) \\
   &\quad - \frac{3}{2(1+\eta^2)(4+\eta^2)} (\xi-\eta) \cosech\Bigl(\frac{\pi}{2}(\xi-\eta)\Bigr) \\
   &\quad + \frac{1}{2 (4+\eta^2)} \xi \cosech\Bigl(\frac{\pi}{2} \xi\Bigr) + \frac{\eta}{2(1+\eta^2)} \xi \sech\Bigl(\frac{\pi}{2} \xi\Bigr).
  \end{aligned}
 \end{equation}
\end{lemma}
\begin{proof}
For any $f \in \calS(\bbR)$, we write
\begin{equation*}
 \begin{aligned}
  \wh{\calJ[f]}(\xi) &= \frac{1}{2\pi} \lim_{\tau\to1^+} \int_\bbR e^{-ix\xi} \sech^2(\tau x) \int_0^x \cosh(y) \int_0^y \cosh(z) \int_\bbR e^{iz\eta} \hat{f}(\eta) \, \ud\eta \, \ud z \, \ud y \, \ud x \\
  &=: \int_\bbR \calK(\xi,\eta) \hatf(\eta) \, \ud \eta.
 \end{aligned}
\end{equation*}
Then we compute
\begin{equation*}
 \begin{aligned}
  &\int_0^x \cosh(y) \int_0^y \cosh(z)  e^{iz\eta} \, \ud z \, \ud y \\
  &= \frac14 \biggl( \frac{e^{x(2+i\eta)}-1}{(1+i\eta)(2+i\eta)} + \frac{e^{x(-2+i\eta)}-1}{(1-i\eta)(2-i\eta)} \biggr) + \frac{i\eta\sinh(x)}{1+\eta^2} + \frac{1-e^{i\eta x}}{2(1+\eta^2)} \\
  &= \frac12 \frac{(2-\eta^2)e^{ix\eta}}{(1+\eta^2)(4+\eta^2)}\cosh(2x) - \frac32 i\frac{\eta \, e^{ix\eta}}{(1+\eta^2)(4+\eta^2)}\sinh(2x) + \frac{1}{4+\eta^2} + \frac{i\eta\sinh(x)}{1+\eta^2} - \frac{e^{i\eta x}}{2(1+\eta^2)}.
 \end{aligned}
\end{equation*}
Hence, in the sense of distributional limits
\begin{equation} \label{equ:compute_FT_calJ_calK_def}
 \begin{aligned}
  \calK(\xi,\eta) &= \frac{1}{4\pi}\frac{2-\eta^2}{(1+\eta^2)(4+\eta^2)} \lim_{\tau\to1^+} \int_\bbR e^{-ix(\xi-\eta)} \sech^2(\tau x)\cosh(2x) \, \ud x \\
  &\quad - \frac{3i}{4\pi}\frac{\eta}{(1+\eta^2)(4+\eta^2)} \lim_{\tau\to1^+} \int_\bbR e^{-ix(\xi-\eta)} \sech^2(\tau x)\sinh(2x) \, \ud x \\
  &\quad + \frac{1}{2\pi} \frac{1 }{4+\eta^2} \int_\bbR e^{-ix\xi} \sech^2(x) \, \ud x + \frac{1}{2\pi} \frac{i\eta }{1+\eta^2} \int_\bbR e^{-ix\xi} \sech^2( x) \sinh(x) \, \ud x \\
  &\quad - \frac{1}{4\pi(1+\eta^2)} \int_\bbR e^{-ix(\xi-\eta)} \sech^2(x) \, \ud x.
 \end{aligned}
\end{equation}
For the computation of the limits, note that in the sense of $\calS'(\R)$, 
\begin{equation*}
 \begin{aligned}
  \lim_{\tau\to 1^+} \sech^2(\tau x)\cosh(2x) &= 2 - \sech^2(x), \\
  \lim_{\tau\to 1^+} \sech^2(\tau x)\sinh(2x) &= 2 \tanh(x).
 \end{aligned}
\end{equation*}
Recall from \cite[Lemma 5.6]{LS1} that in the sense of $\calS'(\bbR)$
\begin{equation*}
 \widehat{\tanh}(\xi) = -i \sqrt{\frac{\pi}{2}} \, \pvdots \cosech\Bigl(\frac{\pi}{2} \xi\Bigr),
\end{equation*}
and from \cite[Corollary 5.7]{LS1} that as equalities in $\calS(\bbR)$
\begin{equation*}
 \widehat{\sech^2}(\xi) = \sqrt{\frac{\pi}{2}} \xi \cosech\Bigl(\frac{\pi}{2} \xi\Bigr).
\end{equation*}
Thus, as Fourier transforms in $\calS'(\R)$, we have
\begin{equation*}
 \begin{aligned}
  \lim_{\tau\to1^+}   \calF[ \sech^2(\tau \cdot)\cosh(2 \cdot)](\xi) &= 2\sqrt{2\pi}\,\delta_0(\xi) -  \sqrt{\frac{\pi}{2}} \xi \cosech\Bigl(\frac{\pi}{2} \xi\Bigr), \\
  \lim_{\tau\to1^+}   \calF[ \sech^2(\tau \cdot)\sinh(2\cdot)](\xi) &= - i\sqrt{2\pi} \, \pvdots \cosech\Bigl(\frac{\pi}{2} \xi\Bigr).
 \end{aligned}
\end{equation*}
To simplify the penultimate term in \eqref{equ:compute_FT_calJ_calK_def}, we observe
\begin{equation*}
 \calF\bigl[ \sech^2(\cdot) \sinh(\cdot) \bigr](\xi) = - \calF\bigl[ \px \bigl( \sech(\cdot) \bigr) \bigr](\xi) = -i \sqrt{\frac{\pi}{2}} \xi \sech\Bigl(\frac{\pi}{2}\xi\Bigr).
\end{equation*}
It follows that
\begin{equation*}
 \begin{aligned}
  \calK(\xi,\eta) &= \frac{2-\eta^2}{(1+\eta^2)(4+\eta^2)} \biggl( \delta_0(\xi-\eta) - \frac14(\xi-\eta) \cosech\Bigl(\frac{\pi}{2}(\xi-\eta)\Bigr) \biggr) \\
  &\quad - \frac{3\eta }{2(1+\eta^2)(4+\eta^2)} \, \pvdots \cosech\Bigl(\frac{\pi}{2}(\xi-\eta)\Bigr) \\
  &\quad +\frac{1}{2 (4+\eta^2)} \xi \cosech\Bigl(\frac{\pi}{2} \xi\Bigr) + \frac{\eta}{2(1+\eta^2)} \xi \sech\Bigl(\frac{\pi}{2} \xi\Bigr) \\
  &\quad -\frac{1}{4(1+\eta^2)} (\xi-\eta) \cosech\Bigl(\frac{\pi}{2} (\xi-\eta) \Bigr).
 \end{aligned}
\end{equation*}
Combining the two $(\xi-\eta) \cosech(\frac{\pi}{2}(\xi-\eta))$ terms gives~\eqref{equ:FT_calJ_kernel}.
\end{proof}

In order to state the identities \eqref{equ:FT_calI1} and \eqref{equ:FT_calJ} for the Fourier transforms for the integral operators $\calI_1$ and $\calJ$ more succinctly, we now introduce some short-hand notation.
We define the multipliers
\begin{equation} \label{equ:def_multipliers_m}
 \begin{aligned}
  m_0(\xi) &:= - \frac{1}{2 \jxi^2}, \\
  m_1(\xi) &:= \frac{\xi}{2 \jxi^2}, \\
  m_2(\xi) &:= \frac{1}{2(4+\xi^2)}, \\
  m_3(\xi) &:= \frac{\xi}{2(1+\xi^2)} \\
  m_4(\xi) &:= \frac{2-\xi^2}{(1+\xi^2) (4+\xi^2)}, \\
  m_5(\xi) &:= -\frac{3\xi}{2(1+\xi^2)(4+\xi^2)}, \\
  m_6(\xi) &:= -\frac{3}{2(1+\xi^2)(4+\xi^2)},
 \end{aligned}
\end{equation}
and the Schwartz functions
\begin{equation} \label{equ:def_schwartz_omegas}
 \begin{aligned}
  \omega_1(\xi) &:= \sech\Bigl(\frac{\pi}{2} \xi\Bigr), \\
  \omega_2(\xi) &:= \xi \sech\Bigl(\frac{\pi}{2} \xi\Bigr), \\
  \omega_3(\xi) &:= \xi \cosech\Bigl(\frac{\pi}{2} \xi\Bigr).
 \end{aligned}
\end{equation}
Additionally, we introduce the notation
\begin{equation} \label{equ:def_Bterms}
 \begin{aligned}
  B_1(\hatf) &:= \int_\bbR m_1(\eta) \hatf(\eta) \, \ud \eta, \\
  B_2(\hatf) &:= \int_\bbR m_2(\eta) \hatf(\eta) \, \ud \eta, \\
  B_3(\hatf) &:= \int_\bbR m_3(\eta) \hatf(\eta) \, \ud \eta,
 \end{aligned}
\end{equation}
and we write
\begin{equation*}
 \begin{aligned}
  \Omega := \pvdots \cosech\Bigl(\frac{\pi}{2} \cdot \Bigr).
 \end{aligned}
\end{equation*}

Using the preceding short-hand notation, we can express the Fourier transforms of $\calI_1[f]$ and of $\calJ[f]$ succinctly as
\begin{equation}
 \widehat{\calI_1[f]} = i \Omega \ast (m_0 \hatf) + i \omega_1 B_1(\hatf)
\end{equation}
and 
\begin{equation}
 \begin{aligned}
  \widehat{\calJ[f]} &= \delta_0 \ast (m_4 \hatf) + \Omega \ast (m_5 \hatf) + \omega_3 \ast (m_6 \hatf) + \omega_3 B_2(\hatf) + \omega_2 B_3(\hatf).
 \end{aligned}
\end{equation}
We will distinguish between the singular and the regular parts of the Fourier transforms of $\calI_1[f]$ and $\calJ[f]$. Correspondingly, we write
\begin{equation} \label{equ:decomposition_FT_I1_J_sing_reg}
 \widehat{\calI_1[f]} = \widehat{\calI_1[f]}_S + \widehat{\calI_1[f]}_R, \quad \widehat{\calJ[f]} = \widehat{\calJ[f]}_S + \widehat{\calJ[f]}_R
\end{equation}
with
\begin{equation*}
 \begin{aligned}
  \widehat{\calI_1[f]}_S &:= i \Omega \ast (m_0 \hatf), \\
  \widehat{\calI_1[f]}_R &:= i \omega_1 B_1(\hatf), \\
  \widehat{\calJ[f]}_S &:= \delta_0 \ast (m_4 \hatf) + \Omega \ast (m_5 \hatf), \\
  \widehat{\calJ[f]}_R &:= \omega_3 \ast (m_6 \hatf) + \omega_3 B_2(\hatf) + \omega_2 B_3(\hatf).
 \end{aligned}
\end{equation*}

\medskip

We conclude this section with the derivation of the following convolution identities.

\begin{lemma} \label{lem:convolutions}
 We have as equalities in $\calS(\bbR)$,
 \begin{align}
  (\Omega \ast \omega_1)(\xi) &=  2\xi \sech\Bigl(\frac{\pi}{2} \xi\Bigr), \label{equ:convolution_Omega_omega1} \\
  (\Omega \ast \omega_2)(\xi) &= (\xi^2-1) \sech\Bigl(\frac{\pi}{2} \xi\Bigr), \label{equ:convolution_Omega_omega2} \\
  (\Omega \ast \omega_3)(\xi) &= \xi^2 \cosech\Bigl(\frac{\pi}{2} \xi\Bigr), \label{equ:convolution_Omega_omega3}
 \end{align}
 and as equalities in $\calS'(\bbR)$,
 \begin{equation} \label{equ:convolution_Omega_Omega}
  \Omega \ast \Omega = -4\delta_0 + 2 \omega_3.
 \end{equation}
\end{lemma}
\begin{proof}
 We refer to \cite[Corollary 5.7]{LS1} for the proofs of the identities \eqref{equ:convolution_Omega_omega1} and \eqref{equ:convolution_Omega_Omega}.
For the proof of \eqref{equ:convolution_Omega_omega2}, we compute on the one hand that
\begin{equation} \label{equ:FOmom2}
 \begin{aligned}
  \calF^{-1}\bigl[ \Omega \ast \omega_2\bigr](x) = \sqrt{2\pi} \, \check{\Omega}(x) \, \check{\omega}_2(x) &= 2 \tanh(x) \, \px \biggl( \calF^{-1}\Bigl[ \sech\Bigl( \frac{\pi}{2} \xi \Bigr) \Bigr](x) \biggr) \\
  &= 2 \sqrt{\frac{2}{\pi}} \tanh(x) \, \px \bigl( \sech(x) \bigr) \\
  &= -2\sqrt{\frac{2}{\pi}} \tanh^2(x) \sech(x).
 \end{aligned}
\end{equation}
On the other hand,
\begin{equation*}
 \begin{aligned}
  \calF^{-1}\Bigl[ (\xi^2-1) \sech\Bigl( \frac{\pi}{2} \xi \Bigr) \Bigr](x) &= (-\px^2-1) \biggl( \calF^{-1}\Bigl[ \sech\Bigl( \frac{\pi}{2} \xi \Bigr) \Bigr](x) \biggr) \\
  &= \sqrt{\frac{2}{\pi}}  (-\px^2-1) \bigl( \sech(x) \bigr) \\
  &= 2 \sqrt{\frac{2}{\pi}} (\sech(x)^3-\sech(x)),
 \end{aligned}
\end{equation*}
which agrees with \eqref{equ:FOmom2}.

Finally, for the proof of \eqref{equ:convolution_Omega_omega3}, we compute
\begin{equation} \label{eq:FOmom3}
 \begin{aligned}
  \calF^{-1}\bigl[\Omega\ast\omega_3 \bigr](x) = \sqrt{2\pi} \, \check{\Omega}(x) \, \check{\omega}_3(x) &= 2i\tanh(x) \calF^{-1}\Bigl[ \xi \cosech\Bigl( \frac{\pi}{2} \xi \Bigr) \Bigr](x) \\
  &= 2i\tanh(x) (-i) \px \biggl( \calF^{-1}\Bigl[ \pvdots \cosech\Bigl( \frac{\pi}{2} \xi \Bigr) \Bigr](x) \biggr) \\
  &= 2 \tanh(x) \px \biggl( i \sqrt{\frac{2}{\pi}} \tanh(x) \biggr) \\
  &= 2i \sqrt{\frac{2}{\pi}} \sech^2(x) \tanh(x),
 \end{aligned}
\end{equation}
as well as
\begin{equation*}
 \begin{aligned}
  \calF^{-1}\Bigl[\xi^2 \cosech\Bigl( \frac{\pi}{2} \xi \Bigr) \Bigr](x) = -\px^2 \biggl( \calF^{-1}\Bigl[ \pvdots \cosech\Bigl( \frac{\pi}{2} \xi \Bigr) \Bigr](x) \biggr) &= -i \sqrt{\frac{2}{\pi}} \px^2 \bigl( \tanh(x) \bigr) \\
  &= 2i \sqrt{\frac{2}{\pi}} \sech^2(x) \tanh(x),
 \end{aligned}
\end{equation*}
which is the same as \eqref{eq:FOmom3}.
\end{proof}

\section{The Transformed Equation} \label{sec:transformed_equation}

In this section we use the conjugation identity~\eqref{equ:conjugation_identity_sec_darboux}
to transform the equation for the dispersive part of a perturbation of the soliton of the focusing cubic Klein-Gordon equation~\eqref{equ:focusing_cubic_KG} into a nonlinear Klein-Gordon equation without a potential. Then we carefully analyze the structure of the nonlinearities of the transformed equation.

\subsection{Spectral decomposition}

We consider a solution to the focusing cubic Klein-Gordon equation
\begin{equation*}
 (\pt^2 - \px^2 + 1) \phi = \phi^3
\end{equation*}
with \emph{even} initial conditions as specified in the statement of Theorem~\ref{thm:main},
\begin{equation}
 \bigl(\phi(0), \pt \phi(0) \bigr) = (Q,0) + (\varphi_0, \varphi_1) + d (Y_0, \nu Y_0).
\end{equation}
Then the evolution equation for the \emph{even} perturbation
\begin{equation*}
 \varphi(t,x) := \phi(t,x) - Q(x)
\end{equation*}
of the static soliton $Q(x)$ is given by
\begin{equation} \label{equ:pert_equ_varphi_refer_to}
 \bigl( \partial_t^2 - \partial_x^2 - 6 \sech^2(x) + 1 \bigr) \varphi = 3 Q \varphi^2 + \varphi^3.
\end{equation}
We may write it more succinctly in terms of the linearized operator $L$ defined in~\eqref{equ:linearized_operator} as
\begin{equation*} 
 ( \partial_t^2 + \linop ) \varphi = 3 Q \varphi^2 + \varphi^3.
\end{equation*}
Now we enact a spectral decomposition
\begin{equation*}
 \varphi(t,x) = (P_c \varphi)(t,x) + a(t) Y_0(x),
\end{equation*}
where 
\begin{equation*}
 a(t) := \langle Y_0, \varphi(t) \rangle.
\end{equation*}
This leads to the following coupled PDE/ODE system for the variables $(P_c \varphi, a)$,
\begin{equation} \label{equ:pde_ode_system_for_varphi_a}
 \left\{ \begin{aligned}
          (\pt^2 + L) P_c \varphi &= P_c \bigl( 3 Q (P_c \varphi + a Y_0)^2 + (P_c \varphi + a Y_0)^3 \bigr), \\
          (\pt^2 - \nu^2) a &= \bigl\langle Y_0, 3 Q (P_c \varphi + a Y_0)^2 \bigr\rangle + \bigl\langle Y_0, (P_c \varphi + a Y_0)^3 \bigr\rangle.
         \end{aligned} \right.
\end{equation}
Next, we apply the iterated Darboux transformation $\calD_1 \calD_2$ to the equation for $P_c \varphi$ in \eqref{equ:pde_ode_system_for_varphi_a} and we pass to the new variable
\begin{equation*}
 w := \calD_1 \calD_2 P_c \varphi = \calD_1 \calD_2 \varphi.
\end{equation*} 
Using that $P_c \varphi = P_c \calJ[w]$ by the identity~\eqref{equ:relation_Pcf_PcJDDf} and that the kernel of $\calD_1 \calD_2$ is spanned by $Y_0$ and~$Y_1$, we obtain the following coupled PDE/ODE system for the variables $(w, a)$
\begin{equation} \label{equ:derive_transformed_equ1}
 \left\{ \begin{aligned}
          (\pt^2 -\px^2 + 1) w &= \calD_1 \calD_2 \bigl( 3 Q (P_c \calJ[w] + a Y_0)^2 \bigr) + \calD_1 \calD_2 \bigl( (P_c \calJ[w] + a Y_0)^3 \bigr), \\
          (\pt^2 - \nu^2) a &= \bigl\langle Y_0, 3 Q (P_c \calJ[w] + a Y_0)^2 \bigr\rangle + \bigl\langle Y_0, (P_c \calJ[w] + a Y_0)^3 \bigr\rangle,
         \end{aligned} \right.
\end{equation}
with initial conditions 
\begin{equation*}
 \begin{aligned}
  w(0) &= \calD_1 \calD_2 \bigl( \varphi(0) \bigr) = \calD_1 \calD_2 \bigl( \varphi_0 + d Y_0 \bigr) = \calD_1 \calD_2 \varphi_0, \\
  \pt w(0) &= \calD_1 \calD_2 \bigl( \pt \varphi(0) \bigr) = \calD_1 \calD_2 \bigl( \varphi_1 + d \nu Y_0 \bigr) = \calD_1 \calD_2 \varphi_1, \\
  a(0) &= \langle Y_0, \varphi(0) \rangle = \langle Y_0, \varphi_0 \rangle + d, \\
  \pt a(0) &= \langle Y_0, \pt \varphi(0) \rangle = \langle Y_0, \varphi_1 \rangle + d \nu.
 \end{aligned}
\end{equation*}
Note that here we used $\calD_2 Y_0 = 0$.

For the analysis in the subsequent sections, we further decompose the variable $a(t)$ into its unstable and stable components
\begin{equation} \label{equ:definition_a_plus_a_minus}
 \begin{aligned}
  a_+ := \frac12 \bigl( a + \nu^{-1} \pt a \bigr), \quad a_- := \frac12 \bigl( a - \nu^{-1} \pt a \bigr).
 \end{aligned}
\end{equation}
Then it holds that
\begin{equation*}
 a(t) = a_+(t) + a_-(t),
\end{equation*}
and the unstable coefficient $a_+(t)$, respectively the stable coefficient $a_-(t)$, satisfy the first-order differential equations
\begin{equation*}
 \begin{aligned}
  (\pt - \nu) a_+ &= (2\nu)^{-1} \langle Y_0, (3 Q \varphi^2 + \varphi^3) \rangle, \\
  (\pt + \nu) a_- &= - (2\nu)^{-1} \langle Y_0, (3 Q \varphi^2 + \varphi^3) \rangle,
 \end{aligned}
\end{equation*} 
with initial conditions
\begin{equation*}
 \begin{aligned}
  a_+(0) &= d, \\
  a_-(0) &= \frac12 \langle Y_0, \varphi_0 - \nu^{-1} \varphi_1 \rangle.
 \end{aligned}
\end{equation*}
Note that here the condition $\langle Y_0, \nu \varphi_0 + \varphi_1 \rangle = 0$ in the statement of Theorem~\ref{thm:main} entered.

\medskip

In the remainder of this section, we analyze the fine structure of the nonlinearities on the right-hand side of the nonlinear Klein-Gordon equation for $w(t)$ in \eqref{equ:derive_transformed_equ1}.
To this end we recall that
\begin{equation*}
 \varphi = P_c \calJ[w] + a Y_0 = \calI_2\bigl[ \calI_1[w] \bigr] - \langle Y_0, \calJ[w] \rangle Y_0 + a Y_0,
\end{equation*}
which implies
\begin{equation*}
 \calD_2 \varphi = \calI_1[w].
\end{equation*}
Moreover, we will use the easily verified identities
\begin{equation*}
 \begin{aligned}
  \px K &= 1 - K^2 = \frac12 Q^2, \quad \px Q &= - Q K,
 \end{aligned}
\end{equation*}
where $K(x) := \tanh(x)$ and $Q(x) = \sqrt{2} \sech(x)$.

\subsection{Structure of the transformed quadratic nonlinearity}

The purpose of this subsection is to compute 
\begin{equation*}
 \calD_1 \calD_2 \bigl( 3 Q \varphi^2 \bigr) = \calD_1 \calD_2 \bigl( 3 Q (P_c \calJ[w] + a Y_0)^2 \bigr),
\end{equation*}
and to structure the resulting quadratic nonlinearities.
We have 
\begin{equation*}
 \begin{aligned}
  \calD_2 \bigl( Q \varphi^2 \bigr) &= 2 Q \varphi (\calD_2 \varphi) - 3 Q K \varphi^2,
 \end{aligned}
\end{equation*}
and 
\begin{equation*}
 \begin{aligned}
  \calD_1 \bigl( 2 Q \varphi (\calD_2 \varphi) \bigr) &= -6 Q K \varphi (\calD_2 \varphi) + 2 Q (\calD_2 \varphi)^2 + 2 Q \varphi (\calD_1 \calD_2 \varphi), \\
  \calD_1 \bigl( - 3 Q K \varphi^2 \bigr) &= (-3Q + 15 Q K^2) \varphi^2 - 6 Q K \varphi (\calD_2 \varphi).
 \end{aligned}
\end{equation*}
Thus, we obtain that
\begin{equation*}
 \begin{aligned}
  \calD_1 \calD_2 \bigl( Q \varphi^2 \bigr) &= (-3Q + 15 Q K^2) \varphi^2 - 12 Q K \varphi (\calD_2 \varphi) + 2 Q (\calD_2 \varphi)^2 + 2 Q \varphi (\calD_1 \calD_2 \varphi).
 \end{aligned}
\end{equation*}
In view of the preceding, we find
\begin{equation} \label{equ:transformed_quadratic}
 \begin{aligned}
  \calD_1 \calD_2 \bigl( 3 Q \varphi^2 \bigr) &= 9 (-Q + 5 Q K^2) (P_c \calJ[w] + a Y_0)^2 - 36 Q K (P_c \calJ[w] + a Y_0) \calI_1[w] \\   
  &\quad + 6 Q (\calI_1[w])^2 + 6 Q (P_c \calJ[w] + a Y_0) w.
 \end{aligned}
\end{equation}

Next, we group together all terms in~\eqref{equ:transformed_quadratic} that are quadratic in $w$,
\begin{equation*}
 \calD_1 \calD_2 \bigl( 3 Q \varphi^2 \bigr) = \calQ(w) + \calQ(w,a)
\end{equation*}
with 
\begin{equation} \label{equ:definition_calQ_w}
 \begin{aligned}
  \calQ(w) &:= 9 (-Q + 5 Q K^2) (P_c \calJ[w])^2 - 36 Q K (P_c \calJ[w]) \calI_1[w] \\
  &\quad \quad + 6 Q (\calI_1[w])^2 + 6 Q (P_c \calJ[w]) w
 \end{aligned}
\end{equation}
and
\begin{equation} \label{equ:definition_calQ_wa}
 \begin{aligned}
   \calQ(w,a) &:= 18 (-Q + 5 Q K^2) Y_0  (P_c \calJ[w]) a + 9 (-Q + 5 Q K^2) Y_0^2 a^2 \\
   &\quad \quad - 36 Q K  Y_0 \calI_1[w] a + 6 Q Y_0 w a.
 \end{aligned}
\end{equation}

We further isolate the resonant part of the quadratic nonlinearities in $\calQ(w)$ in the sense that we peel off parts of $\calQ(w)$ that are easily seen to have better cubic-type decay. To this end we insert the identities \eqref{equ:calI1_in_terms_of_wtilcalI1} and \eqref{equ:calJ_in_terms_of_wtilcalJ}, i.e.,
\begin{equation*}
 \calI_1[w] = K w + \widetilde{\calI}_1[\px w], \quad P_c \calJ[w] = P_c \Bigl( \frac12 K^2 w \Bigr) + P_c \Bigl( \widetilde{\calJ}[\px w] \Bigr),
\end{equation*}
which leads to the decomposition
\begin{equation*}
 \calQ(w) = \calQ_r(w) + \calQ_{nr}(w)
\end{equation*}
with 
\begin{equation} \label{equ:calQr_definition}
 \calQ_r(w) := \alpha_1(x) w^2 + \alpha_2(x) w \langle G, w \rangle + \alpha_3(x) \bigl( \langle G, w \rangle \bigr)^2, \quad G := K^2 Y_0,
\end{equation}
for Schwartz functions
\begin{align*}
 \alpha_1(x) &:= \frac94 (-Q + 5QK^2) K^4 - 18 Q K^4 + 6 Q K^2 + 3 Q K^2 \\
  &\, =  -\frac{9\sqrt{2}}{4}\sinh^2(x) \bigl(\cosh^2(x)-5\bigr)\sech^7(x), \\
 \alpha_2(x) &:= -\frac92 (-Q+5QK^2) K^2 Y_0 + 18 Q K^2 Y_0 - 3 Q Y_0 \\
  &\, = -\frac{3\sqrt{6}}{4} \bigl(2\cosh^4(x) - 15\cosh^2(x) + 15 \bigr) \sech^7(x), \\
  \alpha_3(x) &:= \frac94 (-Q+5QK^2) Y_0^2 \\
  &\, = \frac{27\sqrt{2}}{16} \bigl(4\cosh^2(x)-5\bigr) \sech^7(x).
\end{align*}
All terms in $\calQ_{nr}(w)$ have at least one input of $\wtilcalI_1[\px w]$ or $\wtilcalJ[\px w]$, which thanks to the expected improved local decay of $\px w$ turns all these nonlinear terms into localized terms with cubic-type decay.
Explicitly, we have
 \begin{equation} \label{equ:Qnr_explicit_formula}
 \begin{aligned}
  \calQ_{nr}(w) &:= 9(-Q+5QK^2)P_c \Bigl( \widetilde{\calJ}[\px w] \Bigr) \biggl( P_c \Bigl( K^2 w \Bigr)+P_c \Bigl( \widetilde{\calJ}[\px w] \Bigr) \biggr) \\
  &\quad \, \, - 36QK P_c \Bigl( \widetilde{\calJ}[\px w] \Bigr)  \big(K w + \widetilde{\calI}_1[\px w]\big) \\
  &\quad \, \, + 6Q \widetilde{\calI}_1[\px w] (2Kw+  \widetilde{\calI}_1[\px w]) +6Qw P_c \Bigl( \widetilde{\calJ}[\px w] \Bigr).
 \end{aligned}
\end{equation}

Summarizing, we have obtained a decomposition of the quadratic nonlinearities into
\begin{equation}
 \calD_1 \calD_2 \bigl( 3 Q \varphi^2 \bigr) = \calQ_r(w) + \calQ_{nr}(w) + \calQ(w,a).
\end{equation} 
 
\begin{remark}
The Fourier transforms of the variable coefficients $\alpha_j(x)$, $1 \leq j \leq 3$, are
\begin{align*}
 \widehat{\alpha}_1(\xi) &= - \frac{\sqrt{\pi}}{64} (1+\xi^2) (-1+2\sqrt{7}+\xi^2)(-1-2\sqrt{7}+\xi^2) \sech\Bigl(\frac{\pi \xi}{2}\Bigr), \\
 \widehat{\alpha}_2(\xi) &= -\frac{\sqrt{3\pi}}{64}(1+\xi^2)^2(3+\xi^2) \sech\Bigl(\frac{\pi \xi}{2}\Bigr), \\ 
 \widehat{\alpha}_3(\xi) &= -\frac{3\sqrt{\pi}}{256} (1+\xi^2)^2(9+\xi^2) \sech\Bigl(\frac{\pi \xi}{2}\Bigr). \\ 
\end{align*}
Clearly, we have $\widehat{\alpha}_j(\pm \sqrt{3}) \neq 0$ for $1 \leq j \leq 3$.
Moreover, we compute that $\int_\bbR G(x) \, \ud x = \frac{1}{\sqrt{3}}$ and
\begin{equation*}
 \begin{aligned}
  \widehat{\alpha}_1(\xi)+\widehat{\alpha}_2(\xi) \biggl( \int_\bbR G \biggr) +  \widehat{\alpha}_3(\xi) \biggl( \int_\bbR G \biggr)^2 = - \frac{3 \sqrt{\pi}}{256} \bigl( -29-23\xi^2+9\xi^4+3\xi^6 \bigr) \sech\Bigl(\frac{\pi}{2} \xi\Bigr),
 \end{aligned}
\end{equation*}
which is \eqref{equ:intro_resonance_condition_calFD1D2} up to a constant multiple.
The analogue of the resonance condition~\eqref{equ:intro_resonance_condition} from Lemma~\ref{lem:resonance_condition_focusing_cubic_source_term} for the transformed equation for the variable $w$
now amounts to
 \[
 \widehat{\alpha}_1(\pm \sqrt{3})+\widehat{\alpha}_2(\pm \sqrt{3}) \biggl( \int_\bbR G \biggr) +  \widehat{\alpha}_3(\pm \sqrt{3}) \biggl( \int_\bbR G \biggr)^2 = -\frac{3 \sqrt{\pi}}{4} \sech\Bigl(\frac{\sqrt{3} \pi}{2} \Bigr) \ne 0.
 \]
We used the Wolfram Mathematica software system to compute the preceding identities.
\end{remark}

\subsection{Structure of the transformed cubic nonlinearity} \label{subsec:transformed_cubic}

The transformed cubic term is
\begin{equation*} 
 \begin{aligned}
  \calD_1 \calD_2 ( \varphi^3 ) &= 3 \varphi^2 (\calD_1 \calD_2 \varphi) + 6 \varphi (\calD_2 \varphi)^2 - 24 K \varphi^2 (\calD_2 \varphi) + (-4 + 24K^2) \varphi^3 \\
  &= 3(P_c \calJ[w] + a Y_0)^2 w + 6 (P_c \calJ[w] + a Y_0) (\calI_1[w])^2 \\
  &\quad - 24 K (P_c \calJ[w] + a Y_0)^2 \calI_1[w] + 4(6K^2-1) (P_c \calJ[w] + a Y_0)^3. 
 \end{aligned}
\end{equation*}
Using that $K^2 = 1 - \frac12 Q^2$, whence $4 (6K^2-1) = 20 - 12 Q^2$, we find 
\begin{equation} \label{equ:transformed_cubic1}
 \begin{aligned}
  \calD_1 \calD_2 ( \varphi^3 ) &= 3(P_c \calJ[w] + a Y_0)^2 w + 6 (P_c \calJ[w] + a Y_0) (\calI_1[w])^2 \\
  &\quad - 24 K (P_c \calJ[w] + a Y_0)^2 \calI_1[w] + (20-12 Q^2) (P_c \calJ[w] + a Y_0)^3. 
 \end{aligned}
\end{equation}
Next, we group together all terms in~\eqref{equ:transformed_cubic1} that are cubic in $w$,
\begin{equation*}
 \calD_1 \calD_2 ( \varphi^3 ) = \calC(w) + \calC(w,a)
\end{equation*}
with 
\begin{equation} \label{equ:calCw_definition}
 \begin{aligned}
  \calC(w) &:= 3 (P_c \calJ[w])^2 w + 6 (P_c \calJ[w]) (\calI_1[w])^2 \\
  &\quad \quad - 24 K (P_c \calJ[w])^2 \calI_1[w] + (20-12 Q^2) (P_c \calJ[w])^3
 \end{aligned}
\end{equation}
as well as
\begin{equation*}
 \begin{aligned}
  \calC(w,a) &:=  6 Y_0 (P_c \calJ[w]) w a + 3 Y_0^2 w a^2 + 6 Y_0 (\calI_1[w])^2 a \\
  &\quad - 48 K Y_0 (P_c \calJ[w]) \calI_1[w] a - 24 K Y_0^2 \calI_1[w] a^2 + 3 (20-12 Q^2) Y_0 (P_c \calJ[w])^2 a \\
  &\quad + 3 (20-12 Q^2) Y_0^2 (P_c \calJ[w]) a^2 + (20-12 Q^2) Y_0^3 a^3. 
 \end{aligned}
\end{equation*}

We now want to arrive at a refined decomposition of $\calC(w)$ into singular parts and regular parts (spatially localized terms).
Due to the spatial localization of $Q$, the corresponding contribution of the last term in~\eqref{equ:calCw_definition} is spatially localized. Moreover, the term $\langle Y_0, \calJ[w] \rangle Y_0$ in $P_c \calJ[w]$ is spatially localized. We correspondingly write
\begin{equation*}
 \calC(w) = \calC_{nl}(w) + \calC_l(w)
\end{equation*}
with 
\begin{equation*}
 \begin{aligned}
  \calC_{nl}(w) &:= 3 (\calJ[w])^2 w + 6 \calJ[w] (\calI_1[w])^2 - 24 K (\calJ[w])^2 \calI_1[w] + 20 (\calJ[w])^3 \\
   &\, =: \, \calC_{nl,1}(w) + \calC_{nl,2}(w) + \calC_{nl,3}(w) + \calC_{nl,4}(w) 
 \end{aligned}
\end{equation*}
and
\begin{equation*}
 \begin{aligned}
  \calC_l(w) &:=  3\bigl( - 2 \langle Y_0, \calJ[w] \rangle \calJ[w] Y_0 + (\langle Y_0, \calJ[w] \rangle)^2 Y_0^2 \bigr) w - 6 \langle Y_0, \calJ[w] \rangle Y_0 (\calI_1[w])^2\\
  &\quad -24K \bigl( -2 \calJ[w] \langle Y_0, \calJ[w] \rangle Y_0 + (\langle Y_0, \calJ[w] \rangle)^2 Y_0^2\bigr) \calI_1[w] - 12 Q^2 (\calJ[w] - \langle Y_0, \calJ[w] \rangle Y_0)^3 \\
  &\quad + (20-12Q^2) \bigl( -3(\calJ[w])^2 \langle Y_0, \calJ[w] \rangle Y_0 + 3\calJ[w](\langle Y_0, \calJ[w] \rangle)^2 Y_0^2 - (\langle Y_0, \calJ[w] \rangle)^3 Y_0^3 \bigr).
 \end{aligned}
\end{equation*}

The cubic nonlinearities in $\calC_{nl}(w)$ should be thought of as ``not obviously localized''.
In order to uncover their fine structure, we next compute their Fourier transforms. This will unveil further localized terms.
In the resulting expressions, we will separate the singular and the regular (spatially localized) parts, and correspondingly arrive at decompositions
\begin{equation*}
 \calC_{nl,j}(w) = \calC_{nl,j;S}(w) + \calC_{nl,j;R}(w), \quad 1 \leq j \leq 4.
\end{equation*}
Our analysis will be based on the decompositions~\eqref{equ:decomposition_FT_I1_J_sing_reg} of the Fourier transforms of $\calI_1[w]$ and $\calJ[w]$ into singular and regular parts,
\begin{equation*}
 \widehat{\calI_1[w]} = \widehat{\calI_1[w]}_S + \widehat{\calI_1[w]}_R, \quad \widehat{\calJ[w]} = \widehat{\calJ[w]}_S + \widehat{\calJ[w]}_R
\end{equation*}
with
\begin{equation*}
 \begin{aligned}
  \widehat{\calI_1[w]}_S &:= i \Omega \ast (m_0 \whatw), \\
  \widehat{\calI_1[w]}_R &:= i \omega_1 B_1(\whatw), \\
  \widehat{\calJ[w]}_S &:= \delta_0 \ast (m_4 \whatw) + \Omega \ast (m_5 \whatw), \\
  \widehat{\calJ[w]}_R &:= \omega_3 \ast (m_6 \whatw) + \omega_3 B_2(\whatw) + \omega_2 B_3(\whatw).
 \end{aligned}
\end{equation*}
Since $\Omega \ast \omega_j$ are Schwartz functions for $1 \leq j \leq 3$  by Lemma~\ref{lem:convolutions}, we observe that the singular parts of the Fourier transforms of the cubic nonlinearities $\calC_{nl}(w)$ can only result from the convolutions of the singular parts of the Fourier transforms of the inputs.

\medskip 

\noindent {\it Fourier transform of $\calC_{nl,1}(w)$}: 
We have 
\begin{equation*}
 \begin{aligned}
  \calF\bigl[ \calC_{nl,1}(w) \bigr](\xi) &= \frac{3}{2\pi} \bigl( \widehat{\calJ[w]}_S + \widehat{\calJ[w]}_R \bigr) \ast \bigl( \widehat{\calJ[w]}_S + \widehat{\calJ[w]}_R \bigr) \ast \whatw.
 \end{aligned}
\end{equation*}
Now observe that by \eqref{equ:convolution_Omega_Omega},
\begin{equation*}
 \begin{aligned}
  \widehat{\calJ[w]}_S \ast \widehat{\calJ[w]}_S &= \bigl( \delta_0 \ast (m_4 \whatw) + \Omega \ast (m_5 \whatw) \bigr) \ast \bigl( \delta_0 \ast (m_4 \whatw) + \Omega \ast (m_5 \whatw) \bigr) \\
  &= \delta_0 \ast (m_4 \whatw) \ast (m_4 \whatw) - 4 \delta_0 \ast (m_5 \whatw) \ast (m_5 \whatw) \\
  &\quad + 2 \Omega \ast (m_4 \whatw) \ast (m_5 \whatw) + 2 \omega_3 \ast (m_5 \whatw) \ast (m_5 \whatw).
 \end{aligned}
\end{equation*}
It follows that the singular part of the Fourier transform of $\calC_{nl,1}(w)$ is given by
\begin{equation*}
 \begin{aligned}
  &\calF\bigl[ \calC_{nl,1;S}(w) \bigr](\xi) \\
  &\quad = \frac{3}{2\pi} \Bigl( \delta_0 \ast (m_4 \whatw) \ast (m_4 \whatw) \ast \whatw - 4 \delta_0 \ast (m_5 \whatw) \ast (m_5 \whatw) \ast \whatw + 2 \Omega \ast (m_4 \whatw) \ast (m_5 \whatw) \ast \whatw \Bigr),
 \end{aligned}
\end{equation*}
and that the regular part of the Fourier transform of $\calC_{nl,1}(w)$ is given by
\begin{equation} \label{equ:def_regular_part_calCnl1}
 \begin{aligned}
  &\calF\bigl[ \calC_{nl,1;R}(w) \bigr](\xi) \\
  &\quad = \frac{3}{2\pi} \Bigl( 2\omega_3 \ast (m_5\whatw) \ast (m_5\whatw) \ast \whatw + 2 \widehat{\calJ[w]}_R \ast \widehat{\calJ[w]}_S \ast \whatw + \widehat{\calJ[w]}_R \ast \widehat{\calJ[w]}_R \ast \whatw \Bigr) \\
  &\quad = \frac{3}{2\pi} \Bigl( 2\omega_3 \ast (m_5\whatw) \ast (m_5\whatw)  + 2 \bigl( \omega_3 \ast (m_6 \whatw) + \omega_3 B_2(\whatw) + \omega_2 B_3(\whatw) \bigr) \ast \bigl( m_4 \whatw + \Omega \ast (m_5 \whatw) \bigr) \\
  &\quad \quad \quad \quad + \bigl( \omega_3 \ast (m_6 \whatw) + \omega_3 B_2(\whatw) + \omega_2 B_3(\whatw) \bigr) \ast \bigl( \omega_3 \ast (m_6 \whatw) + \omega_3 B_2(\whatw) + \omega_2 B_3(\whatw) \bigr) \Bigr) \ast \whatw.
 \end{aligned}
\end{equation}
We emphasize again that the convolutions $\Omega \ast \omega_j$, $1 \leq j \leq 3$, are Schwartz by Lemma~\ref{lem:convolutions}.

\medskip 

\noindent {\it Fourier transform of $\calC_{nl,2}(w)$}: 
We have 
\begin{equation*}
 \begin{aligned}
  \calF\bigl[ \calC_{nl,2}(w) \bigr](\xi) &= \frac{3}{\pi} \bigl( \widehat{\calJ[w]}_S + \widehat{\calJ[w]}_R \bigr) \ast \widehat{\calI_1[w]} \ast \widehat{\calI_1[w]}.
 \end{aligned}
\end{equation*}
Observe that by \eqref{equ:convolution_Omega_Omega},
\begin{equation*}
 \begin{aligned}
  \widehat{\calI_1[w]} \ast \widehat{\calI_1[w]} &= \bigl( i \Omega \ast (m_0 \whatw) + i \omega_1 B_1(\whatw) \bigr) \ast \bigl( i \Omega \ast (m_0 \whatw) + i \omega_1 B_1(\whatw) \bigr) \\
  &= 4 \delta_0 \ast (m_0 \whatw) \ast (m_0 \whatw) - 2 \omega_3 \ast (m_0 \whatw) \ast (m_0 \whatw) \\
  &\quad - 2 B_1(\whatw) (\Omega \ast \omega_1) \ast (m_0 \whatw) - \bigl( B_1(\whatw) \bigr)^2 (\omega_1 \ast \omega_1).
 \end{aligned}
\end{equation*}
It follows that the singular part of the Fourier transform of $\calC_{nl,2}(w)$ is given by
\begin{equation*}
 \begin{aligned}
  \calF\bigl[ \calC_{nl,2;S}(w) \bigr](\xi) &= \frac{3}{\pi} \widehat{\calJ[w]}_S \ast (4\delta_0) \ast (m_0 \whatw) \ast (m_0 \whatw) \\
  &= \frac{12}{\pi} \bigl( \delta_0 \ast (m_4 \whatw) + \Omega \ast (m_5 \whatw) \bigr) \ast \delta_0 \ast (m_0 \whatw) \ast (m_0 \whatw) \\
  &= \frac{12}{\pi} \Bigl( \delta_0 \ast (m_4 \whatw) \ast (m_0 \whatw) \ast (m_0 \whatw) + \Omega \ast (m_5 \whatw) \ast (m_0 \whatw) \ast (m_0 \whatw) \Bigr),
 \end{aligned}
\end{equation*}
and that the regular part of the Fourier transform of $\calC_{nl,2}(w)$ is
\begin{equation} \label{equ:def_regular_part_calCnl2}
\begin{aligned}
 &\calF\bigl[ \calC_{nl,2;R}(w) \bigr](\xi) \\
 &\quad =  \frac{3}{\pi}  \Big( \widehat{\calJ[w]}_R  \ast (\widehat{\calI_1[w]} \ast \widehat{\calI_1[w]})_R+ \widehat{\calJ[w]}_R  \ast (\widehat{\calI_1[w]} \ast \widehat{\calI_1[w]})_S+ \widehat{\calJ[w]}_S  \ast (\widehat{\calI_1[w]} \ast \widehat{\calI_1[w]})_R  \Big) \\
 &\quad = \frac{3}{\pi}  \Big( \bigl( \omega_3 \ast (m_6 \whatw) + \omega_3 B_2(\whatw) + \omega_2 B_3(\whatw) \bigr) \\
 &\quad \quad \quad \quad \quad \ast \bigl(- 2 \omega_3 \ast (m_0 \whatw) \ast (m_0 \whatw) - 2 B_1(\whatw) (\Omega \ast \omega_1) \ast (m_0 \whatw) - \bigl( B_1(\whatw) \bigr)^2 (\omega_1 \ast \omega_1) \bigr) \\
 &\quad \quad \quad \quad  + 4 \bigl( \omega_3 \ast (m_6 \whatw) + \omega_3 B_2(\whatw) + \omega_2 B_3(\whatw) \bigr) \ast  (m_0 \whatw) \ast (m_0 \whatw)   \\
 &\quad \quad \quad \quad + \bigl( m_4 \whatw + \Omega \ast (m_5 \whatw) \bigr) \\
 &\quad \quad \quad \quad \quad \ast \bigl( - 2 \omega_3 \ast (m_0 \whatw) \ast (m_0 \whatw) - 2 B_1(\whatw) (\Omega \ast \omega_1) \ast (m_0 \whatw) - \bigl( B_1(\whatw) \bigr)^2 (\omega_1 \ast \omega_1) \bigr) \Big).
\end{aligned}
\end{equation}
As in the case of $\calC_{nl,1;R}(w)$, the coefficient functions are Schwartz.

\medskip 

\noindent {\it Fourier transform of $\calC_{nl,3}(w)$}: 
Using that $\widehat{K} = - i \sqrt{\frac{\pi}{2}} \Omega$ by \cite[Lemma 5.6]{LS1}, we find
\begin{equation*}
 \begin{aligned}
  \calF\bigl[ \calC_{nl,3}(w) \bigr](\xi) &= - \frac{24}{(2 \pi)^{\frac32}} \whatK \ast \widehat{\calJ[w]} \ast \widehat{\calJ[w]} \ast \widehat{\calI_1[w]} = \frac{6 i}{\pi} \Omega \ast  \widehat{\calJ[w]} \ast \widehat{\calJ[w]} \ast \widehat{\calI_1[w]}.
 \end{aligned}
\end{equation*}
Then we compute
\begin{equation*}
 \begin{aligned}
  &\Omega \ast \widehat{\calJ[w]}_S \ast \widehat{\calJ[w]}_S \ast \widehat{\calI_1[w]}_S \\
  &= \Omega \ast \Bigl( \delta_0 \ast (m_4 \whatw) \ast (m_4 \whatw) - 4 \delta_0 \ast (m_5 \whatw) \ast (m_5 \whatw) \\
  &\quad \quad \quad \quad + 2 \Omega \ast (m_4 \whatw) \ast (m_5 \whatw) + 2 \omega_3 \ast (m_5 \whatw) \ast (m_5 \whatw) \Bigr) \ast \bigl( i \Omega \ast (m_0 \whatw) \bigr) \\
  &= \bigl( \delta_0 \ast (m_4 \whatw) \ast (m_4 \whatw) - 4 \delta_0 \ast (m_5 \whatw) \ast (m_5 \whatw) + 2 \Omega \ast (m_4 \whatw) \ast (m_5 \whatw) \bigr) \ast \bigl( i (\Omega \ast \Omega) \ast (m_0 \whatw)  \bigr) \\
  &\quad \quad + 2 \omega_3 \ast (m_5 \whatw) \ast (m_5 \whatw) \ast \bigl( i (\Omega \ast \Omega) \ast (m_0 \whatw) \bigr).
 \end{aligned}
\end{equation*}
Using \eqref{equ:convolution_Omega_Omega}, we find that the singular part of the Fourier transform of $\calC_{nl,3}(w)$ is given by
\begin{equation*}
 \begin{aligned}
  \calF\bigl[ \calC_{nl,3;S}(w) \bigr](\xi) &= \frac{6 i}{\pi} \Bigl( \delta_0 \ast (m_4 \whatw) \ast (m_4 \whatw) - 4 \delta_0 \ast (m_5 \whatw) \ast (m_5 \whatw) + 2 \Omega \ast (m_4 \whatw) \ast (m_5 \whatw)  \Bigr) \\
  &\qquad \qquad \qquad \qquad \qquad \qquad \qquad \qquad \qquad \qquad \qquad \qquad \qquad \ast (-4i \delta_0) \ast (m_0 \whatw) \\
  &= \frac{24}{\pi} \Bigl( \delta_0 \ast (m_4 \whatw) \ast (m_4 \whatw) \ast (m_0 \whatw) - 4 \delta_0 \ast (m_5 \whatw) \ast (m_5 \whatw) \ast (m_0 \whatw) \\
  &\qquad \qquad \qquad \qquad \qquad \qquad \qquad \qquad \qquad \qquad + 2 \Omega \ast (m_4 \whatw) \ast (m_5 \whatw) \ast (m_0 \whatw) \Bigr).
 \end{aligned}
\end{equation*}
The regular part of the Fourier transform of $\calC_{nl,3}(w)$ is
\begin{equation} \label{equ:def_regular_part_calCnl3}
\begin{aligned}
 &\calF\bigl[ \calC_{nl,3;R}(w) \bigr] \\
 &\quad = \frac{6 i}{\pi} \Bigl( \bigl( \delta_0 \ast (m_4 \whatw) \ast (m_4 \whatw) - 4 \delta_0 \ast (m_5 \whatw) \ast (m_5 \whatw) + 2 \Omega \ast (m_4 \whatw) \ast (m_5 \whatw) \bigr) \ast \bigl( 2 i \omega_3 \ast (m_0 \whatw) \bigr) \\
 &\quad \quad \quad \quad + 2 \omega_3 \ast (m_5 \whatw) \ast (m_5 \whatw) \ast \bigl( i (\Omega \ast \Omega) \ast (m_0 \whatw) \Bigr) \\
 &\quad \quad + \frac{6 i}{\pi} \Omega \ast  \Big( \widehat{\calJ[w]}_R \ast \widehat{\calJ[w]} \ast \widehat{\calI_1[w]} + \widehat{\calJ[w]}_S \ast \widehat{\calJ[w]}_R \ast \widehat{\calI_1[w]}+ \widehat{\calJ[w]}_S \ast \widehat{\calJ[w]}_S \ast \widehat{\calI_1[w]}_R\Big).
\end{aligned}
\end{equation}
By inspection, the coefficient functions are in the Schwartz class.

\medskip 

\noindent {\it Fourier transform of $\calC_{nl,4}(w)$}: 
We have 
\begin{equation*}
 \begin{aligned}
  \calF\bigl[ \calC_{nl,4}(w) \bigr](\xi) &= \frac{20}{2\pi} \bigl( \widehat{\calJ[w]}_S + \widehat{\calJ[w]}_R \bigr) \ast \bigl( \widehat{\calJ[w]}_S + \widehat{\calJ[w]}_R \bigr) \ast \bigl( \widehat{\calJ[w]}_S + \widehat{\calJ[w]}_R \bigr).
 \end{aligned}
\end{equation*}
Next, using~\eqref{equ:convolution_Omega_Omega} we compute
\begin{equation*}
 \begin{aligned}
  &\widehat{\calJ[w]}_S \ast \widehat{\calJ[w]}_S \ast \widehat{\calJ[w]}_S \\
  &= \bigl( \delta_0 \ast (m_4 \whatw) + \Omega \ast (m_5 \whatw) \bigr) \ast \bigl( \delta_0 \ast (m_4 \whatw) + \Omega \ast (m_5 \whatw) \bigr) \ast \bigl( \delta_0 \ast (m_4 \whatw) + \Omega \ast (m_5 \whatw) \bigr) \\
  &= \bigl( \delta_0 \ast (m_4 \whatw) \ast (m_4 \whatw) - 4 \delta_0 \ast (m_5 \whatw) \ast (m_5 \whatw) + 2 \Omega \ast (m_4 \whatw) \ast (m_5 \whatw) + 2 \omega_3 \ast (m_5 \whatw) \ast (m_5 \whatw) \bigr) \\
  &\qquad \qquad \qquad \qquad \qquad \qquad \qquad \qquad \qquad \qquad \qquad \qquad \qquad \qquad \qquad \ast \bigl( \delta_0 \ast (m_4 \whatw) + \Omega \ast (m_5 \whatw) \bigr) \\
  &= \delta_0 \ast (m_4 \whatw) \ast (m_4 \whatw) \ast (m_4 \whatw) + \Omega \ast (m_4 \whatw) \ast (m_4 \whatw) \ast (m_5 \whatw) \\
  &\quad - 4 \delta_0 \ast (m_5 \whatw) \ast (m_5 \whatw) \ast (m_4 \whatw) - 4 \Omega \ast (m_5 \whatw) \ast (m_5 \whatw) \ast (m_5 \whatw) \\
  &\quad + 2 \Omega \ast (m_4 \whatw) \ast (m_5 \whatw) \ast (m_4 \whatw) + 2 (\Omega \ast \Omega) \ast (m_4 \whatw) \ast (m_5 \whatw) \ast (m_5 \whatw) \\
  &\quad + 2 \omega_3 \ast (m_5 \whatw) \ast (m_5 \whatw) \ast (m_4 \whatw) + 2 (\omega_3 \ast \Omega) \ast (m_5 \whatw) \ast (m_5 \whatw) \ast (m_5 \whatw).  
 \end{aligned}
\end{equation*}
Discarding the last two terms that have a convolution with the Schwartz function $\omega_3$ and noting that $\Omega \ast \Omega = -4\delta_0 + 2\omega_3$ by \eqref{equ:convolution_Omega_Omega}, we find that the singular part of the Fourier transform of $\calC_{nl,4}(w)$ is given by
\begin{equation*}
 \begin{aligned}
  \calF\bigl[ \calC_{nl,4;S}(w) \bigr](\xi) &= \frac{10}{\pi} \Bigl( \delta_0 \ast (m_4 \whatw) \ast (m_4 \whatw) \ast (m_4 \whatw) + \Omega \ast (m_4 \whatw) \ast (m_4 \whatw) \ast (m_5 \whatw) \\
  &\qquad \quad - 4 \delta_0 \ast (m_5 \whatw) \ast (m_5 \whatw) \ast (m_4 \whatw) - 4 \Omega \ast (m_5 \whatw) \ast (m_5 \whatw) \ast (m_5 \whatw) \\
  &\qquad \quad + 2 \Omega \ast (m_4 \whatw) \ast (m_5 \whatw) \ast (m_4 \whatw) - 8 \delta_0 \ast (m_4 \whatw) \ast (m_5 \whatw) \ast (m_5 \whatw) \Bigr) \\
  &= \frac{10}{\pi} \Bigl( \delta_0 \ast (m_4 \whatw) \ast (m_4 \whatw) \ast (m_4 \whatw) - 12 \delta_0 \ast (m_4 \whatw) \ast (m_5 \whatw) \ast (m_5 \whatw) \\
  &\qquad \quad + 3 \Omega \ast (m_4 \whatw) \ast (m_4 \whatw) \ast (m_5 \whatw) - 4 \Omega \ast (m_5 \whatw) \ast (m_5 \whatw) \ast (m_5 \whatw) \Bigr).
 \end{aligned}
\end{equation*}
Correspondingly, the regular part of the Fourier transform of $\calC_{nl,4}(w)$ is
\begin{equation} \label{equ:def_regular_part_calCnl4}
 \begin{aligned}
  &\calF\bigl[ \calC_{nl,4;R}(w) \bigr] \\
  &\quad = \frac{10}{\pi} \Bigl( 4 \omega_3 \ast (m_4 \whatw) \ast (m_5 \whatw) \ast (m_5 \whatw) + 2 \omega_3 \ast (m_5 \whatw) \ast (m_5 \whatw) \ast (m_4 \whatw) \\
  &\quad \quad \quad \quad \quad + 2 (\omega_3 \ast \Omega) \ast (m_5 \whatw) \ast (m_5 \whatw) \ast (m_5 \whatw) \Bigr) \\
  &\quad \quad + \frac{10}{\pi} \Bigl( \widehat{\calJ[w]}_R \ast \widehat{\calJ[w]} \ast \widehat{\calJ[w]} + \widehat{\calJ[w]}_S \ast \widehat{\calJ[w]}_R \ast \widehat{\calJ[w]} + \widehat{\calJ[w]}_S \ast \widehat{\calJ[w]}_S \ast \widehat{\calJ[w]}_R \Bigr).
 \end{aligned}
\end{equation}
As in the three preceding regular terms, one immediately verifies that the coefficient functions are Schwartz.

\medskip 

Putting things together, we arrive at the following expression for the singular part of the Fourier transform of the cubic nonlinearities $\calC_{nl}(w)$,
\begin{equation*}
 \begin{aligned}
  &\calF\bigl[ \calC_{nl;S}(w) \bigr](\xi) \\
  &= \frac{3}{2\pi} \Bigl( \delta_0 \ast (m_4 \whatw) \ast (m_4 \whatw) \ast \whatw - 4 \delta_0 \ast (m_5 \whatw) \ast (m_5 \whatw) \ast \whatw + 2 \Omega \ast (m_4 \whatw) \ast (m_5 \whatw) \ast \whatw \Bigr) \\
  &\quad + \frac{12}{\pi} \Bigl( \delta_0 \ast (m_4 \whatw) \ast (m_0 \whatw) \ast (m_0 \whatw) + \Omega \ast (m_5 \whatw) \ast (m_0 \whatw) \ast (m_0 \whatw) \Bigr) \\
  &\quad + \frac{24}{\pi} \Bigl( \delta_0 \ast (m_4 \whatw) \ast (m_4 \whatw) \ast (m_0 \whatw) - 4 \delta_0 \ast (m_5 \whatw) \ast (m_5 \whatw) \ast (m_0 \whatw) \\
  &\qquad \qquad \qquad \qquad \qquad \qquad \qquad \qquad \qquad \qquad + 2 \Omega \ast (m_4 \whatw) \ast (m_5 \whatw) \ast (m_0 \whatw) \Bigr) \\
  &\quad + \frac{10}{\pi} \Bigl( \delta_0 \ast (m_4 \whatw) \ast (m_4 \whatw) \ast (m_4 \whatw) - 12 \delta_0 \ast (m_4 \whatw) \ast (m_5 \whatw) \ast (m_5 \whatw) \\
  &\qquad \qquad + 3 \Omega \ast (m_4 \whatw) \ast (m_4 \whatw) \ast (m_5 \whatw) - 4 \Omega \ast (m_5 \whatw) \ast (m_5 \whatw) \ast (m_5 \whatw) \Bigr).
 \end{aligned}
\end{equation*}
Ordering the terms we find 
\begin{equation*}
 \calF\bigl[ \calC_{nl;S}(w) \bigr](\xi) = \calF\bigl[ \calC_{\delta_0}(w) \bigr](\xi) + \calF\bigl[ \calC_{\pvdots}(w) \bigr](\xi),
\end{equation*}
where 
\begin{equation} \label{equ:FT_cubic_interactions_dirac}
 \begin{aligned}
  \calF\bigl[ \calC_{\delta_0}(w) \bigr](\xi) &:= \frac{1}{\pi} \delta_0 \ast \biggl( \frac32 (m_4 \whatw) \ast (m_4 \whatw) \ast \whatw - 6 (m_5 \whatw) \ast (m_5 \whatw) \ast \whatw \\
  &\quad \qquad \qquad \quad + 12 (m_4 \whatw) \ast (m_0 \whatw) \ast (m_0 \whatw) + 24 (m_4 \whatw) \ast (m_4 \whatw) \ast (m_0 \whatw) \\
  &\quad \qquad \qquad \quad - 96 (m_5 \whatw) \ast (m_5 \whatw) \ast (m_0 \whatw) + 10 (m_4 \whatw) \ast (m_4 \whatw) \ast (m_4 \whatw) \\
  &\quad \qquad \qquad \quad - 120 (m_4 \whatw) \ast (m_5 \whatw) \ast (m_5 \whatw) \biggr),
 \end{aligned}
\end{equation}  
and 
\begin{equation} \label{equ:FT_cubic_interactions_pv}
 \begin{aligned}  
  \calF\bigl[ \calC_{\pvdots}(w) \bigr](\xi) &:= \frac{1}{\pi} \Omega \ast \biggl( 3 (m_4 \whatw) \ast (m_5 \whatw) \ast \whatw + 12 (m_5 \whatw) \ast (m_0 \whatw) \ast (m_0 \whatw) \\
  &\quad \qquad \qquad \quad + 48 (m_4 \whatw) \ast (m_5 \whatw) \ast (m_0 \whatw) + 30 (m_4 \whatw) \ast (m_4 \whatw) \ast (m_5 \whatw) \\
  &\quad \qquad \qquad \quad - 40 (m_5 \whatw) \ast (m_5 \whatw) \ast (m_5 \whatw) \biggr).
 \end{aligned}
\end{equation}
We also group together all spatially localized terms that arose in the preceding analysis of the cubic nonlinearities $\calC(w)$, and define
\begin{equation} \label{equ:definition_calCR_w}
 \calC_R(w) := \calC_l(w) + \calC_{nl,1;R}(w) + \calC_{nl,2;R}(w) + \calC_{nl,3;R}(w) + \calC_{nl,4;R}(w).
\end{equation}
Summarizing, we have obtained a refined decomposition of the cubic nonlinearities $\calC(w)$ into
\begin{equation}
 \calC(w) = \calC_{\delta_0}(w) + \calC_{\pvdots}(w) + \calC_R(w).
\end{equation}

\subsection{Final decomposition of the transformed equation}

We have arrived at the following nonlinear Klein-Gordon equation for the transformed variable $w$,
\begin{equation}
 \begin{aligned}
  (\pt^2 - \px^2 + 1) w &= \calQ(w) + \calQ(w,a) + \calC(w) + \calC(w,a),
 \end{aligned}
\end{equation}
which we can write in more refined form as
\begin{equation} \label{equ:w_equ_refer_to}
 \begin{aligned}
  (\pt^2 - \px^2 + 1) w &= \calQ_r(w) + \calQ_{nr}(w) + \calQ(w,a) + \calC_{\delta_0}(w) + \calC_{\pvdots}(w) + \calC_R(w) + \calC(w,a).
 \end{aligned}
\end{equation}
To analyze the long-time behavior of solutions to~\eqref{equ:w_equ_refer_to} it is convenient to pass to the variable
\begin{equation} \label{equ:definition_v}
 v(t) := \frac12 \bigl( w(t) - i\jD^{-1} \pt w(t) \bigr).
\end{equation}
We have $w(t) = v(t) + \bv(t)$, and the variable $v(t)$ is a solution to the first-order nonlinear Klein-Gordon equation
\begin{equation} \label{equ:v_equ_simple_refer_to}
 \begin{aligned}
  (\pt - i\jD) v &= (2i\jD)^{-1} \Bigl( \calQ(v + \bar{v}) + \calQ(v + \bar{v},a) + \calC(v + \bar{v}) + \calC(v + \bar{v},a) \Bigr),
 \end{aligned}
\end{equation}
or in more refined form
\begin{equation} \label{equ:v_equ_refer_to}
 \begin{aligned}
  (\pt - i\jD) v &= (2i\jD)^{-1} \Bigl( \calQ_r(v + \bar{v}) + \calQ_{nr}(v + \bar{v}) + \calQ(v + \bar{v},a) \\
  &\quad \quad \quad \quad \quad \quad \, \, + \calC_{\delta_0}(v + \bar{v}) + \calC_{\pvdots}(v + \bar{v}) + \calC_R(v + \bar{v}) + \calC(v + \bar{v},a) \Bigr),
 \end{aligned}
\end{equation}
subject to the initial condition
\begin{equation*}
 v(0) = v_0 := \frac12 \bigl( w(0) - i \jD^{-1} \pt w(0) \bigr) = \frac12 \bigl( \calD_1 \calD_2 \varphi_0 - i \jD^{-1} \calD_1 \calD_2 \varphi_1 \bigr). 
\end{equation*}
Then the evolution equation for the profile 
\begin{equation} \label{equ:definition_f}
 f(t) := e^{-it\jD} v(t)
\end{equation}
of the solution $v(t)$ is given by
\begin{equation} \label{equ:f_equ_simple_refer_to}
 \begin{aligned}
  \pt f(t) &= (2i\jD)^{-1} e^{-it\jD} \Bigl( \calQ(v + \bar{v}) + \calQ(v + \bar{v},a) + \calC(v + \bar{v}) + \calC(v + \bar{v},a) \Bigr),
 \end{aligned}
\end{equation}
or in more refined form by
\begin{equation} \label{equ:f_equ_refer_to}
 \begin{aligned}
  \pt f(t) &= (2i\jD)^{-1} e^{-it\jD} \Bigl( \calQ_r(v + \bar{v}) + \calQ_{nr}(v + \bar{v}) + \calQ(v + \bar{v},a) \\
  &\qquad \qquad \qquad \quad \quad \quad \, \, + \calC_{\delta_0}(v + \bar{v}) + \calC_{\pvdots}(v + \bar{v}) + \calC_R(v + \bar{v}) + \calC(v + \bar{v},a) \Bigr).
 \end{aligned}
\end{equation}
Recall that the evolution equations for the solution $v(t)$, respectively for its profile $f(t)$, are coupled to the following first-order ODEs for the unstable, respectively stable, coefficients $a_+(t)$ and $a_-(t)$,
\begin{align}
  (\pt - \nu) a_+ &= (2\nu)^{-1} \langle Y_0, (3 Q \varphi^2 + \varphi^3) \rangle, \label{equ:aplus_equ_refer_to} \\
  (\pt + \nu) a_- &= - (2\nu)^{-1} \langle Y_0, (3 Q \varphi^2 + \varphi^3) \rangle, \label{equ:aminus_equ_refer_to}
\end{align}
where 
\begin{equation*}
 \varphi = P_c \calJ[v + \barv] + (a_+ + a_-) Y_0,
\end{equation*}
and subject to the initial conditions
\begin{equation*}
 \begin{aligned}
  a_+(0) &= d, \\
  a_-(0) &= \frac12 \langle Y_0, \varphi_0 - \nu^{-1} \varphi_1 \rangle.
 \end{aligned}
\end{equation*}

\section{Bootstrap Setup and Overview of the Proof of Theorem~\ref{thm:main}} \label{sec:bootstrap_setup}

In this section we formulate the main bootstrap bounds that go into the proof of Theorem~\ref{thm:main},
and we provide an overview of the organization of the remainder of this paper.

\begin{proposition}[Main bootstrap bounds] \label{prop:main_bootstrap}
 There exist absolute constants $0 < \varepsilon_1 \ll 1$, $0 < c \ll 1$, and $C_0 \geq 1$ with the following property:
 Let $(\varphi_0, \varphi_1) \in H^4_x \times H^3_x$ be even and satisfy 
 \begin{equation*}
  \langle Y_0, \nu \varphi_0 + \varphi_1 \rangle = 0.
 \end{equation*}
 Suppose that 
 \begin{equation*}
  \varepsilon := \|\jx (\varphi_0, \varphi_1)\|_{H^4_x \times H^3_x} \leq \varepsilon_1. 
 \end{equation*}
 For $d \in \bbR$ with $|d| \leq (\log(2))^{-2} \varepsilon^{\frac32}$, let $(\varphi, \pt \varphi) \in C([0,T]; H^4_x \times H^3_x)$ be the solution to \eqref{equ:pert_equ_varphi_refer_to} with initial data
 \begin{equation*}
  (\varphi, \pt \varphi)|_{t=0} = (\varphi_0, \varphi_1) + d (Y_0, \nu Y_0)
 \end{equation*}
 on the time interval $[0,T]$ for some 
 \begin{equation*}
  0 < T \leq \exp\bigl(c \varepsilon^{-\frac13}\bigr).
 \end{equation*}
 Define $v(t)$ as in \eqref{equ:definition_v}, and the coefficients $a_-(t)$ and $a_+(t)$ as in \eqref{equ:definition_a_plus_a_minus}. Let $f(t) := e^{-it\jD} v(t)$ be the profile of $v(t)$.
 Set 
 \begin{equation} \label{equ:definition_NT_norm}
  \|f\|_{N_T} := \sup_{0 \leq t \leq T} \, \biggl( \bigl\| \jD^2 f(t) \bigr\|_{L^2_x} + \sup_{n \geq 1} \, \sup_{0 \leq \ell \leq n} \, 2^{-\frac12 \ell} \tau_n(t) \bigl\| \varphi_\ell^{(n)}(\xi) \jxi^2 \partial_\xi \hatf(t, \xi)\bigr\|_{L^2_\xi} \biggr).
 \end{equation} 
 Suppose that the following estimates hold
 \begin{align}
  \|f\|_{N_T} &\leq 4 C_0 \varepsilon, \label{equ:bootstrap1} \\
  \sup_{0 \leq t \leq T} \, \jt \bigl( \log(2+t) \bigr)^{-2} \cdot |a_-(t)| &\leq 4C_0 \varepsilon. \label{equ:bootstrap2}
 \end{align}
 Moreover, assume that the following trapping condition is satisfied 
 \begin{equation} \label{equ:trapping}
  \sup_{0 \leq t \leq T} \, \jt \bigl( \log(2+t) \bigr)^{-2} \cdot |a_+(t)| \leq \bigl( \log(2) \bigr)^{-2} \varepsilon^{\frac32}.
 \end{equation}
 Then the following stronger estimates hold
 \begin{align}
  \|f\|_{N_T} &\leq 2 C_0 \varepsilon, \label{equ:stronger_bootstrap1} \\
  \sup_{0 \leq t \leq T} \, \jt \bigl( \log(2+t) \bigr)^{-2} \cdot |a_-(t)| &\leq 2C_0 \varepsilon. \label{equ:stronger_bootstrap2}
 \end{align} 
\end{proposition}

Improving the bootstrap bounds \eqref{equ:bootstrap1} and \eqref{equ:bootstrap2} will occupy the majority of the remainder of this paper. The proof of Theorem~\ref{thm:main} will be a consequence of these bootstrap bounds and a standard topological shooting argument to select initial data so that the trapping condition~\eqref{equ:trapping} is satisfied.

The next sections are organized as follows:
\begin{itemize}
 \item In Section~\ref{sec:preparation_for_weighted_estimates}, we assemble several technical estimates that will be used repeatedly in the derivation of the weighted energy estimates.

 \item In Section~\ref{sec:basic_bounds}, we derive decay estimates for the solution $v(t)$ and several basic estimates for the profile $f(t)$. Moreover, we establish the $H^2_x$ energy estimate for the profile $f(t)$, improving the $H^2_x$ energy part of the bound~\eqref{equ:bootstrap1}. Additionally, we obtain (stronger) weighted energy estimates for all spatially localized nonlinearities with cubic-type decay, i.e., for the terms $\calQ_{nr}(v + \bar{v})$, $\calQ(v + \bar{v},a)$, $\calC_R(v + \bar{v})$, and $\calC(v + \bar{v},a)$, thus improving the weighted energy parts of the bound \eqref{equ:bootstrap1} for their contributions.
 We also derive decay for the stable coefficient $a_-(t)$, improving the bound~\eqref{equ:bootstrap2}.

 \item In Section~\ref{sec:weighted_main_quadratic}, we establish the weighted energy estimates for the resonant quadratic interactions $\calQ_r(v+\bv)$.

 \item In Section~\ref{sec:weighted_main_cubic}, we deduce the weighted energy estimates for the singular cubic interactions $\calC_{\delta_0}(v+\bv)$ and $\calC_{\pvdots}(v+\bv)$.

 \item In Section~\ref{sec:conclusion_of_proof}, we put together the results from Sections~\ref{sec:basic_bounds}--\ref{sec:weighted_main_cubic} to prove Proposition~\ref{prop:main_bootstrap}. Then we conclude the proof of Theorem~\ref{thm:main} by combining the bootstrap bounds from Proposition~\ref{prop:main_bootstrap} with a topological shooting argument.
\end{itemize}

\section{Preparations for the Weighted Energy Estimates} \label{sec:preparation_for_weighted_estimates}

In this section we collect several technical estimates that will be used frequenctly in the derivation of the weighted energy estimates in the next sections.
We begin by recalling a version of \cite[Proposition 4.9]{LS1}, which furnishes (stronger) weighted energy estimates for the contributions of spatially localized nonlinearities with (at least) cubic-type decay.

\begin{proposition} \label{prop:prop49}
 Let $T > 0$ and let $A \colon [0,T] \to [0,\infty)$ be a monotone increasing function. Assume uniformly for $0 \leq t \leq T$ that 
 \begin{equation} \label{equ:prop49_input_assumption}
  \bigl\| \jx^2 \jD \calN(t) \bigr\|_{L^2_x} \leq \frac{A(t)}{\jt^\thf}.
 \end{equation}
 Then we have uniformly for $0 \leq t \leq T$ that
 \begin{equation} \label{equ:prop49_slow_growth}
  \biggl\| \jxi^2 \pxi \int_0^t (2i\jxi)^{-1} e^{-is\jxi} \widehat{\calN}(s,\xi) \, \ud s \biggr\|_{L^2_\xi} \lesssim A(t) \sqrt{\log(2+t)}.
 \end{equation}
\end{proposition}

\begin{proof}
By direct computation
\begin{equation} \label{equ:proof_prop49_1}
 \begin{aligned}
  &\jxi^2 \pxi \int_0^t (2i\jxi)^{-1} e^{-is\jxi} \widehat{\calN}(s,\xi) \, \ud s \\
  &\quad = - \int_0^t 2^{-1} \cdot s \cdot \xi \jxi^{-1} e^{-is\jxi} \jxi \widehat{\calN}(s,\xi) \, \ud s + \int_0^t (2i)^{-1} e^{-is\jxi} \jxi^2 \pxi \bigl( \jxi^{-1} \widehat{\calN}(s,\xi) \bigr) \, \ud s.
 \end{aligned}
\end{equation}
The assumption \eqref{equ:prop49_input_assumption} immediately gives an acceptable bound on the $L^2_\xi$ norm of the second term on the right-hand side of~\eqref{equ:proof_prop49_1} for times $0 \leq t \leq T$,
\begin{equation} \label{equ:proof_prop49_2}
 \begin{aligned}
  \biggl\| \int_0^t (2i)^{-1} e^{-is\jxi} \jxi^2 \pxi \bigl( \jxi^{-1} \widehat{\calN}(s,\xi) \bigr) \, \ud s \biggr\|_{L^2_\xi} &\lesssim \int_0^t \bigl\| \jD \jx \calN(s) \bigr\|_{L^2_x} \, \ud s \lesssim A(t).
 \end{aligned}
\end{equation}
To estimate the growth of the $L^2_\xi$ norm of the first term on the right-hand side of~\eqref{equ:proof_prop49_1}, we compute
\begin{equation*}
 \begin{aligned}
  &\partial_t \Biggl( \biggl\| \int_0^t s \cdot \xi \jxi^{-1} e^{-is\jxi} \jxi \widehat{\calN}(s,\xi) \, \ud s \biggr\|_{L^2_\xi}^2 \Biggr) \\
  &\quad = 2 \Re \int_0^t s \cdot t \cdot \biggl( \int_\bbR \overline{\xi \jxi^{-1} e^{i(t-s)\jxi} \jxi \widehat{\calN}(s,\xi)} \cdot \xi \widehat{\calN}(t,\xi) \, \ud \xi \biggr) \, \ud s.
 \end{aligned}
\end{equation*}
Using Parseval's theorem, the Cauchy-Schwarz inequality, and the crucial improved local decay estimate \eqref{equ:improved_local_decay_stronger_weights} for the linear Klein-Gordon evolution, we obtain uniformly for all $0 \leq t \leq T$ that
\begin{equation*}
 \begin{aligned}
  &\Biggl| \partial_t \Biggl( \biggl\| \int_0^t s \cdot \xi \jxi^{-1} e^{-is\jxi} \jxi \widehat{\calN}(s,\xi) \, \ud s \biggr\|_{L^2_\xi}^2 \Biggr) \Biggr| \\
  &\quad \lesssim \int_0^t s \cdot t \cdot \bigl\| \jx^{-2} \px \jD^{-1} e^{i(t-s)\jD} \jD \calN(s) \bigr\|_{L^2_x} \bigl\| \jx^2 \px \calN(t)\bigr\|_{L^2_x} \, \ud s \\
  &\quad \lesssim \int_0^t s \cdot t \cdot \frac{1}{\jap{t-s}^{\frac32}} \bigl\| \jx^2 \jD \calN(s) \bigr\|_{L^2_x} \bigl\|\jx^2 \px \calN(t)\bigr\|_{L^2_x} \, \ud s. 
 \end{aligned}
\end{equation*}
By the assumption~\eqref{equ:prop49_input_assumption}, the last line can be bounded by
\begin{equation*}
 \int_0^t s \cdot t \cdot \frac{1}{\jap{t-s}^{\frac32}} \frac{A(s)}{\js^{\thf}} \frac{A(t)}{\jt^{\thf}} \, \ud s \lesssim \frac{A(t)^2}{\jt^{\frac12}} \int_0^t \frac{1}{\jap{t-s}^{\frac32}} \frac{1}{\js^{\hf}} \, \ud s \lesssim \frac{A(t)^2}{\jt}.
\end{equation*}
Integrating in time yields uniformly for all $0 \leq t \leq T$ that
\begin{equation} \label{equ:proof_prop49_3}
 \biggl\| \int_0^t s \cdot \xi \jxi^{-1} e^{-is\jxi} \jxi \widehat{\calN}(s,\xi) \, \ud s \biggr\|_{L^2_\xi}^2 \lesssim A(t)^2 \log(2+t).
\end{equation}
Combining \eqref{equ:proof_prop49_2} and \eqref{equ:proof_prop49_3} proves the asserted estimate~\eqref{equ:prop49_slow_growth}.
\end{proof}

\begin{remark}
 The main difficulty in the proof of Proposition~\ref{prop:prop49} is to deal with the problematic term $s \cdot \jxi^{-1} \xi$ in \eqref{equ:proof_prop49_1}, which arises when the derivative $\pxi$ falls onto the phase of $e^{-is\jxi}$. Note that for spatially localized nonlinearities the input and output frequencies are decorrelated so that one cannot hope to transfer this derivative to the inputs via suitable integration by parts arguments.
 On the negative side the term $s \cdot \jxi^{-1} \xi$ features the badly divergent factor of $s$, but on the positive side the factor $\jxi^{-1} \xi$ leads to better low-frequency behavior, which allows us to bring in the crucial improved local decay estimate~\eqref{equ:improved_local_decay_stronger_weights} for the linear Klein-Gordon evolution.
 This observation was already used in \cite{LLS1}.

 A related observation by Chen-Pusateri~\cite{ChenPus22} is that the arising factor of $\xi$ gives access to $L^\infty_x L^2_t$-type smoothing estimates (and their inhomogeneous versions), see Lemma~3.5, Corollary~3.7, and Subsection~6.1.2 in \cite{ChenPus22}. See also \cite[Lemma 4.2]{KairzhanPusateri22}. This would provide an alternative proof of Proposition~\ref{prop:prop49}.
\end{remark}

In order to bound various trilinear terms that arise in Section~\ref{sec:weighted_main_quadratic} and in Section~\ref{sec:weighted_main_cubic}, we will repeatedly use the following standard estimate.

\begin{lemma} \label{lem:frakm_for_delta_three_inputs}
 Let $\frakm \colon \bbR^3 \to \bbC$ be a Schwartz function, and let $T_\frakm$ be the trilinear operator defined by
 \begin{equation*}
  \calF\bigl[ T_\frakm[f,g,h] \bigr](\xi) := \iint \frakm(\xi, \xi_1, \xi_2) \hatf(\xi_1) \hatg(\xi_2) \hath(\xi-\xi_1-\xi_2) \, \ud \xi_1 \, \ud \xi_2.
 \end{equation*}
 For any exponents $1 \leq p, p_1, p_2, p_3 \leq \infty$ satisfying $\frac{1}{p} = \frac{1}{p_1} + \frac{1}{p_2} + \frac{1}{p_3}$, one has
 \begin{equation*}
  \bigl\| T_\frakm[f, g, h] \bigr\|_{L^p_x(\bbR)} \lesssim \| \calF^{-1}[\frakm] \|_{L^1(\bbR^3)} \|f\|_{L^{p_1}_x(\bbR)} \|g\|_{L^{p_2}_x(\bbR)} \|h\|_{L^{p_3}_x(\bbR)}
 \end{equation*}
 for any Schwartz functions $f,g,h$.
\end{lemma}
\begin{proof}
It suffices to consider
\begin{equation*}
\frakm(\xi, \xi_1, \xi_2) = e^{i(\xi y+\xi_1 y_1 + \xi_2 y_2)}
\end{equation*}
for an arbitrary, but fixed choice of $y, y_1, y_2\in\bbR$.
In that case one obtains
\[
 T_\frakm[f,g,h](x) = 2\pi \, f(x+y+y_1) g(x+y+y_2) h(x+y)
\]
and H\"older's inequality finishes the proof.
\end{proof}

By the same method one can verify the following variants of the preceding lemma.
\begin{lemma}\label{lem:m modified}
Let $\frakm \colon \bbR^3 \to \bbC$ be a Schwartz function, and let $T_\frakm^{(a)}$, $T_\frakm^{(b)}$ be the trilinear operators defined by
\begin{equation*}
 \begin{aligned}
  \calF\bigl[ T_\frakm^{(a)}[f,g,h] \bigr](\xi) &:= \iint \frakm(\xi, \xi_1, \xi_2) \hatf(\xi_1) \hatg(\xi_2) \hath(\xi-\xi_1) \, \ud \xi_1 \, \ud \xi_2, \\
  \calF\bigl[ T_\frakm^{(b)}[f,g,h] \bigr](\xi) &:= \iint \frakm(\xi, \xi_1, \xi_2) \hatf(\xi_1) \hatg(\xi_2) \hath(\xi) \, \ud \xi_1 \, \ud \xi_2.
 \end{aligned}
\end{equation*}
For any exponents $1 \leq p, p_1, p_2 \leq \infty$ satisfying $\frac{1}{p} = \frac{1}{p_1} + \frac{1}{p_2}$, we have
\begin{equation*}
 \begin{aligned}
  \bigl\| T_\frakm^{(a)}[f,g,h] \bigr\|_{L^p_x} \lesssim \bigl\|\calF^{-1}[\frakm]\bigr\|_{L^1(\bbR^3)} \|f\|_{L^{p_1}_x} \|g\|_{L^{\infty}_x} \|h\|_{L^{p_2}_x}.
 \end{aligned}
\end{equation*}
Moreover, for any exponent $1 \leq p \leq \infty$, we have
\begin{equation*}
 \begin{aligned}
  \bigl\| T_\frakm^{(b)}[f,g,h] \bigr\|_{L^p_x} \lesssim \bigl\|\calF^{-1}[\frakm]\bigr\|_{L^1(\bbR^3)} \|f\|_{L^{\infty}_x} \|g\|_{L^{\infty}_x} \|h\|_{L^{p}_x}.
 \end{aligned}
\end{equation*}
\end{lemma}

In Subsection~\ref{subsec:weighted_energy_est_cubic_pv} we will also need the following bound for trilinear terms with a Hilbert-type kernel. See also \cite[Lemma 6.7]{GP20}.

\begin{lemma} \label{lem:frakn_for_pv}
 Let $\frakn \colon \bbR^4 \to \bbC$ be a Schwartz function, and let $T_\frakn$ be the trilinear operator defined by
 \begin{equation*}
  \calF\bigl[ T_\frakn[f,g,h] \bigr](\xi) := \iiint \frakn(\xi,\xi_1,\xi_2,\xi_3) \hatf(\xi_1) \hatg(\xi_2) \hath(\xi-\xi_1-\xi_2-\xi_3) \, \pvdots \cosech\Bigl(\frac{\pi}{2} \xi_3 \Bigr) \, \ud \xi_1 \, \ud \xi_2 \, \ud \xi_3.
 \end{equation*}
 For any exponents $1 \leq p, p_1, p_2, p_3 \leq \infty$ satisfying $\frac{1}{p} = \frac{1}{p_1} + \frac{1}{p_2} + \frac{1}{p_3}$, it holds that 
 \begin{equation*}
  \bigl\| T_\frakn[f, g, h] \bigr\|_{L^p_x(\bbR)} \lesssim \| \calF^{-1}[\frakn] \|_{L^1(\bbR^4)} \|f\|_{L^{p_1}_x(\bbR)} \|g\|_{L^{p_2}_x(\bbR)} \|h\|_{L^{p_3}_x(\bbR)}. 
 \end{equation*}
\end{lemma}
\begin{proof}
Recall from \cite[Lemma 5.6]{LS1} that in the sense of $\calS'(\bbR)$,
\[
\calF[\tanh](\xi)=-i\sqrt{\frac{\pi}{2}}\,  \pvdots \cosech\Bigl(\frac{\pi}{2} \xi \Bigr).
\]
We therefore have
\begin{equation*}
 \begin{aligned}
  &T_\frakn[f,g,h](x) \\
  &= i \sqrt{\frac{2}{\pi}} \frac{1}{(2\pi)^2} \int_{\bbR^8} e^{ix\xi} \frakn(\xi, \xi_1, \xi_2,\xi_3) e^{-ix_1\xi_1} f(x_1)  e^{-ix_2\xi_2} g(x_2) e^{-ix_3(\xi-\xi_1-\xi_2-\xi_3)} h(x_3) \\
  &\qquad \qquad \qquad \qquad \qquad \qquad \qquad \qquad \qquad \quad \, \, \, \, \times  e^{-iu\xi_3} \tanh(u) \, \ud \xi \, \ud \xi_1 \, \ud \xi_2 \, \ud \xi_3 \, \ud u \, \ud x_1 \, \ud x_2 \, \ud x_3 \\
  &= i \sqrt{\frac{2}{\pi}} \int_{\bbR^4} \calF^{-1}[\frakn](x-x_3,x_3-x_1,x_3-x_2, x_3-u) \tanh(u) f(x_1) g(x_2) h(x_3) \, \ud u \, \ud x_1 \, \ud x_2 \, \ud x_3 \\
  &= i \sqrt{\frac{2}{\pi}} \int_{\bbR^4} \calF^{-1}[\frakn](x_3, x_1, x_2, x-u-x_3) \tanh(u) f(x-x_1-x_3) g(x-x_2-x_3) h(x-x_3)  \\
  &\qquad\qquad\qquad\qquad \qquad\qquad\qquad\qquad \qquad\qquad\qquad\qquad \qquad\qquad\qquad \quad \, \, \, \, \, \, \times  \, \ud u \, \ud x_1 \, \ud x_2 \, \ud x_3.
 \end{aligned}
\end{equation*}
Thus, 
\begin{equation*}
 \begin{aligned}
  &\bigl| T_\frakn[f,g,h](x) \bigr| \\
  &\lesssim \int_{\bbR^3} \bigl\| \calF^{-1}[\frakn](x_3, x_1, x_2, u) \bigr\|_{L^1_u(\bbR)} \bigl| f(x-x_1-x_3) \bigr| \bigl| g(x-x_2-x_3) \bigr| \bigl|h(x-x_3)\bigr| \, \ud x_1 \, \ud x_2 \, \ud x_3.
 \end{aligned}
\end{equation*}
 Passing the $L^p_x$ norm inside, applying H\"older's inequality,  and integrating out the other variables finishes the proof. 
\end{proof}

\section{Basic Bounds} \label{sec:basic_bounds}

In this section we derive decay estimates for the solution $v(t)$ defined in \eqref{equ:definition_v} and several basic estimates for the profile $f(t)$ defined in \eqref{equ:definition_f}. These will be used repeatedly in the next sections. Moreover, we establish the $H^2_x$ energy estimate for the profile $f(t)$, we obtain (stronger) weighted energy estimates for all spatially localized nonlinearities with cubic-type decay, and we derive a decay estimate for the stable coefficient $a_-(t)$ defined in \eqref{equ:definition_a_plus_a_minus}.

\subsection{Core bounds on the solution $v(t)$ and its profile $f(t)$}

We begin with a dispersive decay estimate for the solution $v(t)$.

\begin{lemma} \label{lem:Linfty_decay_v}
 Under the assumptions of Proposition~\ref{prop:main_bootstrap}, we have uniformly for all $0 \leq t \leq T$ that
 \begin{equation} \label{equ:Linfty_decay_v}
  \|v(t)\|_{L^\infty_x} = \bigl\|e^{it\jD} f(t) \bigr\|_{L^\infty_x} \lesssim \frac{\log(2+t)}{\jt^\hf} \varepsilon.
 \end{equation}
\end{lemma}
\begin{proof}
 This decay estimate is an immediate consequence of the dispersive decay estimate from Proposition~\ref{prop:dispersive_decay_estimate} and the definition of the $N_T$ norm in \eqref{equ:definition_NT_norm}.
 To this end recall that $\{\tau_n\}_{n=1}^\infty$ is a smooth partition of unity on $[0,\infty)$ such that for every fixed $t \geq 0$ there exists an integer $n \geq 1$ with $\tau_n(t) \geq \frac12$. Thus, for any time $0 \leq t \leq T$ and a corresponding $n \geq 1$ with $\tau_n(t) \geq \frac12$ and therefore $2^n \simeq t$, we have by \eqref{eq:5p1} that
 \begin{equation*}
 \begin{aligned}
  \bigl\| e^{it\jD} f \bigr\|_{L^\infty_x}  &\leq C \frac{n}{\jt^{\frac12}} \Bigl( \|\jD^2 f\|_{L^2_x} + \sup_{0 \leq \ell \leq n} \, 2^{- \frac12 \ell} \bigl\| \varphi_\ell^{(n)}(\xi) \jxi^2 \partial_\xi \hatf(\xi)\bigr\|_{L^2_\xi} \Bigr) \\
  &\leq 2 C \frac{n}{\jt^{\frac12}} \Bigl( \|\jD^2 f\|_{L^2_x} + \sup_{0 \leq \ell \leq n} \, 2^{- \frac12 \ell} \tau_n(t) \bigl\| \varphi_\ell^{(n)}(\xi) \jxi^2 \partial_\xi \hatf(\xi)\bigr\|_{L^2_\xi} \Bigr) \\
  &\lesssim \frac{\log(2+t)}{\jt^{\frac12}} \|f\|_{N_T} \lesssim \frac{\log(2+t)}{\jt^\hf} \varepsilon,
 \end{aligned}
 \end{equation*}
 as claimed.
\end{proof}

Next, we obtain local decay estimates for the solution $v(t)$.

\begin{lemma} 
 Under the assumptions of Proposition~\ref{prop:main_bootstrap}, we have uniformly for all $0 \leq t \leq T$ that
 \begin{align}
  \bigl\| \jx^{-1} \px v(t)\bigr\|_{H^1_x} &\lesssim \frac{\log(2+t)}{\jt} \varepsilon, \label{equ:local_decay_px_v} \\
  \bigl\| \jx^{-1} (\jD-1) v(t)\bigr\|_{H^1_x} &\lesssim \frac{\log(2+t)}{\jt} \varepsilon. \label{equ:local_decay_jD_minus_one_v}
 \end{align}
 Moreover, for any integer $k \leq 0$ we have 
 \begin{equation} \label{equ:local_decay_pxPk_v}
  \bigl\| \jx^{-1} \px P_{\leq k} v(t)\bigr\|_{L^2_x} + \bigl\| \jx^{-1} \px P_{>k} v(t)\bigr\|_{L^2_x} \lesssim \frac{\log(2+t)}{\jt} \varepsilon
 \end{equation}
 as well as
 \begin{equation} \label{equ:local_decay_Pk_v}
  \bigl\| \jx^{-1} P_k v(t)\bigr\|_{L^2_x} + \bigl\| \jx^{-1} P_{>k} v(t) \bigr\|_{L^2_x} \lesssim 2^{|k|} \cdot \frac{\log(2+t)}{\jt} \varepsilon.
 \end{equation}
\end{lemma}
\begin{proof}
 By the same reasoning as in the preceding proof of Lemma~\ref{lem:Linfty_decay_v}, the asserted local decay bounds \eqref{equ:local_decay_px_v}, \eqref{equ:local_decay_jD_minus_one_v}, \eqref{equ:local_decay_pxPk_v} and \eqref{equ:local_decay_Pk_v} for $v(t)$ are immediate corollaries of the local decay estimates for the linear Klein-Gordon evolution from Proposition~\ref{prop:local_decay_estimates} and the definition of the $N_T$ norm in~\eqref{equ:definition_NT_norm}.
\end{proof}

We also use the following variants of improved local decay for the solution $v(t)$.

\begin{lemma}
 Under the assumptions of Proposition~\ref{prop:main_bootstrap}, we have uniformly for all $0 \leq t \leq T$ that
 \begin{align}
  \bigl\| \jx^{-2} \bigl( v(t,x) - v(t,0) \bigr) \bigr\|_{L^2_x} &\lesssim\frac{\log(2+t)}{\jt} \varepsilon, \label{equ:improved_L2decay_v_minus_v0} \\
  \bigl\| \jx^{-2} \bigl( v(t,x) - v(t,0) \bigr) \bigr\|_{L^\infty_x} &\lesssim \frac{\log(2+t)}{\jt} \varepsilon. \label{equ:improved_Linftydecay_v_minus_v0}
 \end{align}
\end{lemma}
\begin{proof}
 We begin with the proof of~\eqref{equ:improved_L2decay_v_minus_v0}. By the fundamental theorem of calculus we can write
 \begin{equation*}
 \begin{aligned}
  \jx^{-2} \bigl( v(t,x) - v(t,0) \bigr) &= \int_\bbR K(x,y) \jap{y}^{-1} \partial_y v(s,y) \, \ud y
 \end{aligned}
 \end{equation*}
 with 
 \begin{equation*}
  K(x,y) := \bigl( \mathds{1}_{[0,\infty)}(x) \mathds{1}_{[0,x]}(y) - \mathds{1}_{(-\infty,0)}(x) \mathds{1}_{[x,0]}(y) \bigr) \jx^{-2} \jap{y}.
 \end{equation*}
 Using Schur's test and the local decay estimate~\eqref{equ:local_decay_px_v}, we then conclude for $0 \leq t \leq T$ that
 \begin{equation*}
  \bigl\| \jx^{-2} \bigl( v(t,x) - v(t,0) \bigr) \bigr\|_{L^2_x} \lesssim \bigl\| \jx^{-1} \px v(t) \bigr\|_{L^2_x} \lesssim \jt^{-1} \log(2+t) \varepsilon.
 \end{equation*}
 This proves \eqref{equ:improved_L2decay_v_minus_v0}. The estimate~\eqref{equ:improved_Linftydecay_v_minus_v0} follows by Sobolev embedding from \eqref{equ:improved_L2decay_v_minus_v0} and the local decay estimate~\eqref{equ:local_decay_px_v}.
\end{proof}

Next, we infer several basic bounds for the time derivative of the profile $f(t)$.

\begin{lemma} \label{lem:pt_f_bounds}
 Under the assumptions of Proposition~\ref{prop:main_bootstrap}, we have uniformly for all $0 \leq t \leq T$ that
 \begin{align}
  \bigl\| \jD \pt f(t) \bigr\|_{L^2_x} &\lesssim \frac{\bigl(\log(2+t)\bigr)^2}{\jt} \varepsilon^2, \label{equ:pt_f_H1_bound} \\
  \bigl\| \pt f(t) \bigr\|_{L^\infty_x} &\lesssim \frac{\bigl(\log(2+t)\bigr)^2}{\jt} \varepsilon^2, \label{equ:pt_f_Linftyx_bound} \\
  \bigl\| \jxi \pxi \pt \hatf(t) \bigr\|_{L^2_\xi} &\lesssim \bigl(\log(2+t)\bigr)^2 \varepsilon^2. \label{equ:pt_f_pxi_L2_bound}
 \end{align}
\end{lemma}
\begin{proof}
 We begin with the proof of the bound~\eqref{equ:pt_f_H1_bound}. In view of the evolution equation~\eqref{equ:f_equ_simple_refer_to} for the profile $f(t)$, we have 
 \begin{equation} \label{equ:pt_f_H1_bound_equ1}
 \begin{aligned}
  \bigl\| \jD \pt f(t) \bigr\|_{L^2_x} &\lesssim \bigl\| \calQ\bigl(v(t) + \bar{v}(t)\bigr)\bigr\|_{L^2_x} + \bigl\|\calQ\bigl(v(t) + \bar{v}(t), a(t)\bigr)\bigr\|_{L^2_x} \\
  &\quad + \bigl\|\calC\bigl( v(t) + \bar{v}(t) \bigr)\bigr\|_{L^2_x} + \bigl\|\calC\bigl(v(t) + \bar{v}(t),a(t)\bigr)\bigr\|_{L^2_x}.
 \end{aligned}
 \end{equation}
 Invoking the decay estimate~\eqref{equ:Linfty_decay_v} for $v(t)$ along with the assumed decay estimates \eqref{equ:bootstrap2} and \eqref{equ:trapping} for the coefficients $a_-(t)$, $a_+(t)$, and the simple bounds from Lemma~\ref{lem:I1_and_J_bounds} for the integral operators $\calI_1$ and $\calJ$, it is straightforward to conclude that all terms on the right-hand side of \eqref{equ:pt_f_H1_bound_equ1} decay at least like $C \jt^{-1} (\log(2+t))^2 \varepsilon^2$. For instance, for the first term in the expression \eqref{equ:definition_calQ_w} for $\calQ(v+\barv)$, we have 
 \begin{equation*}
 \begin{aligned}
  \bigl\| 9 (-Q + 5 Q K^2) \bigl(P_c \calJ[v(t)+\bar{v}(t)]\bigr)^2 \bigr\|_{L^2_x} &\lesssim \|Q\|_{L^2_x} \|P_c \calJ[v(t) + \bar{v}(t)]\|_{L^\infty_x}^2 \\
  &\lesssim \|Q\|_{L^2_x} \|v(t)\|_{L^\infty_x}^2 \lesssim \jt^{-1} \bigl(\log(2+t)\bigr)^2 \varepsilon^2.
 \end{aligned}
 \end{equation*}
 All other estimates to conclude the proof of~\eqref{equ:pt_f_H1_bound} proceed similarly.
 Moreover, the asserted estimate~\eqref{equ:pt_f_Linftyx_bound} follows from~\eqref{equ:pt_f_H1_bound} by Sobolev embedding.
 
 It remains to prove~\eqref{equ:pt_f_pxi_L2_bound}. From the evolution equation~\eqref{equ:f_equ_simple_refer_to} for $f(t)$, we compute that
 \begin{equation*}
 \begin{aligned}
  \jxi \pxi \pt \hatf(t,\xi) &= - 2^{-1} \cdot t \cdot \xi \jxi^{-1} e^{-it\jxi} \calF\bigl[ \calQ(v + \bar{v}) + \calQ(v + \bar{v},a) + \calC(v + \bar{v}) + \calC(v + \bar{v},a) \bigr](\xi) \\
  &\quad + (2i)^{-1} e^{-it\jxi} \pxi \calF\bigl[ \calQ(v + \bar{v}) + \calQ(v + \bar{v},a) + \calC(v + \bar{v}) + \calC(v + \bar{v},a) \bigr](\xi) \\
  &\quad + \bigl\{ \text{lower order terms} \bigr\}.
 \end{aligned}
 \end{equation*}
 Hence, 
 \begin{equation*}
 \begin{aligned}
  \bigl\| \jxi \pxi \pt \hatf(t,\xi) \bigr\|_{L^2_\xi} &\lesssim \jt \Bigl( \bigl\| \calQ\bigl(v(t) + \bar{v}(t)\bigr)\bigr\|_{L^2_x} + \bigl\|\calQ\bigl(v(t) + \bar{v}(t), a(t)\bigr)\bigr\|_{L^2_x} \\
  &\quad \quad \quad + \bigl\|\calC\bigl( v(t) + \bar{v}(t) \bigr)\bigr\|_{L^2_x} + \bigl\|\calC\bigl(v(t) + \bar{v}(t),a(t)\bigr)\bigr\|_{L^2_x} \Bigr) \\
  &\quad + \Bigl( \bigl\| \jx \calQ\bigl(v(t) + \bar{v}(t)\bigr)\bigr\|_{L^2_x} + \bigl\|\jx \calQ\bigl(v(t) + \bar{v}(t), a(t)\bigr)\bigr\|_{L^2_x} \\
  &\quad \quad \quad + \bigl\|\jx \calC\bigl( v(t) + \bar{v}(t) \bigr)\bigr\|_{L^2_x} + \bigl\|\jx \calC\bigl(v(t) + \bar{v}(t),a(t)\bigr)\bigr\|_{L^2_x} \Bigr) \\
  &=: \jt I(t) + II(t).
 \end{aligned}
 \end{equation*}
 By the same arguments as in the preceding proof of~\eqref{equ:pt_f_H1_bound}, we have $I(t) \lesssim \jt^{-1} (\log(2+t))^2 \varepsilon^2$, which implies the acceptable bound $\jt I(t) \lesssim (\log(2+t))^2 \varepsilon^2$. We claim that the second term obeys the bound $II(t) \lesssim (\log(2+t))^2 \varepsilon^2$. To see this, we first observe that thanks to the spatial localization of all quadratic terms $\calQ(v+\bar{v})$ and $\calQ(v+\bar{v},a)$ as well as of the cubic terms $\calC(v+\barv, a)$, we obtain by similar arguments as in the preceding proof of~\eqref{equ:pt_f_H1_bound} the bound
 \begin{equation*}
 \begin{aligned}
  \bigl\| \jx \calQ\bigl(v(t) + \bar{v}(t)\bigr)\bigr\|_{L^2_x} + \bigl\|\jx \calQ\bigl(v(t) + \bar{v}(t), a(t)\bigr)\bigr\|_{L^2_x} + &\bigl\|\jx \calC\bigl(v(t) + \bar{v}(t),a(t)\bigr)\bigr\|_{L^2_x} \\
  &\quad \quad \quad \quad \lesssim \jt^{-1} \bigl( \log(2+t) \bigr)^2 \varepsilon^2.
 \end{aligned}
 \end{equation*}
 To estimate the contributions of the non-localized cubic terms $\calC(v+\barv)$, we record that in view of the definition~\eqref{equ:definition_NT_norm} of the $N_T$ norm, we have the crude bound 
 \begin{equation*}
 \begin{aligned}
  \|\jx v(t)\|_{L^2_x} &\lesssim \|f(t)\|_{L^2_x} + \bigl\| \pxi \bigl( e^{it\jxi} \hatf(t,\xi) \bigr) \bigr\|_{L^2_\xi} \lesssim \jt \|f(t)\|_{L^2_x} + \| \pxi \hatf(t,\xi) \|_{L^2_\xi} \lesssim \jt \|f\|_{N_T} \lesssim \jt \varepsilon.
 \end{aligned}
 \end{equation*}
 Then together with the bounds from Lemma~\ref{lem:I1_and_J_bounds} on the integral operators $\calI_1$ and $\calJ$, it is straightforward to conclude 
 \begin{equation*}
 \begin{aligned}
   \bigl\|\jx \calC\bigl( v(t) + \bar{v}(t) \bigr)\bigr\|_{L^2_x} \lesssim \|v(t)\|_{L^\infty_x}^2 \|\jx v(t)\|_{L^2_x} \lesssim \jt^{-1} \bigl( \log(2+t) \bigr)^2 \varepsilon^2 \cdot \jt \varepsilon \lesssim \bigl( \log(2+t) \bigr)^2 \varepsilon^3.
 \end{aligned}
\end{equation*}
 Combining the preceding estimates finishes the proof of~\eqref{equ:pt_f_pxi_L2_bound}.
\end{proof}

In the derivation of the weighted energy estimates for the main quadratic interactions in Section~\ref{sec:weighted_main_quadratic}, we use on a few occasions the following improved decay estimate (at the origin) for the time derivative of the phase-filtered solution.

\begin{lemma} \label{lem:pt_phase_filtered_vt0_bound}
 Under the assumptions of Proposition~\ref{prop:main_bootstrap}, we have uniformly for all $0 \leq t \leq T$ that
 \begin{equation} \label{equ:pt_phase_filtered_vt0_bound}
  \bigl| \pt \bigl( e^{-it} v(t,0) \bigr) \bigr| \lesssim \frac{1}{\jt} \Bigl( \log(2+t) \varepsilon + \bigl( \log(2+t) \bigr)^2 \varepsilon^2 \Bigr).
 \end{equation}
\end{lemma}
\begin{proof}
 We write
 \begin{equation*}
  \pt \bigl( e^{-it} v(t,0) \bigr) = i e^{-it} \bigl[ (\jD - 1) v \bigr](t,0) + e^{-it} \bigl[ (\pt - i\jD) v \bigr](t,0).
 \end{equation*}
 Then we obtain by Sobolev embedding, \eqref{equ:local_decay_jD_minus_one_v}, and \eqref{equ:pt_f_H1_bound},
 \begin{equation*}
 \begin{aligned}
  \bigl| \pt \bigl( e^{-it} v(t,0) \bigr) \bigr| &\lesssim \bigl\| \jx^{-1} (\jD - 1)v(t) \bigr\|_{L^\infty_x} + \bigl\|(\pt - i\jD) v(t) \bigr\|_{L^\infty_x} \\
  &\lesssim \bigl\| \jx^{-1} (\jD - 1)v(t) \bigr\|_{H^1_x} + \bigl\|(\pt - i\jD) v(t) \bigr\|_{H^1_x} \\
  &\lesssim \bigl\| \jx^{-1} (\jD - 1)v(t) \bigr\|_{H^1_x} + \bigl\|\jD \pt f(t) \bigr\|_{L^2_x} \\
  &\lesssim \jt^{-1} \Bigl( \log(2+t) \varepsilon + \bigl( \log(2+t) \bigr)^2 \varepsilon^2 \Bigr),
 \end{aligned}
 \end{equation*}
 as claimed.
\end{proof}

\subsection{$H^2_x$ energy estimate}

We are now in the position to establish the $H^2_x$ energy estimate for the profile $f(t)$.

\begin{proposition} \label{prop:H2_energy_estimate}
 Under the assumptions of Proposition~\ref{prop:main_bootstrap}, we have uniformly for all $0 \leq t \leq T$,
 \begin{equation} \label{equ:H2_energy_estimate}
  \bigl\| \jD^2 f(t) \bigr\|_{L^2_x} \lesssim \varepsilon + \bigl( \log(2+t) \bigr)^3 \varepsilon^2.
 \end{equation}
\end{proposition}
\begin{proof}
From the integral formulation of the evolution equation~\eqref{equ:f_equ_simple_refer_to} for the profile $f(t)$, we obtain 
\begin{equation} \label{equ:H2_energy_estimate_equ1}
 \begin{aligned}
  \bigl\|\jD^2 f(t)\bigr\|_{L^2_x} &\lesssim \|\jD^2 v_0\|_{L^2_x} \\
  &\quad + \int_0^t \Bigl( \bigl\| \jD \calQ\bigl( v(s) + \bar{v}(s) \bigr) \bigr\|_{L^2_x} + \bigl\|\jD \calQ\bigl(v(s) + \bar{v}(s), a(s)\bigr)\bigr\|_{L^2_x} \\
  &\quad \quad \quad \quad \quad + \bigl\| \jD \calC\bigl(v(s)+\bar{v}(s)\bigr)\bigr\|_{L^2_x} + \bigl\|\jD \calC\bigl(v(s) + \bar{v}(s),a(s)\bigr)\bigr\|_{L^2_x} \Bigr) \, \ud s.
 \end{aligned}
\end{equation}
Then we have 
\begin{equation} \label{equ:H2_energy_estimate_equ2}
 \|\jD^2 v_0\|_{L^2_x} \lesssim \|\jD^2 \calD_1 \calD_2 \varphi_0\|_{L^2_x} + \|\jD \calD_1 \calD_2 \varphi_1\|_{L^2_x} \lesssim \|(\varphi_0, \varphi_1)\|_{H^4_x \times H^3_x} \lesssim \varepsilon.
\end{equation}
Moreover, using the bounds from Lemma~\ref{lem:I1_and_J_bounds} for the integral operators $\calI_1$ and $\calJ$ along with the decay estimates \eqref{equ:Linfty_decay_v}, \eqref{equ:local_decay_px_v}, \eqref{equ:bootstrap2}, and \eqref{equ:trapping}, we infer for $0 \leq s \leq t \leq T$ the following bounds for the (spatially localized) quadratic terms among the integrands in \eqref{equ:H2_energy_estimate_equ1},
\begin{equation} \label{equ:H2_energy_estimate_equ3}
 \begin{aligned}
  &\bigl\| \jD \calQ\bigl( v(s) + \bar{v}(s) \bigr) \bigr\|_{L^2_x} + \bigl\|\jD \calQ\bigl(v(s) + \bar{v}(s), a(s)\bigr)\bigr\|_{L^2_x} \\
  &\lesssim \bigl( \|v(s)\|_{L^\infty_x} + \|\jx^{-1} \px v(s)\|_{L^2_x} \bigr) \bigl( \|v(s)\|_{L^\infty_x} + |a_+(s)| + |a_-(s)| \bigr) \\
  &\lesssim \js^{-1} \bigl( \log(2+s) \bigr)^2 \varepsilon^2,
 \end{aligned}
\end{equation}
and the following bounds for the cubic terms among the integrands in \eqref{equ:H2_energy_estimate_equ1}
\begin{equation} \label{equ:H2_energy_estimate_equ4}
 \begin{aligned}
  &\bigl\| \jD \calC\bigl( v(s) + \bar{v}(s) \bigr) \bigr\|_{L^2_x} + \bigl\| \jD \calC\bigl( v(s) + \bar{v}(s), a(s) \bigr) \bigr\|_{L^2_x} \\
  &\lesssim \bigl( \|\jD v(s)\|_{L^2_x} + \|v(s)\|_{L^\infty_x} + |a_+(s)| + |a_-(s)| \bigr) \bigl( \|v(s)\|_{L^\infty_x} + |a_+(s)| + |a_-(s)| \bigr)^2 \\
  &\lesssim \js^{-1} \bigl( \log(2+s) \bigr)^2 \varepsilon^3.
 \end{aligned}
\end{equation}
The asserted energy estimate \eqref{equ:H2_energy_estimate} now follows from \eqref{equ:H2_energy_estimate_equ2}, \eqref{equ:H2_energy_estimate_equ3}, and \eqref{equ:H2_energy_estimate_equ4} upon integrating in time.
\end{proof}

\subsection{Weighted energy estimates for localized interactions with cubic-type decay}

In this subsection we consider the weighted energy estimates for all spatially localized interactions with cubic-type decay, namely for $\calQ_{nr}(v+\bar{v})$, $\calQ(v+\bar{v},a)$, $\calC_R(v+\bar{v})$, and $\calC(v+\bar{v},a)$. To state the outcome in a succinct manner, we introduce the short-hand notation
\begin{equation} \label{equ:definition_calRc_cubic_type}
 \calR_c(v+\barv,a) := \calQ_{nr}(v+\bar{v}) + \calQ(v+\bar{v},a) + \calC_R(v+\bar{v}) + \calC(v+\bar{v},a).
\end{equation}
Using Proposition~\ref{prop:prop49} we obtain the following stronger weighted energy estimates.

\begin{proposition} \label{prop:weighted_energy_localized_cubic_type}
 Under the assumptions of Proposition~\ref{prop:main_bootstrap}, we have uniformly for all $0 \leq t \leq T$,
 \begin{equation} \label{equ:weighted_energy_localized_cubic_type}
  \begin{aligned}
   &\biggl\| \jxi^2 \partial_\xi \int_0^t (2i\jxi)^{-1} e^{-is\jxi} \calF\bigl[ \calR_c\bigl(v(s)+\barv(s),a(s)\bigr) \bigr](\xi) \, \ud s \biggr\|_{L^2_\xi} \\
   &\qquad \qquad \qquad \qquad \qquad \qquad \qquad  \lesssim \bigl( \log(2+t) \bigr)^{\frac52} \varepsilon^2 + \bigl( \log(2+t) \bigr)^{\frac72} \varepsilon^{\frac52}.
  \end{aligned}
 \end{equation}
\end{proposition}
\begin{proof}
The asserted stronger weighted energy estimate~\eqref{equ:weighted_energy_localized_cubic_type} is an immediate consequence of Proposition~\ref{prop:prop49} upon showing that for all $0 \leq t \leq T$, 
\begin{equation} \label{equ:weighted_energy_localized_cubic_type_proof_equ1}
 \bigl\| \jx^2 \jD \calR_c\bigl(v(t)+\barv(t),a(t)\bigr) \bigr\|_{L^2_x} \lesssim \jt^{-\frac32} \Bigl( \bigl( \log(2+t) \bigr)^2 \varepsilon^2 + \bigl( \log(2+t) \bigr)^3 \varepsilon^{\frac52} \Bigr).
\end{equation} 
We discuss the proof of \eqref{equ:weighted_energy_localized_cubic_type_proof_equ1} separately for all four types of nonlinear terms in~\eqref{equ:definition_calRc_cubic_type}, starting with the contributions of the ``non-resonant quadratic terms'' $\calQ_{nr}(v+\bar{v})$ defined in~\eqref{equ:Qnr_explicit_formula}.
For instance, the first term on the right-hand side of~\eqref{equ:Qnr_explicit_formula} features the following quadratic interactions
\begin{equation} \label{equ:weighted_energy_localized_cubic_type_proof_equ2}
 \begin{aligned}
  &9 (-Q+5QK^2) P_c \Bigl( K^2 (v+\bv) \Bigr) P_c \Bigl( \widetilde{\calJ}[\px v + \px \bar{v}] \Bigr) \\
  &= 9 (-Q+5QK^2) \Bigl( K^2 (v+\bv) - \langle Y_0, K^2 (v+\bv) \rangle Y_0 \Bigr) \\
  &\qquad \qquad \qquad \qquad \quad \times \Bigl( \widetilde{\calJ}[\px v + \px \bar{v}] - \langle Y_0, \widetilde{\calJ}[\px v + \px \bar{v}] \rangle Y_0 \Bigr) \\
  &= \gamma_1(x) (v+\bv) \widetilde{\calJ}[\px v + \px \bv] + \gamma_2(x) (v+\bv) \langle Y_0, \widetilde{\calJ}[\px v + \px \bar{v}] \rangle \\
  &\quad + \gamma_3(x) \langle K^2 Y_0, (v+\bv) \rangle \widetilde{\calJ}[\px v + \px \bar{v}] + \gamma_4(x) \langle K^2 Y_0, (v+\bv) \rangle \langle Y_0, \widetilde{\calJ}[\px v + \px \bar{v}] \rangle 
 \end{aligned}
\end{equation}
with the Schwartz functions
\begin{equation*}
 \begin{aligned}
  \gamma_1(x) &= 9 (-Q+5QK^2) K^2, \quad \gamma_2(x) = - 9 (-Q+5QK^2) K^2 Y_0, \\
  \gamma_3(x) &= 9 (-Q+5QK^2) Y_0, \quad \, \, \gamma_4(x) = 9 (-Q+5QK^2) Y_0^2.
 \end{aligned}
\end{equation*}
Then we have for the first term on the right-hand side of~\eqref{equ:weighted_energy_localized_cubic_type_proof_equ2}, using Lemma~\ref{lem:I1_and_J_bounds} and the decay estimates \eqref{equ:Linfty_decay_v}, \eqref{equ:local_decay_px_v},
\begin{equation*}
 \begin{aligned}
  &\bigl\|\jx^2 \jD \bigl( \gamma_1(x) \bigl(v(t)+\bv(t)\bigr) \widetilde{\calJ}\bigl[\px v(t) + \px \bv(t)\bigr]\bigr) \bigr\|_{L^2_x} \\
  &\lesssim \bigl\| \jx^3 \gamma_1(x) \bigr\|_{W^{1,\infty}_x} \|v(t)\|_{L^\infty_x} \bigl\|\jx^{-1} \widetilde{\calJ}[\px v(t)] \bigr\|_{L^2_x} \\
  &\quad + \bigl\| \jx^4 \gamma_1(x) \bigr\|_{L^{\infty}_x} \bigl\|\jx^{-1} \px v(t)\bigr\|_{L^2_x} \bigl\|\jx^{-1} \widetilde{\calJ}[\px v(t)] \bigr\|_{L^\infty_x} \\
  &\quad + \bigl\| \jx^3 \gamma_1(x) \bigr\|_{L^{\infty}_x} \|v(t)\|_{L^\infty_x} \bigl\|\jx^{-1} \px \widetilde{\calJ}[\px v(t)] \bigr\|_{L^2_x} \\
  &\lesssim \bigl\| \jx^4 \gamma_1(x) \bigr\|_{W^{1,\infty}_x} \bigl( \|v(t)\|_{L^\infty_x} + \bigl\|\jx^{-1} \px v(t)\bigr\|_{L^2_x} \bigr) \bigl\|\jx^{-1} \px v(t)\bigr\|_{L^2_x} \\
  &\lesssim \jt^{-\frac32} \bigl( \log(2+t) \bigr)^2 \varepsilon^2.
 \end{aligned}
\end{equation*}
In a similar manner, we obtain the same decay estimates for all other terms on the right-hand side of \eqref{equ:weighted_energy_localized_cubic_type_proof_equ2}, and in fact for all other terms in $\calQ_{nr}(v(t)+\bar{v}(t))$, whence
\begin{equation*}
 \begin{aligned}
  \bigl\|\jx^2 \jD \calQ_{nr}\bigl(v(t)+\bar{v}(t)\bigr) \bigr\|_{L^2_x} &\lesssim \bigl( \|v(t)\|_{L^\infty_x} + \bigl\|\jx^{-1} \px v(t)\bigr\|_{L^2_x} \bigr) \bigl\|\jx^{-1} \px v(t)\bigr\|_{L^2_x} \\
  &\lesssim \jt^{-\frac32} \bigl( \log(2+t) \bigr)^2 \varepsilon^2.
 \end{aligned}
\end{equation*}

Proceeding analogously, we obtain for the quadratic terms $\calQ\bigl(v(t)+\bar{v}(t), a(t) \bigr)$ that
\begin{equation}
\begin{aligned}
 &\bigl\|\jx^2 \jD \calQ\bigl(v(t)+\bar{v}(t), a(t) \bigr) \bigr\|_{L^2_x} \\
 &\lesssim \bigl( \|v(t)\|_{L^\infty_x} + \|\jx^{-1} \px v(t)\|_{L^2_x} + |a_+(t)| + |a_-(t)| \bigr) \bigl( |a_+(t)| + |a_-(t)| \bigr) \\
 &\lesssim \jt^{-\frac12} \log(2+t) \varepsilon \Bigl( \jt^{-1} \bigl( \log(2+t) \bigr)^2 \varepsilon^{\frac32} + e^{-\nu t} \varepsilon + \jt^{-1} \bigl( \log(2+t) \bigr)^2 \varepsilon^2 \Bigr) \\
 &\lesssim \jt^{-\frac32} \Bigl( \log(2+t) \varepsilon^2 + \bigl( \log(2+t) \bigr)^3 \varepsilon^{\frac52} \Bigr),
\end{aligned}
\end{equation}
where we used the bound \eqref{equ:stable_coefficient_bound} from Lemma~\ref{lem:stable_coefficient} below.
For the cubic terms $\calC\bigl(v(t)+\bar{v}(t), a(t) \bigr)$ we obtain that
\begin{equation}
\begin{aligned}
 &\bigl\|\jx^2 \jD \calC\bigl(v(t)+\bar{v}(t), a(t) \bigr) \bigr\|_{L^2_x} \\
 &\lesssim \bigl( \|v(t)\|_{L^\infty_x} + \|\jx^{-1} \px v(t)\|_{L^2_x} + |a_+(t)| + |a_-(t)| \bigr)^2 \bigl( |a_+(t)| + |a_-(t)| \bigr) \\
 &\lesssim \jt^{-2} \bigl( \log(2+t) \bigr)^3 \varepsilon^3.
\end{aligned}
\end{equation}

It remains to treat the contributions of the cubic terms $\calC_R(v+\bar{v})$, for which a more detailed discussion is in order. Recall from \eqref{equ:definition_calCR_w} the decomposition
\begin{equation*}
 \calC_R(v+\bv) = \calC_l(v+\bv) + \sum_{j=1}^4 \calC_{nl,j;R}(v+\bv).
\end{equation*}
For the terms $\calC_l(v+\bv)$ we obtain in a very similar manner as for the preceding cubic terms that 
\begin{equation}
\begin{aligned}
 \bigl\|\jx^2 \jD \calC_l\bigl(v(t)+\bar{v}(t)\bigr) \bigr\|_{L^2_x} &\lesssim \bigl( \|v(t)\|_{L^\infty_x} + \|\jx^{-1} \px v(t)\|_{L^2_x} \bigr) \|v(t)\|_{L^\infty_x}^2 \\
 &\lesssim \jt^{-\frac32} \bigl( \log(2+t) \bigr)^3 \varepsilon^3.
\end{aligned}
\end{equation}
In view of \eqref{equ:def_regular_part_calCnl1}, \eqref{equ:def_regular_part_calCnl2}, \eqref{equ:def_regular_part_calCnl3}, and \eqref{equ:def_regular_part_calCnl4}, each cubic term in $\calC_{nl,j;R}(v+\bv)$, $1 \leq j \leq 4$, is of the form
\begin{equation*}
 \gamma(x) u_1(t) u_2(t) u_3(t)
\end{equation*}
for some Schwartz function $\gamma(x)$ and with each input $u_i(t)$, $1 \leq i \leq 3$, given by
\begin{equation*}
 v(t), \quad \text{or} \quad m_a(D) v(t), \quad a \in \{0, 4, 5, 6\}, \quad \text{or} \quad B_b(\hatv(t)), \quad b \in \{1, 2, 3\},
\end{equation*}
or complex conjugates thereof, with $m_a(D)$ defined in \eqref{equ:def_multipliers_m} and $B_b(\hatv(t))$ defined in \eqref{equ:def_Bterms}.
Since the multipliers $m_c(D)$, $0 \leq c \leq 6$, are bounded on $L^p_x(\bbR)$, $1 \leq p \leq \infty$, and since 
\begin{equation*}
 |B_b(\hatv(t))| = \biggl| \int_\bbR m_b(\eta) \hatv(t,\eta) \, \ud \eta \biggr| = \sqrt{2\pi} \bigl| \bigl(m_b(D) v\bigr)(0)\bigr| \lesssim \|m_b(D) v(t)\|_{L^\infty_x} \lesssim \|v(t)\|_{L^\infty_x},
\end{equation*}
it is clear that for $1 \leq j \leq 4$ we have the estimates
\begin{equation*}
 \begin{aligned}
  \bigl\| \jx^2 \jD \calC_{nl,j;R}\bigl( v(t) + \bv(t) \bigr) \bigr\|_{L^2_x} &\lesssim \bigl( \|v(t)\|_{L^\infty_x} + \|\jx^{-1} \px v(t)\|_{L^2_x} \bigr) \|v(t)\|_{L^\infty_x}^2 \\
  &\lesssim \jt^{-\frac32} \bigl( \log(2+t) \bigr)^3 \varepsilon^3.
 \end{aligned}
\end{equation*}
This finishes the proof of the proposition.
\end{proof}

\subsection{Controlling the stable coefficient}

Next, we derive decay for the stable coefficient.

\begin{lemma} \label{lem:stable_coefficient}
 Under the assumptions of Proposition~\ref{prop:main_bootstrap}, we have uniformly for all $0 \leq t \leq T$,
 \begin{equation} \label{equ:stable_coefficient_bound}
  |a_-(t)| \lesssim e^{-\nu t} \varepsilon + \frac{\bigl(\log(2+t)\bigr)^2}{\jt} \varepsilon^2. 
 \end{equation}
\end{lemma}
\begin{proof}
From the integral formulation of the differential equation~\eqref{equ:aminus_equ_refer_to} for the stable coefficient $a_-(t)$, we obtain
\begin{equation*}
 \begin{aligned}
  |a_-(t)| &\lesssim e^{-\nu t} |a_-(0)| + \int_0^t e^{-\nu(t-s)} \bigl| \langle Y_0, (3Q \varphi(s)^2 + \varphi(s)^3) \rangle \bigr| \, \ud s
 \end{aligned}
\end{equation*}
with 
\begin{equation*}
 \varphi(s) = P_c \calJ\bigl[v(s)+\barv(s)\bigr] + \bigl(a_+(s) + a_-(s)\bigr) Y_0.
\end{equation*}
Now we have 
\begin{equation*}
 |a_-(0)| = \frac12 \bigl|\langle Y_0, \varphi_0 - \nu^{-1} \varphi_1 \rangle\bigr| \lesssim \|\varphi_0\|_{L^2_x} + \|\varphi_1\|_{L^2_x} \lesssim \varepsilon.
\end{equation*}
Moreover, by Lemma~\ref{lem:I1_and_J_bounds}, by the dispersive decay~\eqref{equ:Linfty_decay_v} of $v(t)$ and by the decay bounds \eqref{equ:bootstrap2} and \eqref{equ:trapping} for the stable, respectively unstable coefficients, we have
\begin{equation*}
 \begin{aligned}
  |\varphi(s)| \lesssim \|v(s)\|_{L^\infty_x} + |a_+(s)| + |a_-(s)| \lesssim \js^{-\frac12} \log(2+s) \varepsilon.
 \end{aligned}
\end{equation*}
Hence, we find 
\begin{equation*}
 \begin{aligned}
  |a_-(t)| \lesssim e^{-\nu t} \varepsilon + \int_0^t e^{-\nu(t-s)} \js^{-1} \bigl(\log(2+s)\bigr)^2 \varepsilon^2 \, \ud s \lesssim e^{-\nu t} \varepsilon + \jt^{-1} \bigl(\log(2+t)\bigr)^2 \varepsilon^2, 
 \end{aligned}
\end{equation*}
as claimed.
\end{proof}

\subsection{Decomposition of the evolution equation for the profile}

In the derivation of the weighted energy estimates for the main quadratic and cubic interactions in Section~\ref{sec:weighted_main_quadratic}, respectively in Section~\ref{sec:weighted_main_cubic}, we will repeatedly use normal form arguments, more precisely, we will integrate by parts in time and insert the equation for the time derivative of the profile again. Here we derive a representation of the evolution equation of the Fourier transform of the profile that will be useful in those instances.

\begin{lemma} \label{lem:decomposition_FT_pt_hatf} 
Under the assumptions of Proposition~\ref{prop:main_bootstrap}, the Fourier transform of the profile $\hatf(t,\xi)$ satisfies the evolution equation
\begin{equation} \label{equ:decomposition_FT_pt_hatf} 
 \begin{aligned}
   \pt \hatf(t,\xi) &= (2i\jxi)^{-1} e^{-it\jxi} \whatbeta(\xi) \bigl( v(t,0) + \barv(t,0) \bigr)^2 + (2i\jxi)^{-1} e^{-it\jxi} \widehat{\calN}_c(t,\xi),  
 \end{aligned}
\end{equation}
with the Schwartz function 
\begin{equation} \label{equ:betahat_definition}
 \begin{aligned}
  \whatbeta(\xi) := \whatalpha_1(\xi) + \whatalpha_2(\xi) \langle G, 1 \rangle + \whatalpha_3(\xi) ( \langle G, 1 \rangle)^2, \quad G := K^2 Y_0,
 \end{aligned}
\end{equation}
and where uniformly for all $0 \leq t \leq T$,
\begin{equation} \label{equ:decomposition_FT_pt_hatf_Nc_cubic_decay}
 \bigl\| \calN_c(t) \bigr\|_{L^\infty_x} \lesssim \frac{\bigl(\log(2+t)\bigr)^3}{\jt^\thf} \varepsilon^2.
\end{equation}
\end{lemma}
\begin{proof}
We begin by recalling the evolution equation \eqref{equ:f_equ_refer_to} for the profile
\begin{equation} \label{equ:decomposition_FT_pt_hatf_proof1} 
 \begin{aligned}
  \pt f(t) &= (2i\jD)^{-1} e^{-it\jD} \Bigl( \calQ_r(v + \bar{v}) + \calQ_{nr}(v + \bar{v}) + \calQ(v + \bar{v},a) \\
  &\qquad \qquad \qquad \quad \quad \quad \, \, + \calC_{\delta_0}(v + \bar{v}) + \calC_{\pvdots}(v + \bar{v}) + \calC_R(v + \bar{v}) + \calC(v + \bar{v},a) \Bigr).
 \end{aligned}
\end{equation}
By similar estimates as in the preceding subsections, one readily sees that apart from $\calQ_r(v(t)+\bv(t))$, all nonlinear terms on the right-hand side of~\eqref{equ:decomposition_FT_pt_hatf_proof1} decay in $L^\infty_x$ at least like $C \jt^{-\frac32} \bigl(\log(2+t)\bigr)^3 \varepsilon^2$.
We arrive at the representation \eqref{equ:decomposition_FT_pt_hatf} after peeling off further parts of $\calQ_r(v(t)+\bv(t))$ with cubic decay by subtracting off 
\begin{equation*}
 \begin{aligned}
  \bigl( \alpha_1(x) + \alpha_2(x) \langle G, 1 \rangle + \alpha_3(x) ( \langle G, 1 \rangle)^2 \bigr) \bigl( v(t,0) + \bv(t,0) \bigr)^2.
 \end{aligned}
\end{equation*}
Then each term in the difference 
\begin{equation} \label{equ:decomposition_FT_pt_hatf_proof2} 
 \begin{aligned}
  \calQ_r\bigl(v(t) + \bv(t)\bigr) - \bigl( \alpha_1(x) + \alpha_2(x) \langle G, 1 \rangle + \alpha_3(x) ( \langle G, 1 \rangle)^2 \bigr) \bigl( v(t,0) + \bv(t,0) \bigr)^2
 \end{aligned}
\end{equation}
has at least one input of the form
$v(t,x) - v(t,0)$
or complex conjugates thereof, which along with the spatial localization of the coefficients $\alpha_j(x)$, $1 \leq j \leq 3$, gives access to the improved local decay estimate~\eqref{equ:improved_Linftydecay_v_minus_v0}. Using the latter we conclude that \eqref{equ:decomposition_FT_pt_hatf_proof2} enjoys the cubic-type decay $C \jt^{-\thf} \bigl(\log(2+t)\bigr)^2 \varepsilon^2$ in $L^\infty_x$.
Hence, setting
\begin{equation*}
 \begin{aligned}
  \calN_c(t) &:= \calQ_r\bigl(v(t) + \bv(t)\bigr) - \bigl( \alpha_1(x) + \alpha_2(x) \langle G, 1 \rangle + \alpha_3(x) ( \langle G, 1 \rangle)^2 \bigr) \bigl( v(t,0) + \bv(t,0) \bigr)^2 \\
  &\quad + \calQ_{nr}\bigl(v(t) + \bv(t)\bigr) + \calQ\bigl(v(t) + \bv(t), a(t)\bigr) \\
  &\quad + \calC_{\delta_0}\bigl(v(t) + \bv(t)\bigr) + \calC_{\pvdots}\bigl(v(t) + \bv(t)\bigr) + \calC_R\bigl(v(t) + \bv(t)\bigr) + \calC\bigl(v(t) + \bv(t), a(t)\bigr),
 \end{aligned}
\end{equation*}
yields the representation \eqref{equ:decomposition_FT_pt_hatf} along with the asserted decay estimate~\eqref{equ:decomposition_FT_pt_hatf_Nc_cubic_decay}.
\end{proof}

\section{Weighted Energy Estimates for the Main Quadratic Interactions} \label{sec:weighted_main_quadratic}

In this section we establish the weighted energy estimates for the main quadratic interactions~$\calQ_r(v+\bv)$.

\begin{proposition} \label{prop:weighted_energy_est_main_quadratic}
 Under the assumptions of Proposition~\ref{prop:main_bootstrap} we have for all $0 \leq t \leq T$ that
 \begin{equation} \label{equ:weighted_energy_est_main_quadratic}
 \begin{aligned}
  &\sup_{n \geq 1} \, \sup_{0 \leq \ell \leq n} \, 2^{-\frac12 \ell} \tau_n(t) \biggl\| \varphi_\ell^{(n)}(\xi) \jxi^2 \partial_\xi \int_0^t (2i\jxi)^{-1} e^{-is\jxi} \calF\bigl[ \calQ_r(v(s)+\bar{v}(s)) \bigr](\xi) \, \ud s \biggr\|_{L^2_\xi} \\
  &\qquad \qquad \qquad \qquad \qquad \qquad \qquad \qquad \qquad \qquad \qquad \lesssim \bigl( \log(2+t) \bigr)^3 \varepsilon^2 + \bigl( \log(2+t) \bigr)^5 \varepsilon^3.
  \end{aligned}
 \end{equation}
\end{proposition}

We recall that
\begin{equation*}
 \calQ_r(v+\bv) = \alpha_1(x) (v+\bv)^2 + \alpha_2(x) (v+\bv) \langle G, v+\bv \rangle + \alpha_3(x) \bigl( \langle G, v+\bv \rangle \bigr)^2, \quad G := K^2 Y_0.
\end{equation*}
For the proof of Proposition~\ref{prop:weighted_energy_est_main_quadratic} we introduce the bilinear operators
\begin{equation*}
\begin{aligned}
 \calB_1[a,b](t) &:= \int_0^t (2i\jD)^{-1} e^{-is\jD} \bigl( \alpha_1(\cdot) a(s) b(s) \bigr) \, \ud s, \\
 \calB_2[a,b](t) &:= \int_0^t (2i\jD)^{-1} e^{-is\jD} \bigl( \alpha_2(\cdot) a(s) \bigr) \langle G, b(s) \rangle \, \ud s, \\
 \calB_3[a,b](t) &:= \int_0^t (2i\jD)^{-1} \bigl( e^{-is\jD} \alpha_3(\cdot) \bigr) \langle G, a(s) \rangle \langle G, b(s) \rangle \, \ud s.
\end{aligned}
\end{equation*}
Then we have
\begin{equation} \label{equ:weighted_energy_est_Qr_decomposition_into_Bs}
 \begin{aligned}
  \int_0^t (2i\jD)^{-1} e^{-is\jD} \calQ_r(v(s) + \bv(s)) \, \ud s = \sum_{j=1}^3 \biggl( \calB_j[v, v](t) + 2 \calB_j[v, \bv](t) + \calB_j[\bv, \bv](t) \biggr).
 \end{aligned}
\end{equation}
The problematic fully resonant quadratic interactions are in the terms $\calB_j[v,v]$, $1 \leq j \leq 3$, which exhibit a space-time resonance. They dictate the design of the adapted functional framework defined in~\eqref{equ:definition_NT_norm}.
The quadratic interactions in $\calB_j[v, \bv](t)$ and in $\calB_j[\bv, \bv](t)$, $1 \leq j \leq 3$, are milder.

In the next propositions we establish the weighted energy estimates for all terms on the right-hand side of~\eqref{equ:weighted_energy_est_Qr_decomposition_into_Bs}. Put together, they furnish a proof of Proposition~\ref{prop:weighted_energy_est_main_quadratic}. We begin with the weighted energy estimates for the problematic quadratic interactions $\calB_1[v,v]$.

\begin{proposition} \label{prop:weighted_energy_est_calB1_bad}
 Under the assumptions of Proposition~\ref{prop:main_bootstrap} we have for all $0 \leq t \leq T$ that
 \begin{equation}
 \begin{aligned}
  &\sup_{n \geq 1} \, \sup_{0 \leq \ell \leq n} \, 2^{-\frac12 \ell} \tau_n(t) \bigl\| \varphi_\ell^{(n)}(\xi) \jxi^2 \partial_\xi \calF\bigl[ \calB_1[v,v](t) \bigr](\xi) \bigr\|_{L^2_\xi} \\
  &\qquad \qquad \qquad \qquad \qquad \qquad \lesssim \bigl( \log(2+t) \bigr)^3 \varepsilon^2 + \bigl( \log(2+t) \bigr)^5 \varepsilon^3.
  \end{aligned}
 \end{equation}
\end{proposition}
\begin{proof}
Fix $0 \leq t \leq T$.  Let $n \geq 1$ be an integer such that $t \in \mathrm{supp}(\tau_n)$.
By direct computation we have
\begin{equation} \label{equ:pxi_of_calFB1vv}
 \begin{aligned}
  \jxi^2 \pxi \calF\bigl[ \calB_1[v,v](t) \bigr](\xi) &= - 2^{-1} \int_0^t s \cdot \xi e^{-is\jxi} \calF\bigl[ \alpha_1(\cdot) v(s) v(s) \bigr](\xi) \, \ud s \\
  &\quad + (2i)^{-1} \int_0^t e^{-is\jxi} \jxi \partial_\xi \calF\bigl[ \alpha_1(\cdot) v(s) v(s) \bigr](\xi) \, \ud s \\
  &\quad - (2i)^{-1} \int_0^t e^{-is\jxi} \jxi^{-1} \xi \calF\bigl[ \alpha_1(\cdot) v(s) v(s) \bigr](\xi) \, \ud s.
 \end{aligned}
\end{equation}
We separately treat the cases $\ell = n$, $1 \leq \ell \leq n-1$, and $\ell = 0$, for the localization of the output frequency relative to the problematic frequencies $\pm \sqrt{3}$, i.e., we distinguish the cases $||\xi|-\sqrt{3}| \lesssim 2^{-n-100}$, $||\xi|-\sqrt{3}| \simeq 2^{-\ell-100}$ for $1 \leq \ell \leq n-1$, and $||\xi|-\sqrt{3}| \gtrsim 2^{-100}$.

\noindent \underline{\it Case 1: $\ell = n$.}
Since $|\xi| \lesssim 1$ on the support of $\varphi_n^{(n)}(\xi)$ and since $t \simeq 2^n$, we infer from~\eqref{equ:pxi_of_calFB1vv} and the decay estimate~\eqref{equ:Linfty_decay_v} that
\begin{equation*}
 \begin{aligned}
  &2^{-\frac12 n} \bigl\| \varphi_n^{(n)}(\xi) \jxi^2 \pxi \calF\bigl[ \calB_1[v,v](t) \bigr](\xi) \bigr\|_{L^2_\xi}  \\
  &\lesssim 2^{-\frac12 n} \bigl\| \varphi_n^{(n)}(\xi) \bigr\|_{L^2_\xi} \int_0^t \js \bigl\| \jap{\pxi} \calF\bigl[ \alpha_1(\cdot) v(s) v(s) \bigr](\xi) \bigr\|_{L^\infty_\xi} \, \ud s \\
  &\lesssim 2^{-\frac12 n} \cdot 2^{-\frac12 n} \int_0^t \js \bigl\| \jx \alpha_1(x) \bigr\|_{L^1_x} \|v(s)\|_{L^\infty_x}^2  \, \ud s \\
  &\lesssim 2^{-n} \int_0^t \js \cdot \js^{-1} \bigl( \log(2+s) \bigr)^2 \varepsilon^2 \, \ud s \\
  &\lesssim \bigl( \log(2+t) \bigr)^2 \varepsilon^2,
 \end{aligned}
\end{equation*}
which is acceptable.

\medskip
\noindent \underline{\it Case 2: $1 \leq \ell \leq n-1$.}
We decompose the inputs $v(t)$ for $\calB_1[v,v](t)$ into a low-frequency and a high-frequency piece, depending on the scale of the localization $||\xi|-\sqrt{3}| \simeq 2^{-\ell-100}$ of the output frequency,
\begin{equation*}
 v(t) = P_{\leq -\frac12 \ell -100} v(t) + P_{>-\frac12 \ell -100} v(t).
\end{equation*}
See \eqref{equ:weighted_energy_est_main_quadratic_phase_size} below for how the choice of the low-frequency projection $P_{\leq -\frac12 \ell-100}$ comes up naturally.
As soon as one input is frequency localized away from zero, the corresponding quadratic interaction has cubic-type time decay thanks to the improved local decay estimates \eqref{equ:local_decay_pxPk_v} and \eqref{equ:local_decay_Pk_v}, but at the expense of a loss in terms of the distance to zero frequency, see \eqref{equ:local_decay_Pk_v}. The weighted energy estimates for those contributions can still be obtained using just Proposition~\ref{prop:prop49}, because the losses get just about compensated by the weights built into the functional framework.
Indeed, invoking the improved local decay estimates \eqref{equ:local_decay_pxPk_v} and \eqref{equ:local_decay_Pk_v}, we have for $0 \leq s \leq T$ that
\begin{equation*}
 \begin{aligned}
  &\Bigl\| \jx^2 \jD \Bigl( \alpha_1(x) \bigl(P_{\leq -\frac12 \ell -100} v\bigr)(s) \bigl(P_{>-\frac12 \ell -100} v\bigr)(s) \Bigr) \Bigr\|_{L^2_x} \\
  &\lesssim \| \jx^3 \alpha_1 \|_{W^{1,\infty}_x} \|v(s)\|_{L^\infty_x} \bigl\| \jx^{-1} \bigl(P_{>-\frac12 \ell -100} v\bigr)(s) \bigr\|_{L^2_x} \\
  &\quad + \| \jx^3 \alpha_1(x) \|_{L^\infty_x} \bigl\| \jx^{-1} \px P_{\leq -\frac12 \ell -100} v(s) \bigr\|_{L^2_x} \|v(s)\|_{L^\infty_x} \\
  &\quad + \| \jx^3 \alpha_1(x) \|_{L^\infty_x} \|v(s)\|_{L^\infty_x} \bigl\| \jx^{-1} \px P_{> -\frac12 \ell -100} v(s) \bigr\|_{L^2_x} \\
  &\lesssim 2^{\frac12 \ell} \js^{-\frac32} \bigl( \log(2+s) \bigr)^2 \varepsilon^2.
 \end{aligned}
\end{equation*}
By Proposition~\ref{prop:prop49} we therefore obtain the acceptable bound
\begin{equation*}
 \begin{aligned}
  &2^{-\frac12 \ell} \, \Bigl\| \varphi_\ell^{(n)}(\xi) \jxi^2 \partial_\xi \calF\Bigl[ \calB_1\bigl[ P_{\leq -\frac12 \ell -100} v, P_{>-\frac12 \ell -100} v \bigr](t) \Bigr]\Bigr\|_{L^2_\xi} \\
  &\quad \lesssim 2^{-\frac12 \ell} \cdot 2^{\frac12 \ell} \bigl( \log(2+t) \bigr)^{\frac52} \varepsilon^2 \lesssim \bigl( \log(2+t) \bigr)^{\frac52} \varepsilon^2.
 \end{aligned}
\end{equation*}
Analogously, we can estimate all other combinations where at least one input for $\calB_1[v,v](t)$ is localized to frequencies $\gtrsim 2^{-\frac12 \ell -100}$.

Thus, it remains to deal with the scenario, where both inputs are localized to very small frequencies $\lesssim 2^{-\frac12 \ell-100} \ll 1$. We only discuss how to estimate the first term on the right-hand side of~\eqref{equ:pxi_of_calFB1vv}, the other two terms being much easier to treat. We have
\begin{equation*}
 \begin{aligned}
  &\varphi_\ell^{(n)}(\xi) \int_0^t s \cdot \xi e^{-is\jxi} \calF\Bigl[ \alpha_1(\cdot) (P_{\leq -\frac12\ell -100} v)(s) (P_{\leq -\frac12\ell-100}v)(s) \Bigr](\xi) \, \ud s \\
  &= (2\pi)^{-1} \int_0^t s \cdot \xi \iint e^{is(-\jxi + \jxione + \jxitwo)} \varphi_\ell^{(n)}(\xi) \varphi_{\leq -\frac12\ell-100}(\xi_1) \varphi_{\leq -\frac12\ell-100}(\xi_2) \\
  &\qquad \qquad \qquad \qquad \qquad \qquad \qquad \qquad \qquad \qquad \times \whatalpha_1(\xi_3) \hatf(s,\xi_1) \hatf(s,\xi_2) \, \ud \xi_1 \, \ud \xi_2 \, \ud s
 \end{aligned}
\end{equation*}
with
\begin{equation*}
 \xi_3 := \xi - \xi_1 - \xi_2.
\end{equation*}
Next, we would like to integrate by parts in time, which necessitates to carefully analyze the size of the phase function $-\jxi + \jap{\xi_1} + \jap{\xi_2}$ by Taylor expansion of $\jxi$ around $\xi = \pm \sqrt{3}$ and of $\jap{\xi_j}$ around $\xi_j = 0$, $j = 1, 2$.
Let us consider the case where the output frequency $\xi$ is close to $+\sqrt{3}$, the other case being analogous. Using that $\sqrt{1+x} = 1 + \frac12 x + \calO(x^2)$ for $|x| \ll 1$, we obtain for $|\xi-\sqrt{3}| \ll 1$,
\begin{equation*}
 \begin{aligned}
  -\jxi = - \bigl( 1 + (\sqrt{3} + \xi - \sqrt{3})^2 \bigr)^{\frac12} 
  = - 2 \Bigl( 1 + \frac{\sqrt{3}}{4} (\xi-\sqrt{3}) + \calO\bigl( (\xi-\sqrt{3})^2 \bigr) \Bigr),
 \end{aligned}
\end{equation*}
and for $|\xi_j| \ll 1$, $j = 1, 2$,
\begin{equation*}
 \begin{aligned}
  \jap{\xi_j} = 1 + \frac12 \xi_j^2 + \calO( \xi_j^4 ).
 \end{aligned}
\end{equation*}
In the current frequency configuration $|\xi-\sqrt{3}| \simeq 2^{-\ell-100}$ and $|\xi_1| + |\xi_2| \lesssim 2^{-\frac12\ell-100}$, we thus have
\begin{equation} \label{equ:weighted_energy_est_main_quadratic_phase_size}
 \begin{aligned}
  &\bigl| -\jxi + \jap{\xi_1} + \jap{\xi_2} \bigr| \\
  &\quad = \biggl| - \frac{\sqrt{3}}{2} (\xi-\sqrt{3}) + \frac12 \xi_1^2 + \frac12 \xi_2^2 + \calO\bigl( (\xi-\sqrt{3})^2 \bigr) + \calO(\xi_1^4) + \calO(\xi_2^4) \biggr| \simeq 2^{-\ell-100}.
 \end{aligned}
\end{equation}
Hence, integrating by parts in time we find
\begin{equation} \label{equ:pxi_of_calFB1vv_main1}
 \begin{aligned}
  &\int_0^t s \cdot \xi \iint e^{is(-\jxi + \jxione + \jxitwo)} \varphi_\ell^{(n)}(\xi) \varphi_{\leq -\frac12\ell-100}(\xi_1) \varphi_{\leq -\frac12\ell-100}(\xi_2) \\
  &\qquad \qquad \qquad \qquad \qquad \qquad \qquad \qquad \qquad \qquad \times \whatalpha_1(\xi_3) \hatf(s,\xi_1) \hatf(s,\xi_2) \, \ud \xi_1 \, \ud \xi_2 \, \ud s \\
  &\quad = i \int_0^t s \cdot \xi \iint e^{is(-\jxi + \jxione + \jxitwo)} \frakm(\xi, \xi_1, \xi_2) \whatalpha_1(\xi_3) \ps \hatf(s,\xi_1) \hatf(s,\xi_2) \, \ud \xi_1 \, \ud \xi_2 \, \ud s \\
  &\quad \quad + i\int_0^t 1 \cdot \xi \iint e^{is(-\jxi + \jxione + \jxitwo)} \frakm(\xi, \xi_1, \xi_2) \whatalpha_1(\xi_3) \hatf(s,\xi_1) \hatf(s,\xi_2) \, \ud \xi_1 \, \ud \xi_2 \, \ud s \\
  &\quad \quad - i \, t \cdot \xi \iint e^{is(-\jxi + \jxione + \jxitwo)} \frakm(\xi, \xi_1, \xi_2) \whatalpha_1(\xi_3) \hatf(s,\xi_1) \hatf(s,\xi_2) \, \ud \xi_1 \, \ud \xi_2 \\
  &\quad \quad + \{\text{symmetric terms}\} \\
  &\quad = I(t,\xi) + II(t,\xi) + III(t,\xi) + \{\text{symmetric terms}\}
 \end{aligned}
\end{equation}
with
\begin{equation*}
 \frakm(\xi, \xi_1, \xi_2) := (-\jxi + \jxione + \jxitwo)^{-1} \varphi_\ell^{(n)}(\xi) \varphi_{\leq -\frac12\ell-100}(\xi_1) \varphi_{\leq -\frac12\ell-100}(\xi_2).
\end{equation*}

At this point we observe that
\begin{equation} \label{equ:weighted_energy_est_calB1_calFinvfrakm_L1bound}
 \bigl\| \calF^{-1}[\frakm] \bigr\|_{L^1(\bbR^3)} \lesssim 2^\ell.
\end{equation}
Indeed, after a change of variables we have
\begin{equation*}
 \begin{aligned}
  \bigl\| \calF^{-1}\bigl[ \frakm \bigr]\bigr\|_{L^1(\bbR^3)}
  &= \biggl\| \iiint e^{i(x\eta + y\eta_1 + z\eta_2)} \bigl(-\jap{2^{-\ell-100}\eta} + \jap{2^{-\frac12\ell-100} \eta_1} + \jap{2^{-\frac12\ell-100} \eta_2} \bigr)^{-1} \\
  &\qquad \qquad \qquad \qquad \qquad \qquad \quad \times \psi\bigl(|\eta| - 2^{\ell+100} \sqrt{3}\bigr) \varphi(\eta_1) \varphi(\eta_2) \, \ud \eta \, \ud \eta_1 \, \ud \eta_2 \biggr\|_{L^1_{x,y,z}(\bbR^3)}.
 \end{aligned}
\end{equation*}
Then the asserted bound~\eqref{equ:weighted_energy_est_calB1_calFinvfrakm_L1bound} follows by repeated integration by parts, exploiting that on the support of $\psi\bigl(|\eta| - 2^{\ell+100} \sqrt{3}\bigr) \varphi(\eta_1) \varphi(\eta_2)$ it holds that
\begin{equation*}
 |-\jap{2^{-\ell-100}\eta} + \jap{2^{-\frac12\ell-100} \eta_1} + \jap{2^{-\frac12\ell-100} \eta_2} \bigr| \gtrsim 2^{-\ell}
\end{equation*}
as well as
\begin{equation*}
 \begin{aligned}
  \bigl| \partial_\eta \bigl( \jap{2^{-\ell-100} \eta} \bigr) \bigr| \lesssim 2^{-\ell}, \quad \bigl| \partial_{\eta_j} \bigl( \jap{2^{-\frac12\ell-100} \eta_j} \bigr) \bigr| \lesssim 2^{- \ell}, \quad j = 1,2.
 \end{aligned}
\end{equation*}
For the final estimate we used that $|\eta_j| \lesssim 1$ for $j = 1, 2$ implies
\begin{equation*}
 \bigl| \partial_{\eta_j} \bigl( \jap{2^{-\frac12\ell-100} \eta_j} \bigr) \bigr| = \biggl| \frac{2^{-\frac12\ell-100} \eta_j}{\jap{2^{-\frac12\ell-100} \eta_j}} 2^{-\frac12\ell-100} \biggr| \leq  \frac{|\eta_j|}{\jap{2^{-\frac12\ell-100} \eta_j}} 2^{-\ell} \lesssim 2^{-\ell}.
\end{equation*}

The last two terms $II(t,\xi)$ and $III(t,\xi)$ in~\eqref{equ:pxi_of_calFB1vv_main1} are now straightforward to estimate.
We denote by $\widetilde{\varphi}_\ell^{(n)}(\xi)$ a slight fattening of the cut-off $\varphi_\ell^{(n)}(\xi)$, and we recall that $|\xi| \lesssim 1$ on its support.
By Lemma~\ref{lem:frakm_for_delta_three_inputs} and the decay estimate~\eqref{equ:Linfty_decay_v}, we obtain for $0 \leq t \leq T$
\begin{equation*}
 \begin{aligned}
  &2^{-\frac12 \ell} \bigl\|II(t,\xi)\bigr\|_{L^2_\xi} \\
  &\lesssim 2^{-\frac12 \ell} \bigl\|\xi \, \widetilde{\varphi}_\ell^{(n)}(\xi) \bigr\|_{L^2_\xi} \int_0^t \, \biggl\| \iint \frakm(\xi, \xi_1, \xi_2) \whatalpha_1(\xi_3) \hatv(s,\xi_1) \hatv(s,\xi_2) \, \ud \xi_1 \, \ud \xi_2 \biggr\|_{L^\infty_\xi} \, \ud s \\
  &\lesssim 2^{-\frac12 \ell} \bigl\|\xi \, \widetilde{\varphi}_\ell^{(n)}(\xi) \bigr\|_{L^2_\xi} \int_0^t \, \biggl\| \calF^{-1} \biggl[ \iint \frakm(\xi, \xi_1, \xi_2) \whatalpha_1(\xi_3) \hatv(s,\xi_1) \hatv(s,\xi_2) \, \ud \xi_1 \, \ud \xi_2 \biggr] \biggr\|_{L^1_x} \, \ud s \\
  &\lesssim 2^{-\frac12 \ell} \cdot 2^{-\frac12 \ell} \cdot \bigl\|\calF^{-1}[\frakm]\bigr\|_{L^1(\bbR^3)} \|\alpha_1(x)\|_{L^1_x}  \int_0^t  \|v(s)\|_{L^\infty_x}^2 \, \ud s \\
  &\lesssim 2^{- \ell} \cdot 2^\ell \cdot \bigl( \log(2+t) \bigr)^3 \varepsilon^2 \lesssim \bigl( \log(2+t) \bigr)^3 \varepsilon^2,
 \end{aligned}
\end{equation*}
and
\begin{equation*}
 \begin{aligned}
  &2^{-\frac12 \ell} \bigl\|III(t,\xi)\bigr\|_{L^2_\xi} \\
  &\lesssim 2^{-\frac12 \ell} \, \bigl\|\xi \, \widetilde{\varphi}_\ell^{(n)}(\xi) \bigr\|_{L^2_\xi} \cdot t \cdot \biggl\| \iint \frakm(\xi, \xi_1, \xi_2) \whatalpha_1(\xi_3) \hatv(t,\xi_1) \hatv(t,\xi_2) \, \ud \xi_1 \, \ud \xi_2 \biggr\|_{L^\infty_\xi}  \\
  &\lesssim 2^{-\frac12 \ell} \, \bigl\|\xi \, \widetilde{\varphi}_\ell^{(n)}(\xi) \bigr\|_{L^2_\xi} \cdot t \cdot \biggl\| \calF^{-1} \biggl[ \iint \frakm(\xi, \xi_1, \xi_2) \whatalpha_1(\xi_3) \hatv(t,\xi_1) \hatv(t,\xi_2) \, \ud \xi_1 \, \ud \xi_2 \biggr] \biggr\|_{L^1_x}  \\
  &\lesssim 2^{-\frac12 \ell} \cdot 2^{-\frac12 \ell} \cdot t \cdot \bigl\|\calF^{-1}[\frakm]\bigr\|_{L^1(\bbR^3)} \|\alpha_1(x)\|_{L^1_x} \|v(t)\|_{L^\infty_x}^2 \\
  &\lesssim 2^{- \ell} \cdot t \cdot 2^\ell \cdot \jt^{-1} \bigl( \log(2+t) \bigr)^2 \varepsilon^2 \lesssim \bigl( \log(2+t) \bigr)^2 \varepsilon^2,
 \end{aligned}
\end{equation*}
which is acceptable.

In order to estimate the main term $I(t,\xi)$ in~\eqref{equ:pxi_of_calFB1vv_main1}, we insert the equation for $\ps \hatf(s,\xi_1)$ in the form \eqref{equ:decomposition_FT_pt_hatf} from Lemma~\ref{lem:decomposition_FT_pt_hatf}. Recall that \eqref{equ:decomposition_FT_pt_hatf} reads
\begin{equation} \label{equ:B1vv_pshatf_decomposition}
 \begin{aligned}
   \ps \hatf(s,\xi_1) &= (2i\jxione)^{-1} e^{-is\jxione} \whatbeta(\xi_1) \bigl( v(s,0) + \barv(s,0) \bigr)^2 + (2i\jxione)^{-1} e^{-is\jxione} \widehat{\calN}_c(s,\xi_1),
 \end{aligned}
\end{equation}
with the Schwartz function $\whatbeta(\xi_1)$ defined in \eqref{equ:betahat_definition} and with $\|\calN_c(s)\|_{L^\infty_x} \lesssim \js^{-\thf} (\log(2+s))^3 \varepsilon^2$ for $0 \leq s \leq T$.
Upon inserting \eqref{equ:B1vv_pshatf_decomposition} into $I(t,\xi)$ and expanding $(v(s,0) + \bv(s,0))^2$, we find that
\begin{equation*}
 \begin{aligned}
  I(t,\xi) &= i \int_0^t s \cdot \xi \iint e^{is(-\jxi + 2 + \jxitwo)} \frakm(\xi, \xi_1, \xi_2) \whatalpha_1(\xi_3) (2i\jxione)^{-1} \whatbeta(\xi_1) \bigl( e^{-is} v(s,0) \bigr)^2 \hatf(s,\xi_2) \, \ud \xi_1 \, \ud \xi_2 \, \ud s \\
  &\quad + 2 i \int_0^t s \cdot \xi \iint e^{is(-\jxi + \jxitwo)} \frakm(\xi, \xi_1, \xi_2) \whatalpha_1(\xi_3) (2i\jxione)^{-1} \whatbeta(\xi_1) \bigl| e^{-is} v(s,0) \bigr|^2 \hatf(s,\xi_2) \, \ud \xi_1 \, \ud \xi_2 \, \ud s \\
  &\quad + i \int_0^t s \cdot \xi \iint e^{is(-\jxi - 2 + \jxitwo)} \frakm(\xi, \xi_1, \xi_2) \whatalpha_1(\xi_3) (2i\jxione)^{-1} \whatbeta(\xi_1) \bigl( \overline{e^{-is} v(s,0)} \bigr)^2 \hatf(s,\xi_2) \, \ud \xi_1 \, \ud \xi_2 \, \ud s \\
  &\quad + i \int_0^t s \cdot \xi \iint e^{is(-\jxi + \jxitwo)} \frakm(\xi, \xi_1, \xi_2) \whatalpha_1(\xi_3) (2i\jxione)^{-1} \widehat{\calN}_c(s,\xi_1) \hatf(s,\xi_2) \, \ud \xi_1 \, \ud \xi_2 \, \ud s \\
  &= I_1(t,\xi) + I_2(t,\xi) + I_3(t,\xi) + I_4(t,\xi).
 \end{aligned}
\end{equation*}
The last term $I_4(t,\xi)$ is again straightforward to estimate. Using Lemma~\ref{lem:frakm_for_delta_three_inputs}, we obtain
\begin{equation*}
 \begin{aligned}
  &2^{-\frac12 \ell} \bigl\|I_4(t,\xi)\bigr\|_{L^2_\xi} \\
  &\lesssim 2^{-\frac12 \ell} \, \bigl\| \xi \, \widetilde{\varphi}_\ell^{(n)}(\xi) \bigr\|_{L^2_\xi} \int_0^t s \cdot \biggl\| \calF^{-1} \biggl[ \iint \frakm(\xi, \xi_1, \xi_2) \whatalpha_1(\xi_3)  (2i\jxione)^{-1} \widehat{\calN}_c(s,\xi_1) \hatv(s,\xi_2) \, \ud \xi_1 \, \ud \xi_2 \biggr] \biggr\|_{L^1_x} \, \ud s \\
  &\lesssim 2^{-\frac12 \ell} \cdot 2^{-\frac12 \ell} \int_0^t s \cdot \bigl\|\calF^{-1}[\frakm]\bigr\|_{L^1(\bbR^3)} \|\alpha_1(x)\|_{L^1_x} \bigl\| (2i\jD)^{-1} \calN_c(s)\bigr\|_{L^\infty_x} \|v(s)\|_{L^\infty_x} \, \ud s \\
  &\lesssim 2^{-\ell} \int_0^t s \cdot 2^\ell \cdot \js^{-2} \bigl( \log(2+s) \bigr)^4 \varepsilon^3 \, \ud s \\
  &\lesssim \bigl( \log(2+t) \bigr)^5 \varepsilon^3.
 \end{aligned}
\end{equation*}
In order to estimate the other three terms $I_j(t,\xi)$, $1 \leq j \leq 3$, we exploit that on the support of the symbol $\frakm(\xi, \xi_1, \xi_2)$, the corresponding phase functions are of size one. Indeed, in our current frequency configuration
\begin{equation*}
 \bigl||\xi|-\sqrt{3}\bigr| \simeq 2^{-\ell-100} \ll 1, \quad |\xi_1| + |\xi_2| \ll 2^{-\frac12 \ell-100} \ll 1,
\end{equation*}
we have that
\begin{equation*}
 \begin{aligned}
  -\jxi + 2 + \jxitwo &\approx -2 + 2 + 1 = 1,  \\
  -\jxi + \jxitwo &\approx -2 + 1 = -1, \\
  -\jxi - 2 + \jxitwo &\approx -2 -2 + 1 = -3.
 \end{aligned}
\end{equation*}
We can therefore integrate by parts in time once more, which then leads to enough time decay to conclude the estimates.
We provide the details for the term $I_1(t,\xi)$, the treatment of the other two terms $I_2(t,\xi)$ and $I_3(t,\xi)$ being analogous.
For the term $I_1(t,\xi)$ we obtain upon integrating by parts in time that
\begin{equation*}
 \begin{aligned}
  I_1(t,\xi) &= - \int_0^t s \cdot \xi \iint e^{is(-\jxi + 2 + \jxitwo)} \widetilde{\frakm}(\xi, \xi_1, \xi_2) \whatalpha_1(\xi_3) (2i\jxione)^{-1} \whatbeta(\xi_1) \\
  &\qquad \qquad \qquad \qquad \qquad \qquad \qquad \times 2 \bigl( e^{-is} v(s,0) \bigr) \ps \bigl( e^{-is} v(s,0) \bigr) \hatf(s,\xi_2) \, \ud \xi_1 \, \ud \xi_2 \, \ud s \\
  &\quad - \int_0^t s \cdot \xi \iint e^{is(-\jxi + 2 + \jxitwo)} \widetilde{\frakm}(\xi, \xi_1, \xi_2) \whatalpha_1(\xi_3) (2i\jxione)^{-1} \whatbeta(\xi_1) \\
  &\qquad \qquad \qquad \qquad \qquad \qquad \qquad \times \bigl( e^{-is} v(s,0) \bigr)^2 \ps \hatf(s,\xi_2) \, \ud \xi_1 \, \ud \xi_2 \, \ud s \\
  &\quad - \int_0^t 1 \cdot \xi \iint e^{is(-\jxi + 2 + \jxitwo)} \widetilde{\frakm}(\xi, \xi_1, \xi_2) \whatalpha_1(\xi_3) (2i\jxione)^{-1} \whatbeta(\xi_1) \\
  &\qquad \qquad \qquad \qquad \qquad \qquad \qquad \times \bigl( e^{-is} v(s,0) \bigr)^2 \hatf(s,\xi_2) \, \ud \xi_1 \, \ud \xi_2 \, \ud s \\
  &\quad+ t \cdot \xi \iint e^{i t (-\jxi + 2 + \jxitwo)} \widetilde{\frakm}(\xi, \xi_1, \xi_2) \whatalpha_1(\xi_3) (2i\jxione)^{-1} \whatbeta(\xi_1) \\
  &\qquad \qquad \qquad \qquad \qquad \qquad \qquad \times \bigl( e^{-it} v(t,0) \bigr)^2 \hatf(t,\xi_2) \, \ud \xi_1 \, \ud \xi_2 \\
  &= I_1^{(a)}(t,\xi) + I_1^{(b)}(t,\xi) + I_1^{(c)}(t,\xi) + I_1^{(d)}(t,\xi),
 \end{aligned}
\end{equation*}
where
\begin{equation*}
 \widetilde{\frakm}(\xi, \xi_1, \xi_2) := (-\jxi + 2 + \jxitwo)^{-1} \frakm(\xi,\xi_1,\xi_2).
\end{equation*}
Since $|-\jxi + 2 + \jxitwo| \simeq 1$ on the support of $\frakm(\xi,\xi_1,\xi_2)$, we can conclude as in the proof of the bound~\eqref{equ:weighted_energy_est_calB1_calFinvfrakm_L1bound} that
\begin{equation*}
 \bigl\|\calF^{-1}[\widetilde{\frakm}]\bigr\|_{L^1(\bbR^3)} \lesssim 2^\ell.
\end{equation*}
Hence, by Lemma~\ref{lem:frakm_for_delta_three_inputs}, and the bounds~\eqref{equ:Linfty_decay_v} as well as~\eqref{equ:pt_phase_filtered_vt0_bound}, we obtain
\begin{equation*}
 \begin{aligned}
  &2^{-\frac12 \ell} \bigl\|I_1^{(a)}(t,\xi)\bigr\|_{L^2_\xi} \\
  &\lesssim 2^{-\frac12 \ell} \, \bigl\| \xi \widetilde{\varphi}_\ell^{(n)}(\xi) \bigr\|_{L^2_\xi} \\
  &\quad \times \int_0^t s \cdot \biggl\| \calF^{-1} \biggl[ \iint \widetilde{\frakm}(\xi, \xi_1, \xi_2) \whatalpha_1(\xi_3) (2i\jxione)^{-1} \widehat{\beta}(\xi_1) \hatv(s,\xi_2) \, \ud \xi_1 \, \ud \xi_2 \biggr] \biggr\|_{L^1_x} |v(s,0)| \bigl| \ps \bigl(e^{-is} v(s,0)\bigr)\bigr| \, \ud s \\
  &\lesssim 2^{-\frac12 \ell} \cdot 2^{-\frac12\ell} \int_0^t s \cdot \bigl\|\calF^{-1}[\widetilde{\frakm}]\bigr\|_{L^1(\bbR^3)} \|\alpha_1(x)\|_{L^2_x} \|(2i\jD)^{-1} \beta\|_{L^2_x} \|v(s)\|_{L^\infty_x} |v(s,0)| \bigl| \ps \bigl(e^{-is} v(s,0)\bigr)\bigr| \, \ud s \\
  &\lesssim 2^{-\ell} \int_0^t s \cdot 2^\ell \cdot \js^{-2} \Bigl( \bigl( \log(2+s) \bigr)^3 \varepsilon^3 + \bigl( \log(2+s) \bigr)^4 \varepsilon^4 \Bigr) \, \ud s \\
  &\lesssim \bigl( \log(2+t) \bigr)^4 \varepsilon^3 + \bigl( \log(2+t) \bigr)^5 \varepsilon^4.
 \end{aligned}
\end{equation*}
The remaining terms $I_1^{(b)}(t,\xi)$, $I_1^{(c)}(t,\xi)$, and $I_1^{(d)}(t,\xi)$ can be treated analogously. For $I_1^{(b)}(t,\xi)$ we use the bound
\[
 \| \ps f(s)\|_{L^\infty_x} \les \jap{s}^{-1}\bigl( \log(2+s) \bigr)^2 \varepsilon^3,
\]
see \eqref{equ:pt_f_Linftyx_bound}, which leads to the same estimate as the one for $I_1^{(a)}(t,\xi)$. On the other hand,  $I_1^{(c)}(t,\xi)$, and $I_1^{(d)}(t,\xi)$ exhibit faster decay by a factor of~$\jap{t}^{-\frac12}$.
This concludes the estimates for the case $1 \leq \ell \leq n-1$.

\medskip
\noindent \underline{\it Case 3: $\ell = 0$.}
Here we write
\begin{equation} \label{equ:pxi_of_calFB1vv_case1_1}
 \begin{aligned}
  \varphi_0^{(n)}(\xi) \jxi^2 \pxi \calF\bigl[ \calB_1[v,v](t) \bigr](\xi) &= \jxi^2 \pxi \Bigl( \varphi_0^{(n)}(\xi) \calF\bigl[ \calB_1[v,v](t) \bigr](\xi) \Bigr) \\
  &\quad \quad - \jxi^2 \bigl(\varphi_0^{(n)}(\xi)\bigr)' \calF\bigl[ \calB_1[v,v](t) \bigr](\xi).
 \end{aligned}
\end{equation}
Then the contribution of the second term on the right-hand side of~\eqref{equ:pxi_of_calFB1vv_case1_1} is straightforward to bound. Using that $(\varphi_0^{(n)}(\xi))' = 0$ for $|\xi| \geq 10$, we obtain
\begin{equation*}
 \begin{aligned}
  \bigl\| \jxi^2 \bigl(\varphi_0^{(n)}(\xi)\bigr)' \calF\bigl[ \calB_1[v,v](t) \bigr](\xi) \bigr\|_{L^2_\xi} &\lesssim \int_0^t \bigl\| \bigl(\varphi_0^{(n)}(\xi)\bigr)' \jxi \calF\bigl[ \alpha_1(x) v(s)^2 \bigr](\xi) \bigr\|_{L^2_\xi} \, \ud s \\
  &\lesssim \int_0^t \bigl\| \alpha_1(x) v(s)^2 \bigr\|_{L^2_x} \, \ud s \\
  &\lesssim \int_0^t \|\alpha_1(x)\|_{L^2_x} \|v(s)\|_{L^\infty_x}^2 \, \ud s \\
  &\lesssim \int_0^t \js^{-1} \bigl( \log(2+s) \bigr)^2 \varepsilon^2 \, \ud s \lesssim \bigl( \log(2+t) \bigr)^3 \varepsilon^2.
 \end{aligned}
\end{equation*}
In order to estimate the contribution of the first term on the right-hand side of~\eqref{equ:pxi_of_calFB1vv_case1_1}, we first decompose it as
\begin{equation} \label{equ:pxi_of_calFB1vv_case1_2}
 \begin{aligned}
  \varphi_0^{(n)}(\xi) \calF\bigl[ \calB_1[v,v](t) \bigr](\xi) &= \varphi_0^{(n)}(\xi) \int_0^t (2i\jxi)^{-1} e^{-is\jxi} \calF\bigl[ \alpha_1(x) v(s)^2 \bigr](\xi) \, \ud s \\
  &= \varphi_0^{(n)}(\xi) \int_0^t (2i\jxi)^{-1} e^{is(2-\jxi)} \widehat{\alpha}_1(\xi) \bigl( e^{-is} v(s,0) \bigr)^2 \, \ud s \\
  &\quad + \varphi_0^{(n)}(\xi) \int_0^t (2i\jxi)^{-1} e^{-is\jxi} \calF\bigl[ \alpha_1(\cdot) \bigl( v(s)^2 - v(s,0)^2 \bigr)\bigr](\xi) \, \ud s \\
  &=: I(t,\xi) + II(t,\xi).
 \end{aligned}
\end{equation}
Since $|2-\jxi| \gtrsim 2^{-100}$ on the support of $\varphi_0^{(n)}(\xi)$, we can integrate by parts in time in the first term $I(t,\xi)$ and obtain
\begin{equation*}
 \begin{aligned}
  I(t,\xi) &= -i \, \varphi_0^{(n)}(\xi) (2i\jxi)^{-1} (2-\jxi)^{-1} e^{-it\jxi} \whatalpha_1(\xi) v(t,0)^2 \\
  &\quad + i \, \varphi_0^{(n)}(\xi) (2i\jxi)^{-1} (2-\jxi)^{-1} \whatalpha_1(\xi) v(0,0)^2 \\
  &\quad + 2i \int_0^t (2i\jxi)^{-1} e^{-is\jxi} \varphi_0^{(n)}(\xi) (2-\jxi)^{-1} \whatalpha_1(\xi) e^{is} \ps \bigl( e^{-is} v(s,0) \bigr) v(s,0) \, \ud s \\
  &=: I_1(t,\xi) + I_2(t,\xi) + I_3(t,\xi).
 \end{aligned}
\end{equation*}
To estimate $\|\jxi \pxi I_1(t,\xi)\|_{L^2_\xi}$, we observe that the worst term occurs when $\pxi$ falls onto $e^{-it\jxi}$. This produces a factor of $t$, which however gets mostly compensated by the decay $|v(t,0)|^2 \lesssim \jt^{-1} (\log(2+t))^2 \varepsilon^2$. Hence, $\|\jxi \pxi I_1(t,\xi)\|_{L^2_\xi} \lesssim (\log(2+t))^2 \varepsilon^2$. Clearly, we have $\|\jxi \pxi I_2(t,\xi)\|_{L^2_\xi} \lesssim \varepsilon^2$. In order to bound the contribution of $I_3(t,\xi)$, we observe that by the bound~\eqref{equ:pt_phase_filtered_vt0_bound}, we have uniformly for all $0 \leq s \leq T$,
\begin{equation*}
 \begin{aligned}
  &\bigl\| \jx^2 \jD \varphi_0^{(n)}(D) (2-\jap{D})^{-1} \alpha_1(\cdot) e^{is} \ps \bigl( e^{-is} v(s,0) \bigr) v(s,0) \bigr\|_{L^2_x} \\
  &\quad \lesssim \bigl| \ps \bigl( e^{-is} v(s,0) \bigr) \bigr| |v(s,0)|
  \lesssim \js^{-\thf} \Bigl( \bigl( \log(2+s) \bigr)^2 \varepsilon^2 + \bigl( \log(2+s) \bigr)^3 \varepsilon^3 \Bigr).
 \end{aligned}
\end{equation*}
Thus, Proposition~\ref{prop:prop49} gives the bound $\|\jxi^2 \pxi I_3(t,\xi)\|_{L^2_\xi} \lesssim (\log(2+t))^{\frac52} \varepsilon^2 + (\log(2+t))^{\frac72} \varepsilon^3$.

Finally, we can also invoke Proposition~\ref{prop:prop49} to obtain the estimate $\|\jxi^2 \pxi II(t,\xi)\|_{L^2_\xi} \lesssim (\log(2+t))^{\frac52} \varepsilon^2$ for the second term on the right-hand side of \eqref{equ:pxi_of_calFB1vv_case1_2} since by the improved local decay estimates~\eqref{equ:local_decay_px_v} and \eqref{equ:improved_L2decay_v_minus_v0}, we have for $0 \leq s \leq T$,
\begin{equation*}
 \begin{aligned}
  &\bigl\| \jx^2 \jD \bigl( \alpha_1(\cdot) ( v(s)^2 - v(s,0)^2 ) \bigr) \bigr\|_{L^2_x} \\
  &\lesssim \| \jx^4 \alpha_1(x) \|_{W^{1,\infty}_x} \bigl\| \jx^{-2} \bigl( v(s) - v(s,0) \bigr) \bigr\|_{L^2_x} \|v(s)\|_{L^\infty_x} \\
  &\quad + \| \jx^3 \alpha_1(x) \|_{L^\infty_x} \|\jx^{-1} \px v(s)\|_{L^2_x} \|v(s)\|_{L^\infty_x} \\
  &\lesssim \js^{-\thf} \bigl( \log(2+s) \bigr)^2 \varepsilon^2.
 \end{aligned}
\end{equation*}
This finishes the estimates for the case $\ell=0$ and concludes the proof of Proposition~\ref{prop:weighted_energy_est_calB1_bad}.
\end{proof}

Next, we turn to the milder non-resonant quadratic interactions in $\calB_1[v,\bv](t)$ and in $\calB_1[\bv,\bv](t)$, for which we establish the following stronger weighted energy estimates.
\begin{proposition} \label{prop:weighted_energy_est_calB1_mild}
 Under the assumptions of Proposition~\ref{prop:main_bootstrap} we have for all $0 \leq t \leq T$ that
 \begin{align}
  \bigl\| \jxi^2 \partial_\xi \calF\bigl[ \calB_1[v,\bv](t) \bigr](\xi) \bigr\|_{L^2_\xi} &\lesssim \bigl( \log(2+t) \bigr)^{\frac52} \varepsilon^2 + \bigl( \log(2+t) \bigr)^{\frac72} \varepsilon^3, \label{equ:calB1_vbv_weighted_estimate} \\
  \bigl\| \jxi^2 \partial_\xi \calF\bigl[ \calB_1[\bv,\bv](t) \bigr](\xi) \bigr\|_{L^2_\xi} &\lesssim \bigl( \log(2+t) \bigr)^{\frac52} \varepsilon^2 + \bigl( \log(2+t) \bigr)^{\frac72} \varepsilon^3. \label{equ:calB1_bvbv_weighted_estimate}
 \end{align}
\end{proposition}
\begin{proof}
 We begin with the estimate for $\calB_1[v,\bv](t)$. To this end we write
 \begin{equation} \label{equ:calB1_vbv_decomposition}
 \begin{aligned}
  \calF\bigl[ \calB_1[v,\bv](t) \bigr](\xi) &= \int_0^t (2i\jxi)^{-1} e^{-is\jxi} \whatalpha_1(\xi) \bigl( e^{-is} v(s,0) \bigr) \overline{\bigl(e^{-is} v(s,0)\bigr)} \, \ud s \\
  &\quad + \int_0^t (2i\jxi)^{-1} e^{-is\jxi} \calF\bigl[ \alpha_1(\cdot) \bigl( v(s,\cdot) \bv(s,\cdot) - v(s,0)\bv(s,0) \bigr) \bigr](\xi) \, \ud s \\
  &=: I(t,\xi) + II(t,\xi).
 \end{aligned}
 \end{equation}
 In the first term $I(t,\xi)$ we can integrate by parts in time to obtain
 \begin{equation}
 \begin{aligned}
  I(t,\xi) &= i (2i\jxi)^{-1} e^{-it\jxi} \jxi^{-1} \widehat{\alpha}_1(\xi) v(t,0) \bv(t,0) \\
  &\quad - i (2i\jxi)^{-1} \jxi^{-1} \widehat{\alpha}_1(\xi) v(0,0) \bv(0,0) \\
  &\quad - i \int_0^t (2i\jxi)^{-1} e^{-is\jxi} \jxi^{-1} \widehat{\alpha}_1(\xi) \ps \bigl( e^{-is} v(s,0) \bigr) \overline{\bigl( e^{-is} v(s,0) \bigr)} \, \ud s \\
  &\quad - i \int_0^t (2i\jxi)^{-1} e^{-is\jxi} \jxi^{-1} \widehat{\alpha}_1(\xi) \bigl( e^{-is} v(s,0) \bigr) \overline{\ps \bigl( e^{-is} v(s,0) \bigr)} \, \ud s \\
  &= I_1(t,\xi) + I_2(t,\xi) + I_3(t,\xi) + I_4(t,\xi).
 \end{aligned}
 \end{equation}
 Using the decay estimate~\eqref{equ:Linfty_decay_v}, it is straightforward to infer for the terms $I_1(t,\xi)$ and $I_2(t,\xi)$ for times $0 \leq t \leq T$ that
 \begin{equation*}
  \bigl\| \jxi^2 \pxi I_1(t,\xi) \bigr\|_{L^2_\xi} + \bigl\| \jxi^2 \pxi I_2(t,\xi) \bigr\|_{L^2_\xi} \lesssim \bigl( \log(2+t) \bigr)^2 \varepsilon^2.
 \end{equation*}
 To bound the contributions of the terms $I_3(t,\xi)$ and $I_4(t,\xi)$, we observe that their integrands have cubic-type decay in the sense that by \eqref{equ:pt_phase_filtered_vt0_bound} we have for $0 \leq s \leq T$ that
 \begin{equation*}
 \begin{aligned}
  \bigl\| \jx^2 \jD \jD^{-1} \alpha_1(\cdot) \ps \bigl( e^{-is} v(s,0) \bigr) \overline{\bigl( e^{-is} v(s,0) \bigr)} \bigr\|_{L^2_x} &\lesssim \|\jx^2 \alpha_1(x)\|_{L^2_x} \bigl| \ps \bigl( e^{-is} v(s,0) \bigr) \bigr| |v(s,0)| \\
  &\lesssim \js^{-\thf} \Bigl( \bigl( \log(2+s) \bigr)^2 \varepsilon^2 + \bigl( \log(2+s) \bigr)^3 \varepsilon^3 \Bigr).
 \end{aligned}
 \end{equation*}
 Thus, Proposition~\ref{prop:prop49} gives for $0 \leq t \leq T$ the acceptable bounds
 \begin{equation*}
  \bigl\| \jxi^2 \pxi I_j(t,\xi) \bigr\|_{L^2_\xi} \lesssim \bigl( \log(2+t) \bigr)^{\frac52} \varepsilon^2 + \bigl( \log(2+t) \bigr)^{\frac72} \varepsilon^3, \quad j = 3,4.
 \end{equation*}
 Similarly, we can use Proposition~\ref{prop:prop49} to estimate the contributions of the second term $II(t,\xi)$ on the right-hand side of \eqref{equ:calB1_vbv_decomposition} for times $0 \leq t \leq T$ by
 \begin{equation*}
  \bigl\| \jxi^2 \pxi II(t,\xi) \bigr\|_{L^2_\xi} \lesssim \bigl( \log(2+t) \bigr)^{\frac52} \varepsilon^2,
 \end{equation*}
 upon noting that through \eqref{equ:Linfty_decay_v} and \eqref{equ:improved_L2decay_v_minus_v0} we have for $0 \leq s \leq T$,
 \begin{equation*}
  \bigl\| \jx^2 \jD \bigl( \alpha_1(\cdot) ( v(s,\cdot) \bv(s,\cdot) - v(s,0)\bv(s,0) ) \bigr) \bigr\|_{L^2_x} \lesssim \js^{-\thf} \bigl( \log(2+s) \bigr)^2 \varepsilon^2.
 \end{equation*}
 This finishes the proof of the weighted estimate \eqref{equ:calB1_vbv_weighted_estimate} for $\calB_1[v,\bv]$. The proof of~\eqref{equ:calB1_bvbv_weighted_estimate} is analogous. The  phase function in that case equals $e^{-is(2+\jxi)}$, so we can once again perform integration by parts in time.
\end{proof}

We now turn to the weighted energy estimates for the terms $\calB_2[\cdot,\cdot](t)$ and $\calB_3[\cdot,\cdot](t)$. The proofs are largely identical to the the proofs of Proposition~\ref{prop:weighted_energy_est_calB1_bad} and Proposition~\ref{prop:weighted_energy_est_calB1_mild}.
In what follows, we therefore only indicate the main differences in the proofs. We begin with the problematic resonant quadratic interactions in $\calB_2[v,v](t)$ and in $\calB_3[v,v](t)$.
\begin{proposition} \label{prop:weighted_energy_est_calB2andB3_bad}
 Under the assumptions of Proposition~\ref{prop:main_bootstrap} we have for all $0 \leq t \leq T$ that
 \begin{align*}
  \sup_{n \geq 1} \, \sup_{0 \leq \ell \leq n} \, 2^{-\frac12 \ell} \tau_n(t) \bigl\| \varphi_\ell^{(n)}(\xi) \jxi^2 \partial_\xi \calF\bigl[ \calB_2[v,v](t) \bigr](\xi) \bigr\|_{L^2_\xi} &\lesssim \bigl( \log(2+t) \bigr)^3 \varepsilon^2 + \bigl( \log(2+t) \bigr)^5 \varepsilon^3, \\
  \sup_{n \geq 1} \, \sup_{0 \leq \ell \leq n} \, 2^{-\frac12 \ell} \tau_n(t) \bigl\| \varphi_\ell^{(n)}(\xi) \jxi^2 \partial_\xi \calF\bigl[ \calB_3[v,v](t) \bigr](\xi) \bigr\|_{L^2_\xi} &\lesssim \bigl( \log(2+t) \bigr)^3 \varepsilon^2 + \bigl( \log(2+t) \bigr)^5 \varepsilon^3.
 \end{align*}
\end{proposition}
\begin{proof}
 The proofs are analogous to the case of $\calB_1[v,v](t)$. There is, however, the following modification. On the one hand, recall that for $\calB_1[v,v](t)$ in the case $1 \leq \ell \leq n-1$, we write
\begin{equation*}
 \begin{aligned}
  \varphi_\ell^{(n)}(\xi) \int_0^t s \cdot \xi e^{-is\jxi} \calF\Bigl[ \alpha_1(\cdot) (P_{\leq -\frac12\ell -100} v)(s) (P_{\leq -\frac12\ell-100}v)(s) \Bigr](\xi) \, \ud s
 \end{aligned}
\end{equation*}
as
\begin{equation*}
 \begin{aligned}
  &\int_0^t s \cdot \xi \iint e^{is(-\jxi + \jxione + \jxitwo)} \varphi_\ell^{(n)}(\xi) \varphi_{\leq -\frac12\ell-100}(\xi_1) \varphi_{\leq -\frac12\ell-100}(\xi_2) \\
  &\qquad \qquad \qquad \qquad \qquad \qquad \qquad \qquad \qquad \qquad \times \whatalpha_1(\xi-\xi_1-\xi_2) \hatf(s,\xi_1) \hatf(s,\xi_2) \, \ud \xi_1 \, \ud \xi_2 \, \ud s
 \end{aligned}
\end{equation*}
On the other hand, for $\calB_2[v,v](t)$ and $\calB_3[v,v](t)$, due to the inner products $\langle G, v(s) \rangle$, we arrive at the following expressions:
 \begin{equation*}
  \varphi_\ell^{(n)}(\xi) \int_0^t s \cdot \xi e^{-is\jxi} \calF\Bigl[ \alpha_2(\cdot) (P_{\leq -\frac12\ell -100} v)(s) \bigl\langle G, (P_{\leq -\frac12\ell -100} v)(s) \bigr\rangle \Bigr](\xi) \, \ud s
 \end{equation*}
 becomes
 \begin{equation*}
 \begin{aligned}
  &\int_0^t s \cdot \xi \iint e^{is(-\jxi + \jxione + \jxitwo)} \varphi_\ell^{(n)}(\xi) \varphi_{\leq -\frac12\ell-100}(\xi_1) \varphi_{\leq -\frac12\ell-100}(\xi_2) \check{G}(\xi_2) \\
  &\qquad \qquad \qquad \qquad \qquad \qquad \qquad \qquad \qquad \qquad \times\whatalpha_2(\xi-\xi_1) \hatf(s,\xi_1) \hatf(s,\xi_2) \, \ud \xi_1 \, \ud \xi_2 \, \ud s,
 \end{aligned}
 \end{equation*}
 while
 \begin{equation*}
  \varphi_\ell^{(n)}(\xi) \int_0^t s \cdot \xi e^{-is\jxi} \calF\Bigl[ \alpha_3(\cdot) \bigl\langle G, (P_{\leq -\frac12\ell -100} v)(s) \bigr\rangle \bigl\langle G, (P_{\leq -\frac12\ell -100} v)(s) \bigr\rangle \Bigr](\xi) \, \ud s
 \end{equation*}
 becomes
 \begin{equation*}
 \begin{aligned}
  &\int_0^t s \cdot \xi \iint e^{is(-\jxi + \jxione + \jxitwo)} \varphi_\ell^{(n)}(\xi) \varphi_{\leq -\frac12\ell-100}(\xi_1) \varphi_{\leq -\frac12\ell-100}(\xi_2) \check{G}(\xi_1) \check{G}(\xi_2) \\
  &\qquad \qquad \qquad \qquad \qquad \qquad \qquad \qquad \qquad \qquad \times\whatalpha_3(\xi) \hatf(s,\xi_1) \hatf(s,\xi_2) \, \ud \xi_1 \, \ud \xi_2 \, \ud s.
 \end{aligned}
 \end{equation*}
 After integrating by parts in time in both cases, we therefore use the variants of Lemma~\ref{lem:frakm_for_delta_three_inputs} as stated in Lemma~\ref{lem:m modified}. Note that the bounds provided by that lemma require those input functions being paired with $G$ to be placed into~$L^\infty_x$. However, this gives the desired estimate in all cases. Otherwise, everything else is essentially identical.
\end{proof}

For the milder quadratic interactions $\calB_2[v, \bv](t)$, $\calB_2[\bv, \bv](t)$, $\calB_3[v, \bv](t)$, and $\calB_3[\bv, \bv](t)$, we can again establish stronger weighted energy estimates.
\begin{proposition} \label{prop:weighted_energy_est_calB2and3_mild}
 Under the assumptions of Proposition~\ref{prop:main_bootstrap} we have for all $0 \leq t \leq T$ that
 \begin{align*}
  \bigl\| \jxi^2 \partial_\xi \calF\bigl[ \calB_2[v,\bv](t) \bigr](\xi) \bigr\|_{L^2_\xi} &\lesssim \bigl( \log(2+t) \bigr)^{\frac52} \varepsilon^2 + \bigl( \log(2+t) \bigr)^{\frac72} \varepsilon^3, \\
  \bigl\| \jxi^2 \partial_\xi \calF\bigl[ \calB_2[\bv,\bv](t) \bigr](\xi) \bigr\|_{L^2_\xi} &\lesssim  \bigl( \log(2+t) \bigr)^{\frac52} \varepsilon^2 + \bigl( \log(2+t) \bigr)^{\frac72} \varepsilon^3, \\
  \bigl\| \jxi^2 \partial_\xi \calF\bigl[ \calB_3[v,\bv](t) \bigr](\xi) \bigr\|_{L^2_\xi} &\lesssim \bigl( \log(2+t) \bigr)^{\frac52} \varepsilon^2 + \bigl( \log(2+t) \bigr)^{\frac72} \varepsilon^3, \\
  \bigl\| \jxi^2 \partial_\xi \calF\bigl[ \calB_3[\bv,\bv](t) \bigr](\xi) \bigr\|_{L^2_\xi} &\lesssim  \bigl( \log(2+t) \bigr)^{\frac52} \varepsilon^2 + \bigl( \log(2+t) \bigr)^{\frac72} \varepsilon^3. 
 \end{align*}
\end{proposition}
\begin{proof}
 This is again the easier, non-resonant case. The details are essentially identical with Proposition~\ref{prop:weighted_energy_est_calB1_mild}, and we leave them to the reader.
\end{proof}

 Combining Proposition~\ref{prop:weighted_energy_est_calB1_bad}, Proposition~\ref{prop:weighted_energy_est_calB1_mild}, Proposition~\ref{prop:weighted_energy_est_calB2andB3_bad}, and Proposition~\ref{prop:weighted_energy_est_calB2and3_mild} furnishes the proof of the weighted energy estimates for the main quadratic interactions $\calQ_r(v+\bv)$ asserted in Proposition~\ref{prop:weighted_energy_est_main_quadratic}.

\section{Weighted Energy Estimates for the Main Cubic Interactions} \label{sec:weighted_main_cubic}

In this section we establish the weighted energy estimates for the singular cubic interactions $\calC_{\delta_0}(v+\bv)$ and $\calC_{\pvdots}(v+\bv)$.
We will separately treat the cubic interactions with a Dirac kernel and those with a Hilbert-type kernel.
The results are summarized in the following two propositions.

\begin{proposition} \label{prop:weighted_energy_est_delta_cubic}
 Under the assumptions of Proposition~\ref{prop:main_bootstrap} we have for all $0 \leq t \leq T$ that
 \begin{equation} \label{equ:weighted_energy_est_delta_cubic}
  \sup_{n \geq 1} \, \sup_{0 \leq \ell \leq n} \, 2^{-\frac12 \ell} \tau_n(t) \biggl\| \varphi_\ell^{(n)}(\xi) \jxi^2 \partial_\xi \int_0^t (2i\jxi)^{-1} e^{-is\jxi} \calF\bigl[ \calC_{\delta_0}\bigl(v(s)+\bar{v}(s)\bigr) \bigr](\xi) \, \ud s \biggr\|_{L^2_\xi} \lesssim \bigl( \log(2+t) \bigr)^6 \varepsilon^3.
 \end{equation}
\end{proposition}

\begin{proposition} \label{prop:weighted_energy_est_pv_cubic}
 Under the assumptions of Proposition~\ref{prop:main_bootstrap} we have for all $0 \leq t \leq T$ that
 \begin{equation} \label{equ:weighted_energy_est_pv_cubic}
  \sup_{n \geq 1} \, \sup_{0 \leq \ell \leq n} \, 2^{-\frac12 \ell} \tau_n(t) \biggl\| \varphi_\ell^{(n)}(\xi) \jxi^2 \partial_\xi \int_0^t (2i\jxi)^{-1} e^{-is\jxi} \calF\bigl[ \calC_{\pvdots}\bigl(v(s)+\bar{v}(s)\bigr) \bigr](\xi) \, \ud s \biggr\|_{L^2_\xi} \lesssim \bigl( \log(2+t) \bigr)^6 \varepsilon^3.
 \end{equation}
\end{proposition}

\subsection{Preliminaries}

In view of the structures \eqref{equ:FT_cubic_interactions_dirac}, respectively \eqref{equ:FT_cubic_interactions_pv}, of the singular cubic interactions $\calC_{\delta_0}(v+\bv)$, respectively $\calC_{\pvdots}(v+\bv)$, the proofs of Proposition~\ref{prop:weighted_energy_est_delta_cubic} and Proposition~\ref{prop:weighted_energy_est_pv_cubic} amount to weighted energy estimates for trilinear terms whose inputs are
\begin{equation} \label{equ:g_inputs_gathering_part}
 g(t) := f(t) \quad \text{or} \quad g(t) = m_a(D) f(t) \quad \text{for some} \quad a \in \{0, 4, 5\},
\end{equation}
or complex conjugates thereof, with the multipliers $m_a(D)$ defined in \eqref{equ:def_multipliers_m}.
In the next lemma, we gather several estimates for the inputs~\eqref{equ:g_inputs_gathering_part} that will be used frequently throughout this section.

\begin{lemma} \label{lem:g_bounds_repeated}
 Suppose that the assumptions in the statement of Proposition~\ref{prop:main_bootstrap} are in place.
 Let $f(t) = e^{-it\jD} v(t)$ be the profile of the solution $v(t)$ to~\eqref{equ:v_equ_refer_to} and let
 \begin{equation*}
  g(t) := f(t) \quad \text{or} \quad g(t) := m_a(D) f(t) \quad \text{for some} \quad a \in \{0, 4, 5\},
 \end{equation*}
 with the multipliers $m_a(D)$ defined in \eqref{equ:def_multipliers_m}.
 Then we have uniformly for all $0 \leq t \leq T$ that
 \begin{align}
  \bigl\| \jxi \pxi \hatg(t,\xi) \bigr\|_{L^1_\xi} &\lesssim \log(2+t) \varepsilon, \label{equ:g_bound_pxiL1} \\
  \bigl\| \hatg(t,\xi) \bigr\|_{L^\infty_\xi} &\lesssim \log(2+t) \varepsilon, \label{equ:g_bound_Linftyxi}
 \end{align}
 as well as
 \begin{align}
  \bigl\| e^{it\jD} g(t) \bigr\|_{L^\infty_x} &\lesssim \frac{\log(2+t)}{\jt^\hf} \varepsilon,\label{equ:g_bound_dispersive_est} \\
  \bigl\| e^{it\jD} P_k g(t) \bigr\|_{L^\infty_x} &\lesssim 2^{-\frac12 k} \frac{\log(2+t)}{\jt^\hf} \varepsilon, \quad k \geq 0, \label{equ:g_bound_dispersive_est_kgain} \\
  \bigl\| \jx^{-1} (\jD^{-1} D) e^{it\jD} g(t) \bigr\|_{L^2_x} &\lesssim \frac{\log(2+t)}{\jt} \varepsilon. \label{equ:g_bound_improved_local_decay}
 \end{align}
 Moreover, we have
 \begin{align}
  \bigl\| \jD^2 g(t) \bigr\|_{L^2_x} &\lesssim \varepsilon, \label{equ:g_bound_jD2_L2} \\
  \bigl\| \pxi \hatg(t,\xi) \bigr\|_{L^2_\xi} &\lesssim \jt^{\frac12} \varepsilon, \label{equ:g_bound_pxi_crude} \\
  \bigl\| \jD \pt g(t) \bigr\|_{L^2_x} &\lesssim \frac{\bigl( \log(2+t) \bigr)^2}{\jt} \varepsilon^2, \label{equ:g_bound_jD_pt_L2} \\
  \bigl\| \pt g(t) \bigr\|_{L^\infty_x} &\lesssim \frac{\bigl( \log(2+t) \bigr)^2}{\jt} \varepsilon^2, \label{equ:g_bound_jD_pt_Linfty} \\
  \bigl\| \jxi \pxi \pt \hatg(t,\xi) \bigr\|_{L^2_\xi} &\lesssim \bigl( \log(2+t) \bigr)^2 \varepsilon^2. \label{equ:g_bound_pxi_pt_L2xi}
 \end{align}
 Finally, the Fourier transform $\hatg(t,\xi)$ satisfies the evolution equation
 \begin{equation} \label{equ:decomposition_FT_pt_hatg}
  \pt \hatg(t,\xi) = (2i\jxi)^{-1} e^{-it\jxi} \widehat{\beta}(\xi) \bigl( v(s,0) + \bv(s,0) \bigr)^2 + (2i\jxi)^{-1} e^{-it\jxi} \widehat{\calN}_{c}(t,\xi),
 \end{equation}
 where $\widehat{\beta}(\xi)$ is a Schwartz function and where uniformly for all $0 \leq t \leq T$,
 \begin{equation} \label{equ:decomposition_FT_pt_hatg_Ncdecay}
  \bigl\| \calN_{c}(t) \bigr\|_{L^\infty_x} \lesssim \frac{\bigl(\log(2+t)\bigr)^3}{\jt^\thf} \varepsilon^2.
 \end{equation}
\end{lemma}
\begin{proof}
We first record that $\|g\|_{N_T} \lesssim \|f\|_{N_T} \lesssim \varepsilon$ by the boundedness of the multipliers $m_a(\xi)$ and their derivatives. Now we begin with the proof of \eqref{equ:g_bound_pxiL1}. Fix $0 \leq t \leq T$. Let $n \geq 1$ be an integer such that $\tau_n(t) \geq \frac12$ and therefore $2^n \simeq t$. Then we have by the Cauchy-Schwarz inequality
\begin{equation*}
 \begin{aligned}
  \bigl\| \jxi \pxi \hatg(t,\xi) \bigr\|_{L^1_\xi} &\leq \sum_{0 \leq \ell \leq n} \, \bigl\| \varphi_\ell^{(n)}(\xi) \jxi \pxi \hatg(t,\xi) \bigr\|_{L^1_\xi} \\
  &\leq 2 \sum_{0 \leq \ell \leq n} \tau_n(t) \bigl\| \varphi_\ell^{(n)}(\xi) \jxi \pxi \hatg(t,\xi) \bigr\|_{L^1_\xi} \\
  &\lesssim n \cdot \sup_{0 \leq \ell \leq n} 2^{-\frac12 \ell} \tau_n(t) \bigl\| \varphi_\ell^{(n)}(\xi) \jxi \pxi \hatg(t,\xi) \bigr\|_{L^2_\xi} \\
  &\lesssim \log(2+t) \|g\|_{N_T}.
 \end{aligned}
\end{equation*}
Since all constants in the preceding estimate are uniform for all times $0 \leq t \leq T$, the asserted bound~\eqref{equ:g_bound_pxiL1} follows.
The estimate~\eqref{equ:g_bound_Linftyxi} is a consequence of Sobolev embedding and the bound \eqref{equ:g_bound_pxiL1}. Indeed, for any $0 \leq t \leq T$ we have
\begin{equation*}
 \begin{aligned}
  \bigl\| \hatg(t,\xi) \bigr\|_{L^\infty_\xi} &\lesssim \bigl\| \hatg(t,\xi) \bigr\|_{L^1_\xi} + \bigl\| \pxi \hatg(t,\xi) \bigr\|_{L^1_\xi} \lesssim \bigl\| \jxi \hatg(t,\xi) \bigr\|_{L^2_\xi} + \bigl\| \pxi \hatg(t,\xi) \bigr\|_{L^1_\xi} \lesssim \log(2+t) \|g\|_{N_T}.
 \end{aligned}
\end{equation*}
Next, the decay estimate~\eqref{equ:g_bound_dispersive_est} can be established by proceeding exactly as in the proof of Lemma~\ref{lem:Linfty_decay_v}. The related decay estimate~\eqref{equ:g_bound_dispersive_est_kgain} with high-frequency gain follows similarly using Lemma~\ref{lem:decwithgain}, and similarly for the improved local decay estimate~\eqref{equ:g_bound_improved_local_decay} using Proposition~\ref{prop:local_decay_estimates}.
The bound~\eqref{equ:g_bound_jD2_L2} is obvious, and the crude estimate~\eqref{equ:g_bound_pxi_crude} follows from the definition of the $N_T$ norm. The bounds~\eqref{equ:g_bound_jD_pt_L2}, \eqref{equ:g_bound_jD_pt_Linfty}, and \eqref{equ:g_bound_pxi_pt_L2xi} are corollaries of Lemma~\ref{lem:pt_f_bounds}.
Finally, by a slight abuse of notation the representation~\eqref{equ:decomposition_FT_pt_hatg} of the equation for $\pt \hatg(t,\xi)$ follows from Lemma~\ref{lem:decomposition_FT_pt_hatf}.
\end{proof}

\subsection{Cubic interactions with a Dirac kernel}

In this subsection we turn to the proof of Proposition~\ref{prop:weighted_energy_est_delta_cubic} and establish the weighted energy estimate for the singular cubic interactions $\calC_{\delta_0}(v+\bv)$.
In view of their structure~\eqref{equ:FT_cubic_interactions_dirac}, their contribution to the profile
\begin{equation*}
 \int_0^t (2i\jxi)^{-1} e^{-is\jxi} \calF\bigl[ \calC_{\delta_0}\bigl(v(s)+\bar{v}(s)\bigr) \bigr](\xi) \, \ud s
\end{equation*}
is a linear combination of trilinear terms of the form
\begin{equation} \label{equ:cubic_dirac_precise_trilinear_terms}  
 \begin{aligned}
  \int_0^t \jxi^{-1} \iint e^{is(-\jxi \pm_1 \jxione \pm_2 \jxitwo \pm_3 \jxithree)} \widehat{g^{\pm_1}_1}(s,\xi_1) \widehat{g^{\pm_2}_2}(s,\xi_2) \widehat{g^{\pm_3}_3}(s,\xi_3) \, \ud \xi_1 \, \ud \xi_2 \, \ud s,
 \end{aligned}
\end{equation}
with
\begin{equation*}
 \xi_3 := \xi-\xi_1-\xi_2,
\end{equation*}
and where for $1 \leq j \leq 3$,
\begin{equation} \label{equ:g_inputs_dirac_precise_form}
 \widehat{g^{\pm_j}_j}(s,\xi_j) = \widehat{f^\pm}(s,\xi_j) \quad \text{or} \quad \widehat{g^{\pm_j}_j}(s,\xi_j) = m_{a_j}(\xi_j) \widehat{f^\pm}(s,\xi_j) \quad \text{for some} \quad a_j \in \{0, 4, 5\},
\end{equation}
with the multipliers $m_a(D)$ defined in \eqref{equ:def_multipliers_m}.
Here we use the convention
\begin{equation*}
 \widehat{f^+}(s,\xi) := \hatf(s,\xi), \quad \widehat{f^-}(s,\xi) := \hatbarf(s,\xi).
\end{equation*}
In order to bound the contributions of the terms~\eqref{equ:cubic_dirac_precise_trilinear_terms} to the weighted energy estimate~\eqref{equ:weighted_energy_est_delta_cubic} for the singular cubic interactions $\calC_{\delta_0}(v+\bv)$, we only use the bounds from Lemma~\ref{lem:g_bounds_repeated}. It therefore does not matter which one of the four possible types in~\eqref{equ:g_inputs_dirac_precise_form} every input in~\eqref{equ:cubic_dirac_precise_trilinear_terms} precisely assumes.
For this reason, it actually suffices to establish the weighted energy estimates for the following four terms
\begin{equation} \label{equ:cubic_dirac_four_trilinear_terms}
 \begin{aligned}
  \calF\bigl[\calT_1^{\delta_0}[g](t)\bigr](\xi) &:=  \int_0^t \jxi^{-1} \iint e^{is\phi_1(\xi,\xi_1,\xi_2)} \hatg(s,\xi_1) \hatg(s,\xi_2) \hatg(s,\xi_3) \, \ud \xi_1 \, \ud \xi_2 \, \ud s, \\
  \calF\bigl[\calT_2^{\delta_0}[g](t)\bigr](\xi) &:=  \int_0^t \jxi^{-1} \iint e^{is\phi_2(\xi,\xi_1,\xi_2)} \hatg(s,\xi_1) \hatbarg(s,\xi_2) \hatg(s,\xi_3) \, \ud \xi_1 \, \ud \xi_2 \, \ud s, \\
  \calF\bigl[\calT_3^{\delta_0}[g](t)\bigr](\xi) &:=  \int_0^t \jxi^{-1} \iint e^{is\phi_3(\xi,\xi_1,\xi_2)} \hatg(s,\xi_1) \hatbarg(s,\xi_2) \hatbarg(s,\xi_3) \, \ud \xi_1 \, \ud \xi_2 \, \ud s, \\
  \calF\bigl[\calT_4^{\delta_0}[g](t)\bigr](\xi) &:=  \int_0^t \jxi^{-1} \iint e^{is\phi_4(\xi,\xi_1,\xi_2)} \hatbarg(s,\xi_1) \hatbarg(s,\xi_2) \hatbarg(s,\xi_3) \, \ud \xi_1 \, \ud \xi_2 \, \ud s,
 \end{aligned}
\end{equation}
with phases
\begin{equation*}
 \begin{aligned}
  \phi_1(\xi,\xi_1,\xi_2) &:= -\jxi + \jxione + \jxitwo + \jap{\xi_3}, \\
  \phi_2(\xi,\xi_1,\xi_2) &:= -\jxi + \jxione - \jxitwo + \jap{\xi_3}, \\
  \phi_3(\xi,\xi_1,\xi_2) &:= -\jxi + \jxione - \jxitwo - \jap{\xi_3}, \\
  \phi_4(\xi,\xi_1,\xi_2) &:= -\jxi - \jxione - \jxitwo - \jap{\xi_3},
 \end{aligned}
\end{equation*}
and inputs
\begin{equation} \label{equ:g_inputs_dirac_set_to_same_form}
 \hatg(s,\xi) = \hatf(s,\xi) \quad \text{or} \quad \hatg(s,\xi) = m_a(\xi) \hatf(s,\xi) \quad \text{for some} \quad a \in \{0, 4, 5\}.
\end{equation}
The three inputs in the trilinear terms~\eqref{equ:cubic_dirac_precise_trilinear_terms} are of course not necessarily all of the same one type in~\eqref{equ:g_inputs_dirac_set_to_same_form}. But since their fine structure is not relevant for establishing the weighted energy estimates, we decided not to introduce additional cumbersome notation to keep track of this.

Among the terms $\calT_j^{\delta_0}[g]$, $1 \leq j \leq 4$, the weighted energy estimates for $\calT_2^{\delta_0}[g]$ are the most delicate, because the phase $\phi_2(\xi,\xi_1,\xi_2)$ has a space-time resonance at $(\xi, \xi_1,\xi_2) = (\xi, \xi, -\xi)$. We therefore provide full details for the treatment of $\calT_2^{\delta_0}[g]$, and leave the analogous weighted energy estimates for the other terms $\calT_1^{\delta_0}[g]$, $\calT_3^{\delta_0}[g]$, and $\calT_4^{\delta_0}[g]$ to the reader.

\begin{proposition} \label{prop:weighted_energy_est_delta_T2}
 Suppose that the assumptions in the statement of Proposition~\ref{prop:main_bootstrap} are in place.
 Let $f(t) = e^{-it\jD} v(t)$ be the profile of the solution $v(t)$ to~\eqref{equ:v_equ_refer_to} and let
 \begin{equation*}
  g(t) := f(t) \quad \text{or} \quad g(t) = m_a(D) f(t) \quad \text{for some} \quad a \in \{0, 4, 5\},
 \end{equation*}
 with the multipliers $m_a(D)$ defined in \eqref{equ:def_multipliers_m}.
 Then we have uniformly for all $0 \leq t \leq T$ that
 \begin{equation} \label{equ:weighted_energy_est_delta_T2}
  \sup_{n \geq 1} \, \sup_{0 \leq \ell \leq n} \, 2^{-\frac12 \ell} \tau_n(t) \Bigl\| \varphi_\ell^{(n)}(\xi) \jxi^2 \partial_\xi \calF\bigl[\calT_2^{\delta_0}[g](t)\bigr](\xi) \Bigr\|_{L^2_\xi} \lesssim \bigl( \log(2+t) \bigr)^6 \varepsilon^3.
 \end{equation}
\end{proposition}
\begin{proof}
Fix $0 \leq t \leq T$. Let $n \geq 1$ be an integer such that $t \in \supp(\tau_n)$ and therefore $2^n \simeq t$.
We consider for every integer $1 \leq m \leq n+5$ the time-localized version of $\calT_2^{\delta_0}[g](t)$ given by
\begin{equation*}
 \begin{aligned}
  \calF\bigl[\calT_{2;m}^{\delta_0}[g](t)\bigr](\xi) &:=  \int_0^t \tau_m(s) \jxi^{-1} \iint e^{is\phi_j(\xi,\xi_1,\xi_2)} \hatg(s,\xi_1) \hatbarg(s,\xi_2) \hatg(s,\xi_3) \, \ud \xi_1 \, \ud \xi_2 \, \ud s.
 \end{aligned}
\end{equation*}
In what follows we prove for all $1 \leq m \leq n+5$ that
\begin{equation} \label{equ:weighted_energy_est_delta_calT2m}
 \sup_{0 \leq \ell \leq n} \, 2^{-\frac12 \ell} \Bigl\| \varphi_\ell^{(n)}(\xi) \jxi^2 \partial_\xi \calF\bigl[\calT_{2;m}^{\delta_0}[g](t)\bigr](\xi) \Bigr\|_{L^2_\xi} \lesssim m^5 \varepsilon^3.
\end{equation}
Since clearly $\calT_2^{\delta_0}[g](t) = \sum_{1 \leq m \leq n+5} \calT_{2;m}^{\delta_0}[g](t)$ for $t \simeq 2^n$ and since all implied absolute constants are independent of $0 \leq t \leq T$,
\eqref{equ:weighted_energy_est_delta_calT2m} immediately yields the assertion~\eqref{equ:weighted_energy_est_delta_T2} of Proposition~\ref{prop:weighted_energy_est_delta_T2}.

We now begin in earnest with the proof of \eqref{equ:weighted_energy_est_delta_calT2m} by computing
\begin{equation} \label{equ:delta_T2m_pxi_action}
 \begin{aligned}
  \jxi^2 \partial_\xi \calF\bigl[\calT_{2,m}^{\delta_0}[g](t)\bigr](\xi) &= \int_0^t \tau_m(s) \iint i s \jxi (\pxi \phi_2) e^{is\phi_2} \hatg(s,\xi_1) \hatbarg(s,\xi_2) \hatg(s,\xi_3) \, \ud \xi_1 \, \ud \xi_2 \, \ud s \\
  &\quad + \int_0^t \tau_m(s) \iint e^{is\phi_2} \hatg(s,\xi_1) \hatbarg(s,\xi_2) \jxi \pxi \bigl( \hatg(s,\xi_3) \bigr) \, \ud \xi_1 \, \ud \xi_2 \, \ud s \\
  &\quad + \{ \text{lower order terms} \},
 \end{aligned}
\end{equation} 
where a lower order term arises when $\jxi^2 \pxi$ falls onto the weight $\jxi^{-1}$. Since such lower order contributions are much simpler to estimate, we do not keep track of them explicitly.
Using the identity
\begin{equation*}
 \jxi (\pxi \phi_2) + \jxione (\pxione \phi_2) - \jxitwo (\pxitwo \phi_2) = - \phi_2 \frac{\xi_3}{\jxithree},
\end{equation*}
we rewrite~\eqref{equ:delta_T2m_pxi_action} as
\begin{equation} \label{equ:delta_T2m_pxi_action2}
 \begin{aligned}
  &\jxi^2 \partial_\xi \calF\bigl[\calT_{2,m}^{\delta_0}[g](t)\bigr](\xi) \\
  &= \int_0^t \tau_m(s) \iint (-i) s \phi_2 \frac{\xi_3}{\jxithree} e^{is\phi_2} \hatg(s,\xi_1) \hatbarg(s,\xi_2) \hatg(s,\xi_3) \, \ud \xi_1 \, \ud \xi_2 \, \ud s \\
  &\quad + \int_0^t \tau_m(s) \iint i s \bigl( -\jap{\xi_1} (\pxione \phi_2) + \jxitwo (\pxitwo \phi_2) \bigr) e^{is\phi_2} \hatg(s,\xi_1) \hatbarg(s,\xi_2) \hatg(s,\xi_3) \, \ud \xi_1 \, \ud \xi_2 \, \ud s \\
  &\quad + \int_0^t \tau_m(s) \iint e^{is\phi_2} \hatg(s,\xi_1) \hatbarg(s,\xi_2) \jxi \pxi \bigl( \hatg(s,\xi_3) \bigr) \, \ud \xi_1 \, \ud \xi_2 \, \ud s \\
  &\quad + \{\text{lower order terms}\}.
 \end{aligned}
\end{equation}
In the first term on the right-hand side of~\eqref{equ:delta_T2m_pxi_action2} we can integrate by parts in time $s$,
\begin{equation*} 
 \begin{aligned}
  &\int_0^t \tau_m(s) \iint (-i) s \phi_2 \frac{\xi_3}{\jxithree} e^{is\phi_2} \hatg(s,\xi_1) \hatbarg(s,\xi_2) \hatg(s,\xi_3) \, \ud \xi_1 \, \ud \xi_2 \, \ud s \\
  &= - \biggl[ \tau_m(s) \cdot s \iint e^{is\phi_2} \hatg(s,\xi_1) \hatbarg(s,\xi_2) \frac{\xi_3}{\jxithree} \hatg(s,\xi_3) \, \ud \xi_1 \, \ud \xi_2 \biggr]_{s=0}^{s=t} \\
  &\quad + \int_0^t \tau_m(s) \cdot s \iint e^{is\phi_2} \ps \Bigl( \hatg(s,\xi_1) \hatbarg(s,\xi_2) \frac{\xi_3}{\jxithree} \hatg(s,\xi_3) \Bigr) \, \ud \xi_1 \, \ud \xi_2 \, \ud s \\
  &\quad + \int_0^t \ps \bigl( \tau_m(s) \cdot s \bigr) \iint e^{is\phi_2} \hatg(s,\xi_1) \hatbarg(s,\xi_2) \frac{\xi_3}{\jxithree} \hatg(s,\xi_3) \, \ud \xi_1 \, \ud \xi_2 \, \ud s.
 \end{aligned}
\end{equation*}
To rewrite the second term on the right-hand side of~\eqref{equ:delta_T2m_pxi_action2} we use that 
\begin{equation*}
 i s \bigl( -\jap{\xi_1} (\pxione \phi_2) + \jxitwo (\pxitwo \phi_2) \bigr) e^{is\phi_2} = \bigl( -\jap{\xi_1} \pxione + \jxitwo \pxitwo \bigr) e^{is\phi_2},
\end{equation*}
and integrate by parts in $\xi_1$ and $\xi_2$ to obtain
\begin{equation} \label{equ:delta_T2m_pxi_action4}
 \begin{aligned}
  &\int_0^t \tau_m(s) \iint i s \bigl( -\jap{\xi_1} (\pxione \phi_2) + \jxitwo (\pxitwo \phi_2) \bigr) e^{is\phi_2} \hatg(s,\xi_1) \hatbarg(s,\xi_2) \hatg(s,\xi_3) \, \ud \xi_1 \, \ud \xi_2 \, \ud s \\
  &= \int_0^t \tau_m(s) \iint e^{is\phi_2} \jxione (\pxione \hatg)(s,\xi_1) \hatbarg(s,\xi_2) \hatg(s,\xi_3) \, \ud \xi_1 \, \ud \xi_2 \, \ud s \\
  &\quad - \int_0^t \tau_m(s) \iint e^{is\phi_2} \hatg(s,\xi_1) \jxitwo (\pxitwo \hatbarg)(s,\xi_2) \hatg(s,\xi_3) \, \ud \xi_1 \, \ud \xi_2 \, \ud s \\
  &\quad + \int_0^t \tau_m(s) \iint e^{is\phi_2} \hatg(s,\xi_1) \hatbarg(s,\xi_2) \bigl( \jxione \pxione - \jxitwo \pxitwo \bigr) \hatg(s,\xi_3) \, \ud \xi_1 \, \ud \xi_2 \, \ud s \\
  &\quad + \{\text{lower order terms}\},
 \end{aligned}
\end{equation}
where the lower order terms come from 
\begin{equation*}
 \bigl( -\jap{\xi_1} \pxione + \jxitwo \pxitwo \bigr)^\ast = \jxione \pxione - \jxitwo \pxitwo + \frac{\xi_1}{\jxione} - \frac{\xi_2}{\jxitwo}.
\end{equation*}
Upon inserting \eqref{equ:delta_T2m_pxi_action4} back into \eqref{equ:delta_T2m_pxi_action2}, and recalling that $\xi_3 = \xi-\xi_1-\xi_2$, the third term on the right-hand side of~\eqref{equ:delta_T2m_pxi_action2} and the third term on the right-hand side of~\eqref{equ:delta_T2m_pxi_action4} combine to 
\begin{equation} \label{equ:delta_T2m_pxi_action5}
 \begin{aligned}
  &\int_0^t \tau_m(s) \iint e^{is\phi_2} \hatg(s,\xi_1) \hatbarg(s,\xi_2) \bigl( \jxi \pxi + \jxione \pxione - \jxitwo \pxitwo \bigr) \hatg(s,\xi_3) \, \ud \xi_1 \, \ud \xi_2 \, \ud s \\
  &= \int_0^t \tau_m(s) \iint e^{is\phi_2} \hatg(s,\xi_1) \hatbarg(s,\xi_2) \bigl( \jxi - \jxione + \jxitwo \bigr) (\pxithree \hatg)(s,\xi_3) \, \ud \xi_1 \, \ud \xi_2 \, \ud s \\
  &= \int_0^t \tau_m(s) \iint e^{is\phi_2} \hatg(s,\xi_1) \hatbarg(s,\xi_2) \jxithree (\pxithree \hatg)(s,\xi_3) \, \ud \xi_1 \, \ud \xi_2 \, \ud s \\
  &\quad + \int_0^t \tau_m(s) \iint e^{is\phi_2} \hatg(s,\xi_1) \hatbarg(s,\xi_2) (-\phi_2) (\pxithree \hatg)(s,\xi_3) \, \ud \xi_1 \, \ud \xi_2 \, \ud s.
 \end{aligned}
\end{equation}
In the second term on the right-hand side of~\eqref{equ:delta_T2m_pxi_action5}, we can now integrate by parts in time again to find that
\begin{equation}
 \begin{aligned}
  &\int_0^t \tau_m(s) \iint e^{is\phi_2} \hatg(s,\xi_1) \hatbarg(s,\xi_2) (-\phi_2) (\pxithree \hatg)(s,\xi_3) \, \ud \xi_1 \, \ud \xi_2 \, \ud s \\
  &= \biggl[ i \tau_m(s) \iint e^{is\phi_2} \hatg(s,\xi_1) \hatbarg(s,\xi_2) (\pxithree \hatg)(s,\xi_3) \, \ud \xi_1 \, \ud \xi_2 \biggr]_{s=0}^{s=t} \\
  &\quad - i \int_0^t \ps \bigl( \tau_m(s) \bigr) \iint e^{is\phi_2} \hatg(s,\xi_1) \hatbarg(s,\xi_2) (\pxithree \hatg)(s,\xi_3) \, \ud \xi_1 \, \ud \xi_2 \, \ud s \\
  &\quad - i \int_0^t \tau_m(s) \iint e^{is\phi_2} \ps \bigl( \hatg(s,\xi_1) \hatbarg(s,\xi_2) (\pxithree \hatg)(s,\xi_3) \bigr) \, \ud \xi_1 \, \ud \xi_2 \, \ud s.
 \end{aligned}
\end{equation}
Combining the preceding computations, we conclude that
\begin{equation*}
 \begin{aligned}
  \jxi^2 \partial_\xi \calF\bigl[\calT_{2,m}^{\delta_0}[g](t)\bigr](\xi) &= \calI_{2,m}^1(t,\xi) + \calI_{2,m}^2(t,\xi) + \calI_{2,m}^3(t,\xi) \\
  &\quad \quad + \calR_{2,m}^1(t,\xi) + \calR_{2,m}^2(t,\xi) + \{\text{lower order terms}\},
 \end{aligned}
\end{equation*}
where
\begin{equation*}
 \begin{aligned}
  \calI_{2,m}^1(t,\xi) &:= \int_0^t \tau_m(s) \iint e^{is\phi_2} \jxione (\pxione \hatg)(s,\xi_1) \hatbarg(s,\xi_2) \hatg(s,\xi_3) \, \ud \xi_1 \, \ud \xi_2 \, \ud s, \\
  \calI_{2,m}^2(t,\xi) &:= - \int_0^t \tau_m(s) \iint e^{is\phi_2} \hatg(s,\xi_1) \jxitwo (\pxitwo \hatbarg)(s,\xi_2) \hatg(s,\xi_3) \, \ud \xi_1 \, \ud \xi_2 \, \ud s, \\
  \calI_{2,m}^3(t,\xi) &:= \int_0^t \tau_m(s) \iint e^{is\phi_2} \hatg(s,\xi_1) \hatbarg(s,\xi_2) \jxithree (\pxithree \hatg)(s,\xi_3) \, \ud \xi_1 \, \ud \xi_2 \, \ud s,
 \end{aligned}
\end{equation*}
and 
\begin{equation*}
 \begin{aligned}
  \calR_{2,m}^1(t,\xi) &:= - \biggl[ \tau_m(s) \cdot s \iint e^{is\phi_2} \hatg(s,\xi_1) \hatbarg(s,\xi_2) \frac{\xi_3}{\jxithree} \hatg(s,\xi_3) \, \ud \xi_1 \, \ud \xi_2 \biggr]_{s=0}^{s=t} \\
  &\quad \quad + \int_0^t \tau_m(s) \cdot s \iint e^{is\phi_2} \ps \Bigl( \hatg(s,\xi_1) \hatbarg(s,\xi_2) \frac{\xi_3}{\jxithree} \hatg(s,\xi_3) \Bigr) \, \ud \xi_1 \, \ud \xi_2 \, \ud s \\
  &\quad \quad + \int_0^t \ps \bigl( \tau_m(s) \cdot s \bigr) \iint e^{is\phi_2} \hatg(s,\xi_1) \hatbarg(s,\xi_2) \frac{\xi_3}{\jxithree} \hatg(s,\xi_3) \, \ud \xi_1 \, \ud \xi_2 \, \ud s \\
  &=: \calR_{2,m}^{1,1}(t,\xi) + \calR_{2,m}^{1,2}(t,\xi) + \calR_{2,m}^{1,3}(t,\xi),
 \end{aligned}
\end{equation*}
as well as 
\begin{equation*}
 \begin{aligned}
  \calR_{2,m}^2(t,\xi) &:= \biggl[ i \tau_m(s) \iint e^{is\phi_2} \hatg(s,\xi_1) \hatbarg(s,\xi_2) (\pxithree \hatg)(s,\xi_3) \, \ud \xi_1 \, \ud \xi_2 \biggr]_{s=0}^{s=t} \\
  &\quad \quad - i \int_0^t \ps \bigl( \tau_m(s) \bigr) \iint e^{is\phi_2} \hatg(s,\xi_1) \hatbarg(s,\xi_2) (\pxithree \hatg)(s,\xi_3) \, \ud \xi_1 \, \ud \xi_2 \, \ud s \\
  &\quad \quad - i \int_0^t \tau_m(s) \iint e^{is\phi_2} \ps \bigl( \hatg(s,\xi_1) \hatbarg(s,\xi_2) (\pxithree \hatg)(s,\xi_3) \bigr) \, \ud \xi_1 \, \ud \xi_2 \, \ud s \\
  &=: \calR_{2,m}^{2,1}(t,\xi) + \calR_{2,m}^{2,2}(t,\xi) + \calR_{2,m}^{2,3}(t,\xi).
 \end{aligned}
\end{equation*}
The main work goes into estimating the terms $\calI_{2, m}^j(t,\xi)$, $1 \leq j \leq 3$, while the terms in $\calR_{2,m}^1(t,\xi)$ and in $\calR_{2,m}^2(t,\xi)$ can be considered as milder remainder terms, which are amenable to stronger weighted estimates.
We begin with the bounds for the terms in $\calR_{2,m}^1(t,\xi)$ and $\calR_{2,m}^2(t,\xi)$.

In what follows, we use the notation $\sup_{s \, \simeq \, 2^m}$ in the sense that $s \simeq 2^m$ with $s \leq T$.

\medskip 

\noindent {\bf Step 1: Weighted energy estimates for the terms in $\calR_{2,m}^1(t,\xi)$.}
For the term $\calR_{2,m}^{1,1}(t,\xi)$ we use an $L^\infty_x \times L^\infty_x \times L^2_x$ estimate and the bounds \eqref{equ:g_bound_dispersive_est}, \eqref{equ:g_bound_jD2_L2}, to infer for all $1 \leq m \leq n+5$ that
\begin{equation*}
 \begin{aligned}
  \bigl\| \calR_{2,m}^{1,1}(t,\xi) \bigr\|_{L^2_\xi} &\lesssim 2^m \cdot \sup_{s \, \simeq \, 2^m} \, \bigl\| e^{is\jD} g(s) \bigr\|_{L^\infty_x}^2 \| \jD^{-1} D g(s) \|_{L^2_x} \lesssim 2^m \cdot \bigl(2^{-\frac12 m} m \varepsilon \bigr)^2 \cdot \varepsilon \lesssim m^2 \varepsilon^3.
 \end{aligned}
\end{equation*}
We estimate the term $\calR_{2,m}^{1,2}(t,\xi)$ similarly, placing the input onto which $\partial_s$ falls into $L^2_x$, so that by the bounds \eqref{equ:g_bound_dispersive_est}, \eqref{equ:g_bound_jD_pt_L2},
\begin{equation*}
 \begin{aligned}
  \bigl\| \calR_{2,m}^{1,2}(t,\xi) \bigr\|_{L^2_\xi} &\lesssim 2^m \cdot 2^m \cdot \sup_{s \, \simeq \, 2^m} \, \bigl\| e^{is\jD} g(s) \bigr\|_{L^\infty_x}^2 \| \ps g(s) \|_{L^2_x} \\
  &\lesssim 2^m \cdot 2^m \cdot \bigl( 2^{-\frac12 m} m \varepsilon \bigr)^2 \cdot 2^{-m} m^2 \varepsilon^2 \lesssim m^4 \varepsilon^4.
 \end{aligned}
\end{equation*}
Finally, a simple $L^\infty_x \times L^\infty_x \times L^2_x$ estimate also suffices to bound
\begin{equation*}
 \bigl\| \calR_{2,m}^{1,3}(t,\xi) \bigr\|_{L^2_\xi} \lesssim 2^m \cdot \sup_{s \, \simeq \, 2^m} \, \bigl\| e^{is\jD} g(s) \bigr\|_{L^\infty_x}^2 \|\jD^{-1} D g(s)\|_{L^2_x} \lesssim 2^m \cdot 2^{-m} m^2 \varepsilon^2 \cdot \varepsilon \lesssim m^2 \varepsilon^3.
\end{equation*}
Thus, we have obtained for all $1 \leq m \leq n+5$ that
\begin{equation*}
 \bigl\| \calR_{2,m}^{1}(t,\xi) \bigr\|_{L^2_\xi} \lesssim m^4 \varepsilon^3,
\end{equation*}
which suffices.

\medskip 

\noindent {\bf Step 2: Weighted energy estimates for the terms in $\calR_{2,m}^2(t,\xi)$.}
By an $L^\infty_x \times L^\infty_x \times L^2_x$ estimate, the decay estimate~\eqref{equ:g_bound_dispersive_est}, and the crude bound~\eqref{equ:g_bound_pxi_crude}, we obtain for all $1 \leq m \leq n+5$ that
\begin{equation*}
 \begin{aligned}
  \bigl\| \calR_{2,m}^{2,1}(t,\xi) \bigr\|_{L^2_\xi} &\lesssim \sup_{s \, \simeq \, 2^m} \, \bigl\| e^{is\jD} g(s) \bigr\|_{L^\infty_x}^2 \|\pxi \hatg(s,\xi)\|_{L^2_\xi} \lesssim 2^{-m} m^2 \varepsilon^2 \cdot 2^{\frac12 m} \varepsilon \lesssim \varepsilon^3
 \end{aligned}
\end{equation*}
as well as
\begin{equation*}
 \begin{aligned}
  \bigl\| \calR_{2,m}^{2,2}(t,\xi) \bigr\|_{L^2_\xi} &\lesssim 2^m \cdot 2^{-m} \sup_{s \, \simeq \, 2^m} \, \bigl\| e^{is\jD} g(s) \bigr\|_{L^\infty_x}^2 \|\pxi \hatg(s,\xi)\|_{L^2_\xi} \lesssim 2^{-m} m^2 \varepsilon^2 \cdot 2^{\frac12 m} \varepsilon \lesssim \varepsilon^3.
 \end{aligned}
\end{equation*}
Finally, for the third term $\calR_{2,m}^{2,3}(t,\xi)$ we use Sobolev embedding and the bounds~\eqref{equ:g_bound_jD_pt_L2}, \eqref{equ:g_bound_dispersive_est}, \eqref{equ:g_bound_pxi_crude}, and \eqref{equ:g_bound_pxi_pt_L2xi} to find
\begin{equation*}
 \begin{aligned}
  \bigl\| \calR_{2,m}^{2,3}(t,\xi) \bigr\|_{L^2_\xi} &\lesssim 2^m \sup_{s \, \simeq \, 2^m} \, \bigl\| e^{is\jD} \ps g(s) \bigr\|_{L^\infty_x} \bigl\| e^{is\jD} g(s) \bigr\|_{L^\infty_x} \|\pxi \hatg(s,\xi)\|_{L^2_\xi} \\
  &\quad + 2^m \sup_{s \, \simeq \, 2^m} \, \bigl\| e^{is\jD} g(s) \bigr\|_{L^\infty_x}^2 \|\pxi \ps \hatg(s,\xi)\|_{L^2_\xi} \\
  &\lesssim 2^m \cdot 2^{-m} m^2 \varepsilon^2 \cdot 2^{-\frac12 m} m \varepsilon \cdot 2^{\frac12 m} \varepsilon + 2^m \cdot 2^{-m} m^2 \varepsilon^2 \cdot m^2 \varepsilon^2 \lesssim m^4 \varepsilon^4.
 \end{aligned}
\end{equation*}
Hence, we have found that for all $1 \leq m \leq n+5$,
\begin{equation*}
 \bigl\| \calR_{2,m}^{2}(t,\xi) \bigr\|_{L^2_\xi} \lesssim m^4 \varepsilon^3.
\end{equation*}

\medskip 

We can now turn to the weighted energy estimates for the main terms $\calI_{2,m}^j(t,\xi)$, $1 \leq j \leq 3$. Note that by symmetry, the proofs of the bounds for $\calI_{2,m}^1(t,\xi)$ and $\calI_{2,m}^3(t,\xi)$ are identical. To conclude the proof of Proposition~\ref{prop:weighted_energy_est_delta_T2} it therefore suffices to consider $\calI_{2,m}^1(t,\xi)$ and $\calI_{2,m}^2(t,\xi)$ in the remaining two steps.

\medskip 

\noindent {\bf Step 3: Weighted energy estimate for the term $\calI_{2,m}^1(t,\xi)$.}
Here we seek to show for all $1 \leq m \leq n+5$ that
\begin{equation*}
 \sup_{0 \leq \ell \leq n} \, 2^{-\frac12 \ell} \bigl\| \varphi_\ell^{(n)}(\xi) \calI_{2,m}^1(t,\xi) \bigr\|_{L^2_\xi} \lesssim m^3 \varepsilon^3.
\end{equation*}
We distinguish the cases $\ell = n$, $1 \leq \ell \leq n-1$, and $\ell = 0$.

\medskip 
\noindent \underline{\it Case 3.1: $\ell = n$.}
Using H\"older's inequality in the frequency variables and the bounds~\eqref{equ:g_bound_pxiL1}, \eqref{equ:g_bound_jD2_L2}, we obtain for all $1 \leq m \leq n+5$ that
\begin{equation*}
 \begin{aligned}
  &2^{-\frac12 n} \bigl\| \varphi_n^{(n)}(\xi) \calI_{2,m}^1(t,\xi) \bigr\|_{L^2_\xi} \\
  &\lesssim 2^{-\frac12 n} \bigl\| \varphi_n^{(n)}(\xi) \bigr\|_{L^2_\xi} \cdot 2^m \cdot \sup_{s \simeq 2^m} \, \biggl\| \iint e^{is\phi_2} \jxione (\pxione \hatg)(s,\xi_1) \hatbarg(s,\xi_2) \hatg(s,\xi_3) \, \ud \xi_1 \, \ud \xi_2 \biggr\|_{L^\infty_\xi} \\
  &\lesssim 2^{-\frac12 n} \cdot 2^{-\hf n} \cdot 2^m \cdot \sup_{s \simeq 2^m} \, \bigl\| \jxione \pxione \hatg(s,\xi_1) \bigr\|_{L^1_{\xi_1}} \|\hatg(s,\xi_2)\|_{L^2_{\xi_2}} \|\hatg(s,\xi_3)\|_{L^2_{\xi_3}} \\
  &\lesssim 2^{-n} \cdot 2^m \cdot m \varepsilon \cdot \varepsilon^2 \lesssim m \varepsilon^3,
 \end{aligned}
\end{equation*}
which is acceptable.

\medskip 
\noindent \underline{\it Case 3.2: $1 \leq \ell \leq n-1$.}
We insert a smooth partition of unity to distinguish how close the input frequency variable $\xi_1$ is to the problematic frequencies $\pm \sqrt{3}$, and write
\begin{equation*}
 \begin{aligned}
  \varphi_{\ell}^{(n)}(\xi) \calI_{2,m}^1(t,\xi) = \sum_{0 \leq \ell_1 \leq m} \calI_{2,m; \ell, \ell_1}^{1}(t,\xi)
 \end{aligned}
\end{equation*}
with 
\begin{equation*}
 \begin{aligned}
  \calI_{2,m; \ell, \ell_1}^{1}(t,\xi) := \varphi_{\ell}^{(n)}(\xi) \int_0^t \tau_m(s) \iint e^{is\phi_2} \varphi_{\ell_1}^{(m)}(\xi_1) \jxione (\pxione \hatg)(s,\xi_1) \hatbarg(s,\xi_2) \hatg(s,\xi_3) \, \ud \xi_1 \, \ud \xi_2 \, \ud s.
 \end{aligned}
\end{equation*}
In what follows, we may assume that $\ell \leq m-10$. We then distinguish the subcases $0 \leq \ell_1 < \ell+10$ and $\ell+10 \leq \ell_1 \leq m$. If $\ell > m-10$, we can just proceed as in the former subcase for all $0 \leq \ell_1 \leq m$.

\medskip 
\noindent \underline{\it Subcase 3.2.1: $1 \leq \ell \leq n-1$, $0 \leq \ell_1 < \ell+10$.}
We use a simple $L^2_x \times L^\infty_x \times L^\infty_x$ estimate and the decay estimate \eqref{equ:g_bound_dispersive_est} to get
\begin{equation*}
 \begin{aligned}
  &2^{-\hf \ell} \biggl\| \sum_{0 \leq \ell_1 < \ell+10} \calI_{2,m; \ell, \ell_1}^{1}(t,\xi) \biggr\|_{L^2_\xi} \\
  &\lesssim 2^{-\hf \ell} \sum_{0 \leq \ell_1 < \ell+10} 2^{\hf \ell_1} \cdot 2^m \cdot \sup_{s \, \simeq \, 2^m} \, 2^{-\hf \ell_1} \bigl\| \varphi_{\ell_1}^{(m)}(\xi_1) \jxione (\pxione \hatg)(s,\xi_1) \bigr\|_{L^2_{\xi_1}} \bigl\| e^{is\jD} g(s) \bigr\|_{L^\infty_x}^2 \\
  &\lesssim 2^m \cdot \|g\|_{N_T} \cdot 2^{-m} m^2 \varepsilon^2 \lesssim m^2 \varepsilon^3.
 \end{aligned}
\end{equation*}

\medskip 
\noindent \underline{\it Subcase 3.2.2: $1 \leq \ell \leq n-1$, $\ell+10 \leq \ell_1 \leq m$.}
We have reduced to the most delicate interactions 
\begin{equation*}
 \begin{aligned}
  \calI_{2,m;\ell\leq\ell_1}^1(t,\xi) &:= \sum_{\ell+10 \leq \ell_1 \leq m} \calI_{2,m; \ell, \ell_1}^{1}(t,\xi) \\
  &= \varphi_{\ell}^{(n)}(\xi) \int_0^t \tau_m(s) \iint e^{is\phi_2} \varphi_{\geq \ell+10}^{(m)}(\xi_1) \jxione (\pxione \hatg)(s,\xi_1) \hatbarg(s,\xi_2) \hatg(s,\xi_3) \, \ud \xi_1 \, \ud \xi_2 \, \ud s.
 \end{aligned}
\end{equation*}
Since the first input $\hatg(s,\xi_1)$ is already differentiated, we can try to integrate by parts in $\xi_2$. In this frequency configuration, $\pxitwo \phi_2$ cannot vanish, because
\begin{equation*}
 \pxitwo \phi_2 = -\frac{\xi_2}{\jxitwo} - \frac{\xi_3}{\jxithree} = 0 \quad \Leftrightarrow \quad \xi-\xi_1 = 0,
\end{equation*}
which cannot occur since we consider the regime $\ell_1 \geq \ell+10$.
We find that
\begin{equation*}
 \begin{aligned}
  &\calI_{2,m; \ell \leq \ell_1}^{1}(t,\xi) \\
  &= i \int_0^t \tau_m(s) \frac{1}{s} \iint e^{is\phi_2} \frac{1}{\pxitwo \phi_2} \varphi_{\ell}^{(n)}(\xi) \varphi_{\geq \ell+10}^{(m)}(\xi_1) \jxione (\pxione \hatg)(s,\xi_1) (\pxitwo \hatbarg)(s,\xi_2) \hatg(s,\xi_3) \, \ud \xi_1 \, \ud \xi_2 \, \ud s \\
  &\quad -i \int_0^t \tau_m(s) \frac{1}{s} \iint e^{is\phi_2} \frac{1}{\pxitwo \phi_2} \varphi_{\ell}^{(n)}(\xi) \varphi_{\geq \ell+10}^{(m)}(\xi_1) \jxione (\pxione \hatg)(s,\xi_1) \hatbarg(s,\xi_2) (\pxithree \hatg)(s,\xi_3) \, \ud \xi_1 \, \ud \xi_2 \, \ud s \\
  &\quad +i \int_0^t \tau_m(s) \frac{1}{s} \iint e^{is\phi_2} \pxitwo \biggl( \frac{1}{\pxitwo \phi_2} \biggr) \varphi_{\ell}^{(n)}(\xi) \varphi_{\geq \ell+10}^{(m)}(\xi_1) \jxione (\pxione \hatg)(s,\xi_1) \hatbarg(s,\xi_2) \hatg(s,\xi_3) \, \ud \xi_1 \, \ud \xi_2 \, \ud s \\
  &=: \calI_{2,m; \ell \leq \ell_1}^{1,(a)}(t,\xi) + \calI_{2,m; \ell \leq \ell_1}^{1,(b)}(t,\xi) + \calI_{2,m; \ell \leq \ell_1}^{1,(c)}(t,\xi).
 \end{aligned}
\end{equation*}
The terms $\calI_{2,m; \ell \leq \ell_1}^{1,(a)}(t,\xi)$ and $\calI_{2,m; \ell \leq \ell_1}^{1,(b)}(t,\xi)$ can be treated identically, because the complex conjugation signs on the input profiles are not relevant here.
So we only consider the term $\calI_{2,m; \ell \leq \ell_1}^{1,(a)}(t,\xi)$, and we will see that the term $\calI_{2,m; \ell \leq \ell_1}^{1,(c)}(t,\xi)$ can be treated similarly.

Without loss of generality, we may assume that $\xi > 0$, i.e., $|\xi-\sqrt{3}| \simeq 2^{-\ell-100}$. 
Then we have to distinguish the cases $\xi_1 \approx -\sqrt{3}$ and $\xi_1 \approx \sqrt{3}$. Since the phase function $\phi_2$ has a space-time resonance at $(\xi,\xi_1,\xi_2,\xi_3) = (\xi,\xi,-\xi,\xi)$, the more delicate case is when $\xi_1 \approx +\sqrt{3}$. We treat this case in detail now and leave the other case to the reader.

When $|\xi_1 - \sqrt{3}| \simeq 2^{-\ell_1-100}$ and $\ell_1 \geq \ell+10$, we have $|\xi-\xi_1| = |(\xi-\sqrt{3}) - (\xi_1-\sqrt{3})| \simeq 2^{-\ell-100}$. Recall that
\begin{equation} \label{equ:calI2m1_xis_relation}
 \xi - \xi_1 = \xi_2 + \xi_3, 
\end{equation}
and 
\begin{equation} \label{equ:calI2m1_pxitwo_phi2}
 \pxitwo \phi_2 = -\frac{\xi_2}{\jxitwo} - \frac{\xi_3}{\jxithree},
\end{equation}
whence
\begin{equation} \label{equ:calI2m1_pxitwo_phi2_inverse}
 \frac{1}{\pxitwo \phi_2} = - \frac{\jxitwo \jxithree (\xi_2 \jxithree - \xi_3 \jxitwo)}{(\xi_2+\xi_3) (\xi_2-\xi_3)} = - \frac{1}{\xi-\xi_1} \cdot \jxitwo \jxithree \cdot \frac{\xi_2 \jxithree - \xi_3 \jxitwo}{\xi_2-\xi_3}.
\end{equation}
We distinguish the subcases (1) $|\xi_2| \lesssim 2^{-\ell-10}$, (2) $2^{-\ell-10} \lesssim |\xi_2| \lesssim 2^{10}$, and (3) $|\xi_2| \gtrsim 2^{10}$. Correspondingly, we consider the following three components of the term $\calI_{2,m; \ell, \ell_1}^{1,(a)}(t,\xi)$,
\begin{equation*}
 \begin{aligned}
  &i \int_0^t \tau_m(s) \frac{1}{s} \iint e^{is\phi_2} \frac{1}{\pxitwo \phi_2} \varphi_{\ell}^{(n),+}(\xi) \varphi_{\geq \ell+10}^{(m),+}(\xi_1) \varphi_{\leq -\ell-5}(\xi_2) \jxione (\pxione \hatg)(s,\xi_1) (\pxitwo \hatbarg)(s,\xi_2) \hatg(s,\xi_3) \, \ud \xi_1 \, \ud \xi_2 \, \ud s \\
  &+ i \int_0^t \tau_m(s) \frac{1}{s} \iint e^{is\phi_2} \frac{1}{\pxitwo \phi_2} \varphi_{\ell}^{(n),+}(\xi) \varphi_{\geq \ell+10}^{(m),+}(\xi_1) \varphi_{[-\ell-5,5]}(\xi_2) \jxione (\pxione \hatg)(s,\xi_1) (\pxitwo \hatbarg)(s,\xi_2) \hatg(s,\xi_3) \, \ud \xi_1 \, \ud \xi_2 \, \ud s \\
  &+ i \int_0^t \tau_m(s) \frac{1}{s} \iint e^{is\phi_2} \frac{1}{\pxitwo \phi_2} \varphi_{\ell}^{(n),+}(\xi) \varphi_{\geq \ell+10}^{(m),+}(\xi_1) \varphi_{>5}(\xi_2) \jxione (\pxione \hatg)(s,\xi_1) (\pxitwo \hatbarg)(s,\xi_2) \hatg(s,\xi_3) \, \ud \xi_1 \, \ud \xi_2 \, \ud s \\
  &=: \calI_{2,m; \ell \leq \ell_1}^{1,(a),low}(t,\xi) + \calI_{2,m; \ell \leq \ell_1}^{1,(a),med}(t,\xi) + \calI_{2,m; \ell \leq \ell_1}^{1,(a),high}(t,\xi),
 \end{aligned}
\end{equation*}
where we refer the reader to Subsection~\ref{subsec:projection_operators} for the precise definitions of the frequency cut-offs in the preceding expressions.
We enact an analogous decomposition for the term $\calI_{2,m; \ell \leq \ell_1}^{1,(c)}(t,\xi)$.

\medskip 
\noindent \underline{\it Subcase 3.2.2.1: $1 \leq \ell \leq n-1$, $\ell+10 \leq \ell_1 \leq m$, $|\xi_2| \lesssim 2^{-\ell-10}$.}
If $|\xi_2| \lesssim 2^{-\ell-10}$, then we must also have $|\xi_3| \lesssim 2^{-\ell-10}$ in view of~\eqref{equ:calI2m1_xis_relation} and $|\xi-\xi_1| \simeq 2^{-\ell-100}$ in this frequency configuration. Thus, by Taylor expansion and \eqref{equ:calI2m1_xis_relation},
\begin{equation*}
 \begin{aligned}
  \pxitwo \phi_2 = -\frac{\xi_2}{\jxitwo} - \frac{\xi_3}{\jxithree} = - \xi_2 - \xi_3 + \calO(|\xi_2|^3 +|\xi_3|^3) = -(\xi-\xi_1) + \calO(|\xi_2|^3 +|\xi_3|^3) \simeq 2^{-\ell-100},
 \end{aligned}
\end{equation*}
whence 
\begin{equation*}
 \frac{1}{|\pxitwo \phi_2|} \lesssim \frac{1}{|\xi-\xi_1|} \lesssim 2^{\ell}.
\end{equation*}
Moreover, since in this case 
\begin{equation*}
 \begin{aligned}
  |\pxitwo^2 \phi_2| = \biggl| -\frac{1}{\jxitwo^3} + \frac{1}{\jxithree^3} \biggr| \lesssim \bigl| |\xi_2| - |\xi_3| \bigr| \lesssim |\xi_2+\xi_3| \simeq |\xi-\xi_1| \simeq 2^{-\ell-100},
 \end{aligned}
\end{equation*}
we also have 
\begin{equation*}
  \biggl| \pxitwo \biggl( \frac{1}{\pxitwo \phi_2} \biggr) \biggr| \lesssim 2^\ell.
\end{equation*}
Thus, denoting by $\wtilvarphi^{(n)}_\ell(\xi)$ a slight fattening of the cut-off $\varphi^{(n)}_\ell(\xi)$, we obtain by H\"older's inequality in the frequency variables and by the bounds \eqref{equ:g_bound_pxiL1}, \eqref{equ:g_bound_Linftyxi} that
\begin{equation*}
 \begin{aligned}
  &2^{-\hf \ell} \bigl\| \calI_{2,m; \ell \leq \ell_1}^{1,(a),low}(t,\xi) \bigr\|_{L^2_\xi} \\
  &\lesssim 2^{-\hf \ell} \bigl\| \wtilvarphi_\ell^{(n)}(\xi) \bigr\|_{L^2_\xi} \cdot 2^m \cdot 2^{-m} \cdot 2^\ell \cdot \sup_{s \simeq 2^m} \, \bigl\| \jxione (\pxione \hatg)(s,\xi_1) \bigr\|_{L^1_{\xi_1}} \bigl\| (\pxitwo \hatbarg)(s,\xi_2) \bigr\|_{L^1_{\xi_2}} \bigl\| \hatg(s,\xi_3) \bigr\|_{L^\infty_{\xi_3}} \\
  &\lesssim 2^{-\hf \ell} \cdot 2^{-\hf \ell} \cdot 2^m \cdot 2^{-m} \cdot 2^\ell \cdot m^3 \varepsilon^3 \lesssim m^3 \varepsilon^3.
 \end{aligned}
\end{equation*}

The estimate for the term $\calI_{2,m; \ell \leq \ell_1}^{1,(c), low}(t,\xi)$ is analogous.

\medskip 
\noindent \underline{\it Subcase 3.2.2.2: $1 \leq \ell \leq n-1$, $\ell+10 \leq \ell_1 \leq m$, $2^{-\ell-10} \lesssim |\xi_2| \lesssim 2^{10}$.}
If instead $2^{-\ell-10} \lesssim |\xi_2| \lesssim 2^{10}$, then in view of~\eqref{equ:calI2m1_xis_relation} and $|\xi-\xi_1| \simeq 2^{-\ell-100}$, we must have $|\xi_3| \simeq |\xi_2|$ and $\xi_2 \xi_3 < 0$, whence 
\begin{equation*}
 \begin{aligned}
  \frac{1}{|\pxitwo \phi_2|} = \frac{1}{|\xi-\xi_1|} \cdot \jxitwo \jxithree \cdot \frac{|\xi_2| \jxithree + |\xi_3| \jxitwo}{|\xi_2|+|\xi_3|} \lesssim 2^\ell.
 \end{aligned}
\end{equation*}
Analogously, to the previous subcase we also infer
\begin{equation*}
  \biggl| \pxitwo \biggl( \frac{1}{\pxitwo \phi_2} \biggr) \biggr| \lesssim 2^\ell.
\end{equation*}
Then the estimates for both terms $\calI_{2,m; \ell \leq \ell_1}^{1,(a), med}(t,\xi)$ and $\calI_{2,m; \ell \leq \ell_1}^{1,(c), med}(t,\xi)$ are analogous to the preceding subcase $|\xi_2| \lesssim 2^{-\ell-10}$.

\medskip 
\noindent \underline{\it Subcase 3.2.2.3: $1 \leq \ell \leq n-1$, $\ell+10 \leq \ell_1 \leq m$, $|\xi_2| \gtrsim 2^{10}$.}
Finally, if $|\xi_2| \gtrsim 2^{10}$, in view of~\eqref{equ:calI2m1_xis_relation} and $|\xi-\xi_1| \simeq 2^{-\ell-100} \ll 1$, we must also have $|\xi_3| \simeq |\xi_2|$ and $\xi_2 \xi_3 < 0$ so that 
\begin{equation*}
 \begin{aligned}
  \frac{1}{|\pxitwo \phi_2|} = \frac{1}{|\xi-\xi_1|} \cdot \jxitwo \jxithree \cdot \frac{|\xi_2| \jxithree + |\xi_3| \jxitwo}{|\xi_2|+|\xi_3|} \lesssim 2^\ell \jxitwo^\thf \jxithree^\thf.
 \end{aligned}
\end{equation*}
Moreover, when $|\xi_2| \simeq |\xi_3| \gtrsim 2^{10}$, 
\begin{equation*}
 \begin{aligned}
  |\pxitwo^2 \phi_2| = \biggl| -\frac{1}{\jxitwo^3} + \frac{1}{\jxithree^3} \biggr| \lesssim \frac{\bigl| |\xi_2| - |\xi_3| \bigr|}{\jxitwo^4} \lesssim \frac{|\xi_2+\xi_3|}{\jxitwo^4} \simeq \frac{|\xi-\xi_1|}{\jxitwo^4} \simeq \frac{2^{-\ell}}{\jxitwo^4},
 \end{aligned}
\end{equation*}
and hence, 
\begin{equation*}
  \biggl| \pxitwo \biggl( \frac{1}{\pxitwo \phi_2} \biggr) \biggr| \lesssim \bigl( 2^\ell \jxitwo^\thf \jxithree^\thf \bigr)^2 \frac{2^{-\ell}}{\jxitwo^4} \lesssim 2^\ell \jxitwo \jxithree.
\end{equation*}
Then, by H\"older's inequality in the frequency variables, the bound \eqref{equ:g_bound_pxiL1}, and since $|\xi_2| \gtrsim 2^{10}$,
\begin{equation*}
 \begin{aligned}
  &2^{-\hf \ell} \bigl\| \calI_{2,m; \ell \leq \ell_1}^{1,(a),high}(t,\xi) \bigr\|_{L^2_\xi} \\
  &\lesssim 2^{-\hf \ell} \bigl\| \wtilvarphi_\ell^{(n)}(\xi) \bigr\|_{L^2_\xi} \cdot 2^m \cdot 2^{-m} \cdot 2^\ell \cdot \sup_{s \, \simeq \, 2^m} \, \bigl\| \jxione (\pxione \hatg)(s,\xi_1) \bigr\|_{L^1_{\xi_1}} \\
  &\qquad \qquad \qquad \qquad \qquad \qquad \times \bigl\| \varphi_0^{(m)}(\xi_2) \jxitwo^\thf (\pxitwo \hatbarg)(s,\xi_2) \bigr\|_{L^2_{\xi_2}} \bigl\| \jxithree^\thf \hatg(s,\xi_3) \bigr\|_{L^2_{\xi_3}} \\
  &\lesssim 2^{-\hf \ell} \cdot 2^{-\hf \ell} \cdot 2^m \cdot 2^{-m} \cdot 2^\ell \cdot m \varepsilon \cdot \varepsilon^2 \lesssim m \varepsilon^3.
 \end{aligned}
\end{equation*}
Observe that thanks to the cut-off $\varphi_0^{(m)}(\xi_2)$, here we could right away place the input $\pxitwo \hatbarg(s,\xi_2)$ into $L^2_{\xi_2}$.

The bound for $\calI_{2,m; \ell \leq \ell_1}^{1,(c),high}(t,\xi)$ is similar, and in fact simpler.

\medskip 
\noindent \underline{\it Case 3.3: $\ell = 0$.}
We proceed as at the beginning of Case~3.2 and distinguish how close the input frequency variable $\xi_1$ is to the problematic frequencies $\pm \sqrt{3}$. By a simple $L^2 \times L^\infty \times L^\infty$ estimate, we can reduce to the case where $||\xi_1|-\sqrt{3}| \lesssim 2^{-\ell_1-100}$ with $\ell_1 \geq 10$, and without loss of generality, we may assume that $\xi_1 > 0$. We are thus led to consider the weighted energy estimate for the term
\begin{equation*}
 \begin{aligned}
  \calI^{1,+}_{2,m;0 \leq \ell_1}(t,\xi) := \int_0^t \tau_m(s) \iint e^{is\phi_2} \varphi_0^{(n)}(\xi) \varphi^{(m),+}_{\geq 10}(\xi_1) \, \jxione \pxione \hatg(s,\xi_1) \hatbarg(s,\xi_2) \hatg(s,\xi_3) \, \ud \xi_1 \, \ud \xi_2 \, \ud s.
 \end{aligned}
\end{equation*}
Integrating by parts in $\xi_2$ again, we obtain
\begin{equation*}
 \begin{aligned}
  &\calI_{2,m; 0 \leq \ell_1}^{1,+}(t,\xi) \\
  &= i \int_0^t \tau_m(s) \frac{1}{s} \iint e^{is\phi_2} \frac{1}{\pxitwo \phi_2} \varphi_{0}^{(n)}(\xi) \varphi_{\geq 10}^{(m),+}(\xi_1) \, \jxione (\pxione \hatg)(s,\xi_1) (\pxitwo \hatbarg)(s,\xi_2) \hatg(s,\xi_3) \, \ud \xi_1 \, \ud \xi_2 \, \ud s \\
  &\quad -i \int_0^t \tau_m(s) \frac{1}{s} \iint e^{is\phi_2} \frac{1}{\pxitwo \phi_2} \varphi_{0}^{(n)}(\xi) \varphi_{\geq 10}^{(m),+}(\xi_1) \, \jxione (\pxione \hatg)(s,\xi_1) \hatbarg(s,\xi_2) (\pxithree \hatg)(s,\xi_3) \, \ud \xi_1 \, \ud \xi_2 \, \ud s \\
  &\quad +i \int_0^t \tau_m(s) \frac{1}{s} \iint e^{is\phi_2} \pxitwo \biggl( \frac{1}{\pxitwo \phi_2} \biggr) \varphi_{0}^{(n)}(\xi) \varphi_{\geq 10}^{(m),+}(\xi_1) \, \jxione (\pxione \hatg)(s,\xi_1) \hatbarg(s,\xi_2) \hatg(s,\xi_3) \, \ud \xi_1 \, \ud \xi_2 \, \ud s \\
  &=: \calI_{2,m; 0\leq \ell_1}^{1,+,(a)}(t,\xi) + \calI_{2,m; 0 \leq \ell_1}^{1,+,(b)}(t,\xi) + \calI_{2,m; 0 \leq \ell_1}^{1,+,(c)}(t,\xi).
 \end{aligned}
\end{equation*}
We claim that on the support of $\varphi_{0}^{(n)}(\xi) \varphi_{\geq 10}^{(m),+}(\xi_1)$, we have the bounds
\begin{align}
 \frac{1}{|\pxitwo \phi_2|} &\lesssim \jxi^{-1} \jxitwo \jxithree \min\bigl\{ \jxitwo, \jxithree \bigr\}, \label{equ:calI2m1_case33_pxitwo_phi_bound_claim} 
\end{align}
and 
\begin{align}
 \biggl| \pxitwo \biggl( \frac{1}{\pxitwo \phi_2} \biggr) \biggr| &\lesssim \jxi^{-2} \jxitwo^2 \jxithree^2. \label{equ:calI2m1_case33_pxitwo_phi_diffed_bound_claim}
\end{align}
To see these, first note that $|\xi-\xi_1| = |(\xi-\sqrt{3})-(\xi_1-\sqrt{3})| \gtrsim 2^{-100}$ in the current frequency configuration. Since $\xi-\xi_1 = \xi_2 + \xi_3$ we must therefore have $|\xi_2| + |\xi_3| \gtrsim 2^{-100}$ on the support of $\varphi_{0}^{(n)}(\xi) \varphi_{\geq 10}^{(m),+}(\xi_1)$. If $\xi_2 \xi_3 > 0$, then \eqref{equ:calI2m1_pxitwo_phi2} implies $|\pxitwo \phi_2| \gtrsim 1$, which is consistent with \eqref{equ:calI2m1_case33_pxitwo_phi_bound_claim}. If on the other hand $\xi_2 \xi_3 < 0$, then we infer the bound~\eqref{equ:calI2m1_case33_pxitwo_phi_bound_claim} from \eqref{equ:calI2m1_pxitwo_phi2_inverse}.
Finally, the estimate \eqref{equ:calI2m1_case33_pxitwo_phi_diffed_bound_claim} follows from \eqref{equ:calI2m1_case33_pxitwo_phi_bound_claim} since $\pxitwo^2 \phi_2 = - \jxitwo^{-3} + \jxithree^{-3}$.

The weighted energy estimate for the term $\calI_{2,m; 0 \leq \ell_1}^{1,+,(a)}(t,\xi)$ now follows from \eqref{equ:calI2m1_case33_pxitwo_phi_bound_claim} and \eqref{equ:g_bound_pxiL1}, \eqref{equ:g_bound_jD2_L2},
\begin{equation*}
 \begin{aligned}
  \bigl\| \calI_{2,m; 0 \leq \ell_1}^{1,+,(a)}(t,\xi) \bigr\|_{L^2_\xi} &\lesssim 2^m \cdot 2^{-m} \cdot \sup_{s \, \simeq \, 2^m} \, \bigl\| \jxione \pxione \hatg(s,\xi_1)\bigr\|_{L^1_{\xi_1}} \bigl\| \jxitwo \pxitwo \hatg(s,\xi_2)\bigr\|_{L^1_{\xi_2}} \bigl\| \jxithree \hatg(s,\xi_3) \bigr\|_{L^2_{\xi_3}} \lesssim m^2 \varepsilon^3.
 \end{aligned}
\end{equation*}
The bound for $\calI_{2,m; 0 \leq \ell_1}^{1,+,(b)}(t,\xi)$ is analogous. 
For the weighted energy estimate of the last term  $\calI_{2,m; 0 \leq \ell_1}^{1,+,(c)}(t,\xi)$ we use the bound \eqref{equ:calI2m1_case33_pxitwo_phi_diffed_bound_claim}, apply Cauchy-Schwarz in the integration variable $\xi_2$, and invoke the estimates \eqref{equ:g_bound_pxiL1}, \eqref{equ:g_bound_jD2_L2} to obtain
\begin{equation*}
 \begin{aligned}
  &\bigl\| \calI_{2,m; 0\leq \ell_1}^{1,+,(c)}(t,\xi) \bigr\|_{L^2_\xi} \\
  &\lesssim 2^m \cdot 2^{-m} \cdot \sup_{s \, \simeq \, 2^m} \, \bigl\| \jxi^{-2} \bigr\|_{L^2_\xi} \bigl\| \jxione \pxione \hatg(s,\xi_1)\bigr\|_{L^1_{\xi_1}} \bigl\| \jxitwo^2 \hatg(s,\xi_2)\bigr\|_{L^2_{\xi_2}} \bigl\| \jxithree^2 \hatg(s,\xi_3)\bigr\|_{L^2_{\xi_3}} \lesssim m \varepsilon^3.
 \end{aligned}
\end{equation*}

\medskip 
\noindent {\bf Step 4: Weighted energy estimate for the term $\calI_{2,m}^2(t,\xi)$.}
We now aim to prove for all $0 \leq m \leq n+5$ that
\begin{equation*}
 \sup_{0 \leq \ell \leq n} \, 2^{-\hf \ell} \bigl\| \varphi_\ell^{(n)}(\xi) \calI_{2,m}^2(t,\xi) \bigr\|_{L^2_\xi} \lesssim m^5 \varepsilon^3.
\end{equation*}
Again, we distinguish the cases $\ell = n$, $1 \leq \ell \leq n-1$, and $\ell = 0$.

\medskip 
\noindent \underline{\it Case 4.1: $\ell = n$.}
Analogously to Case 3.1 of the treatment of the term $\calI_{2,m}^1(t,\xi)$, here we can just use H\"older's inequality in the frequency variables and the estimates from Lemma~\ref{lem:g_bounds_repeated} to obtain the desired bound.

\medskip 
\noindent \underline{\it Case 4.2: $1 \leq \ell \leq n-1$.}
We insert a smooth partition of unity, this time to distinguish how close the frequency variable $\xi_2$ is to the problematic frequencies $\pm \sqrt{3}$, and write 
\begin{equation*}
 \begin{aligned}
  \varphi_{\ell}^{(n)}(\xi) \calI_{2,m}^2(t,\xi) = \sum_{0 \leq \ell_2 \leq m} \calI_{2,m; \ell, \ell_2}^{2}(t,\xi)
 \end{aligned}
\end{equation*}
with 
\begin{equation*}
 \begin{aligned}
  \calI_{2,m; \ell, \ell_2}^{2}(t,\xi) := \varphi_{\ell}^{(n)}(\xi) \int_0^t \tau_m(s) \iint e^{is\phi_2} \varphi_{\ell_2}^{(m)}(\xi_2) \hatg(s,\xi_1) \jxitwo \pxitwo \hatbarg(s,\xi_2) \hatg(s,\xi_3) \, \ud \xi_1 \, \ud \xi_2 \, \ud s.
 \end{aligned}
\end{equation*}
In what follows, we may assume that $\ell \leq m-100$. We then distinguish the subcases $0 \leq \ell_2 < \ell+100$ and $\ell+100 \leq \ell_2 \leq m$. If $\ell > m-100$, we can just proceed as in the former subcase for all $0 \leq \ell_2 \leq m$.

\medskip 
\noindent \underline{\it Subcase 4.2.1: $1 \leq \ell \leq n-1$, $0 \leq \ell_2 < \ell+100$.}
This subcase can be dealt with via a simple $L^\infty_x \times L^2_x \times L^\infty_x$ estimate, analogously to Case 3.2.1 for the term $\calI_{2,m;\ell\leq\ell_1}^1(t,\xi)$.

\medskip 
\noindent \underline{\it Subcase 4.2.2: $1 \leq \ell \leq n-1$, $\ell+100 \leq \ell_2 \leq m$.}
We have now reduced to the most delicate interactions when $||\xi|-\sqrt{3}| \simeq 2^{-\ell-100}$ with $1 \leq \ell \leq n-1$ and $||\xi_2|-\sqrt{3}| \simeq 2^{-\ell_2-100}$ with $\ell+100 \leq \ell_2 \leq m$. Without loss of generality, we may assume that $\xi > 0$, i.e., that $\xi \approx \sqrt{3}$. Then we have to distinguish the cases $\xi_2 \approx -\sqrt{3}$ and $\xi_2 \approx \sqrt{3}$. In view of the space-time resonance of the phase $\phi_2(\xi,\xi_1,\xi_2)$ at $(\xi, \xi_1, \xi_2, \xi_3) = (\sqrt{3}, \sqrt{3}, - \sqrt{3}, \sqrt{3})$, we discuss the more difficult case $\xi_2 \approx -\sqrt{3}$ in detail, and leave the case $\xi_2 \approx \sqrt{3}$ to the reader. 

Since the second input $\hatbarg(s,\xi_2)$ is already differentiated, we can try to integrate by parts in $\xi_1$ in this subcase. However, $\pxione \phi_2$ can vanish here, because
\begin{equation} \label{equ:phi2_delta_subcase422_pxione_vanishing_discussion}
 \pxione \phi_2 = \frac{\xi_1}{\jxione} - \frac{\xi_3}{\jxithree} = 0 \quad \Leftrightarrow \quad \xi_1 - \xi_3 = 0 \quad \Leftrightarrow \quad \xi_1 = \frac12 (\xi-\xi_2),
\end{equation}
which is not ruled out in the configuration $\ell_2 \geq \ell+100$.
Correspondingly, we need to make a further distinction relative to the size of
\begin{equation*}
 |\xi_1-\xi_3| \simeq 2^{-\ell_4}, \quad \ell_4 \in \bbZ,
\end{equation*} 
and we separately treat the subcases (1) $\ell_4 \geq \ell + 1000$ and (2) $\ell_4 < \ell + 1000$. In the former case integration by parts in time is feasible, while in the latter case integration by parts in $\xi_1$ is possible and pays off.

\medskip 
\noindent \underline{\it Subcase 4.2.2.1: $1 \leq \ell \leq n-1$, $\ell+100 \leq \ell_2 \leq m$, $\ell_4 \geq \ell+1000$.}
Here we consider the term
\begin{equation*}
 \begin{aligned}
  \calI_{2,m; \ell\leq\ell_2}^{2, \ell\leq\ell_4}(t,\xi) := \int_0^t \tau_m(s) \iint e^{is\phi_2} \frakm(\xi,\xi_1,\xi_2) \hatg(s,\xi_1) \pxitwo \hatbarg(s,\xi_2) \hatg(s,\xi_3) \, \ud \xi_1 \, \ud \xi_2 \, \ud s
 \end{aligned}
\end{equation*}
with
\begin{equation*}
 \frakm(\xi,\xi_1,\xi_2) := \varphi_{\ell}^{(n),+}(\xi) \varphi_{\geq \ell+100}^{(m),-}(\xi_2) \varphi_{\leq-\ell-1000}(\xi_1-\xi_3) \, \jxitwo.
\end{equation*}
Note that we included $\jxitwo$ into the symbol $\frakm(\xi,\xi_1,\xi_2)$, and that $\jxitwo$ is of size $\calO(1)$ on the support of $\frakm(\xi,\xi_1,\xi_2)$.
Observe that since $|\xi_2+\sqrt{3}| \ll |\xi-\sqrt{3}| \simeq 2^{-\ell-100}$, the relation $\xi-\xi_2 = \xi_1 + \xi_3$ gives
\begin{equation*} 
 2^{-\ell-100} \simeq \bigl| (\xi - \sqrt{3}) - (\xi_2+\sqrt{3}) \bigr| = \bigl| (\xi_1-\sqrt{3}) + (\xi_3-\sqrt{3}) \bigr|.
\end{equation*}
Since $|\xi_1-\xi_3| \lesssim 2^{-\ell-1000}$, we must have 
\begin{equation} \label{equ:phi2_delta_subcase422_xij_close_to_xi}
 \Bigl| (\xi_j - \sqrt{3}) - \frac12 (\xi-\sqrt{3}) \Bigr| \lesssim 2^{-\ell-500}, \quad j = 1, 3.
\end{equation}
In particular, this means that $|\xi_j-\sqrt{3}| \simeq 2^{-\ell-100}$ for $j=1,3$.
Our only resort in this frequency configuration is to integrate by parts in time, which will require a careful analysis of the size of the phase function $\phi_2$.
In fact, we first integrate by parts in $\xi_2$, because this will lead to a better balance of all inputs when we later insert the equation for $\ps \hatg(s)$ after having integrated by parts in time.
We obtain upon integrating by parts in $\xi_2$,
\begin{equation*}
 \begin{aligned}
  \calI_{2,m; \ell\leq\ell_2}^{2, \ell\leq\ell_4}(t,\xi) &= \int_0^t \tau_m(s) \iint e^{is\phi_2} \frakm(\xi,\xi_1,\xi_2) \hatg(s,\xi_1) \hatbarg(s,\xi_2) \pxithree \hatg(s,\xi_3) \, \ud \xi_1 \, \ud \xi_2 \, \ud s \\
  &\quad - \int_0^t \tau_m(s) \iint e^{is\phi_2} \, \pxitwo \frakm(\xi,\xi_1,\xi_2) \hatg(s,\xi_1)  \hatbarg(s,\xi_2) \hatg(s,\xi_3) \, \ud \xi_1 \, \ud \xi_2 \, \ud s \\
  &\quad -i \int_0^t \tau_m(s) \cdot s \iint e^{is\phi_2} \, (\pxitwo \phi_2) \, \frakm(\xi,\xi_1,\xi_2) \hatg(s,\xi_1)  \hatbarg(s,\xi_2) \hatg(s,\xi_3) \, \ud \xi_1 \, \ud \xi_2 \, \ud s \\
  &=: \calI_{2,m; \ell\leq\ell_2}^{2, \ell\leq\ell_4, (a)}(t,\xi) + \calI_{2,m; \ell\leq\ell_2}^{2, \ell\leq\ell_4, (b)}(t,\xi) + \calI_{2,m; \ell\leq\ell_2}^{2, \ell\leq\ell_4, (c)}(t,\xi).
 \end{aligned}
\end{equation*}
The last term on the right-hand side is the most delicate one and requires the integration by parts in time. Before we turn to it, we dispense of the first two terms on the right-hand side.

To estimate the term $\calI_{2,m; \ell\leq\ell_2}^{2, \ell\leq\ell_4, (a)}(t,\xi)$ we integrate by parts using the identity
\begin{equation} \label{equ:phi2_delta_int_by_parts_xi1_and_xi2}
 e^{is\phi_2} = \frac{1}{is} \frac{1}{(\pxione - \pxitwo)\phi_2} (\pxione - \pxitwo) \bigl( e^{is\phi_2} \bigr).
\end{equation}
Note that $(\pxione - \pxitwo) \pxithree \hatg(s,\xi_3) = 0$. We find
\begin{equation*}
 \begin{aligned}
  &\calI_{2,m; \ell\leq\ell_2}^{2, \ell\leq\ell_4, (a)}(t,\xi) \\
  &= i \int_0^t \tau_m(s) \frac{1}{s} \iint e^{is\phi_2} \frac{\frakm(\xi,\xi_1,\xi_2)}{(\pxione - \pxitwo)\phi_2}  \pxione \hatg(s,\xi_1) \hatbarg(s,\xi_2) \pxithree \hatg(s,\xi_3) \, \ud \xi_1 \, \ud \xi_2 \, \ud s \\
  &\quad -i \int_0^t \tau_m(s) \frac{1}{s} \iint e^{is\phi_2} \frac{\frakm(\xi,\xi_1,\xi_2)}{(\pxione - \pxitwo)\phi_2} \hatg(s,\xi_1) \pxitwo \hatbarg(s,\xi_2) \pxithree \hatg(s,\xi_3) \, \ud \xi_1 \, \ud \xi_2 \, \ud s \\
  &\quad + i \int_0^t \tau_m(s) \frac{1}{s} \iint e^{is\phi_2} \biggl( (\pxione - \pxitwo) \frac{\frakm(\xi,\xi_1,\xi_2)}{(\pxione - \pxitwo)\phi_2} \biggr) \hatg(s,\xi_1) \hatbarg(s,\xi_2) \pxithree \hatg(s,\xi_3) \, \ud \xi_1 \, \ud \xi_2 \, \ud s \\
  &=: \calI_{2,m; \ell\leq\ell_2}^{2, \ell\leq\ell_4, (a), 1}(t,\xi) + \calI_{2,m; \ell\leq\ell_2}^{2, \ell\leq\ell_4, (a), 2}(t,\xi) + \calI_{2,m; \ell\leq\ell_2}^{2, \ell\leq\ell_4, (a), 3}(t,\xi).
 \end{aligned}
\end{equation*}
Now observe that by Taylor expansion and \eqref{equ:phi2_delta_subcase422_xij_close_to_xi},
\begin{equation*}
 \begin{aligned}
  (\pxione - \pxitwo)\phi_2 &= \frac{\xi_1}{\jxione} + \frac{\xi_2}{\jxitwo} 
  = \frac{1}{16} (\xi-\sqrt{3}) + \frac18 (\xi_2+\sqrt{3}) + \calO\bigl( 2^{-2(\ell+100)} \bigr) + + \calO\bigl( 2^{-\ell-500} \bigr)
  \simeq 2^{-\ell-100}.
 \end{aligned}
\end{equation*}
Similarly, we obtain that
\begin{equation*}
 \begin{aligned}
  \biggl| (\pxione - \pxitwo) \frac{\frakm(\xi,\xi_1,\xi_2)}{(\pxione - \pxitwo)\phi_2} \biggr| \lesssim 2^{2\ell}.
 \end{aligned}
\end{equation*}
Hence, we can conclude by H\"older's inequality in the frequency variables and by the bounds \eqref{equ:g_bound_pxiL1}, \eqref{equ:g_bound_Linftyxi} that
\begin{equation*}
 \begin{aligned}
  &2^{-\frac12 \ell} \bigl\| \calI_{2,m; \ell\leq\ell_2}^{2, \ell\leq\ell_4, (a), 1}(t,\xi) \bigr\|_{L^2_\xi} \\
  &\lesssim 2^{-\frac12 \ell} \bigl\| \varphi^{(n),+}_\ell(\xi) \bigr\|_{L^2_\xi} \cdot 2^m \cdot 2^{-m} \cdot 2^\ell \cdot \sup_{s \, \simeq \, 2^m} \, \bigl\| \pxione \hatg(s,\xi_1) \bigr\|_{L^1_{\xi_1}} \| \hatg(s,\xi_2)\|_{L^\infty_{\xi_2}} \bigl\| \pxithree \hatg(s,\xi_3)\bigr\|_{L^1_{\xi_3}} \\
  &\lesssim 2^{-\frac12 \ell} \cdot 2^{-\frac12 \ell} \cdot 2^m \cdot 2^{-m} \cdot 2^\ell \cdot m^3 \varepsilon^3 \lesssim m^3 \varepsilon^3.
 \end{aligned}
\end{equation*}
The bound for $\calI_{2,m; \ell\leq\ell_2}^{2, \ell\leq\ell_4, (a), 2}(t,\xi)$ is analogous. For estimating $\calI_{2,m; \ell\leq\ell_2}^{2, \ell\leq\ell_4, (a), 3}(t,\xi)$, we place both inputs $\hatg(s,\xi_1)$ and $\hatbarg(s,\xi_2)$ into $L^\infty$, while we place $\pxithree \hatg(s,\xi_3)$ into $L^1$. This allows us to gain back a factor $2^{-\ell}$ from the size of the frequency support of $\frakm(\xi,\xi_1,\xi_2)$ in the variable $\xi_2$.

For the weighted estimate of the term $\calI_{2,m; \ell\leq\ell_2}^{2, \ell\leq\ell_4, (b)}(t,\xi)$ we first observe that
\begin{equation*}
 \Bigl| \pxi^\kappa \pxione^{\kappa_1} \pxitwo^{\kappa_2} \Bigl( \pxitwo \frakm(\xi,\xi_1,\xi_2) \Bigr) \Bigr| \lesssim 2^\ell 2^{(\kappa + \kappa_1 + \kappa_2)\ell},
\end{equation*}
whence
\begin{equation*}
 \Bigl\| \calF^{-1}\Bigl[ \pxitwo \frakm(\xi,\xi_1,\xi_2) \Bigr] \Bigr\|_{L^1(\bbR^3)} \lesssim 2^\ell.
\end{equation*}
Then using Lemma~\ref{lem:frakm_for_delta_three_inputs} and invoking the bounds \eqref{equ:g_bound_Linftyxi}, \eqref{equ:g_bound_dispersive_est}, we obtain
\begin{equation*}
 \begin{aligned}
  &2^{-\frac12 \ell} \bigl\| \calI_{2,m; \ell\leq\ell_2}^{2, \ell\leq\ell_4, (b)}(t,\xi) \bigr\|_{L^2_\xi} \\
  &\lesssim 2^{-\frac12 \ell} \cdot 2^m \cdot \Bigl\| \calF^{-1}\Bigl[ \pxitwo \frakm(\xi,\xi_1,\xi_2) \Bigr] \Bigr\|_{L^1(\bbR^3)} \\
  &\quad \quad \cdot \sup_{s \, \simeq \, 2^m} \, \bigl\| e^{is\jD} g(s) \bigr\|_{L^\infty_x} \bigl\| \widetilde{\varphi}_{\geq \ell+100}^{(m),-}(\xi_2) \hatbarg(s,\xi_2) \bigr\|_{L^2_{\xi_2}} \bigl\| e^{is\jD} g(s) \bigr\|_{L^\infty_x} \\
  &\lesssim 2^{-\frac12 \ell} \cdot 2^m \cdot 2^\ell \cdot \sup_{s \, \simeq \, 2^m} \, \bigl\| e^{is\jD} g(s) \bigr\|_{L^\infty_x} \cdot 2^{-\frac12 \ell} \bigl\| \hatbarg(s,\xi_2) \bigr\|_{L^\infty_{\xi_2}} \bigl\| e^{is\jD} g(s) \bigr\|_{L^\infty_x} \\
  &\lesssim 2^{-\frac12 \ell} \cdot 2^m \cdot 2^\ell \cdot 2^{-\frac12 m} m \varepsilon \cdot 2^{-\frac12 \ell} \cdot m \varepsilon \cdot 2^{-\frac12 m} m \varepsilon \lesssim m^3 \varepsilon^3.
 \end{aligned}
\end{equation*}

Finally, in order to bound the delicate term $\calI_{2,m; \ell\leq\ell_2}^{2, \ell\leq\ell_4, (c)}(t,\xi)$, we integrate by parts in time,
\begin{equation*}
 \begin{aligned}
  \calI_{2,m; \ell\leq\ell_2}^{2, \ell\leq\ell_4, (c)}(t,\xi)  &= \int_0^t \tau_m(s) \cdot s \iint e^{is\phi_2} \, \frac{\pxitwo \phi_2}{\phi_2} \, \frakm(\xi,\xi_1,\xi_2) \, \ps \bigl( \hatg(s,\xi_1)  \hatbarg(s,\xi_2) \hatg(s,\xi_3) \bigr) \, \ud \xi_1 \, \ud \xi_2 \, \ud s \\
  &\quad + \int_0^t \ps \bigl( \tau_m(s) \cdot s \bigr) \iint e^{is\phi_2} \, \frac{\pxitwo \phi_2}{\phi_2} \, \frakm(\xi,\xi_1,\xi_2) \hatg(s,\xi_1) \hatbarg(s,\xi_2) \hatg(s,\xi_3) \, \ud \xi_1 \, \ud \xi_2 \, \ud s \\
  &\quad - \tau_m(s) \cdot s \iint e^{is\phi_2} \, \frac{\pxitwo \phi_2}{\phi_2} \, \frakm(\xi,\xi_1,\xi_2) \, \ps \bigl( \hatg(s,\xi_1)  \hatbarg(s,\xi_2) \hatg(s,\xi_3) \bigr) \, \ud \xi_1 \, \ud \xi_2 \, \ud s \bigg|_{s=0}^{s=t} \\
  &=: \calI_{2,m; \ell\leq\ell_2}^{2, \ell\leq\ell_4, (c), 1}(t,\xi) + \calI_{2,m; \ell\leq\ell_2}^{2, \ell\leq\ell_4, (c), 2}(t,\xi) + \calI_{2,m; \ell\leq\ell_2}^{2, \ell\leq\ell_4, (c), 3}(t,\xi).
 \end{aligned}
\end{equation*}
Recall that in the current frequency configuration
\begin{equation*}
 \begin{aligned}
  |\xi-\sqrt{3}| \simeq 2^{-\ell-100}, \quad |\xi_2+\sqrt{3}| \ll |\xi-\sqrt{3}|, \quad
 \Bigl| (\xi_j - \sqrt{3}) - \frac12 (\xi-\sqrt{3}) \Bigr| \lesssim 2^{-\ell-500}, \quad j = 1, 3.
 \end{aligned}
\end{equation*}
Moreover, we have $|\xi_j-\sqrt{3}| \simeq 2^{-\ell-100}$ for $j = 1, 3$.
Correspondingly, we obtain by Taylor expansion around $\xi \approx \sqrt{3}$, respectively around $\xi_3 \approx \sqrt{3}$,
\begin{equation*}
 \pxi \phi_2 = -\frac{\xi}{\jxi} + \frac{\xi_3}{\jxithree} = -\frac18 (\xi-\sqrt{3}) + \frac{1}{8}(\xi_3-\sqrt{3}) + \calO\bigl(2^{-2(\ell+100)}\bigr).
\end{equation*}
It follows that $|\pxi \phi_2| \lesssim 2^{-\ell}$ on the support of $\frakm(\xi,\xi_1,\xi_2)$. We can estimate $|\pxione \phi_2|$ and $|\pxitwo \phi_2|$ analogously. Thus, on the support of $\frakm(\xi,\xi_1,\xi_2)$ we have 
\begin{equation} \label{equ:phi2_delta_subcase4221_est1}
 |\pxi \phi_2| + |\pxione \phi_2| + |\pxitwo \phi_2| \lesssim 2^{-\ell}
\end{equation}
Next, we determine the size of the phase function $\phi_2(\xi,\xi_1,\xi_2)$ on the support of $\frakm(\xi,\xi_1,\xi_2)$.
To this end, we compute 
\begin{equation*}
 \begin{aligned}
  \bigl[\mathrm{Hess} \, (\phi_2)\bigr](\sqrt{3},\sqrt{3},-\sqrt{3}) = \frac18 \begin{bmatrix} 0 & -1 & -1 \\ -1 & 2 & 1 \\ -1 & 1 & 0 \end{bmatrix}.
 \end{aligned}
\end{equation*}
Since $\phi_2(\sqrt{3},\sqrt{3},-\sqrt{3})=0$ and 
\begin{equation*}
 \pxi \phi_2(\sqrt{3},\sqrt{3},-\sqrt{3}) = \pxione \phi_2(\sqrt{3},\sqrt{3},-\sqrt{3}) = \pxitwo \phi_2(\sqrt{3},\sqrt{3},-\sqrt{3}) = 0,
\end{equation*}
a Taylor expansion of the phase $\phi_2(\xi,\xi_1,\xi_2)$ around $(\xi,\xi_1,\xi_2) = (\sqrt{3}, \sqrt{3},-\sqrt{3})$ yields that on the support of $\frakm(\xi,\xi_1,\xi_2)$,
\begin{equation*}
 \begin{aligned}
  \phi_2(\xi,\xi_1,\xi_2) &= \frac{1}{16} \Bigl( -2 (\xi-\sqrt{3}) (\xi_1-\sqrt{3}) -2 (\xi-\sqrt{3}) (\xi_2+\sqrt{3}) \\
  &\qquad \quad + 2 (\xi_1-\sqrt{3})^2 + 2 (\xi_1-\sqrt{3}) (\xi_2+\sqrt{3}) \Bigr) + \calO\bigl( 2^{-3(\ell+100)} \bigr).
 \end{aligned}
\end{equation*}
Observing that $|\xi_2+\sqrt{3}| \ll |\xi-\sqrt{3}|$ and that $\xi_1-\sqrt{3} = \frac12 (\xi-\sqrt{3}) + \calO(2^{-\ell-500})$, we conclude that on the support of $\frakm(\xi,\xi_1,\xi_2)$,
\begin{equation} \label{equ:phi2_delta_subcase4221_est2}
  \phi_2(\xi,\xi_1,\xi_2) = \frac{1}{16} \Bigl( -\frac12 (\xi-\sqrt{3})^2 + \calO\bigl( 2^{-100} 2^{-2(\ell+100)} \bigr) \Bigr) \simeq 2^{-2(\ell+100)}.
\end{equation}
From~\eqref{equ:phi2_delta_subcase4221_est1} and \eqref{equ:phi2_delta_subcase4221_est2} we infer
\begin{equation*}
 \biggl| \pxi^\kappa \pxione^{\kappa_1} \pxitwo^{\kappa_2} \biggl( \frac{\pxitwo \phi_2}{\phi_2} \, \frakm(\xi,\xi_1,\xi_2) \biggr)\biggr| \lesssim 2^\ell 2^{(\kappa + \kappa_1 + \kappa_2)\ell},
\end{equation*}
whence 
\begin{equation} \label{equ:phi2_delta_subcase4221_est3}
 \biggl\| \calF^{-1}\biggl[ \frac{\pxitwo \phi_2}{\phi_2} \, \frakm(\xi,\xi_1,\xi_2) \biggr] \biggr\|_{L^1(\bbR^3)} \lesssim 2^\ell.
\end{equation}

Now we are prepared to estimate the two terms $\calI_{2,m; \ell\leq\ell_2}^{2, \ell\leq\ell_4, (c), 1}(t,\xi)$ and $\calI_{2,m; \ell\leq\ell_2}^{2, \ell\leq\ell_4, (c), 2}(t,\xi)$. We begin with the term $\calI_{2,m; \ell\leq\ell_2}^{2, \ell\leq\ell_4, (c), 1}(t,\xi)$ when the time derivative falls onto the first input $\hatg(s,\xi_1)$. The cases when it falls onto the other two inputs $\hatbarg(s,\xi_2)$ and $\hatg(s,\xi_3)$ can be estimated in an identical manner.
Inserting the equation~\eqref{equ:decomposition_FT_pt_hatg} for $\ps \hatg(s,\xi_1)$, using Lemma~\ref{lem:frakm_for_delta_three_inputs} with \eqref{equ:phi2_delta_subcase4221_est3}, and invoking the bounds \eqref{equ:Linfty_decay_v}, \eqref{equ:g_bound_dispersive_est}, \eqref{equ:g_bound_Linftyxi}, \eqref{equ:decomposition_FT_pt_hatg_Ncdecay} we obtain the weighted estimate
\begin{equation*}
 \begin{aligned}
  &2^{-\frac12 \ell} \, \biggl\| \int_0^t \tau_m(s) \cdot s \iint e^{is\phi_2} \, \frac{\pxitwo \phi_2}{\phi_2} \, \frakm(\xi,\xi_1,\xi_2) \, (2i\jxione)^{-1} e^{-is\jxione} \\
  &\qquad \qquad \qquad \qquad \times \Bigl( \whatbeta(\xi_1) \bigl( v(s,0) + \barv(s,0) \bigr)^2 + \widehat{\calN}_c(s,\xi_1) \Bigr)  \hatbarg(s,\xi_2) \hatg(s,\xi_3)  \, \ud \xi_1 \, \ud \xi_2 \, \ud s \biggr\|_{L^2_\xi} \\
  &\lesssim 2^{-\frac12 \ell} \cdot 2^{2m} \cdot 2^\ell \cdot \sup_{s \, \simeq \, 2^m} \, \biggl( \bigl\| \varphi_{\leq -\ell+100}(\xi_1-\sqrt{3}) (2i\jxione)^{-1} \whatbeta(\xi_1) \bigr\|_{L^2_{\xi_1}} |v(s,0)|^2 \bigl\|e^{is\jD} g(s)\bigr\|_{L^\infty_x}^2 \\
  &\qquad \qquad \qquad \qquad \qquad \qquad + \bigl\| (2i\jD)^{-1} \calN_c(s) \bigr\|_{L^\infty_x} \bigl\| \varphi_{\geq \ell+100}^{(m),-}(\xi_2) \hatg(s,\xi_2) \bigr\|_{L^2_{\xi_2}} \bigl\|e^{is\jD} g(s)\bigr\|_{L^\infty_x} \biggr) \\
  &\lesssim 2^{-\frac12 \ell} \cdot 2^{2m} \cdot 2^\ell \cdot \Bigl( 2^{-\frac12 \ell} \cdot 2^{-2m} m^4 \varepsilon^4 + 2^{-\frac32 m} m^3 \varepsilon^2 \cdot 2^{-\frac12 \ell} m \varepsilon \cdot 2^{-\frac12 m} m \varepsilon \Bigr) \lesssim  m^5 \varepsilon^3.
 \end{aligned}
\end{equation*}
Here we could freely insert the frequency cut-off $\varphi_{\leq -\ell+100}(\xi_1-\sqrt{3})$, and we invoked the bounds
\begin{equation*}
 \bigl\| \varphi_{\leq -\ell+100}(\xi_1-\sqrt{3}) (2i\jxione)^{-1} \whatbeta(\xi_1) \bigr\|_{L^2_{\xi_1}} \lesssim 2^{-\frac12 \ell} \|\whatbeta(\xi_1)\|_{L^\infty_{\xi_1}} \lesssim 2^{-\frac12 \ell}
\end{equation*}
and
\begin{equation*}
 \bigl\|\varphi_{\geq \ell+100}^{(m),-}(\xi_2) \hatg(s,\xi_2) \bigr\|_{L^2_{\xi_2}} \lesssim 2^{-\frac12 \ell} \bigl\|\hatg(s,\xi_2)\bigr\|_{L^\infty_{\xi_2}}.
\end{equation*}

The weighted estimated for the term $\calI_{2,m; \ell\leq\ell_2}^{2, \ell\leq\ell_4, (c), 2}(t,\xi)$ is simpler. Using Lemma~\ref{lem:frakm_for_delta_three_inputs} with \eqref{equ:phi2_delta_subcase4221_est3}, along with the bounds \eqref{equ:g_bound_Linftyxi}, \eqref{equ:g_bound_dispersive_est} we find
\begin{equation*}
 \begin{aligned}
  2^{-\frac12 \ell} \bigl\| \calI_{2,m; \ell\leq\ell_2}^{2, \ell\leq\ell_4, (c), 2}(t,\xi) \bigr\|_{L^2_\xi}  &\lesssim 2^{-\frac12 \ell} \cdot 2^m \cdot 2^\ell \cdot \sup_{s \, \simeq \, 2^m} \, \bigl\| \varphi_{\geq \ell+100}^{(m),-}(\xi_2) \hatg(s,\xi_2) \bigr\|_{L^2_{\xi_2}} \bigl\| e^{is\jD} g(s) \bigr\|_{L^\infty_x}^2 \\
  &\lesssim 2^{-\frac12 \ell} \cdot 2^m \cdot 2^\ell \cdot 2^{-\frac12 \ell} \cdot \sup_{s \, \simeq \, 2^m} \, \bigl\| \hatg(s,\xi_2) \bigr\|_{L^\infty_{\xi_2}} \bigl\| e^{is\jD} g(s) \bigr\|_{L^\infty_x}^2 \\
  &\lesssim 2^{-\frac12 \ell} \cdot 2^m \cdot 2^\ell \cdot 2^{-\frac12 \ell} \cdot m \varepsilon \cdot 2^{-m} m^2 \varepsilon^2 \lesssim m^3 \varepsilon^3.
 \end{aligned}
\end{equation*}
The weighted estimate for the boundary term $\calI_{2,m; \ell\leq\ell_2}^{2, \ell\leq\ell_4, (c), 3}(t,\xi)$ is analogous.

\medskip 
\noindent \underline{\it Subcase 4.2.2.2: $1 \leq \ell \leq n-1$, $\ell+100 \leq \ell_2 \leq m$, $\ell_4 < \ell+1000$.}
Here we carry out the weighted energy estimate for the term
\begin{equation*}
 \begin{aligned}
  \calI_{2,m; \ell\leq\ell_2}^{2, \ell\geq\ell_4}(t,\xi) := \int_0^t \tau_m(s) \iint e^{is\phi_2} \frakn(\xi,\xi_1,\xi_2) \hatg(s,\xi_1) \jxitwo \pxitwo \hatbarg(s,\xi_2) \hatg(s,\xi_3) \, \ud \xi_1 \, \ud \xi_2 \, \ud s
 \end{aligned}
\end{equation*}
with
\begin{equation*}
 \frakn(\xi,\xi_1,\xi_2) := \varphi_{\ell}^{(n),+}(\xi) \varphi_{\geq \ell+100}^{(m),-}(\xi_2) \varphi_{>-\ell-1000}(\xi_1-\xi_3).
\end{equation*}
In this frequency configuration, $\xi_1-\xi_3$ cannot become zero, whence $\pxione \phi_2$ cannot vanish in view of \eqref{equ:phi2_delta_subcase422_pxione_vanishing_discussion}, and we can integrate by parts in $\xi_1$,
\begin{equation*}
 \begin{aligned}
  &\calI_{2,m; \ell\leq\ell_2}^{2, \ell\geq\ell_4}(t,\xi) \\
  &= i \int_0^t \tau_m(s) \frac{1}{s} \iint e^{is\phi_2} \frac{1}{\pxione \phi_2} \frakn(\xi,\xi_1,\xi_2) \pxione \hatg(s,\xi_1) \jxitwo \pxitwo \hatbarg(s,\xi_2) \hatg(s,\xi_3) \, \ud \xi_1 \, \ud \xi_2 \, \ud s \\
  &\quad - i \int_0^t \tau_m(s) \frac{1}{s} \iint e^{is\phi_2} \frac{1}{\pxione \phi_2} \frakn(\xi,\xi_1,\xi_2) \hatg(s,\xi_1) \jxitwo \pxitwo \hatbarg(s,\xi_2) \pxithree \hatg(s,\xi_3) \, \ud \xi_1 \, \ud \xi_2 \, \ud s \\
  &\quad + i \int_0^t \tau_m(s) \frac{1}{s} \iint e^{is\phi_2} \pxione \biggl( \frac{1}{\pxione \phi_2} \frakn(\xi,\xi_1,\xi_2) \biggr) \hatg(s,\xi_1) \jxitwo \pxitwo \hatbarg(s,\xi_2) \hatg(s,\xi_3) \, \ud \xi_1 \, \ud \xi_2 \, \ud s \\
  &=: \calI_{2,m; \ell\leq\ell_2}^{2, \ell\geq\ell_4, (a)}(t,\xi) + \calI_{2,m; \ell\leq\ell_2}^{2, \ell\geq\ell_4, (b)}(t,\xi) + \calI_{2,m; \ell\leq\ell_2}^{2, \ell\geq\ell_4, (c)}(t,\xi).
 \end{aligned}
\end{equation*}
The first and the second term on the right-hand side are symmetric, so it suffices to carry out the weighted estimates for the first and the third term. 

We may assume that $\ell > 10$ since the cases $1 \leq \ell \leq 10$ can be subsumed into the case $\ell=0$ discussed below.
We claim that on the support of $\frakn(\xi,\xi_1,\xi_2)$,
\begin{equation} \label{equ:phi2_delta_subcase4222_est1}
 \frac{1}{|\pxione \phi_2|} \lesssim \min \bigl\{ 2^\ell, 2^{\ell_1} \bigr\},
\end{equation}
where we distinguish relative to the size of $||\xi_1|-\sqrt{3}| \simeq 2^{-\ell_1-100}$ for $ \ell_1 \geq 1$ or $||\xi_1|-\sqrt{3}| \gtrsim 2^{-100}$ for $\ell_1=0$. To see~\eqref{equ:phi2_delta_subcase4222_est1}, we separately consider the cases $\ell_1 \geq \ell-10$ and $0 \leq \ell_1 < \ell-10$.
If $||\xi_1|-\sqrt{3}| \lesssim 2^{-\ell_1 -100} \lesssim 1$ for $\ell_1 \geq \ell-10$, we obtain from $2 \sqrt{3} \approx \xi - \xi_2 = \xi_1+\xi_3$ that $|\xi_1| + |\xi_3| \lesssim 1$, whence the assumption $|\xi_1-\xi_3| \gtrsim 2^{-\ell-1000}$ implies
\begin{equation*}
 \biggl| \frac{1}{\pxione \phi_2} \biggr| = \biggl| \frac{1}{\xi_1-\xi_3} \frac{\jxione \jxithree (\xi_1 \jxithree + \xi_3 \jxione)}{\xi_1+\xi_3}\biggr| \lesssim 2^\ell.
\end{equation*}
Now consider $||\xi_1|-\sqrt{3}| \simeq 2^{-\ell_1 -100}$ for $1 \leq \ell_1 < \ell-10$ or $||\xi_1|-\sqrt{3}| \gtrsim 2^{-100}$ for $\ell_1 = 0$. From $2 \sqrt{3} \approx \xi-\xi_2 = \xi_1+\xi_3$, we infer $|\xi_1| + |\xi_3| \gtrsim 1$.  Thus, if $\xi_1 \xi_3 < 0$, then
\begin{equation*}
 |\pxione \phi_2| = \biggl| \frac{\xi_1}{\jxione} - \frac{\xi_3}{\jxithree} \biggr| = \frac{|\xi_1|}{\jxione} + \frac{|\xi_3|}{\jxithree} \gtrsim 1.
\end{equation*}
Instead, if $\xi_1 \xi_3 > 0$, then $2 \sqrt{3} \approx \xi-\xi_2 = \xi_1 + \xi_3$ implies that $|\xi_1| + |\xi_3| \lesssim 1$. Moreover, since $\ell_1 < \ell-10$,
\begin{equation*}
 \begin{aligned}
  |\xi_1-\xi_3| &= \bigl|2 (\xi_1-\sqrt{3}) - (\xi-\sqrt{3}) + (\xi_2+\sqrt{3})\bigr| \gtrsim 2^{-\ell_1-100}.
 \end{aligned}
\end{equation*}
Hence, if $\xi_1 \xi_3 > 0$ we conclude 
\begin{equation*}
 \begin{aligned}
  \biggl| \frac{1}{\pxione \phi_2} \biggr| = \biggl| \frac{1}{\xi_1-\xi_3} \frac{\jxione \jxithree (\xi_1 \jxithree + \xi_3 \jxione)}{\xi_1+\xi_3}\biggr| \lesssim 2^{\ell_1}.
 \end{aligned}
\end{equation*}
Finally, we note that when $\ell_1=0$ with $|\xi_1| \gg 1$, then we must have $\xi_1 \xi_3 < 0$.
This proves \eqref{equ:phi2_delta_subcase4222_est1}.

We deduce the weighted energy estimate for the first term $\calI_{2,m; \ell\leq\ell_2}^{2, \ell\geq\ell_4, 1, (a)}(t,\xi)$ using \eqref{equ:phi2_delta_subcase4222_est1}, H\"older's inequality in the frequency variables, and the bounds \eqref{equ:g_bound_pxiL1}, \eqref{equ:g_bound_Linftyxi},
\begin{equation*}
 \begin{aligned}
  &2^{-\frac12 \ell} \big\| \calI_{2,m; \ell\leq\ell_2}^{2, \ell\geq\ell_4, 1, (a)}(t,\xi) \bigr\|_{L^2_\xi} \\ &\lesssim 2^{-\frac12 \ell} \bigl\| \wtilvarphi^{(n)}_\ell(\xi) \bigr\|_{L^2_\xi} \cdot 2^m \cdot 2^{-m} \cdot 2^{\ell} \cdot \sup_{s \, \simeq \, 2^m} \, \bigl\| \pxione \hatg(s,\xi_1) \bigr\|_{L^1_{\xi_1}} \bigl\| \jxitwo \pxitwo \hatbarg(s,\xi_2) \bigr\|_{L^1_{\xi_2}} \bigl\| \hatg(s,\xi_3) \bigr\|_{L^\infty_{\xi_3}} \\
  &\lesssim 2^{-\frac12 \ell} \cdot 2^{-\frac12 \ell} \cdot 2^m \cdot 2^{-m} \cdot 2^\ell \cdot m^3 \varepsilon^3 \lesssim m^3 \varepsilon^3.
 \end{aligned}
\end{equation*}
For the weighted estimate for the third term $\calI_{2,m; \ell\leq\ell_2}^{2, \ell\geq\ell_4, (c)}(t,\xi)$, we compute
\begin{equation} \label{equ:phi2_delta_subcase4222_pxione_on_symbol}
 \begin{aligned}
  \pxione \biggl( \frac{1}{\pxione \phi_2} \frakn(\xi,\xi_1,\xi_2) \biggr) = - \frac{\pxione^2 \phi_2}{(\pxione \phi_2)^2} \frakn(\xi,\xi_1,\xi_2) + \frac{1}{\pxione \phi_2} \pxione \frakn(\xi,\xi_1,\xi_2).
 \end{aligned}
\end{equation}
To estimate the contributions of the first term on the right-hand side of~\eqref{equ:phi2_delta_subcase4222_pxione_on_symbol} we decompose the integration over $\xi_1$ into the regions $||\xi_1|-\sqrt{3}| \lesssim 2^{-\ell+90}$, $||\xi_1|-\sqrt{3}| \simeq 2^{-\ell_1-100}$ for $1 \leq \ell_1 \leq \ell-10$, and $||\xi_1|-\sqrt{3}| \gtrsim 2^{-100}$.
For the contributions of the second term on the right-hand side of~\eqref{equ:phi2_delta_subcase4222_pxione_on_symbol}, we observe that $|\pxione \frakn(\xi,\xi_1,\xi_2)| \lesssim 2^\ell$, and that on the support of $\pxione \frakn(\xi,\xi_1,\xi_2)$, the variable $\xi_1$ is localized to the frequency interval $|\xi_1-\sqrt{3}| \lesssim 2^{-\ell}$.
Using the estimate~\eqref{equ:phi2_delta_subcase4222_est1}, we then obtain by H\"older's inequality in the frequency variables, and the bounds \eqref{equ:g_bound_pxiL1}, \eqref{equ:g_bound_Linftyxi} that
\begin{equation*}
 \begin{aligned}
  &2^{-\frac12 \ell} \big\| \calI_{2,m; \ell\leq\ell_2}^{2, \ell\geq\ell_4, (c)}(t,\xi) \bigr\|_{L^2_\xi} \\
  &\lesssim 2^{-\frac12 \ell} \bigl\| \wtilvarphi^{(n)}_\ell(\xi) \bigr\|_{L^2_\xi} \cdot 2^m \cdot 2^{-m}  \cdot \biggl( 2^{2 \ell} \bigl\| \varphi_{\leq-\ell+90}(|\xi_1|-\sqrt{3}) \bigr\|_{L^1_{\xi_1}}
   + \sum_{0 < \ell_1 < \ell-10} 2^{2\ell_1} \bigl\| \varphi_{-\ell_1-100}(|\xi_1|-\sqrt{3}) \bigr\|_{L^1_{\xi_1}} \biggr) \\
  &\qquad \qquad \qquad \times \sup_{s \, \simeq \, 2^m} \, \bigl\| \hatg(s,\xi_1) \bigr\|_{L^\infty_{\xi_1}} \bigl\| \jxitwo \pxitwo \hatbarg(s,\xi_2) \bigr\|_{L^1_{\xi_2}} \bigl\| \hatg(s,\xi_3) \bigr\|_{L^\infty_{\xi_3}} \\
  &\quad + 2^{-\frac12 \ell} \bigl\| \wtilvarphi^{(n)}_\ell(\xi) \bigr\|_{L^2_\xi} \cdot 2^m \cdot 2^{-m} \cdot \sup_{s \, \simeq \, 2^m} \, \bigl\| \varphi_{\geq-100}(|\xi_1|-\sqrt{3}) \hatg(s,\xi_1) \bigr\|_{L^1_{\xi_1}} \bigl\| \jxitwo \pxitwo \hatbarg(s,\xi_2) \bigr\|_{L^1_{\xi_2}} \bigl\| \hatg(s,\xi_3) \bigr\|_{L^\infty_{\xi_3}} \\
  &\quad + 2^{-\frac12 \ell} \bigl\| \wtilvarphi^{(n)}_\ell(\xi) \bigr\|_{L^2_\xi} \cdot 2^m \cdot 2^{-m}  \cdot 2^{2\ell} \bigl\| \varphi_{\leq -\ell+10}(\xi_1-\sqrt{3}) \bigr\|_{L^1_{\xi_1}} \\
  &\qquad \qquad \qquad \times \sup_{s \, \simeq \, 2^m} \, \bigl\| \hatg(s,\xi_1) \bigr\|_{L^\infty_{\xi_1}} \bigl\| \jxitwo \pxitwo \hatbarg(s,\xi_2) \bigr\|_{L^1_{\xi_2}} \bigl\| \hatg(s,\xi_3) \bigr\|_{L^\infty_{\xi_3}} \\
  &\lesssim 2^{-\frac12 \ell} \cdot 2^{-\frac12 \ell} \cdot 2^m \cdot 2^{-m} \cdot 2^\ell \cdot m^3 \varepsilon^3 \lesssim m^3 \varepsilon^3.
 \end{aligned}
\end{equation*}
This finishes the discussion of the case $1 \leq \ell \leq n-1$.

\bigskip 

\noindent \underline{\it Case 4.3: $\ell=0$.}
As at the beginning of Case~4.2, we distinguish how close the frequency $\xi_2$ is to the problematic frequencies $\pm \sqrt{3}$. By a simple $L^\infty_x \times L^2_x \times L^\infty_x$ estimate, we may assume that $||\xi_2|-\sqrt{3}| \lesssim 2^{-\ell_2-100}$ for $\ell_2 \geq 100$. Moreover, without loss of generality, we may assume that $\xi_2 < 0$, i.e., $\xi_2 \approx -\sqrt{3}$. Since $||\xi|-\sqrt{3}| \gtrsim 2^{-100}$, we then have
\begin{equation*}
 |\xi-\xi_2| = |(\xi+\sqrt{3})-(\xi_2+\sqrt{3})| \gtrsim 2^{-100}.
\end{equation*}
Since the second input $\hatbarg(s,\xi_2)$ is already differentiated, ideally, we would just like to integrate by parts in $\xi_1$, but $\pxione \phi_2$ can vanish in this configuration in view of \eqref{equ:phi2_delta_subcase422_pxione_vanishing_discussion} when $\xi_1-\xi_3 = 0$. As in Case~4.2 we therefore further distinguish relative to the size of $|\xi_1-\xi_3| \simeq 2^{-\ell_4}$, and separately treat the subcases (1) $\ell_4 \geq 1000$ and (2) $\ell_4 < 1000$.

\medskip 

\noindent \underline{\it Subcase 4.3.1: $\ell=0$, $100\leq \ell_2 \leq m$, $\ell_4 \geq 1000$.}
Here we consider the weighted energy estimate for the term
\begin{equation} \label{equ:subcase431_phi2_term}
 \calI_{2,m;0\leq\ell_2}^{2,\ell_4\geq 0}(t,\xi) := \int_0^t \tau_m(s) \iint e^{is\phi_2} \frakm(\xi,\xi_1,\xi_2) \, \hatg(s,\xi_1) \jxitwo \pxitwo \hatbarg(s,\xi_2) \hatg(s,\xi_3) \, \ud \xi_1 \, \ud \xi_2 \, \ud s
\end{equation}
with 
\begin{equation*}
 \frakm(\xi,\xi_1,\xi_2) := \varphi_0^{(n)}(\xi) \varphi_{\geq 100}^{(m),-}(\xi_2) \varphi_{\leq -1000}(\xi_1-\xi_3).
\end{equation*}
As in Subcase~4.2.2.1 we want to integrate by parts in time in this configuration. To this end we determine a lower bound on the size of the phase function $\phi_2(\xi,\xi_1,\xi_2)$.
Since $|\xi_1-\xi_3| \lesssim 2^{-1000}$, we have
\begin{equation*}
 \Bigl| \xi_j - \frac12 (\xi-\xi_2) \Bigr| \lesssim 2^{-1000}, \quad j = 1, 3,
\end{equation*}
and thus
\begin{equation} \label{equ:subcase431_phi2_phase_expansion}
 \phi_2(\xi,\xi_1,\xi_2) = -\jxi - \jxitwo + 2 \jap{{\textstyle \frac12} (\xi-\xi_2)} + \calO\bigl( 2^{-1000} \bigr).
\end{equation}
For every fixed $\xi_2 \in \bbR$, the function $p(\xi) := -\jxi -\jxitwo + 2 \jap{{\textstyle \frac12} (\xi-\xi_2)}$, $\xi \in \bbR$, satisfies $p(-\xi_2) = 0$ and $p(\xi) \leq 0$ for all $\xi \in \bbR$ as $p(\xi)$ is monotone increasing for $\xi < -\xi_2$ and monotone decreasing for $\xi > -\xi_2$.
Moreover, near $\xi \approx -\xi_2$, we have 
\begin{equation*}
 p(\xi) = -\frac14 \jxitwo^{-3} (\xi+\xi_2)^2 + \calO((\xi+\xi_2)^3). 
\end{equation*}
Since $|\xi+\xi_2| = |(\xi-\sqrt{3})+(\xi_2+\sqrt{3})| \gtrsim 2^{-100}$, we conclude uniformly for all $|\xi_2+\sqrt{3}| \lesssim 2^{-200}$ that $|p(\xi)| \gtrsim 2^{-200}$ on the support of $\frakm(\xi,\xi_1,\xi_2)$. In view of~\eqref{equ:subcase431_phi2_phase_expansion}, it follows that 
\begin{equation} \label{equ:subcase431_phi2_phase_lower_bound}
 |\phi_2(\xi,\xi_1,\xi_2)| \gtrsim 2^{-200}
\end{equation}
on the support of $\frakm(\xi,\xi_1,\xi_2)$.

Hence, we can integrate by parts in time in~\eqref{equ:subcase431_phi2_term}. Additionally, we decompose the symbol $\frakm(\xi,\xi_1,\xi_2)$ dyadically relative to the size of $|\xi| \simeq 2^k$, $k \geq 0$. We obtain
\begin{equation*}
 \begin{aligned}
  &\calI_{2,m;0\leq\ell_2}^{2,\ell_4\geq 0}(t,\xi) \\
  &= \sum_{k \geq 0} i \int_0^t \tau_m(s) \iint e^{is\phi_2} \frac{1}{\phi_2} \frakm_k(\xi,\xi_1,\xi_2) \, \ps \hatg(s,\xi_1) \jxitwo \pxitwo \hatbarg(s,\xi_2) \hatg(s,\xi_3) \, \ud \xi_1 \, \ud \xi_2 \, \ud s \\
  &\quad + \sum_{k \geq 0} i \int_0^t \tau_m(s) \iint e^{is\phi_2} \frac{1}{\phi_2} \frakm_k(\xi,\xi_1,\xi_2) \, \hatg(s,\xi_1) \jxitwo \pxitwo \ps \hatbarg(s,\xi_2) \hatg(s,\xi_3) \, \ud \xi_1 \, \ud \xi_2 \, \ud s \\
  &\quad + \bigl\{ \text{similar or better terms} \bigr\} \\
  &=: \sum_{k \geq 0} \calI_{2,m;0\leq\ell_2}^{2,\ell_4\geq 0, (a), k}(t,\xi) + \sum_{k \geq 0} \calI_{2,m;0\leq\ell_2}^{2,\ell_4\geq 0, (b), k}(t,\xi) + \bigl\{ \text{similar or better terms} \bigr\}
 \end{aligned}
\end{equation*}
with 
\begin{equation*}
 \frakm_k(\xi,\xi_1,\xi_2) := \varphi_k(\xi) \varphi_0^{(n)}(\xi) \varphi_{\geq 100}^{(m),-}(\xi_2) \varphi_{\leq -1000}(\xi_1-\xi_3) \widetilde{\varphi}_k(\xi_1), \quad k \geq 1,
\end{equation*}
and 
\begin{equation*}
 \frakm_0(\xi,\xi_1,\xi_2) := \varphi_{\leq 0}(\xi) \varphi_0^{(n)}(\xi) \varphi_{\geq 100}^{(m),-}(\xi_2) \varphi_{\leq -1000}(\xi_1-\xi_3) \widetilde{\varphi}_{\leq 0}(\xi_1).
\end{equation*}
Note that we could freely introduce the fattened cut-offs $\widetilde{\varphi}_k(\xi_1)$, respectively $\widetilde{\varphi}_{\leq 0}(\xi_1)$, in this frequency configuration.
In view of \eqref{equ:subcase431_phi2_phase_lower_bound} and 
\begin{equation*}
 \bigl| \pxi^\kappa \pxione^{\kappa_1} \pxitwo^{\kappa_2} \frakm_k(\xi,\xi_1,\xi_2) \bigr| \lesssim 2^{-(\kappa + \kappa_1) k},
\end{equation*}
it follows that
\begin{equation} \label{equ:subcase431_phi2_phase_frakm_bound}
 \biggl\| \calF^{-1}\biggl[ \frac{1}{\phi_2} \frakm_k(\xi,\xi_1,\xi_2) \biggr] \biggr\|_{L^1(\bbR^3)} \lesssim 1.
\end{equation}
Then by Lemma~\ref{lem:frakm_for_delta_three_inputs} with \eqref{equ:subcase431_phi2_phase_frakm_bound} and by the bounds \eqref{equ:g_bound_pxiL1}, \eqref{equ:g_bound_dispersive_est}, \eqref{equ:g_bound_jD_pt_L2}
\begin{equation*}
 \begin{aligned}
  \sum_{k \geq 0} \, \bigl\| \calI_{2,m;0\leq\ell_2}^{2,\ell_4\geq 0, (a), k}(t,\xi) \bigr\|_{L^2_\xi} 
  &\lesssim \sum_{k \geq 0} \, 2^m \cdot \sup_{s \, \simeq \, 2^m} \, \bigl\| \ps \widetilde{P}_k g(s) \bigr\|_{L^2_x} \bigl\| e^{is\jD} \jD x \bar{g}(s) \bigr\|_{L^\infty_x} \bigl\| e^{is\jD} g(s) \bigr\|_{L^\infty_x} \\
  &\lesssim \sum_{k \geq 0} \, 2^m \cdot \sup_{s \, \simeq \, 2^m} \, \bigl\| \ps \widetilde{P}_k g(s) \bigr\|_{L^2_x} \bigl\| \jxitwo \pxitwo \hatbarg(s,\xi_2) \bigr\|_{L^1_{\xi_2}} \bigl\| e^{is\jD} g(s) \bigr\|_{L^\infty_x} \\
  &\lesssim \sum_{k \geq 0} \, 2^m \cdot 2^{-k} 2^{-m} m^2 \varepsilon \cdot m \varepsilon \cdot 2^{-\frac12 m} m \varepsilon \lesssim \varepsilon^3.
 \end{aligned}
\end{equation*}
Similarly, by Lemma~\ref{lem:frakm_for_delta_three_inputs} with \eqref{equ:subcase431_phi2_phase_frakm_bound} and the bounds \eqref{equ:g_bound_dispersive_est}, \eqref{equ:g_bound_dispersive_est_kgain}, \eqref{equ:g_bound_pxi_pt_L2xi}, we have
\begin{equation*}
 \begin{aligned}
  \sum_{k \geq 0} \, \bigl\| \calI_{2,m;0\leq\ell_2}^{2,\ell_4\geq 0, (b), k}(t,\xi) \bigr\|_{L^2_\xi} &\lesssim \sum_{k \geq 0} \, 2^m \cdot \sup_{s \, \simeq \, 2^m} \, \bigl\| e^{it\jD} \widetilde{P}_k g(s) \bigr\|_{L^\infty_x} \bigl\| \jxitwo \pxitwo \ps \hatbarg(s,\xi_2) \bigr\|_{L^2_{\xi_2}} \bigl\| e^{is\jD} g(s) \bigr\|_{L^\infty_x} \\
  &\lesssim \sum_{k \geq 0} \, 2^m \cdot 2^{-\frac12 m} 2^{-\frac12 k} m \varepsilon \cdot m^2 \varepsilon^2 \cdot 2^{-\frac12 m} m \varepsilon \lesssim m^4 \varepsilon^4.
 \end{aligned}
\end{equation*}

\medskip 

\noindent \underline{\it Subcase 4.3.2: $\ell=0$, $100\leq \ell_2 \leq m$, $\ell_4 < 1000$.}
Here we consider the weighted energy estimate for the term
\begin{equation*}
 \calI_{2,m;0 \leq \ell_2}^{2, \ell_4 < 0}(t,\xi) := \int_0^t \tau_m(s) \iint e^{is\phi_2} \frakn(\xi,\xi_1,\xi_2) \, \hatg(s,\xi_1) \jxitwo \pxitwo \hatbarg(s,\xi_2) \hatg(s,\xi_3) \, \ud \xi_1 \, \ud \xi_2 \, \ud s
\end{equation*}
with
\begin{equation*}
 \frakn(\xi,\xi_1,\xi_2) := \varphi_0^{(n)}(\xi) \varphi_{\geq 100}^{(m),-}(\xi_2) \varphi_{>-1000}(\xi_1-\xi_3).
\end{equation*}
Since $\pxione \phi_2$ cannot vanish in this frequency configuration, we integrate by parts in $\xi_1$ to obtain 
\begin{equation*}
 \begin{aligned}
  &\calI_{2,m;0 \leq \ell_2}^{2, \ell_4 < 0}(t,\xi) \\
  &= i \int_0^t \tau_m(s) \frac{1}{s} \iint e^{is\phi_2} \frac{1}{\pxione \phi_2} \frakn(\xi,\xi_1,\xi_2) \, \pxione \hatg(s,\xi_1) \jxitwo \pxitwo \hatbarg(s,\xi_2) \hatg(s,\xi_3) \, \ud \xi_1 \, \ud \xi_2 \, \ud s \\
  &\quad - i \int_0^t \tau_m(s) \frac{1}{s} \iint e^{is\phi_2} \frac{1}{\pxione \phi_2} \frakn(\xi,\xi_1,\xi_2) \, \hatg(s,\xi_1) \jxitwo \pxitwo \hatbarg(s,\xi_2) \pxithree \hatg(s,\xi_3) \, \ud \xi_1 \, \ud \xi_2 \, \ud s \\
  &\quad + i \int_0^t \tau_m(s) \frac{1}{s} \iint e^{is\phi_2} \pxione \biggl( \frac{1}{\pxione \phi_2} \frakn(\xi,\xi_1,\xi_2) \biggr) \, \hatg(s,\xi_1) \jxitwo \pxitwo \hatbarg(s,\xi_2) \hatg(s,\xi_3) \, \ud \xi_1 \, \ud \xi_2 \, \ud s \\
  &=: \calI_{2,m;0 \leq \ell_2}^{2, \ell_4 < 0, (a)}(t,\xi) + \calI_{2,m;0 \leq \ell_2}^{2, \ell_4 < 0, (b)}(t,\xi) + \calI_{2,m;0 \leq \ell_2}^{2, \ell_4 < 0, (c)}(t,\xi).
 \end{aligned}
\end{equation*}
We claim that on the support of $\frakn(\xi,\xi_1,\xi_2)$, the following bounds hold
\begin{equation} \label{equ:subcase432_pxione_phi2_inv_bound}
 \frac{1}{|\pxione \phi_2|} \lesssim \jxione \jxithree \min \bigl\{ \jxione, \jxithree \bigr\}
\end{equation}
and 
\begin{equation} \label{equ:subcase432_pxione_phi2_inv_diffed_bound}
 \biggl| \pxione \biggl( \frac{1}{\pxione \phi_2} \biggr) \biggr| \lesssim \jxi^{-1} \jxione^2 \jxithree^2.
\end{equation}
To see~\eqref{equ:subcase432_pxione_phi2_inv_bound}, first note that $|\xi_1| \gtrsim 2^{-100}$ or $|\xi_3| \gtrsim 2^{-100}$ in view of $\xi-\xi_2 = \xi_1 + \xi_3$ and $|\xi-\xi_2| \gtrsim 2^{-100}$. If $\xi_1 \xi_3 < 0$, then we have 
\begin{equation*}
 |\pxione \phi_2| = \frac{|\xi_1|}{\jxione} + \frac{|\xi_3|}{\jxithree} \gtrsim 2^{-100},
\end{equation*}
which is consistent with~\eqref{equ:subcase432_pxione_phi2_inv_bound}. Instead, if $\xi_1 \xi_3 > 0$, then we use the assumption $|\xi_1-\xi_3| \gtrsim 2^{-1000}$ to infer~\eqref{equ:subcase432_pxione_phi2_inv_bound} from
\begin{equation*}
 \frac{1}{|\pxione \phi_2|} = \frac{1}{|\xi_1-\xi_3|} \frac{\jxione \jxithree (|\xi_1| \jxithree + |\xi_3| \jxione)}{|\xi_1| + |\xi_3|}.
\end{equation*}
For the proof of~\eqref{equ:subcase432_pxione_phi2_inv_diffed_bound} we only discuss the case when $|\xi| \gg 1$ and $\xi_1 \xi_3 > 0$, since the other cases are straightforward. In view of $\xi-\xi_2 = \xi_1+\xi_3$, we must have $|\xi_1| + |\xi_3| \lesssim |\xi|$. If $|\xi_1| \simeq |\xi_3| \simeq |\xi|$, we infer~\eqref{equ:subcase432_pxione_phi2_inv_diffed_bound} from~\eqref{equ:subcase432_pxione_phi2_inv_bound} and $|\pxione^2 \phi_2| \lesssim \jxione^{-3} + \jxithree^{-3}$. If $|\xi_1| \simeq |\xi|$ and $|\xi_3| \ll |\xi_1|$, then $|\xi_1^2-\xi_3^2| \simeq \jxione^2$ and $|\pxione^2 \phi_2|\lesssim \jxithree^{-3}$, whence 
\begin{equation*}
 \biggl| \frac{\pxione^2 \phi_2}{(\pxione \phi_2)^2} \biggr| \lesssim \biggl| \frac{\jxione \jxithree (|\xi_1| \jxithree + |\xi_3| \jxione )}{\xi_1^2-\xi_3^2} \biggr|^2 \frac{1}{\jxithree^3} \lesssim \jxi^{-1} \jxione \jxithree,
\end{equation*}
which is consistent with the asserted bound~\eqref{equ:subcase432_pxione_phi2_inv_diffed_bound}.
The remaining possibility $|\xi_1| \ll |\xi_3|$ and $|\xi_3| \simeq |\xi|$ is analogous.

To estimate the term $\calI_{2,m;0 \leq \ell_2}^{2, \ell_4 < 0, (a)}(t,\xi)$, we now use the bound~\eqref{equ:subcase432_pxione_phi2_inv_bound} along with \eqref{equ:g_bound_pxiL1}, \eqref{equ:g_bound_jD2_L2} to obtain
\begin{equation*}
 \begin{aligned}
  &\bigl\| \calI_{2,m;0 \leq \ell_2}^{2, \ell_4 < 0, (a)}(t,\xi) \bigr\|_{L^2_\xi} \\
  &\quad \lesssim 2^m \cdot 2^{-m} \cdot \sup_{s \, \simeq \, 2^m} \, \bigl\| \jxione \pxione \hatg(s,\xi_1) \bigr\|_{L^1_{\xi_1}} \bigl\| \jxitwo \pxitwo \hatbarg(s,\xi_2) \bigr\|_{L^1_{\xi_2}} \bigl\| \jxithree^2 \hatg(s,\xi_3) \bigr\|_{L^2_{\xi_3}} \lesssim m^2 \varepsilon^3.
 \end{aligned}
\end{equation*}
The weighted energy estimate for the term $\calI_{2,m;0 \leq \ell_2}^{2, \ell_4 < 0, (b)}(t,\xi)$ is analogous. Finally, for the term $\calI_{2,m;0 \leq \ell_2}^{2, \ell_4 < 0, (c)}(t,\xi)$ we invoke the bound~\eqref{equ:subcase432_pxione_phi2_inv_diffed_bound}, noting that $|\pxione \frakn(\xi,\xi_1,\xi_2)| \lesssim 1$, and use the Cauchy-Schwarz inequality in the integration variable $\xi_1$ along with \eqref{equ:g_bound_pxiL1}, \eqref{equ:g_bound_jD2_L2},
\begin{equation*}
 \begin{aligned}
  &\bigl\| \calI_{2,m;0 \leq \ell_2}^{2, \ell_4 < 0, (c)}(t,\xi) \bigr\|_{L^2_\xi} \\
  &\quad \lesssim 2^m \cdot 2^{-m} \cdot \bigl\|\jxi^{-1}\bigr\|_{L^2_\xi} \sup_{s \, \simeq \, 2^m} \, \bigl\| \jxione^2 \hatg(s,\xi_1)\bigr\|_{L^2_{\xi_1}} \bigl\| \jxitwo \pxitwo \hatbarg(s,\xi_2) \bigr\|_{L^1_{\xi_2}} \bigl\| \jxithree^2 \hatg(s,\xi_3) \bigr\|_{L^2_{\xi_3}} \lesssim m \varepsilon^3.
 \end{aligned}
\end{equation*}
This finishes the proof of Proposition~\ref{prop:weighted_energy_est_delta_T2}.
\end{proof}

\subsection{Cubic interactions with a Hilbert-type kernel} \label{subsec:weighted_energy_est_cubic_pv}

In this subsection we turn to the proof of Proposition~\ref{prop:weighted_energy_est_pv_cubic} and establish the weighted energy estimate for the singular cubic interactions $\calC_{\pvdots}(v+\bv)$.
In view of their structure~\eqref{equ:FT_cubic_interactions_pv}, their contribution to the profile
\begin{equation*}
 \int_0^t (2i\jxi)^{-1} e^{-is\jxi} \calF\bigl[ \calC_{\pvdots}\bigl(v(s)+\bar{v}(s)\bigr) \bigr](\xi) \, \ud s
\end{equation*}
is a linear combination of trilinear terms of the form
\begin{equation} \label{equ:cubic_pv_precise_trilinear_terms}  
 \begin{aligned}
  \int_0^t \jxi^{-1} \iint e^{is(-\jxi \pm_1 \jxione \pm_2 \jxitwo \pm_4 \jxifour)} \widehat{g^{\pm_1}_1}(s,\xi_1) &\widehat{g^{\pm_2}_2}(s,\xi_2) \widehat{g^{\pm_4}_4}(s,\xi_4) \\
  &\, \, \times \pvdots \cosech\Bigl(\frac{\pi}{2}\xi_3\Bigr) \, \ud \xi_1 \, \ud \xi_2 \, \ud \xi_3 \, \ud s,
 \end{aligned}
\end{equation}
where 
\begin{equation*}
 \xi_4 := \xi-\xi_1-\xi_2-\xi_3,
\end{equation*}
and where for $j = 1, 2, 4$,
\begin{equation} \label{equ:g_inputs_pv_precise_form}
 \widehat{g^{\pm_j}_j}(s,\xi_j) = \widehat{f^\pm}(s,\xi_j) \quad \text{or} \quad \widehat{g^{\pm_j}_j}(s,\xi_j) = m_{a_j}(\xi_j) \widehat{f^\pm}(s,\xi_j) \quad \text{for some} \quad a_j \in \{0, 4, 5\},
\end{equation}
and crucially
\begin{equation*}
 a_j = 5 \text{ for at least one } j \in \{1,2,4\}.
\end{equation*}
Recall that we use the convention
\begin{equation*}
 \widehat{f^+}(s,\xi) := \hatf(s,\xi), \quad \widehat{f^-}(s,\xi) := \hatbarf(s,\xi).
\end{equation*}
The presence of at least one multiplier $m_5(D)$ on one of the inputs is extremely important for the weighted energy estimates, because $m_5(D)$ exhibits a low-frequency improvement and thus makes one input amenable to improved local decay behavior. Recall from \eqref{equ:def_multipliers_m} that we may write
\begin{equation*}
 m_5(\xi) = - \frac{3\xi}{2(1+\xi^2)(4+\xi^2)} =: (\jxi^{-1} \xi) \, \widetilde{m}_5(\xi).
\end{equation*}

In order to bound the contributions of the terms~\eqref{equ:cubic_pv_precise_trilinear_terms} to the weighted energy estimate~\eqref{equ:weighted_energy_est_pv_cubic} for the singular cubic interactions $\calC_{\pvdots}(v+\bv)$, we only use the bounds from Lemma~\ref{lem:g_bounds_repeated} (sometimes with $\widetilde{m}_5(D)$ in place of $m_5(D)$). It therefore does not matter which one of the four possible types in~\eqref{equ:g_inputs_pv_precise_form} every input in~\eqref{equ:cubic_pv_precise_trilinear_terms} precisely assumes, apart from overall having one input with a low-frequency improvement.
For this reason, it suffices to establish the weighted energy estimates for the following four terms
\begin{equation*}
 \begin{aligned}
  \calF\bigl[\calT_1^{\pvdots}[g](t)\bigr](\xi) :=  \int_0^t \jxi^{-1} \iiint e^{is\Phi_1(\xi,\xi_1,\xi_2,\xi_4)} (\jxione^{-1} \xi_1) \hatg(s,\xi_1) &\hatg(s,\xi_2) \hatg(s,\xi_4) \\
  &\times \pvdots \frac{\hatq(\xi_3)}{\xi_3} \, \ud \xi_1 \, \ud \xi_2 \, \ud \xi_3 \, \ud s, \\
  \calF\bigl[\calT_2^{\pvdots}[g](t)\bigr](\xi) := \int_0^t \jxi^{-1} \iiint e^{is\Phi_2(\xi,\xi_1,\xi_2, \xi_4)} (\jxione^{-1} \xi_1) \hatg(s,\xi_1) &\hatbarg(s,\xi_2) \hatg(s,\xi_4) \\
  &\times \pvdots \frac{\hatq(\xi_3)}{\xi_3} \, \ud \xi_1 \, \ud \xi_2 \, \ud \xi_3 \, \ud s, \\
  \calF\bigl[\calT_3^{\pvdots}[g](t)\bigr](\xi) := \int_0^t \jxi^{-1} \iiint e^{is\Phi_3(\xi,\xi_1,\xi_2, \xi_4)} (\jxione^{-1} \xi_1) \hatg(s,\xi_1) &\hatbarg(s,\xi_2) \hatbarg(s,\xi_4) \\
  &\times \pvdots \frac{\hatq(\xi_3)}{\xi_3} \, \ud \xi_1 \, \ud \xi_2 \, \ud \xi_3 \, \ud s, \\
  \calF\bigl[\calT_4^{\pvdots}[g](t)\bigr](\xi) :=  \int_0^t \jxi^{-1} \iiint e^{is\Phi_4(\xi,\xi_1,\xi_2, \xi_4)} (\jxione^{-1} \xi_1) \hatbarg(s,\xi_1) &\hatbarg(s,\xi_2) \hatbarg(s,\xi_4) \\
  &\times \pvdots \frac{\hatq(\xi_3)}{\xi_3} \, \ud \xi_1 \, \ud \xi_2 \, \ud \xi_3 \, \ud s,
 \end{aligned}
\end{equation*}
with phase functions
\begin{equation*}
 \begin{aligned}
  \Phi_1(\xi,\xi_1,\xi_2, \xi_4) &:= -\jxi + \jxione + \jxitwo + \jxifour, \\
  \Phi_2(\xi,\xi_1,\xi_2, \xi_4) &:= -\jxi + \jxione - \jxitwo + \jxifour, \\
  \Phi_3(\xi,\xi_1,\xi_2, \xi_4) &:= -\jxi + \jxione - \jxitwo - \jxifour, \\
  \Phi_4(\xi,\xi_1,\xi_2, \xi_4) &:= -\jxi - \jxione - \jxitwo - \jxifour.
 \end{aligned}
\end{equation*}
Without loss of generality we placed the low-frequency improvement on the first input. Moreover, we introduced the short-hand notation
\begin{equation*}
 \hatq(\xi) := \xi \cosech\Bigl(\frac{\pi}{2} \xi \Bigr).
\end{equation*}
Note that $\hatq(\xi)$ is a Schwartz function.
The three inputs in the trilinear terms~\eqref{equ:cubic_pv_precise_trilinear_terms} are of course not necessarily all of the same one type in~\eqref{equ:g_inputs_pv_precise_form}. But since apart from the low-frequency improvement of one of the inputs, their fine structure is not relevant for establishing the weighted energy estimates, we decided not to introduce additional cumbersome notation to keep track of this.

Among the terms $\calT_{j}^\pvdots[g]$, $1 \leq j \leq 4$, the weighted energy estimate for $\calT_{2}^\pvdots[g]$ is the most delicate. We therefore provide full details for the treatment of $\calT_{2}^\pvdots[g]$, and leave the analogous weighted energy estimates for the other terms $\calT_{1}^\pvdots[g]$, $\calT_{3}^\pvdots[g]$, and $\calT_{4}^\pvdots[g]$ to the reader.

\begin{proposition} \label{prop:weighted_energy_est_pv_T2}
 Suppose that the assumptions in the statement of Proposition~\ref{prop:main_bootstrap} are in place.
 Let $f(t) = e^{-it\jD} v(t)$ be the profile of the solution $v(t)$ to~\eqref{equ:v_equ_refer_to} and let
 \begin{equation*}
  g(t) := f(t), \quad g(t) = m_a(D) f(t) \quad \text{for some} \quad a \in \{0, 4, 5\}, \quad \text{or} \quad g(t) = \widetilde{m}_5(D) f(t),
 \end{equation*}
 with the multipliers $m_a(D)$ defined in \eqref{equ:def_multipliers_m}.
 Then we have uniformly for all $0 \leq t \leq T$ that
 \begin{equation} \label{equ:weighted_energy_est_pv_T2}
  \sup_{n \geq 1} \, \sup_{0 \leq \ell \leq n} \, 2^{-\frac12 \ell} \tau_n(t) \Bigl\| \varphi_\ell^{(n)}(\xi) \jxi^2 \partial_\xi \calF\bigl[\calT_2^{\pvdots}[g](t)\bigr](\xi) \Bigr\|_{L^2_\xi} \lesssim \bigl( \log(2+t) \bigr)^6 \varepsilon^3.
 \end{equation}
\end{proposition}
\begin{proof}
Fix $0 \leq t \leq T$. Let $n \geq 1$ be an integer such that $t \in \supp(\tau_n)$ and therefore $2^n \simeq t$.
We consider for every integer $1 \leq m \leq n+5$ the time-localized version of $\calT_2^{\pvdots}[g](t)$ given by
\begin{equation*}
 \begin{aligned}
  \calF\bigl[\calT_{2;m}^{\pvdots}[g](t)\bigr](\xi) :=  \int_0^t \tau_m(s) \, \jxi^{-1} \iiint e^{is\Phi_2(\xi,\xi_1,\xi_2,\xi_4)} (\jxione^{-1} \xi_1) &\hatg(s,\xi_1) \hatbarg(s,\xi_2) \hatg(s,\xi_4) \\
  &\times \pvdots \frac{\hatq(\xi_3)}{\xi_3} \, \ud \xi_1 \, \ud \xi_2 \, \ud \xi_3 \, \ud s.
 \end{aligned}
\end{equation*}
In what follows we prove for all $1 \leq m \leq n+5$ that
\begin{equation} \label{equ:weighted_energy_est_pv_calT2m}
 \sup_{0 \leq \ell \leq n} \, 2^{-\frac12 \ell} \Bigl\| \varphi_\ell^{(n)}(\xi) \jxi^2 \partial_\xi \calF\bigl[\calT_{2;m}^{\pvdots}[g](t)\bigr](\xi) \Bigr\|_{L^2_\xi} \lesssim m^5 \varepsilon^3.
\end{equation}
Since clearly $\calT_2^{\pvdots}[g](t) = \sum_{1 \leq m \leq n+5} \calT_{2;m}^{\pvdots}[g](t)$ for $t \simeq 2^n$ and since all implied absolute constants are independent of $0 \leq t \leq T$,
\eqref{equ:weighted_energy_est_pv_calT2m} immediately yields the assertion~\eqref{equ:weighted_energy_est_pv_T2} of Proposition~\ref{prop:weighted_energy_est_pv_T2}.

We now begin with the proof of \eqref{equ:weighted_energy_est_pv_calT2m} by computing
\begin{equation} \label{equ:pv_T2m_pxi_action}
 \begin{aligned}
  &\jxi^2 \pxi \calF\bigl[\calT_{2;m}^{\pvdots}[g](t)\bigr](\xi) \\
  &= \int_0^t \tau_m(s) \iiint i s \jxi (\pxi \Phi_2) e^{is\Phi_2} (\jxione^{-1} \xi_1) \hatg(s,\xi_1) \hatbarg(s,\xi_2) \hatg(s,\xi_4) \, \pvdots \frac{\hatq(\xi_3)}{\xi_3} \, \ud \xi_1 \, \ud \xi_2 \, \ud \xi_3 \, \ud s \\
  &\quad + \int_0^t \tau_m(s) \iiint e^{is\Phi_2} (\jxione^{-1} \xi_1) \hatg(s,\xi_1) \hatbarg(s,\xi_2) \jxi \pxi \bigl( \hatg(s,\xi_4) \bigr) \, \pvdots \frac{\hatq(\xi_3)}{\xi_3} \, \ud \xi_1 \, \ud \xi_2 \, \ud \xi_3 \, \ud s \\
  &\quad + \bigl\{\text{lower order terms}\bigr\},
 \end{aligned}
\end{equation}
where a lower order term arises when $\jxi^2 \pxi$ falls onto the weight $\jxi^{-1}$. We do not discuss the straightforward estimates for this and similar lower order terms.

Next, we recast~\eqref{equ:pv_T2m_pxi_action} into a better form to carry out the weighted energy estimates.
To this end we basically integrate by parts in the first term to shift the derivative on the other inputs. However, we have to do this in such a manner to avoid hitting the Hilbert-type kernel with a derivative, in other words, we cannot integrate by parts in $\xi_3$.
Inserting the identity 
\begin{equation*}
 \jxi \pxi \Phi_2 + \jxione \pxione \Phi_2 - \jxitwo \pxitwo \Phi_2 = -\xi_3 - (\jxifour^{-1} \xi_4) \, \Phi_2,
\end{equation*} 
for the first term on the right-hand side of~\eqref{equ:pv_T2m_pxi_action}, we rewrite~\eqref{equ:pv_T2m_pxi_action} as
\begin{equation} \label{equ:pv_T2m_pxi_action2}
 \begin{aligned}
  &\jxi^2 \pxi \calF\bigl[\calT_{2;m}^{\pvdots}[g](t)\bigr](\xi) \\
  &= -i \int_0^t \tau_m(s) \cdot s \iiint e^{is\Phi_2} (\jxione^{-1} \xi_1) \hatg(s,\xi_1) \hatbarg(s,\xi_2) \hatg(s,\xi_4) \, \hatq(\xi_3) \, \ud \xi_1 \, \ud \xi_2 \, \ud \xi_3 \, \ud s \\
  &\quad - \int_0^t \tau_m(s) \cdot s \iiint i \Phi_2 e^{is\Phi_2} (\jxione^{-1} \xi_1) \hatg(s,\xi_1) \hatbarg(s,\xi_2) (\jxifour^{-1} \xi_4) \hatg(s,\xi_4) \\
  &\qquad \qquad \qquad \qquad \qquad \qquad \qquad \qquad \qquad \qquad \qquad \qquad \times \pvdots \frac{\hatq(\xi_3)}{\xi_3} \, \ud \xi_1 \, \ud \xi_2 \, \ud \xi_3 \, \ud s \\
  &\quad + \int_0^t \tau_m(s) \iiint i s \bigl( -\jxione \pxione \Phi_2 + \jxitwo \pxitwo \Phi_2 \bigr) e^{is\Phi_2} (\jxione^{-1} \xi_1) \hatg(s,\xi_1) \hatbarg(s,\xi_2) \hatg(s,\xi_4) \\
  &\qquad \qquad \qquad \qquad \qquad \qquad \qquad \qquad \qquad \qquad \qquad \qquad \times \pvdots \frac{\hatq(\xi_3)}{\xi_3} \, \ud \xi_1 \, \ud \xi_2 \, \ud \xi_3 \, \ud s \\
  &\quad + \int_0^t \tau_m(s) \iiint e^{is\Phi_2} (\jxione^{-1} \xi_1) \hatg(s,\xi_1) \hatbarg(s,\xi_2) \jxi \pxi \bigl( \hatg(s,\xi_4) \bigr) \, \pvdots \frac{\hatq(\xi_3)}{\xi_3} \, \ud \xi_1 \, \ud \xi_2 \, \ud \xi_3 \, \ud s \\
  &\quad + \bigl\{\text{lower order terms}\bigr\}.
 \end{aligned}
\end{equation}
Observe that the first term on the right-hand side of \eqref{equ:pv_T2m_pxi_action2} resulted from inserting $-\xi_3$, which eliminated the Hilbert-type kernel. Since $\hatq(\xi_3)$ is a Schwartz function, the first term on the right-hand side of~\eqref{equ:pv_T2m_pxi_action2} is now spatially localized on the physical side. In order to derive an acceptable weighted energy estimate for this term, the low-frequency improvement $\jxione^{-1} \xi_1$ of the first input $\hatg(s,\xi_1)$ is crucial. It leads to improved local decay for $(\jD^{-1} D) e^{is\jD} g(s)$, which can be exploited thanks to the spatial localization of $q(x)$. We refer to Step~3 below for the details. This is the only place, where the low-frequency improvement of one of the inputs is relevant.

In the second term on the right-hand side of~\eqref{equ:pv_T2m_pxi_action2}, we can integrate by parts in time $s$,
\begin{equation*} 
 \begin{aligned}
  &- \int_0^t \tau_m(s) \cdot s \iiint i \Phi_2 e^{is\Phi_2} (\jxione^{-1} \xi_1) \hatg(s,\xi_1) \hatbarg(s,\xi_2) (\jxifour^{-1} \xi_4) \hatg(s,\xi_4) \, \pvdots \frac{\hatq(\xi_3)}{\xi_3} \, \ud \xi_1 \, \ud \xi_2 \, \ud \xi_3 \, \ud s \\
  &= \int_0^t \tau_m(s) \cdot s \iiint e^{is\Phi_2} \ps \bigl( (\jxione^{-1} \xi_1) \hatg(s,\xi_1) \hatbarg(s,\xi_2) (\jxifour^{-1} \xi_4) \hatg(s,\xi_4) \bigr) \, \pvdots \frac{\hatq(\xi_3)}{\xi_3} \, \ud \xi_1 \, \ud \xi_2 \, \ud \xi_3 \, \ud s \\
  &\quad + \int_0^t \ps \bigl( \tau_m(s) \cdot s \bigr) \iiint e^{is\Phi_2} (\jxione^{-1} \xi_1) \hatg(s,\xi_1) \hatbarg(s,\xi_2) (\jxifour^{-1} \xi_4) \hatg(s,\xi_4) \, \pvdots \frac{\hatq(\xi_3)}{\xi_3} \, \ud \xi_1 \, \ud \xi_2 \, \ud \xi_3 \, \ud s \\
  &\quad - \biggl[ \tau_m(s) \cdot s \iiint e^{is\Phi_2} (\jxione^{-1} \xi_1) \hatg(s,\xi_1) \hatbarg(s,\xi_2) (\jxifour^{-1} \xi_4) \hatg(s,\xi_4) \, \pvdots \frac{\hatq(\xi_3)}{\xi_3} \, \ud \xi_1 \, \ud \xi_2 \, \ud \xi_3 \, \ud s \biggr]_{s=0}^{s=t}.
 \end{aligned}
\end{equation*}
To rewrite the third term on the right-hand side of~\eqref{equ:pv_T2m_pxi_action2} we integrate by parts in $\xi_1$ and $\xi_2$ using the identity
\begin{equation*}
 i s \bigl( -\jxione \pxione \Phi_2 + \jxitwo \pxitwo \Phi_2 \bigr) e^{is\Phi_2} = \bigl( -\jxione \pxione + \jxitwo \pxitwo \bigr) e^{is\Phi_2}.
\end{equation*}
We find
\begin{equation} \label{equ:pv_T2m_pxi_action4}
 \begin{aligned}
  &\int_0^t \tau_m(s) \iiint i s \bigl( -\jxione \pxione \Phi_2 + \jxitwo \pxitwo \Phi_2 \bigr) e^{is\Phi_2} (\jxione^{-1} \xi_1) \hatg(s,\xi_1) \hatbarg(s,\xi_2) \hatg(s,\xi_4) \\
  &\qquad \qquad \qquad \qquad \qquad \qquad \qquad \qquad \qquad \qquad \qquad \qquad \times \pvdots \frac{\hatq(\xi_3)}{\xi_3} \, \ud \xi_1 \, \ud \xi_2 \, \ud \xi_3 \, \ud s \\
  &= \int_0^t \tau_m(s) \iiint e^{is\Phi_2} \, \xi_1 (\pxione \hatg)(s,\xi_1) \hatbarg(s,\xi_2) \hatg(s,\xi_4) \, \pvdots \frac{\hatq(\xi_3)}{\xi_3} \, \ud \xi_1 \, \ud \xi_2 \, \ud \xi_3 \, \ud s \\
  &\quad - \int_0^t \tau_m(s) \iiint e^{is\Phi_2} \, (\jxione^{-1} \xi_1) \hatg(s,\xi_1) \jxitwo (\pxitwo \hatbarg)(s,\xi_2) \hatg(s,\xi_4) \, \pvdots \frac{\hatq(\xi_3)}{\xi_3} \, \ud \xi_1 \, \ud \xi_2 \, \ud \xi_3 \, \ud s \\
  &\quad + \int_0^t \tau_m(s) \iiint e^{is\Phi_2} (\jxione^{-1} \xi_1) \hatg(s,\xi_1) \hatbarg(s,\xi_2) \bigl( \jxione \pxione - \jxitwo \pxitwo \bigr) \hatg(s,\xi_4) \\
  &\qquad \qquad \qquad \qquad \qquad \qquad \qquad \qquad \qquad \qquad \qquad \qquad \times \pvdots \frac{\hatq(\xi_3)}{\xi_3} \, \ud \xi_1 \, \ud \xi_2 \, \ud \xi_3 \, \ud s \\
  &\quad + \bigl\{\text{lower order terms}\bigr\}.
 \end{aligned}
\end{equation}
The lower order terms arose from $\jxione \pxione$ falling onto $\jxione^{-1} \xi_1$ and from
\begin{equation*}
 \bigl( -\jxione \pxione + \jxitwo \pxitwo \bigr)^\ast = \jxione \pxione - \jxitwo \pxitwo + \jxione^{-1} \xi_1 - \jxitwo^{-1} \xi_2.
\end{equation*}
Now inserting~\eqref{equ:pv_T2m_pxi_action4} back into~\eqref{equ:pv_T2m_pxi_action2}, and recalling that $\xi_4 = \xi-\xi_1-\xi_2-\xi_3$, the fourth term on the right-hand side of~\eqref{equ:pv_T2m_pxi_action2} and the fourth term on the right-hand side of~\eqref{equ:pv_T2m_pxi_action4} combine to
\begin{equation} \label{equ:pv_T2m_pxi_action5}
 \begin{aligned}
  &\int_0^t \tau_m(s) \iiint e^{is\Phi_2} (\jxione^{-1} \xi_1) \hatg(s,\xi_1) \hatbarg(s,\xi_2) \bigl( \jxi \pxi + \jxione \pxione - \jxitwo \pxitwo \bigr) \hatg(s,\xi_4) \\
  &\qquad \qquad \qquad \qquad \qquad \qquad \qquad \qquad \qquad \qquad \qquad \qquad \times \pvdots \frac{\hatq(\xi_3)}{\xi_3} \, \ud \xi_1 \, \ud \xi_2 \, \ud \xi_3 \, \ud s \\
  &= \int_0^t \tau_m(s) \iiint e^{is\Phi_2} (\jxione^{-1} \xi_1) \hatg(s,\xi_1) \hatbarg(s,\xi_2) \bigl( \jxi - \jxione + \jxitwo \bigr) (\pxifour \hatg)(s,\xi_4) \\
  &\qquad \qquad \qquad \qquad \qquad \qquad \qquad \qquad \qquad \qquad \qquad \qquad \times \pvdots \frac{\hatq(\xi_3)}{\xi_3} \, \ud \xi_1 \, \ud \xi_2 \, \ud \xi_3 \, \ud s \\  
  &= \int_0^t \tau_m(s) \iiint e^{is\Phi_2} (\jxione^{-1} \xi_1) \hatg(s,\xi_1) \hatbarg(s,\xi_2) \jxifour (\pxifour \hatg)(s,\xi_4) \, \pvdots \frac{\hatq(\xi_3)}{\xi_3} \, \ud \xi_1 \, \ud \xi_2 \, \ud \xi_3 \, \ud s \\  
  &\quad - \int_0^t \tau_m(s) \iiint \Phi_2 e^{is\Phi_2} (\jxione^{-1} \xi_1) \hatg(s,\xi_1) \hatbarg(s,\xi_2) (\pxifour \hatg)(s,\xi_4) \, \pvdots \frac{\hatq(\xi_3)}{\xi_3} \, \ud \xi_1 \, \ud \xi_2 \, \ud \xi_3 \, \ud s.
 \end{aligned}
\end{equation}
Finally, in the second term on the right-hand side of~\eqref{equ:pv_T2m_pxi_action5} we integrate by parts in time again,
\begin{equation*}
 \begin{aligned}
  &- \int_0^t \tau_m(s) \iiint \Phi_2 e^{is\Phi_2} (\jxione^{-1} \xi_1) \hatg(s,\xi_1) \hatbarg(s,\xi_2) (\pxifour \hatg)(s,\xi_4) \, \pvdots \frac{\hatq(\xi_3)}{\xi_3} \, \ud \xi_1 \, \ud \xi_2 \, \ud \xi_3 \, \ud s \\
  &= - i \int_0^t \tau_m(s) \iiint e^{is\Phi_2} \ps \bigl( (\jxione^{-1} \xi_1) \hatg(s,\xi_1) \hatbarg(s,\xi_2) (\pxifour \hatg)(s,\xi_4) \bigr) \, \pvdots \frac{\hatq(\xi_3)}{\xi_3} \, \ud \xi_1 \, \ud \xi_2 \, \ud \xi_3 \, \ud s \\
  &\quad - i \int_0^t \ps \bigl( \tau_m(s) \bigr) \iiint e^{is\Phi_2} (\jxione^{-1} \xi_1) \hatg(s,\xi_1) \hatbarg(s,\xi_2) (\pxifour \hatg)(s,\xi_4) \, \pvdots \frac{\hatq(\xi_3)}{\xi_3} \, \ud \xi_1 \, \ud \xi_2 \, \ud \xi_3 \, \ud s \\
  &\quad + \biggl[ i \tau_m(s) \iiint e^{is\Phi_2} (\jxione^{-1} \xi_1) \hatg(s,\xi_1) \hatbarg(s,\xi_2) (\pxifour \hatg)(s,\xi_4) \, \pvdots \frac{\hatq(\xi_3)}{\xi_3} \, \ud \xi_1 \, \ud \xi_2 \, \ud \xi_3 \, \biggr]_{s=0}^{s=t}.
 \end{aligned}
\end{equation*}
Combining the preceding computations, we find that
\begin{equation} \label{equ:pv_T2m_pxi_decomposition}
 \begin{aligned}
  \jxi^2 \pxi \calF\bigl[\calT_{2;m}^{\pvdots}[g](t)\bigr](\xi) &= \calI_{2,m}^{\pvdots, 1}(t,\xi) + \calI_{2,m}^{\pvdots, 2}(t,\xi) + \calI_{2,m}^{\pvdots, 4}(t,\xi) + \calL_{2,m}^{\pvdots}(t,\xi) \\
  &\quad \quad + \calR_{2,m}^{\pvdots, 1}(t,\xi) + \calR_{2,m}^{\pvdots, 2}(t,\xi) + \bigl\{\text{lower order terms}\bigr\},
 \end{aligned}
\end{equation}
where the main terms are given by
\begin{equation*}
 \begin{aligned}
  \calI_{2,m}^{\pvdots, 1}(t,\xi) &:= \int_0^t \tau_m(s) \iiint e^{is\Phi_2} \, \xi_1 (\pxione \hatg)(s,\xi_1) \hatbarg(s,\xi_2) \hatg(s,\xi_4) \, \pvdots \frac{\hatq(\xi_3)}{\xi_3} \, \ud \xi_1 \, \ud \xi_2 \, \ud \xi_3 \, \ud s, \\
  \calI_{2,m}^{\pvdots, 2}(t,\xi) &:= - \int_0^t \tau_m(s) \iiint e^{is\Phi_2} \, (\jxione^{-1} \xi_1) \hatg(s,\xi_1) \jxitwo (\pxitwo \hatbarg)(s,\xi_2) \hatg(s,\xi_4) \, \pvdots \frac{\hatq(\xi_3)}{\xi_3} \, \ud \xi_1 \, \ud \xi_2 \, \ud \xi_3 \, \ud s, \\
  \calI_{2,m}^{\pvdots, 4}(t,\xi) &:= \int_0^t \tau_m(s) \iiint e^{is\Phi_2} (\jxione^{-1} \xi_1) \hatg(s,\xi_1) \hatbarg(s,\xi_2) \jxifour (\pxifour \hatg)(s,\xi_4) \, \pvdots \frac{\hatq(\xi_3)}{\xi_3} \, \ud \xi_1 \, \ud \xi_2 \, \ud \xi_3 \, \ud s.
 \end{aligned}
\end{equation*}
Moreover, we obtained the spatially localized term
\begin{equation*}
 \begin{aligned}
  \calL_{2,m}^{\pvdots}(t,\xi) := -i \int_0^t \tau_m(s) \cdot s \iiint e^{is\Phi_2} (\jxione^{-1} \xi_1) \hatg(s,\xi_1) \hatbarg(s,\xi_2) \hatg(s,\xi_4) \, \hatq(\xi_3) \, \ud \xi_1 \, \ud \xi_2 \, \ud \xi_3 \, \ud s,
 \end{aligned}
\end{equation*}
and integrating by parts in time led to the following remainder terms 
\begin{equation*}
 \begin{aligned}
  &\calR_{2,m}^{\pvdots, 1}(t,\xi) \\
  &:= \int_0^t \tau_m(s) \cdot s \iiint e^{is\Phi_2} \ps \bigl( (\jxione^{-1} \xi_1) \hatg(s,\xi_1) \hatbarg(s,\xi_2) (\jxifour^{-1} \xi_4) \hatg(s,\xi_4) \bigr) \, \pvdots \frac{\hatq(\xi_3)}{\xi_3} \, \ud \xi_1 \, \ud \xi_2 \, \ud \xi_3 \, \ud s \\
  &\quad + \int_0^t \ps \bigl( \tau_m(s) \cdot s \bigr) \iiint e^{is\Phi_2} (\jxione^{-1} \xi_1) \hatg(s,\xi_1) \hatbarg(s,\xi_2) (\jxifour^{-1} \xi_4) \hatg(s,\xi_4) \, \pvdots \frac{\hatq(\xi_3)}{\xi_3} \, \ud \xi_1 \, \ud \xi_2 \, \ud \xi_3 \, \ud s \\
  &\quad - \biggl[ \tau_m(s) \cdot s \iiint e^{is\Phi_2} (\jxione^{-1} \xi_1) \hatg(s,\xi_1) \hatbarg(s,\xi_2) (\jxifour^{-1} \xi_4) \hatg(s,\xi_4) \, \pvdots \frac{\hatq(\xi_3)}{\xi_3} \, \ud \xi_1 \, \ud \xi_2 \, \ud \xi_3 \, \ud s \, \biggr]_{s=0}^{s=t} \\
  &=: \calR_{2,m}^{\pvdots, 1, 1}(t,\xi) + \calR_{2,m}^{\pvdots, 1, 2}(t,\xi) + \calR_{2,m}^{\pvdots, 1, 3}(t,\xi).
 \end{aligned}
\end{equation*}
and
\begin{equation*}
 \begin{aligned}
  &\calR_{2,m}^{\pvdots, 2}(t,\xi) \\
  &:= - i \int_0^t \tau_m(s) \iiint e^{is\Phi_2} \ps \bigl( (\jxione^{-1} \xi_1) \hatg(s,\xi_1) \hatbarg(s,\xi_2) (\pxifour \hatg)(s,\xi_4) \bigr) \, \pvdots \frac{\hatq(\xi_3)}{\xi_3} \, \ud \xi_1 \, \ud \xi_2 \, \ud \xi_3 \, \ud s \\
  &\quad \, \, - i \int_0^t \ps \bigl( \tau_m(s) \bigr) \iiint e^{is\Phi_2} (\jxione^{-1} \xi_1) \hatg(s,\xi_1) \hatbarg(s,\xi_2) (\pxifour \hatg)(s,\xi_4) \, \pvdots \frac{\hatq(\xi_3)}{\xi_3} \, \ud \xi_1 \, \ud \xi_2 \, \ud \xi_3 \, \ud s \\
  &\quad \, \, + \biggl[ i \tau_m(s) \iiint e^{is\Phi_2} (\jxione^{-1} \xi_1) \hatg(s,\xi_1) \hatbarg(s,\xi_2) (\pxifour \hatg)(s,\xi_4) \, \pvdots \frac{\hatq(\xi_3)}{\xi_3} \, \ud \xi_1 \, \ud \xi_2 \, \ud \xi_3 \, \biggr]_{s=0}^{s=t} \\
  &=: \calR_{2,m}^{\pvdots, 2, 1}(t,\xi) + \calR_{2,m}^{\pvdots, 2, 2}(t,\xi) + \calR_{2,m}^{\pvdots, 2, 3}(t,\xi).
 \end{aligned}
\end{equation*}

Next, we establish the weighted energy estimates for all terms on the right-hand side of~\eqref{equ:pv_T2m_pxi_decomposition}. The main work again goes into the treatment of the terms $\calI_{2,m}^{\pvdots, j}(t,\xi)$, $1 \leq j \leq 3$, while the terms $\calR_{2,m}^{\pvdots, 1}(t,\xi)$ and $\calR_{2,m}^{\pvdots, 2}(t,\xi)$ are milder remainder terms. The weighted energy estimate for the spatially localized term $\calL_{2,m}^{\pvdots}(t,\xi)$ is also relatively straightforward thanks to the crucial low-frequency improvement of the first input and the resulting improved local decay.
We point out that the estimates for the two terms $\calI_{2,m}^{\pvdots, 1}(t,\xi)$ and $\calI_{2,m}^{\pvdots, 3}(t,\xi)$ are essentially identical, because the low-frequency improvement of one of the inputs is not relevant. We therefore only provide the details for the term $\calI_{2,m}^{\pvdots, 1}(t,\xi)$ in Step~4 below.

We also recall that 
\begin{equation*}
 \pvdots \frac{\hatq(\xi_3)}{\xi_3} = i \sqrt{\frac{2}{\pi}} \calF\bigl[ \tanh(\cdot) \bigr](\xi_3).
\end{equation*}
In what follows, we use the notation $\sup_{s \, \simeq \, 2^m}$ in the sense that $s \simeq 2^m$ with $s \leq T$.

\medskip 

\noindent {\bf Step 1: Weighted energy estimates for the terms in $\calR_{2,m}^{\pvdots, 1}(t,\xi)$.}
For the term $\calR_{2,m}^{\pvdots, 1, 1}(t,\xi)$ we place the input onto which $\ps$ falls into $L^2_x$, and we place all other inputs into $L^\infty_x$.
Using the bounds \eqref{equ:g_bound_dispersive_est}, \eqref{equ:g_bound_jD_pt_L2}, we obtain for $1 \leq m \leq n+5$,
\begin{equation*}
 \begin{aligned}
  \bigl\| \calR_{2,m}^{\pvdots, 1, 1}(t,\xi) \bigr\|_{L^2_\xi} &\lesssim 2^{2m} \cdot \sup_{s \, \simeq \, 2^m} \, \|\ps g(s)\|_{L^2_x} \bigl\|e^{is\jD} g(s)\bigr\|_{L^\infty_x}^2 \bigl\|\tanh(x)\bigr\|_{L^\infty_x} \\
  &\lesssim 2^{2m} \cdot 2^{-m} m^2 \varepsilon^2 \cdot \bigl( 2^{-\frac12 m} m \varepsilon \bigr)^2 \lesssim m^4 \varepsilon^4.
 \end{aligned}
\end{equation*}
Similarly, we use an $L^2_x \times L^\infty_x \times L^\infty_x \times L^\infty_x$ estimate and \eqref{equ:g_bound_dispersive_est}, \eqref{equ:g_bound_jD2_L2} to bound $\calR_{2,m}^{\pvdots, 1, 2}(t,\xi)$ by
\begin{equation*}
 \begin{aligned}
  \bigl\| \calR_{2,m}^{\pvdots, 1, 2}(t,\xi) \bigr\|_{L^2_\xi} &\lesssim 2^{m} \cdot \sup_{s \, \simeq \, 2^m} \, \|g(s)\|_{L^2_x} \bigl\|e^{is\jD} g(s)\bigr\|_{L^\infty_x}^2 \bigl\|\tanh(x)\bigr\|_{L^\infty_x} \\
  &\lesssim 2^m \cdot \varepsilon \cdot \bigl( 2^{-\frac12 m} m \varepsilon \bigr)^2 \lesssim m^2 \varepsilon^3.
 \end{aligned}
\end{equation*}
The bound for the last term $\calR_{2,m}^{\pvdots, 1, 3}(t,\xi)$ is analogous.
Thus, we have established the stronger energy estimate for $1 \leq m \leq n+5$,
\begin{equation*}
 \bigl\| \calR_{2,m}^{\pvdots, 1}(t,\xi) \bigr\|_{L^2_\xi} \lesssim m^4 \varepsilon^3.
\end{equation*}

\medskip 

\noindent {\bf Step 2: Weighted energy estimates for the terms in $\calR_{2,m}^{\pvdots, 2}(t,\xi)$.}
We again establish a stronger energy estimate, namely that for all $1 \leq m \leq n+5$,
\begin{equation*}
 \bigl\| \calR_{2,m}^{\pvdots, 2}(t,\xi) \bigr\|_{L^2_\xi} \lesssim m^4 \varepsilon^3.
\end{equation*}
For the term $\calR_{2,m}^{\pvdots, 2, 1}(t,\xi)$
we use the bounds \eqref{equ:g_bound_dispersive_est}, \eqref{equ:g_bound_pxi_crude}, \eqref{equ:g_bound_jD_pt_L2}, \eqref{equ:g_bound_pxi_pt_L2xi} along with Sobolev embedding to obtain
\begin{equation*}
 \begin{aligned}
  \bigl\| \calR_{2,m}^{\pvdots, 2, 1}(t,\xi) \bigr\|_{L^2_\xi} &\lesssim 2^{m} \cdot \sup_{s \, \simeq \, 2^m} \, \bigl\| e^{is\jD} \ps g(s) \bigr\|_{L^\infty_x} \bigl\| e^{is\jD} g(s) \bigr\|_{L^\infty_x} \bigl\| \pxifour \hatg(s,\xi_4) \bigr\|_{L^2_{\xi_4}} \bigl\| \tanh(x) \bigr\|_{L^\infty_x} \\
  &\quad + 2^{m} \cdot \sup_{s \, \simeq \, 2^m} \, \bigl\| e^{is\jD} g(s) \bigr\|_{L^\infty_x}^2 \bigl\| \pxi \ps \hatg(s,\xi) \bigr\|_{L^2_\xi} \bigl\| \tanh(x) \bigr\|_{L^\infty_x} \\
  &\lesssim 2^m \cdot 2^{-m} m^2 \varepsilon^2 \cdot 2^{-\frac12 m} m \varepsilon \cdot 2^{\frac12 m} \varepsilon + 2^m \cdot 2^{-m} m^2 \varepsilon^2 \cdot m^2 \varepsilon^2 \lesssim m^4 \varepsilon^4.
 \end{aligned}
\end{equation*}
For the term $\calR_{2,m}^{\pvdots, 2, 2}(t,\xi)$ we use an $L^\infty_x \times L^\infty_x \times L^2_x \times L^\infty_x$ estimate and \eqref{equ:g_bound_dispersive_est}, \eqref{equ:g_bound_pxi_crude} to infer
\begin{equation*}
 \begin{aligned}
  \bigl\| \calR_{2,m}^{\pvdots, 2, 1}(t,\xi) \bigr\|_{L^2_\xi} &\lesssim 2^{m} \cdot 2^{-m} \sup_{s \, \simeq \, 2^m} \, \bigl\| e^{is\jD} g(s) \bigr\|_{L^\infty_x}^2 \| \pxifour \hatg(s,\xi_4) \|_{L^2_{\xi_4}} \bigl\| \tanh(x) \bigr\|_{L^\infty_x} \\
  &\lesssim 2^{m} \cdot 2^{-m} \cdot 2^{-m} m^2 \varepsilon^2 \cdot 2^{\frac12 m} \varepsilon \lesssim m^2 \varepsilon^3.
 \end{aligned}
\end{equation*}
The estimate for the last term $ \calR_{2,m}^{\pvdots, 2, 3}(t,\xi)$ is analogous.

\medskip 

\noindent {\bf Step 3: Weighted energy estimate for the term $\calL_{2,m}^{\pvdots}(t,\xi)$.}
Here we crucially exploit the spatial localization of the Schwartz function $q(x)$ together with the low frequency improvement of the first input, which allows us to access the improved local \eqref{equ:g_bound_improved_local_decay} of $e^{is\jD} (\jD^{-1} D) g(t)$.
By \eqref{equ:g_bound_dispersive_est} and \eqref{equ:g_bound_improved_local_decay}, we obtain the stronger energy estimate for all $1 \leq m \leq n+5$,
\begin{equation*}
 \begin{aligned}
 \bigl\| \calL_{2,m}^{\pvdots}(t,\xi) \bigr\|_{L^2_\xi} &\lesssim 2^m \cdot 2^m \sup_{s \, \simeq \, 2^m} \, \bigl\| q(x) e^{is\jD} (\jD^{-1} D) g(t) \bigr\|_{L^2_x} \bigl\| e^{is\jD} g(s) \bigr\|_{L^\infty_x}^2 \\
 &\lesssim 2^m \cdot 2^m \cdot 2^{-m} m \varepsilon \cdot 2^{-m} m^2 \varepsilon^2 \lesssim m^3 \varepsilon^3.
 \end{aligned}
\end{equation*}

\medskip 

\noindent {\bf Step 4: Weighted energy estimate for the term $\calI_{2,m}^{\pvdots, 1}(t,\xi)$.}
For the term $\calI_{2,m}^{\pvdots, 1}(t,\xi)$ we aim to show for all $1 \leq m \leq n+5$ that
\begin{equation*}
 \sup_{0 \leq \ell \leq n} \, 2^{-\frac12 \ell} \bigl\| \varphi_\ell^{(n)}(\xi) \calI_{2,m}^{\pvdots, 1}(t,\xi) \bigr\|_{L^2_\xi} \lesssim m^5 \varepsilon^3.
\end{equation*}
As usual, we distinguish the cases $\ell = n$, $1 \leq \ell \leq n-1$, and $\ell = 0$.

\medskip 
\noindent \underline{\it Case 4.1: $\ell = n$.}
Using H\"older's inequality in the frequency variable $\xi$, we find
\begin{equation*}
 \begin{aligned}
  &2^{-\frac12 n} \bigl\| \varphi_n^{(n)}(\xi) \calI_{2,m}^{\pvdots, 1}(t,\xi) \bigr\|_{L^2_\xi} \\
  &\lesssim 2^{-\frac12 n} \bigl\| \varphi_n^{(n)}(\xi) \bigr\|_{L^2_\xi} \cdot 2^m \cdot \sup_{s \, \simeq \, 2^m} \, \biggl\| \iiint e^{is\Phi_2} \, \xi_1 (\pxione \hatg)(s,\xi_1) \hatbarg(s,\xi_2) \hatg(s,\xi_4) \, \pvdots \frac{\hatq(\xi_3)}{\xi_3} \, \ud \xi_1 \, \ud \xi_2 \, \ud \xi_3 \biggr\|_{L^\infty_\xi} \\
  &\lesssim 2^{-\frac12 n} \cdot 2^{-\frac12 n} \cdot 2^m \cdot \sup_{s \, \simeq \, 2^m} \, \biggl\| \calF^{-1} \biggl[ \iiint e^{is\Phi_2} \, \xi_1 (\pxione \hatg)(s,\xi_1) \hatbarg(s,\xi_2) \hatg(s,\xi_4) \, \pvdots \frac{\hatq(\xi_3)}{\xi_3} \, \ud \xi_1 \, \ud \xi_2 \, \ud \xi_3 \biggr] \biggr\|_{L^1_x}.
 \end{aligned}
\end{equation*}
By an $L^\infty_x \times L^2_x \times L^2_x \times L^\infty_x$ estimate, using the bounds \eqref{equ:g_bound_pxiL1}, \eqref{equ:g_bound_jD2_L2}, we then obtain for $1 \leq m \leq n+5$,
\begin{equation*}
 \begin{aligned}
  2^{-\frac12 n} \bigl\| \varphi_n^{(n)}(\xi) \calI_{2,m}^{\pvdots, 1}(t,\xi) \bigr\|_{L^2_\xi}
  &\lesssim 2^{-n} \cdot 2^m \cdot \sup_{s \, \simeq \, 2^m} \, \bigl\| e^{is\jD} D \bigl( x g(s) \bigr) \bigr\|_{L^\infty_x} \|g(s)\|_{L^2_x}^2 \bigl\| \tanh(x) \bigr\|_{L^\infty_x} \\
  &\lesssim 2^{-n} \cdot 2^m \cdot \sup_{s \, \simeq \, 2^m} \, \bigl\| \jxione \pxione \hatg(s,\xi_1) \bigr\|_{L^1_{\xi_1}} \|g(s)\|_{L^2_x}^2 \\
  &\lesssim 2^{-n} \cdot 2^m \cdot m\varepsilon \cdot \varepsilon^2 \lesssim m \varepsilon^3.
 \end{aligned}
\end{equation*}

\medskip 
\noindent \underline{\it Case 4.2: $1 \leq \ell \leq n-1$.}
As usual we insert a smooth partition of unity to distinguish how close the frequency variable $\xi_1$ is to the problematic frequencies $\pm \sqrt{3}$. Without loss of generality, we may assume that the output frequency is positive, i.e., $|\xi-\sqrt{3}| \simeq 2^{-\ell-100}$. Then it is more delicate to treat the case when $\xi_1 \approx \sqrt{3}$ than when $\xi_1 \approx - \sqrt{3}$, because the difference $|\xi-\xi_1|$ can become small in the former case. So we only consider the former case, and we right away jump to the more difficult configuration $|\xi_1 - \sqrt{3}| \ll |\xi-\sqrt{3}|$ when $\xi_1$ is much closer to $\sqrt{3}$ than the output frequency $\xi$ is.
We may also assume that $\ell +10 \leq m$.
Correspondingly, we now carry out the energy estimate for the term
\begin{equation*}
\begin{aligned}
 \calI_{2,m; \ell \leq \ell_1}^{\pvdots, 1}(t,\xi) := \int_0^t \tau_m(s) \iiint e^{is\Phi_2} \varphi_{\ell}^{(n),+}(\xi) \varphi_{\geq \ell+10}^{(m),+}(\xi_1) & \, \xi_1 (\pxione \hatg)(s,\xi_1) \hatbarg(s,\xi_2) \hatg(s,\xi_4) \\
 &\quad \quad \quad \times \, \pvdots \frac{\hatq(\xi_3)}{\xi_3} \, \ud \xi_1 \, \ud \xi_2 \, \ud \xi_3 \, \ud s.
\end{aligned}
\end{equation*}
To this end, we will need to keep track of how close the frequency variable $\xi_2$ is to the problematic frequencies $\pm \sqrt{3}$ in terms of 
\begin{equation*}
 \bigl| |\xi_2| - \sqrt{3} \bigr| \simeq 2^{-\ell_2-100}, \quad 0 \leq \ell_2 \leq m.
\end{equation*}
Moreover, we will sometimes have to distinguish the absolute size of the frequencies 
\begin{equation*}
 |\xi_2| \simeq 2^{k_2}, \quad |\xi_3| \simeq 2^{k_3}, \quad |\xi_4| \simeq 2^{k_4} \quad \text{for } k_2, k_3, k_4 \in \bbZ.
\end{equation*}
The two main subcases that we distinguish are (1) $|\xi_3| \lesssim 2^{-\ell-1000}$ and (2) $|\xi_3| \gtrsim 2^{-\ell-1000}$. In the former case, the input and output frequency variables are still approximately correlated due to the relative smallness of $|\xi_3| \lesssim 2^{-\ell-1000}$, while in the latter case they are decorrelated.

\medskip 

\noindent \underline{\it Subcase 4.2.1: $1 \leq \ell \leq n-1$, $\ell+10 \leq \ell_1 \leq m$, $k_3 \leq -\ell-1000$.}
Due to the approximate correlation between the input and the output frequencies in this configuration, we can proceed similarly to Subcase 3.2.2 in the proof of Proposition~\ref{prop:weighted_energy_est_delta_T2}.
From $|\xi-\xi_1| = |\xi-\sqrt{3}-(\xi_1-\sqrt{3})| \simeq 2^{-\ell-100}$ and $|\xi_3| \ll 2^{-\ell-1000}$, we have 
\begin{equation} \label{equ:calIpv2m1_subcase421_xi2_plus_xi4}
 2^{-\ell-100} \simeq \xi-\xi_1-\xi_3 = \xi_2 + \xi_4.
\end{equation}
Since $\pxitwo \Phi_2 = 0$ if and only if $\xi_2 + \xi_4 = 0$, we can integrate by parts in $\xi_2$ in this case.
Thus,
\begin{align*}
 &\calI_{2,m; \ell \leq \ell_1, k_3 \leq -\ell}^{\pvdots, 1}(t,\xi) \\
 &= -i \int_0^t \tau_m(s) \frac{1}{s} \iiint e^{is\Phi_2} \frac{1}{\pxitwo \Phi_2} \frakn(\xi,\xi_1,\xi_3) \, \xi_1 (\pxione \hatg)(s,\xi_1) (\pxitwo \hatbarg)(s,\xi_2) \hatg(s,\xi_4) \\
 &\qquad \qquad \qquad \qquad \qquad \qquad \qquad \qquad \qquad \qquad \qquad \qquad \qquad \qquad \times \pvdots \frac{\hatq(\xi_3)}{\xi_3} \, \ud \xi_1 \, \ud \xi_2 \, \ud \xi_3 \, \ud s \\
 &\quad + i \int_0^t \tau_m(s) \frac{1}{s} \iiint e^{is\Phi_2} \frac{1}{\pxitwo \Phi_2} \frakn(\xi,\xi_1,\xi_3) \, \xi_1 (\pxione \hatg)(s,\xi_1) \hatbarg(s,\xi_2) (\pxifour \hatg)(s,\xi_4) \\
 &\qquad \qquad \qquad \qquad \qquad \qquad \qquad \qquad \qquad \qquad \qquad \qquad \qquad \qquad \times \pvdots \frac{\hatq(\xi_3)}{\xi_3} \, \ud \xi_1 \, \ud \xi_2 \, \ud \xi_3 \, \ud s \\
 &\quad - i \int_0^t \tau_m(s) \frac{1}{s} \iiint e^{is\Phi_2} \pxitwo \biggl( \frac{1}{\pxitwo \Phi_2} \biggr) \frakn(\xi,\xi_1,\xi_3)\, \xi_1 (\pxione \hatg)(s,\xi_1) \hatbarg(s,\xi_2) \hatg(s,\xi_4) \\
 &\qquad \qquad \qquad \qquad \qquad \qquad \qquad \qquad \qquad \qquad \qquad \qquad \qquad \qquad \times \pvdots \frac{\hatq(\xi_3)}{\xi_3} \, \ud \xi_1 \, \ud \xi_2 \, \ud \xi_3 \, \ud s
\end{align*}
with
\begin{equation*}
 \frakn(\xi,\xi_1,\xi_3) := \varphi_{\ell}^{(n),+}(\xi) \varphi_{\geq \ell+10}^{(m),+}(\xi_1)  \varphi_{\leq -\ell-1000}(\xi_3).
\end{equation*}
We record that
\begin{equation} \label{equ:calIpv2m1_subcase421_pxi_2Phi2_inverse}
 \begin{aligned}
  \frac{1}{\pxitwo \Phi_2} = - \frac{1}{\xi_2 + \xi_4} \frac{\jxitwo \jxifour (\xi_2 \jxifour - \xi_4 \jxitwo)}{\xi_2-\xi_4},
 \end{aligned}
\end{equation}
and that
\begin{equation} \label{equ:calIpv2m1_subcase421_two_pxi2s_Phi2}
 \pxitwo^2 \Phi_2 = - \frac{1}{\jxitwo^3} + \frac{1}{\jxifour^3}.
\end{equation}
We need to further distinguish the subcases (1) $\ell_2 \geq \ell$, (2) $1 \leq \ell_2 \leq \ell$, (3) $\ell_2 = 0$ with $k_2 \leq 10$, and (4) $\ell_2 = 0$ with $k_2 > 10$. Correspondingly, we then have
\[
\calI_{2,m; \ell \leq \ell_1, k_3 \leq -\ell}^{\pvdots, 1} = \sum_{\ast\in\{1,2,3,4\}} \calI_{2,m; \ell \leq \ell_1, k_3 \leq -\ell}^{\pvdots, 1,\ast}
\]
with
\EQ{\label{eq:3teil*}
 &\calI_{2,m; \ell \leq \ell_1, k_3 \leq -\ell}^{\pvdots, 1, \ast}(t,\xi) \\
 &:= -i \int_0^t \tau_m(s) \frac{1}{s} \iiint e^{is\Phi_2} \frac{1}{\pxitwo \Phi_2} \frakn_\ast(\xi,\xi_1,\xi_2,\xi_3) \, \xi_1 (\pxione \hatg)(s,\xi_1) (\pxitwo \hatbarg)(s,\xi_2) \hatg(s,\xi_4) \\
 &\qquad \qquad \qquad \qquad \qquad \qquad \qquad \qquad \qquad \qquad \qquad \qquad \qquad \qquad \times \pvdots \frac{\hatq(\xi_3)}{\xi_3} \, \ud \xi_1 \, \ud \xi_2 \, \ud \xi_3 \, \ud s \\
 &\quad + i \int_0^t \tau_m(s) \frac{1}{s} \iiint e^{is\Phi_2} \frac{1}{\pxitwo \Phi_2} \frakn_\ast(\xi,\xi_1,\xi_2,\xi_3) \, \xi_1 (\pxione \hatg)(s,\xi_1) \hatbarg(s,\xi_2) (\pxifour \hatg)(s,\xi_4) \\
 &\qquad \qquad \qquad \qquad \qquad \qquad \qquad \qquad \qquad \qquad \qquad \qquad \qquad \qquad \times \pvdots \frac{\hatq(\xi_3)}{\xi_3} \, \ud \xi_1 \, \ud \xi_2 \, \ud \xi_3 \, \ud s \\
 &\quad - i \int_0^t \tau_m(s) \frac{1}{s} \iiint e^{is\Phi_2} \pxitwo \biggl( \frac{1}{\pxitwo \Phi_2} \biggr) \frakn_\ast(\xi,\xi_1,\xi_2,\xi_3) \, \xi_1 (\pxione \hatg)(s,\xi_1) \hatbarg(s,\xi_2) \hatg(s,\xi_4) \\
 &\qquad \qquad \qquad \qquad \qquad \qquad \qquad \qquad \qquad \qquad \qquad \qquad \qquad \qquad \times \pvdots \frac{\hatq(\xi_3)}{\xi_3} \, \ud \xi_1 \, \ud \xi_2 \, \ud \xi_3 \, \ud s
}
and we write
\[
\calI_{2,m; \ell \leq \ell_1, k_3 \leq -\ell}^{\pvdots, 1,\ast}(t,\xi) =: \calI_{2,m; \ell \leq \ell_1, k_3 \leq -\ell}^{\pvdots, 1, \ast, (a)}(t,\xi) + \calI_{2,m; \ell \leq \ell_1, k_3 \leq -\ell}^{\pvdots, 1, \ast, (b)}(t,\xi) + \calI_{2,m; \ell \leq \ell_1, k_3 \leq -\ell}^{\pvdots, 1, \ast, (c)}(t,\xi).
\]
where $(a), (b), (c)$ refers to the three terms in~\eqref{eq:3teil*}.
We begin with~$\ast=1$.

\medskip 
\noindent \underline{\it Subcase 4.2.1.1: $1 \leq \ell \leq n-1$, $\ell+10 \leq \ell_1 \leq m$, $k_3 \leq -\ell-1000$, $\ell \leq \ell_2$.}
Say $\xi_2 \approx - \sqrt{3}$, i.e., $|\xi_2+\sqrt{3}| \lesssim 2^{-\ell-100}$, the other case $\xi_2 \approx \sqrt{3}$ being analogous. 
So we consider 
\begin{equation*}
 \frakn_{1}(\xi,\xi_1,\xi_2,\xi_3) := \varphi_{\ell}^{(n),+}(\xi) \varphi_{\geq \ell+10}^{(m),+}(\xi_1) \varphi_{\geq \ell}^{(m),-}(\xi_2) \varphi_{\leq -\ell-1000}(\xi_3).
\end{equation*}
In view of~\eqref{equ:calIpv2m1_subcase421_xi2_plus_xi4}, we must have $|\xi_4-\sqrt{3}| \lesssim 2^{-\ell-100}$.
Since $|\xi_2+\xi_4| \simeq 2^{-\ell-100}$ by \eqref{equ:calIpv2m1_subcase421_xi2_plus_xi4}, we obtain from~\eqref{equ:calIpv2m1_subcase421_pxi_2Phi2_inverse} that $|\pxitwo \Phi_2|^{-1} \lesssim 2^\ell$, and thus
\begin{equation*}
 \begin{aligned}
  \biggl| \pxi^{\kappa} \pxione^{\kappa_1} \pxitwo^{\kappa_2} \pxithree^{\kappa_3} \biggl( \frac{1}{\pxitwo \Phi_2} \frakn_{1}(\xi,\xi_1,\xi_2,\xi_3) \biggr) \biggr| \lesssim 2^\ell 2^{(\kappa + \kappa_1 + \kappa_2 + \kappa_3) \ell}.
 \end{aligned}
\end{equation*}
Moreover, by \eqref{equ:calIpv2m1_subcase421_two_pxi2s_Phi2} we have $|\pxitwo^2 \Phi_2| \lesssim ||\xi_2|-|\xi_4|| = |\xi_2 + \xi_4| \lesssim 2^{-\ell-100}$, whence 
\begin{equation*}
 \begin{aligned}
  \biggl| \pxi^{\kappa} \pxione^{\kappa_1} \pxitwo^{\kappa_2} \pxithree^{\kappa_3} \cdot \pxitwo\biggl( \frac{1}{\pxitwo \Phi_2} \biggr) \frakn_{1}(\xi,\xi_1,\xi_2,\xi_3) \biggr| \lesssim 2^\ell 2^{(\kappa + \kappa_1 + \kappa_2 + \kappa_3) \ell}.
 \end{aligned}
\end{equation*}
It follows that
\begin{equation} \label{equ:calIpv2m1_subcase4211_nlemma_bound1}
 \begin{aligned}
  \biggl\| \calF^{-1} \biggl[ \frac{1}{\pxitwo \Phi_2} \frakn_{1}(\xi,\xi_1,\xi_2,\xi_3) \biggr] \biggr\|_{L^1(\bbR^4)} \lesssim 2^\ell,
 \end{aligned}
\end{equation}
and
\begin{equation} \label{equ:calIpv2m1_subcase4211_nlemma_bound2}
 \begin{aligned}
  \biggl\| \calF^{-1} \biggl[ \pxitwo\biggl( \frac{1}{\pxitwo \Phi_2} \biggr) \frakn_{1}(\xi,\xi_1,\xi_2,\xi_3) \biggr] \biggr\|_{L^1(\bbR^4)} \lesssim 2^\ell.
 \end{aligned}
\end{equation}
We can now turn to the energy estimates. By Lemma~\ref{lem:frakn_for_pv} with \eqref{equ:calIpv2m1_subcase4211_nlemma_bound1} and the bounds \eqref{equ:g_bound_pxiL1}, \eqref{equ:g_bound_Linftyxi}, we obtain for $1 \leq m \leq n+5$,
\begin{equation*}
 \begin{aligned}
  &2^{-\frac12 \ell} \bigl\| \calI_{2,m; \ell \leq \ell_1, k_3 \leq -\ell }^{\pvdots, 1, 1, (a)}(t,\xi) \bigr\|_{L^2_\xi} \\
  &\lesssim 2^{-\frac12 \ell} \cdot 2^m \cdot 2^{-m} \cdot \biggl\| \calF^{-1} \biggl[ \frac{1}{\pxitwo \Phi_2} \frakn_{1}(\xi,\xi_1,\xi_2,\xi_3) \biggr] \biggr\|_{L^1(\bbR^4)} \\
  &\quad \quad \cdot \sup_{s \, \simeq \, 2^m} \, \bigl\| e^{is\jD} D x g(s)\bigr\|_{L^\infty_x} \bigl\| e^{is\jD} x g(s) \bigr\|_{L^\infty_x} \bigl\| \varphi_{\leq -\ell}(D-\sqrt{3}) g(s) \bigr\|_{L^2_x} \\
  &\lesssim 2^{-\frac12 \ell} \cdot 2^m \cdot 2^{-m} \cdot 2^\ell \cdot \sup_{s \, \simeq \, 2^m} \, \bigl\| \xi_1 \pxione \hatg(s,\xi_1)\bigr\|_{L^1_{\xi_1}} \bigl\| \pxitwo \hatg(s,\xi_2)\bigr\|_{L^1_{\xi_2}} 2^{-\frac12 \ell} \bigl\| \hatg(s,\xi_4) \bigr\|_{L^\infty_{\xi_4}} \\
  &\lesssim 2^{-\frac12 \ell} \cdot 2^m \cdot 2^{-m} \cdot 2^\ell \cdot (m\varepsilon)^2 \cdot 2^{-\frac12 \ell} m \varepsilon \lesssim m^3 \varepsilon^3,
 \end{aligned}
\end{equation*}
where we could freely insert a fattended cut-off to $|\xi_4-\sqrt{3}| \lesssim 2^{-\ell}$ on the third input. 
The energy estimate for $\calI_{2,m; \ell \leq \ell_1, k_3 \leq -\ell }^{\pvdots, 1, 1, (b)}(t,\xi)$ is analogous, and the one for $\calI_{2,m; \ell \leq \ell_1, k_3 \leq -\ell }^{\pvdots, 1, 1, (c)}(t,\xi)$ is similar using \eqref{equ:calIpv2m1_subcase4211_nlemma_bound2} in the application of Lemma~\ref{lem:frakn_for_pv}.

\medskip 
\noindent \underline{\it Subcase 4.2.1.2: $1 \leq \ell \leq n-1$, $\ell+10 \leq \ell_1 \leq m$, $k_3 \leq -\ell-1000$, $1 \leq \ell_2 \leq \ell$.}
Say that $\xi_2 < 0$ again, i.e., $|\xi_2 + \sqrt{3}| \simeq 2^{-\ell_2-100}$ for some $1 \leq \ell_2 \leq \ell$. So we take
\begin{equation*}
  \frakn_{2}(\xi,\xi_1,\xi_2,\xi_3) := \sum_{1 \leq \ell_2 \leq \ell} \varphi_{\ell}^{(n),+}(\xi) \varphi_{\geq \ell+10}^{(m),+}(\xi_1) \varphi_{\ell_2}^{(m),-}(\xi_2) \varphi_{\leq -\ell-1000}(\xi_3) =: \sum_{1 \leq \ell_2 \leq \ell} \frakn_{2, \ell_2}(\xi,\xi_1,\xi_2,\xi_3).
\end{equation*}
In view of~\eqref{equ:calIpv2m1_subcase421_xi2_plus_xi4} and $1 \leq \ell_2 \leq \ell$, we must have $|\xi_4-\sqrt{3}| \lesssim 2^{-\ell_2-100}$.
Since $|\xi_2+\xi_4| \simeq 2^{-\ell-100}$ by \eqref{equ:calIpv2m1_subcase421_xi2_plus_xi4}, we obtain from~\eqref{equ:calIpv2m1_subcase421_pxi_2Phi2_inverse} that $|\pxitwo \Phi_2|^{-1} \lesssim 2^\ell$.
Additionally, by \eqref{equ:calIpv2m1_subcase421_two_pxi2s_Phi2} we have $|\pxitwo^2 \Phi_2| \lesssim ||\xi_2|-|\xi_4|| = |\xi_2 + \xi_4| \lesssim 2^{-\ell-100}$, and by analogous arguments $|\pxitwo^{\kappa_2} \Phi_2| \lesssim 2^{-\ell}$ for any $\kappa_2 > 2$.
We conclude
\begin{equation*}
 \begin{aligned}
  \biggl| \pxi^{\kappa} \pxione^{\kappa_1} \pxitwo^{\kappa_2} \pxithree^{\kappa_3} \biggl( \frac{1}{\pxitwo \Phi_2} \frakn_{2, \ell_2}(\xi,\xi_1,\xi_2,\xi_3) \biggr) \biggr| \lesssim 2^\ell 2^{(\kappa + \kappa_1 + \kappa_3) \ell} 2^{\kappa_2 \ell_2},
 \end{aligned}
\end{equation*}
as well as
\begin{equation*}
 \begin{aligned}
  \biggl| \pxi^{\kappa} \pxione^{\kappa_1} \pxitwo^{\kappa_2} \pxithree^{\kappa_3} \biggl( \pxitwo\biggl( \frac{1}{\pxitwo \Phi_2} \biggr) \frakn_{2, \ell_2}(\xi,\xi_1,\xi_2,\xi_3) \biggr) \biggr| \lesssim 2^\ell 2^{(\kappa + \kappa_1 + \kappa_3) \ell} 2^{\kappa_2 \ell_2},
 \end{aligned}
\end{equation*}
whence we have
\begin{equation} \label{equ:calIpv2m1_subcase4212_nlemma_bound1}
 \begin{aligned}
  \biggl\| \calF^{-1} \biggl[ \frac{1}{\pxitwo \Phi_2} \frakn_{2, \ell_2}(\xi,\xi_1,\xi_2,\xi_3) \biggr] \biggr\|_{L^1(\bbR^4)} \lesssim 2^\ell,
 \end{aligned}
\end{equation}
as well as
\begin{equation} \label{equ:calIpv2m1_subcase4212_nlemma_bound2}
 \begin{aligned}
  \biggl\| \calF^{-1} \biggl[ \pxitwo\biggl( \frac{1}{\pxitwo \Phi_2} \biggr) \frakn_{2, \ell_2}(\xi,\xi_1,\xi_2,\xi_3) \biggr] \biggr\|_{L^1(\bbR^4)} \lesssim 2^\ell,
 \end{aligned}
\end{equation}
We begin with the energy estimates in this subcase. By Lemma~\ref{lem:frakn_for_pv} with \eqref{equ:calIpv2m1_subcase4212_nlemma_bound1}, and by the bounds~\eqref{equ:g_bound_pxiL1}, \eqref{equ:g_bound_Linftyxi}, we obtain
\begin{equation*}
 \begin{aligned}
  &2^{-\frac12 \ell} \bigl\| \calI_{2,m; \ell \leq \ell_1, k_3 \leq -\ell }^{\pvdots, 1, 2, (a)}(t,\xi) \bigr\|_{L^2_\xi} \\
  &\lesssim \sum_{1 \leq \ell_2 \leq \ell} 2^{-\frac12 \ell} \bigl\| \wtilvarphi^{(n),+}_\ell(\xi) \bigr\|_{L^2_\xi} \cdot 2^m \cdot 2^{-m} \cdot \sup_{s \, \simeq \, 2^m} \, \biggl\| \iiint e^{is\Phi_2} \frac{1}{\pxitwo \Phi_2} \frakn_{2, \ell_2}(\xi,\xi_1,\xi_2,\xi_3) \\
  &\qquad \qquad \qquad \qquad \qquad \qquad \qquad \qquad \quad \times \xi_1 (\pxione \hatg)(s,\xi_1) (\pxitwo \hatbarg)(s,\xi_2) \hatg(s,\xi_4) \, \pvdots \frac{\hatq(\xi_3)}{\xi_3} \, \ud \xi_1 \, \ud \xi_2 \, \ud \xi_3 \biggr\|_{L^\infty_\xi} \\
  &\lesssim \sum_{1 \leq \ell_2 \leq \ell} 2^{-\frac12 \ell} \cdot 2^{-\frac12 \ell} \cdot 2^m \cdot 2^{-m} \cdot \biggl\| \calF^{-1} \biggl[ \frac{1}{\pxitwo \Phi_2} \frakn_{2, \ell_2}(\xi,\xi_1,\xi_2,\xi_3) \biggr] \biggr\|_{L^1(\bbR^4)} \\
  &\quad \times \sup_{s \, \simeq \, 2^m} \, \bigl\| e^{is\jD} D x g(s)\bigr\|_{L^\infty_x} \bigl\| \varphi_{\ell_2}^{(m),-}(D) x g(s) \bigr\|_{L^2_x} \bigl\| \varphi_{\leq -\ell_2}(D-\sqrt{3}) g(s) \bigr\|_{L^2_x} \\
  &\lesssim  \sum_{1 \leq \ell_2 \leq \ell} 2^{-\frac12 \ell} \cdot 2^{-\frac12 \ell} \cdot 2^m \cdot 2^{-m} \cdot 2^\ell \cdot \bigl\| \xi_1 \pxione \hatg(s,\xi_1)\bigr\|_{L^1_{\xi_1}} \bigl\| \varphi_{\ell_2}^{(m),-}(\xi_2) \pxitwo \hatg(s,\xi_2)\bigr\|_{L^2_{\xi_2}} 2^{-\frac12 \ell_2} \bigl\| \hatg(s,\xi_4) \bigr\|_{L^\infty_{\xi_4}} \\
  &\lesssim \ell \cdot m \varepsilon \cdot \varepsilon \cdot m \varepsilon \lesssim m^3 \varepsilon^3.
 \end{aligned}
\end{equation*}
Here we could freely insert the fattened cut-off to frequencies $|\xi_4-\sqrt{3}| \lesssim 2^{-\ell_2}$ for the third input, and we used  
\begin{equation*}
  \sup_{s \, \simeq \, 2^m} \, 2^{-\frac12 \ell_2} \bigl\| \varphi_{\ell_2}^{(m),-}(\xi_2) \pxitwo \hatg(s,\xi_2)\bigr\|_{L^2_{\xi_2}} \lesssim \|g\|_{N_T} \lesssim \varepsilon.
\end{equation*}
The energy estimate for $\calI_{2,m; \ell \leq \ell_1, k_3 \leq -\ell }^{\pvdots, 1, 2, (b)}(t,\xi)$ is again analogous, and the one for $\calI_{2,m; \ell \leq \ell_1, k_3 \leq -\ell }^{\pvdots, 1, 2, (c)}(t,\xi)$ is similar using \eqref{equ:calIpv2m1_subcase4212_nlemma_bound2} in the application of Lemma~\ref{lem:frakn_for_pv}.

\medskip 
\noindent \underline{\it Subcase 4.2.1.3: $1 \leq \ell \leq n-1$, $\ell+10 \leq \ell_1 \leq m$, $k_3 \leq -\ell-1000$, $\ell_2 = 0$, $k_2 \leq 10$.}
Here we take
\begin{equation*}
 \frakn_3(\xi,\xi_1,\xi_2,\xi_3) := \varphi_{\ell}^{(n),+}(\xi) \varphi_{\geq \ell+10}^{(m),+}(\xi_1) \varphi_{0}^{(m)}(\xi_2) \varphi_{\leq 10}(\xi_2) \varphi_{\leq -\ell-1000}(\xi_3).
\end{equation*}
To deduce acceptable bounds on the phase $\pxitwo \Phi_2$, we need to distinguish the cases $|\xi_2| \lesssim 2^{-\ell-100}$ and $2^{-\ell-100} \ll |\xi_2| \lesssim 2^{10}$. 

In the case $|\xi_2| \lesssim 2^{-\ell-100}$, we must also have $|\xi_4| \lesssim 2^{-\ell-100}$ by \eqref{equ:calIpv2m1_subcase421_xi2_plus_xi4}. Then if $\xi_2 \xi_4 < 0$, we obtain from~\eqref{equ:calIpv2m1_subcase421_pxi_2Phi2_inverse} that $|\pxitwo \Phi_2|^{-1} \lesssim 2^\ell$. Instead if $\xi_2 \xi_4 > 0$, we must have $|\xi_2| + |\xi_4| \simeq 2^{-\ell-100}$, whence 
\begin{equation*}
 |\pxitwo \Phi_2| = \biggl| -\frac{\xi_2}{\jxitwo} - \frac{\xi_4}{\jxifour} \biggr| = |\xi_2| + |\xi_4| + \calO\bigl( |\xi_2|^2 \bigr) + \calO\bigl( |\xi_4|^2 \bigr) \simeq 2^{-\ell-100}.
\end{equation*}
Moreover, we obtain $|\pxitwo^2 \Phi_2| \lesssim | |\xi_2|-|\xi_4| | \lesssim |\xi_2| + |\xi_4| \lesssim 2^{-\ell}$, and by similar arguments $|\pxitwo^{\kappa_2} \Phi_2| \lesssim 2^{-\ell}$ for any $\kappa_2 > 2$. 

Instead, if $2^{-\ell-100} \ll |\xi_2| \lesssim 2^{10}$, then in view of~\eqref{equ:calIpv2m1_subcase421_xi2_plus_xi4} we must have a high-high interaction $|\xi_4| \simeq |\xi_2| \lesssim 2^{10}$ and $\xi_2 \xi_4 < 0$, so that \eqref{equ:calIpv2m1_subcase421_pxi_2Phi2_inverse} implies $|\pxitwo \Phi_2|^{-1} \lesssim 2^\ell$. Moreover, $|\pxitwo^2 \Phi_2| \lesssim | |\xi_2| - |\xi_4| | = |\xi_2 + \xi_4| \lesssim 2^{-\ell}$, and by similar arguments $|\pxitwo^{\kappa_2} \Phi_2| \lesssim 2^{-\ell}$ for any $\kappa_2 > 2$. 

We conclude
\begin{equation*}
 \begin{aligned}
  \biggl| \pxi^{\kappa} \pxione^{\kappa_1} \pxitwo^{\kappa_2} \pxithree^{\kappa_3} \biggl( \frac{1}{\pxitwo \Phi_2} \frakn_3(\xi,\xi_1,\xi_2,\xi_3) \biggr) \biggr| \lesssim 2^\ell 2^{(\kappa + \kappa_1 + \kappa_3) \ell},
 \end{aligned}
\end{equation*}
as well as
\begin{equation*}
 \begin{aligned}
  \biggl| \pxi^{\kappa} \pxione^{\kappa_1} \pxitwo^{\kappa_2} \pxithree^{\kappa_3} \biggl( \pxitwo\biggl( \frac{1}{\pxitwo \Phi_2} \biggr) \frakn_3(\xi,\xi_1,\xi_2,\xi_3) \biggr) \biggr| \lesssim 2^\ell 2^{(\kappa + \kappa_1 + \kappa_3) \ell},
 \end{aligned}
\end{equation*}
whence
\begin{equation} \label{equ:calIpv2m1_subcase4213_nlemma_bound1}
 \begin{aligned}
   \biggl\| \calF^{-1} \biggl[ \frac{1}{\pxitwo \Phi_2} \frakn_3(\xi,\xi_1,\xi_2,\xi_3) \biggr] \biggr\|_{L^1(\bbR^4)} \lesssim 2^\ell,
 \end{aligned}
\end{equation}
as well as
\begin{equation} \label{equ:calIpv2m1_subcase4213_nlemma_bound2}
 \begin{aligned}
  \biggl\| \calF^{-1} \biggl[ \pxitwo\biggl( \frac{1}{\pxitwo \Phi_2} \biggr) \frakn_3(\xi,\xi_1,\xi_2,\xi_3) \biggr] \biggr\|_{L^1(\bbR^4)} \lesssim 2^\ell.
 \end{aligned}
\end{equation}
Thus, by Lemma~\ref{lem:frakn_for_pv} with \eqref{equ:calIpv2m1_subcase4213_nlemma_bound1}, and by the bounds, for the term $\calI_{2,m; \ell \leq \ell_1, k_3 \leq -\ell }^{\pvdots, 1, 3, (a)}(t,\xi)$ we obtain the energy estimate
\begin{equation*}
 \begin{aligned}
  &2^{-\frac12 \ell} \bigl\| \calI_{2,m; \ell \leq \ell_1, k_3 \leq -\ell }^{\pvdots, 1, 3, (a)}(t,\xi) \bigr\|_{L^2_\xi} \\
  &\lesssim 2^{-\frac12 \ell} \bigl\| \wtilvarphi^{(n),+}_\ell(\xi) \bigr\|_{L^2_\xi} \cdot 2^m \cdot 2^{-m} \cdot \biggl\| \calF^{-1} \biggl[ \frac{1}{\pxitwo \Phi_2} \frakn_3(\xi,\xi_1,\xi_2,\xi_3) \biggr] \biggr\|_{L^1(\bbR^4)} \\
  &\quad \quad \times \sup_{s \, \simeq \, 2^m} \, \bigl\| e^{is\jD} D x g(s)\bigr\|_{L^\infty_x} \bigl\| \varphi_{0}^{(m)}(D) x g(s) \bigr\|_{L^2_x} \|g(s)\|_{L^2_x} \\
  &\lesssim 2^{-\frac12 \ell} \cdot 2^{-\frac12 \ell} \cdot 2^m \cdot 2^{-m} \cdot 2^\ell \cdot\sup_{s \, \simeq \, 2^m} \, \bigl\| \xi_1 \pxione \hatg(s,\xi_1)\bigr\|_{L^1_{\xi_1}} \bigl\| \varphi_{0}^{(m)}(\xi_2) \pxitwo \hatg(s,\xi_2)\bigr\|_{L^2_{\xi_2}} \|g(s)\|_{L^2_x} \\
  &\lesssim 2^{-\frac12 \ell} \cdot 2^{-\frac12 \ell} \cdot 2^m \cdot 2^{-m} \cdot 2^\ell \cdot m\varepsilon \cdot \varepsilon^2 \lesssim m \varepsilon^3.
 \end{aligned}
\end{equation*}
As before, the energy estimate for $\calI_{2,m; \ell \leq \ell_1, k_3 \leq -\ell }^{\pvdots, 1, 3, (b)}(t,\xi)$ is analogous, and the one for $\calI_{2,m; \ell \leq \ell_1, k_3 \leq -\ell }^{\pvdots, 1, 3, (c)}(t,\xi)$ is similar using \eqref{equ:calIpv2m1_subcase4213_nlemma_bound2} in the application of Lemma~\ref{lem:frakn_for_pv}.

\medskip 
\noindent \underline{\it Subcase 4.2.1.4: $1 \leq \ell \leq n-1$, $\ell+10 \leq \ell_1 \leq m$, $k_3 \leq -\ell-1000$, $\ell_2 = 0$, $k_2 > 10$.}
Here we take
\begin{equation*}
\begin{aligned}
 \frakn_4(\xi,\xi_1,\xi_2,\xi_3) &:= \sum_{k_2 > 10} \varphi_{\ell}^{(n),+}(\xi) \varphi_{\geq \ell+10}^{(m),+}(\xi_1) \varphi_{0}^{(m)}(\xi_2) \varphi_{k_2}(\xi_2) \varphi_{\leq -\ell-1000}(\xi_3) \\
 &=: \sum_{k_2 > 10} \frakn_{4,k_2}(\xi,\xi_1,\xi_2,\xi_3).
\end{aligned}
\end{equation*}
In view of \eqref{equ:calIpv2m1_subcase421_xi2_plus_xi4}, we must have a high-high interaction $1 \ll |\xi_2| \simeq |\xi_4| \simeq 2^{k_2}$ and $\xi_2 \xi_4 < 0$. Then~\eqref{equ:calIpv2m1_subcase421_pxi_2Phi2_inverse} implies $|\pxitwo \Phi_2|^{-1} \lesssim 2^{\ell} 2^{3k_2}$, while we have 
\begin{equation*}
 |\pxitwo^2 \Phi_2| \lesssim 2^{-4k_2} | |\xi_2| - |\xi_4| | = 2^{-4k_2} |\xi_2 + \xi_4| \simeq 2^{-4k_2} 2^{-\ell}.
\end{equation*}
Thus,
\begin{equation*}
 \begin{aligned}
  \biggl| \pxi^{\kappa} \pxione^{\kappa_1} \pxitwo^{\kappa_2} \pxithree^{\kappa_3} \biggl( \frac{1}{\pxitwo \Phi_2} \frakn_{4, k_2}(\xi,\xi_1,\xi_2,\xi_3) \biggr) \biggr| \lesssim 2^\ell 2^{3 k_2} 2^{(\kappa + \kappa_1 + \kappa_3) \ell} 2^{-\kappa_2 k_2},
 \end{aligned}
\end{equation*}
as well as
\begin{equation*}
 \begin{aligned}
  \biggl| \pxi^{\kappa} \pxione^{\kappa_1} \pxitwo^{\kappa_2} \pxithree^{\kappa_3} \biggl( \pxitwo\biggl( \frac{1}{\pxitwo \Phi_2} \biggr) \frakn_{4, k_2}(\xi,\xi_1,\xi_2,\xi_3) \biggr) \biggr| \lesssim 2^\ell 2^{2k_2} 2^{(\kappa + \kappa_1 + \kappa_3) \ell} 2^{-\kappa_2 k_2}.
 \end{aligned}
\end{equation*}
Hence, we have
\begin{equation} \label{equ:calIpv2m1_subcase4214_nlemma_bound1}
 \begin{aligned}
  \biggl\| \calF^{-1} \biggl[ \frac{1}{\pxitwo \Phi_2} \frakn_{4, k_2}(\xi,\xi_1,\xi_2,\xi_3) \biggr] \biggr\|_{L^1(\bbR^4)} \lesssim 2^\ell 2^{3 k_2},
 \end{aligned}
\end{equation}
as well as
\begin{equation} \label{equ:calIpv2m1_subcase4214_nlemma_bound2}
 \begin{aligned}
  \biggl\| \calF^{-1} \biggl[ \pxitwo\biggl( \frac{1}{\pxitwo \Phi_2} \biggr) \frakn_{4, k_2}(\xi,\xi_1,\xi_2,\xi_3) \biggr] \biggr\|_{L^1(\bbR^4)} \lesssim 2^\ell 2^{2k_2}.
 \end{aligned}
\end{equation}
By Lemma~\ref{lem:frakn_for_pv} with \eqref{equ:calIpv2m1_subcase4214_nlemma_bound1} and the bounds~\eqref{equ:g_bound_pxiL1}, \eqref{equ:g_bound_jD2_L2}, we obtain
\begin{equation*}
 \begin{aligned}
  &2^{-\frac12 \ell} \bigl\| \calI_{2,m; \ell \leq \ell_1, k_3 \leq -\ell }^{\pvdots, 1, 4, (a)}(t,\xi) \bigr\|_{L^2_\xi} \\
  &\lesssim \sum_{k_2 > 10} 2^{-\frac12 \ell} \bigl\| \wtilvarphi^{(n),+}_\ell(\xi) \bigr\|_{L^2_\xi} \cdot 2^m \cdot 2^{-m} \cdot \biggl\| \calF^{-1} \biggl[ \frac{1}{\pxitwo \Phi_2} \frakn_{4, k_2}(\xi,\xi_1,\xi_2,\xi_3) \biggr] \biggr\|_{L^1(\bbR^4)} \\
  &\quad \quad \times \sup_{s \, \simeq \, 2^m} \, \bigl\| e^{is\jD} D x g(s)\bigr\|_{L^\infty_x} \bigl\| \varphi_{k_2}(D) x g(s) \bigr\|_{L^2_x} \|\wtilvarphi_{k_2}(D) g(s)\|_{L^2_x} \\
  &\lesssim \sum_{k_2 > 10} 2^{-\frac12 \ell} \cdot 2^{-\frac12 \ell} \cdot 2^m \cdot 2^{-m} \cdot 2^\ell 2^{3 k_2} \cdot \sup_{s \, \simeq \, 2^m} \, \bigl\| \xi_1 \pxione \hatg(s,\xi_1)\bigr\|_{L^1_{\xi_1}} \bigl\| \varphi_{k_2}(\xi_2) \pxitwo \hatg(s,\xi_2)\bigr\|_{L^2_{\xi_2}} \|\wtilvarphi_{k_2}(D) g(s)\|_{L^2_x} \\
  &\lesssim \sum_{k_2 > 10} 2^{-\frac12 \ell} \cdot 2^{-\frac12 \ell} \cdot 2^m \cdot 2^{-m} \cdot 2^\ell 2^{3 k_2} \cdot m\varepsilon \cdot 2^{-2k_2} \varepsilon \cdot 2^{-2k_2} \varepsilon \lesssim m \varepsilon^3.
 \end{aligned}
\end{equation*}
Finally, the energy estimate for $\calI_{2,m; \ell \leq \ell_1, k_3 \leq -\ell }^{\pvdots, 1, 4, (b)}(t,\xi)$ is again analogous, and the one for the term $\calI_{2,m; \ell \leq \ell_1, k_3 \leq -\ell }^{\pvdots, 1, 4, (c)}(t,\xi)$ is similar using \eqref{equ:calIpv2m1_subcase4214_nlemma_bound2} in the application of Lemma~\ref{lem:frakn_for_pv}.

This finishes the discussion of Subcase~4.2.1.

\medskip 

\noindent \underline{\it Subcase 4.2.2: $1 \leq \ell \leq n-1$, $\ell+10 \leq \ell_1 \leq m$, $k_3 > -\ell-1000$.}
We now begin with the analysis of the regime $|\xi_3| \gtrsim 2^{-\ell}$, where the input and output frequencies are decorrelated. It turns out that for $|\xi_4| \ll 2^{-\ell}$, we can either integrate by parts in time or in $\xi_2$, while for $|\xi_4| \gtrsim 2^{-\ell}$ integration by parts in $\xi_3$ pays off.
Correspondingly, we distinguish the subcases (1) $k_4 \leq -\ell-1000$ with $k_2 \leq -\ell-500$, (2) $k_4 \leq -\ell-1000$ with $k_2 > -\ell-500$, and (3) $k_4 > -\ell-1000$.

\medskip 
\noindent \underline{\it Subcase 4.2.2.1: $1 \leq \ell \leq n-1$, $\ell+10 \leq \ell_1 \leq m$, $k_3 > -\ell-1000$, $k_4 \leq -\ell-1000$, $k_2 \leq -\ell-500$.}
We consider the energy estimate for the term
\begin{equation*}
\begin{aligned}
 \calI_{2,m; \ell \leq \ell_1, k_3 > -\ell}^{\pvdots, 1, 1}(t,\xi) := \int_0^t \tau_m(s) \iiint e^{is\Phi_2} \frakn_1(\xi,\xi_1,\xi_2,\xi_3) & \, \xi_1 (\pxione \hatg)(s,\xi_1) \hatbarg(s,\xi_2) \hatg(s,\xi_4) \\
 &\quad \quad \quad \times \, \pvdots \frac{\hatq(\xi_3)}{\xi_3} \, \ud \xi_1 \, \ud \xi_2 \, \ud \xi_3 \, \ud s
\end{aligned}
\end{equation*}
with
\begin{equation*}
 \begin{aligned}
  \frakn_1(\xi,\xi_1,\xi_2,\xi_3) := \varphi_{\ell}^{(n),+}(\xi) \varphi_{\geq \ell+10}^{(m),+}(\xi_1) \varphi_{\leq -\ell-500}(\xi_2) \varphi_{>-\ell-1000}(\xi_3) \varphi_{\leq -\ell-1000}(\xi_4).
 \end{aligned}
\end{equation*}
Here we will have to integrate by parts in time. But to achieve a better balance among the inputs after reinserting the equation for $\partial_s \hatg(s)$, it is useful to first integrate by parts in $\xi_1$,
\begin{equation*}
 \begin{aligned}
  &\calI_{2,m; \ell \leq \ell_1, k_3 > -\ell}^{\pvdots, 1, 1}(t,\xi) \\
  &= - i \int_0^t \tau_m(s) \cdot s \iiint e^{is\Phi_2} (\pxione \Phi_2) \frakn_1(\xi,\xi_1,\xi_2,\xi_3) \, \xi_1  \hatg(s,\xi_1) \hatbarg(s,\xi_2) \hatg(s,\xi_4) \, \pvdots \frac{\hatq(\xi_3)}{\xi_3} \, \ud \xi_1 \, \ud \xi_2 \, \ud \xi_3 \, \ud s \\
  &\quad - \int_0^t \tau_m(s) \iiint e^{is\Phi_2} \pxione \bigl( \frakn_1(\xi,\xi_1,\xi_2,\xi_3) \, \xi_1 \bigr)  \hatg(s,\xi_1) \hatbarg(s,\xi_2) \hatg(s,\xi_4) \, \pvdots \frac{\hatq(\xi_3)}{\xi_3} \, \ud \xi_1 \, \ud \xi_2 \, \ud \xi_3 \, \ud s \\
  &\quad + \int_0^t \tau_m(s) \iiint e^{is\Phi_2} \frakn_1(\xi,\xi_1,\xi_2,\xi_3) \, \xi_1  \hatg(s,\xi_1) \hatbarg(s,\xi_2) (\pxifour \hatg)(s,\xi_4) \, \pvdots \frac{\hatq(\xi_3)}{\xi_3} \, \ud \xi_1 \, \ud \xi_2 \, \ud \xi_3 \, \ud s \\
  &=: \calI_{2,m; \ell \leq \ell_1, k_3 > -\ell}^{\pvdots, 1, 1, (a)}(t,\xi) + \calI_{2,m; \ell \leq \ell_1, k_3 > -\ell}^{\pvdots, 1, 1, (b)}(t,\xi) + \calI_{2,m; \ell \leq \ell_1, k_3 > -\ell}^{\pvdots, 1, 1, (c)}(t,\xi).
 \end{aligned}
\end{equation*}
The third term $\calI_{2,m; \ell \leq \ell_1, k_3 > -\ell}^{\pvdots, 1, 1, (c)}(t,\xi)$ is straightforward to bound since the input $(\pxifour \hatg)(s,\xi_4)$ is supported at $|\xi_4| \lesssim 2^{-\ell-1000}$, and therefore far away from the bad frequencies. The second term $\calI_{2,m; \ell \leq \ell_1, k_3 > -\ell}^{\pvdots, 1, 1, (b)}(t,\xi)$ is better behaved than the first term because $|\pxione \frakn_1(\xi,\xi_1,\xi_2,\xi_3)| \lesssim 2^\ell \lesssim 2^m$. 
For this reason we only discuss the first term, where we integrate by parts in time 
\begin{equation*}
 \begin{aligned}
  &\calI_{2,m; \ell \leq \ell_1, k_3 > -\ell}^{\pvdots, 1, 1, (a)}(t,\xi) \\
  &= \int_0^t \tau_m(s) \cdot s \iiint e^{is\Phi_2} \frac{\pxione \Phi_2}{\Phi_2} \frakn_1(\xi,\xi_1,\xi_2,\xi_3) \, \ps \bigl( \xi_1  \hatg(s,\xi_1) \hatbarg(s,\xi_2) \hatg(s,\xi_4) \bigr) \, \pvdots \frac{\hatq(\xi_3)}{\xi_3} \, \ud \xi_1 \, \ud \xi_2 \, \ud \xi_3 \, \ud s \\
  &\quad \quad + \bigl\{ \text{similar and boundary terms} \bigr\}.
 \end{aligned}
\end{equation*}
Say that the time derivative $\ps$ falls onto the first input, the other possibilities being analogous.
Then we insert the equation for $\ps \hatg(s,\xi_1)$ in the form \eqref{equ:decomposition_FT_pt_hatg}.
By Taylor expansion we have in this frequency configuration 
\begin{equation*}
\begin{aligned}
 \Phi_2(\xi,\xi_1,\xi_2,\xi_3) 
 = - \frac{\sqrt{3}}{2} (\xi-\sqrt{3}) + \frac{\sqrt{3}}{2} (\xi_1-\sqrt{3}) - \frac12 \xi_2^2 + \frac12 \xi_4^2 +\calO\bigl( 2^{-2\ell-200} \bigr) \simeq 2^{-\ell-100}.
\end{aligned}
\end{equation*}
It follows that 
\begin{equation*}
 \biggl| \pxi^\kappa \pxione^{\kappa_1} \pxitwo^{\kappa_2} \pxithree^{\kappa_3} \biggl( \frac{\pxione \Phi_2}{\Phi_2} \frakn_1(\xi,\xi_1,\xi_2,\xi_3) \biggr) \biggr| \lesssim 2^{\ell} 2^{(\kappa + \kappa_1 + \kappa_2 + \kappa_3) \ell},
\end{equation*}
whence
\begin{equation} \label{equ:subcase4221_frakn_phase_bound}
 \biggl\| \calF^{-1}\biggl[ \frac{\pxione \Phi_2}{\Phi_2} \frakn(\xi,\xi_1,\xi_2,\xi_3) \biggr] \biggr\|_{L^1(\bbR^4)} \lesssim 2^\ell.
\end{equation}
Thus, up to the estimates for the simpler terms and the boundary terms, upon inserting \eqref{equ:decomposition_FT_pt_hatg} we obtain by Lemma~\ref{lem:frakn_for_pv} with \eqref{equ:subcase4221_frakn_phase_bound} and the bounds \eqref{equ:Linfty_decay_v}, \eqref{equ:g_bound_Linftyxi}, \eqref{equ:g_bound_dispersive_est}, \eqref{equ:decomposition_FT_pt_hatg_Ncdecay} that
\begin{equation*}
 \begin{aligned}
  &2^{-\frac12 \ell} \bigl\| \calI_{2,m; \ell \leq \ell_1, k_3 > -\ell}^{\pvdots, 1, 1, (a)}(t,\xi) \bigr\|_{L^2_\xi} \\
  &\lesssim 2^{-\frac12 \ell} \cdot 2^m \cdot 2^m \cdot \biggl\| \calF^{-1}\biggl[ \frac{\pxione \Phi_2}{\Phi_2} \frakn(\xi,\xi_1,\xi_2,\xi_3) \biggr] \biggr\|_{L^1(\bbR^4)} \\
  &\quad \quad \times \sup_{s \, \simeq \, 2^m} \, \biggl( \bigl\| \varphi_{\leq -\ell+100}(\xi_1-\sqrt{3}) (2i\jxione)^{-1} \xi_1 \widehat{\beta}(\xi_1) \bigr\|_{L^2_{\xi_1}} |v(s,0)|^2 \bigl\| e^{is\jD} g(s) \bigr\|_{L^\infty_x}^2 \\
  &\quad \quad \quad \quad \quad \quad \quad + \bigl\|(2i\jD)^{-1} D \calN_c(s)\bigr\|_{L^\infty_x} \bigl\| \varphi_{\leq -\ell+100}(\xi_2) \hatg(s,\xi_2)\bigr\|_{L^2_{\xi_2}} \bigl\| e^{is\jD} g(s) \bigr\|_{L^\infty_x} \biggr) \\
  &\lesssim 2^{-\frac12 \ell} \cdot 2^m \cdot 2^m \cdot 2^{\ell} \cdot \bigl( 2^{-\frac12 \ell} \cdot 2^{-2m} m^4 \varepsilon^4 + 2^{-\frac32 m} m^3 \varepsilon^2 \cdot 2^{-\frac12 \ell} m \varepsilon \cdot 2^{-\frac12 m} m \varepsilon \bigr) \\
  &\lesssim m^5 \varepsilon^4,
 \end{aligned}
\end{equation*}
which is acceptable.

\medskip 
\noindent \underline{\it Subcase 4.2.2.2: $1 \leq \ell \leq n-1$, $\ell+10 \leq \ell_1 \leq m$, $k_3 > -\ell-1000$, $k_4 \leq -\ell-1000$, $k_2 > -\ell-500$.}
Next, we consider the energy estimate for the term
\begin{equation*}
\begin{aligned}
 \calI_{2,m; \ell \leq \ell_1, k_3 > -\ell}^{\pvdots, 1, 2}(t,\xi) := \int_0^t \tau_m(s) \iiint e^{is\Phi_2} \frakn_2(\xi,\xi_1,\xi_2,\xi_3) & \, \xi_1 (\pxione \hatg)(s,\xi_1) \hatbarg(s,\xi_2) \hatg(s,\xi_4) \\
 &\quad \quad \quad \quad \times \frac{\hatq(\xi_3)}{\xi_3} \, \ud \xi_1 \, \ud \xi_2 \, \ud \xi_3 \, \ud s
\end{aligned}
\end{equation*}
with
\begin{equation*}
 \begin{aligned}
  \frakn_2(\xi,\xi_1,\xi_2,\xi_3) := \varphi_{\ell}^{(n),+}(\xi) \varphi_{\geq \ell+10}^{(m),+}(\xi_1) \varphi_{>-\ell-500}(\xi_2) \varphi_{>-\ell-1000}(\xi_3) \varphi_{\leq -\ell-1000}(\xi_4).
 \end{aligned}
\end{equation*}
We note that since $|\xi_3| \gtrsim 2^{-\ell-1000}$, here it is not necessary anymore to treat the integration in $\xi_3$ in a $\pvdots$ sense.
Integrating by parts in $\xi_2$, we find
\begin{equation*}
 \begin{aligned}
  &\calI_{2,m; \ell \leq \ell_1, k_3 > -\ell}^{\pvdots, 1, 2}(t,\xi) \\
  &= i \int_0^t \tau_m(s) \frac{1}{s} \iiint e^{is\Phi_2} \frac{1}{\pxitwo \Phi_2} \frakn_2(\xi,\xi_1,\xi_2,\xi_3) \, \xi_1 (\pxione \hatg)(s,\xi_1) (\pxitwo \hatbarg)(s,\xi_2) \hatg(s,\xi_4) \frac{\hatq(\xi_3)}{\xi_3} \, \ud \xi_1 \, \ud \xi_2 \, \ud \xi_3 \, \ud s \\
  &\quad - i  \int_0^t \tau_m(s) \frac{1}{s} \iiint e^{is\Phi_2} \frac{1}{\pxitwo \Phi_2} \frakn_2(\xi,\xi_1,\xi_2,\xi_3) \, \xi_1 (\pxione \hatg)(s,\xi_1) \hatbarg(s,\xi_2) (\pxifour \hatg)(s,\xi_4) \frac{\hatq(\xi_3)}{\xi_3} \, \ud \xi_1 \, \ud \xi_2 \, \ud \xi_3 \, \ud s \\
  &\quad + i \int_0^t \tau_m(s) \frac{1}{s} \iiint e^{is\Phi_2} \pxitwo \biggl( \frac{1}{\pxitwo \Phi_2} \frakn_2(\xi,\xi_1,\xi_2,\xi_3) \biggr) \xi_1 (\pxione \hatg)(s,\xi_1) \hatbarg(s,\xi_2) \hatg(s,\xi_4) \frac{\hatq(\xi_3)}{\xi_3} \, \ud \xi_1 \, \ud \xi_2 \, \ud \xi_3 \, \ud s \\
  &=: \calI_{2,m; \ell \leq \ell_1, k_3 > -\ell}^{\pvdots, 1, 2, (a)}(t,\xi) + \calI_{2,m; \ell \leq \ell_1, k_3 > -\ell}^{\pvdots, 1, 2, (b)}(t,\xi) + \calI_{2,m; \ell \leq \ell_1, k_3 > -\ell}^{\pvdots, 1, 2, (c)}(t,\xi).
 \end{aligned}
\end{equation*}
Since in this subcase $|\xi_4| \ll 1$ and $|\xi_2| \gg |\xi_4|$, we have 
\begin{equation*}
 |\pxitwo \Phi_2| = \biggl| -\frac{\xi_2}{\jxitwo} - \frac{\xi_4}{\jxifour} \biggr| \gtrsim \min\{ 1, |\xi_2|\} \gtrsim 2^{-\ell}.
\end{equation*}
Starting with the first term, by H\"older's inequality in the frequency variables and the bounds~\eqref{equ:g_bound_pxiL1}, \eqref{equ:g_bound_Linftyxi}, we then obtain the energy estimate
\begin{equation} \label{eq:pv12a}
 \begin{aligned}
  &2^{-\frac12 \ell} \bigl\| \calI_{2,m; \ell \leq \ell_1, k_3 > -\ell}^{\pvdots, 1, 2, (a)}(t,\xi) \bigr\|_{L^2_\xi} \\
  &\lesssim 2^{-\frac12 \ell} \bigl\| \wtilvarphi_\ell^{(n)}(\xi) \bigr\|_{L^2_\xi} \cdot 2^m \cdot 2^{-m} \cdot 2^\ell \\
  &\quad \quad \times \sup_{s \, \simeq \, 2^m} \, \bigl\| \xi_1 (\pxione \hatg)(s,\xi_1) \bigr\|_{L^1_{\xi_1}} \bigl\|\pxitwo \hatbarg(s,\xi_2)\bigr\|_{L^1_{\xi_2}} \bigl\| \hatg(s,\xi_4) \bigr\|_{L^\infty_{\xi_4}} \Bigl\| \varphi_{>-\ell-1000}(\xi_3) \frac{\hatq(\xi_3)}{\xi_3} \Bigr\|_{L^1_{\xi_3}} \\
  &\lesssim 2^{-\frac12 \ell} \cdot 2^{-\frac12 \ell} \cdot 2^m \cdot 2^{-m} \cdot 2^\ell \cdot m^3 \varepsilon^3 \cdot \ell \lesssim m^4 \varepsilon^3.
 \end{aligned}
\end{equation}
The energy estimate for the second term $\calI_{2,m; \ell \leq \ell_1, k_3 > -\ell}^{\pvdots, 1, 2, (a)}(t,\xi)$ is identical.
Finally, since 
\begin{equation*}
 \biggl| \pxitwo \biggl( \frac{1}{\pxitwo \Phi_2} \frakn_2(\xi,\xi_1,\xi_2,\xi_3) \biggr) \biggr| \lesssim 2^{2\ell},
\end{equation*}
we can estimate the third term, using H\"older's inequality in the frequency variables and the bounds \eqref{equ:g_bound_pxiL1}, \eqref{equ:g_bound_Linftyxi}, in terms of
\begin{equation} \label{eq:Ipv12c}
 \begin{aligned}
  &2^{-\frac12 \ell} \bigl\| \calI_{2,m; \ell \leq \ell_1, k_3 > -\ell}^{\pvdots, 1, 2, (c)}(t,\xi) \bigr\|_{L^2_\xi} \\
  &\lesssim 2^{-\frac12 \ell} \bigl\| \wtilvarphi_\ell^{(n)}(\xi) \bigr\|_{L^2_\xi} \cdot 2^m \cdot 2^{-m} \cdot 2^{2\ell} \cdot \sup_{s \, \simeq \, 2^m} \, \bigl\| \xi_1 (\pxione \hatg)(s,\xi_1) \bigr\|_{L^1_{\xi_1}} \\
  &\qquad \qquad \qquad \qquad \times \bigl\|\hatbarg(s,\xi_2)\bigr\|_{L^\infty_{\xi_2}} \bigl\| \varphi_{\leq -\ell-1000}(\xi_4) \hatg(s,\xi_4) \bigr\|_{L^1_{\xi_4}} \Bigl\| \varphi_{>-\ell-1000}(\xi_3) \frac{\hatq(\xi_3)}{\xi_3} \Bigr\|_{L^1_{\xi_3}} \\
  &\lesssim 2^{-\frac12 \ell} \cdot 2^{-\frac12 \ell} \cdot 2^m \cdot 2^{-m} \cdot 2^{2\ell} \cdot m \varepsilon \cdot m \varepsilon \cdot 2^{-\ell} m \varepsilon \cdot \ell \lesssim m^4 \varepsilon^3,
 \end{aligned}
\end{equation}
which is acceptable.

\medskip 
\noindent \underline{\it Subcase 4.2.2.3: $1 \leq \ell \leq n-1$, $\ell+10 \leq \ell_1 \leq m$, $k_3 > -\ell-1000$, $k_4 > -\ell-1000$.}
It remains to consider the subcase $|\xi_4| \gtrsim 2^{-\ell-1000}$, where integration by parts in $\xi_3$ pays off.
In order to take into account the relative smallness of $\pxithree \Phi_2$ for $|\xi_4| \ll 1$, it is convenient to further distinguish the configurations $2^{-\ell-1000} \lesssim |\xi_4| \lesssim 2^{-10}$ and $|\xi_4| \gtrsim 2^{-10}$. We begin with the energy estimate for the former configuration, i.e., for the term
\begin{equation*}
\begin{aligned}
 \calI_{2,m; \ell \leq \ell_1, k_3 > -\ell}^{\pvdots, 1, 3}(t,\xi) := \int_0^t \tau_m(s) \iiint e^{is\Phi_2} \frakn_{3}(\xi,\xi_1,\xi_2,\xi_3) & \, \xi_1 (\pxione \hatg)(s,\xi_1) \hatbarg(s,\xi_2) \hatg(s,\xi_4) \\
 &\quad \quad \quad \quad \quad \times \frac{\hatq(\xi_3)}{\xi_3} \, \ud \xi_1 \, \ud \xi_2 \, \ud \xi_3 \, \ud s
\end{aligned}
\end{equation*}
with
\begin{equation*}
 \begin{aligned}
  \frakn_{3}(\xi,\xi_1,\xi_2,\xi_3) := \varphi_{\ell}^{(n),+}(\xi) \varphi_{\geq \ell+10}^{(m),+}(\xi_1) \varphi_{>-\ell-1000}(\xi_3) \varphi_{[-\ell-1000,-10]}(\xi_4).
 \end{aligned}
\end{equation*}
Since $|\xi_3| \gtrsim 2^{-\ell-1000}$ the integration in $\xi_3$ does not have to be understood in the $\pvdots$ sense. In particular, we can now integrate by parts in $\xi_3$ and obtain
\begin{equation*}
 \begin{aligned}
  &\calI_{2,m; \ell \leq \ell_1, k_3 > -\ell}^{\pvdots, 1, 3}(t,\xi) \\
  &= -i \int_0^t \tau_m(s) \frac{1}{s} \iiint e^{is\Phi_2} \frac{1}{\pxithree \Phi_2} \frakn_3(\xi,\xi_1,\xi_2,\xi_3) \, \xi_1 (\pxione \hatg)(s,\xi_1) \hatbarg(s,\xi_2) (\pxifour \hatg)(s,\xi_4) \frac{\hatq(\xi_3)}{\xi_3} \, \ud \xi_1 \, \ud \xi_2 \, \ud \xi_3 \, \ud s \\
  &\quad + i \int_0^t \tau_m(s) \frac{1}{s} \iiint e^{is\Phi_2} \frac{1}{\pxithree \Phi_2} \frakn_3(\xi,\xi_1,\xi_2,\xi_3) \, \xi_1 (\pxione \hatg)(s,\xi_1) \hatbarg(s,\xi_2) \hatg(s,\xi_4) \pxithree \biggl( \frac{\hatq(\xi_3)}{\xi_3} \biggr) \, \ud \xi_1 \, \ud \xi_2 \, \ud \xi_3 \, \ud s \\  
  &\quad + i \int_0^t \tau_m(s) \frac{1}{s} \iiint e^{is\Phi_2} \pxithree \biggl( \frac{1}{\pxithree \Phi_2}  \frakn_3(\xi,\xi_1,\xi_2,\xi_3) \biggr) \, \xi_1 (\pxione \hatg)(s,\xi_1) \hatbarg(s,\xi_2) \hatg(s,\xi_4) \frac{\hatq(\xi_3)}{\xi_3} \, \ud \xi_1 \, \ud \xi_2 \, \ud \xi_3 \, \ud s \\
  &=: \calI_{2,m; \ell \leq \ell_1, k_3 > -\ell}^{\pvdots, 1, 3, (a)}(t,\xi) + \calI_{2,m; \ell \leq \ell_1, k_3 > -\ell}^{\pvdots, 1, 3, (b)}(t,\xi) + \calI_{2,m; \ell \leq \ell_1, k_3 > -\ell}^{\pvdots, 1, 3, (c)}(t,\xi).
 \end{aligned}
\end{equation*}
Observe that for $-\ell-1000<k_4\leq-10$, we have 
\begin{equation*}
 |\pxithree \Phi_2| = \Bigl|\frac{\xi_4}{\jxifour} \Bigr| \gtrsim |\xi_4| \gtrsim 2^{-\ell}.
\end{equation*}
For the first term we then obtain the following energy estimate by H\"older's inequality in the frequency variables and by the bounds \eqref{equ:g_bound_pxiL1}, \eqref{equ:g_bound_Linftyxi},
\begin{equation*}
 \begin{aligned}
  &2^{-\frac12 \ell} \bigl\| \calI_{2,m; \ell \leq \ell_1, k_3 > -\ell}^{\pvdots, 1, 3, (a)}(t,\xi) \bigr\|_{L^2_\xi} \\
  &\lesssim 2^{-\frac12 \ell} \bigl\| \wtilvarphi_\ell^{(n)}(\xi) \bigr\|_{L^2_\xi} \cdot 2^m \cdot 2^{-m} \cdot 2^\ell \\
  &\quad \quad \times \sup_{s \, \simeq \, 2^m} \, \bigl\| \xi_1 (\pxione \hatg)(s,\xi_1) \bigr\|_{L^1_{\xi_1}} \bigl\|\hatbarg(s,\xi_2)\bigr\|_{L^\infty_{\xi_2}} \bigl\| \pxifour \hatg(s,\xi_4) \bigr\|_{L^1_{\xi_4}} \Bigl\| \varphi_{>-\ell-1000}(\xi_3) \frac{\hatq(\xi_3)}{\xi_3} \Bigr\|_{L^1_{\xi_3}} \\
  &\lesssim 2^{-\frac12 \ell} \cdot 2^{-\frac12 \ell} \cdot 2^m \cdot 2^{-m} \cdot 2^\ell \cdot m^3 \varepsilon^3 \cdot \ell \lesssim m^4 \varepsilon^3.
 \end{aligned}
\end{equation*}
For the second term we find similarly that
\begin{equation*}
 \begin{aligned}
  &2^{-\frac12\ell} \bigl\| \calI_{2,m; \ell \leq \ell_1, k_3 > -\ell}^{\pvdots, 1, 3, (b)}(t,\xi) \bigr\|_{L^2_\xi} \\
  &\lesssim 2^{-\frac12 \ell} \cdot \bigl\| \wtilvarphi_\ell^{(n)}(\xi) \bigr\|_{L^2_\xi} \cdot 2^m \cdot 2^{-m} \cdot \Bigl\| \varphi_{[-\ell-1000,-10]}(\xi_4) \frac{1}{\pxithree \Phi_2} \Bigr\|_{L^1_{\xi_4}}  \\
  &\quad \quad \times \sup_{s \, \simeq \, 2^m} \, \bigl\| \xi_1 (\pxione \hatg)(s,\xi_1) \bigr\|_{L^1_{\xi_1}} \bigl\|\hatbarg(s,\xi_2)\bigr\|_{L^\infty_{\xi_2}} \bigl\| \hatg(s,\xi_4) \bigr\|_{L^\infty_{\xi_4}} \biggl\| \varphi_{>-\ell-1000}(\xi_3) \pxithree \biggl( \frac{\hatq(\xi_3)}{\xi_3} \biggr) \biggr\|_{L^1_{\xi_3}} \\
  &\lesssim 2^{-\frac12 \ell} \cdot 2^{-\frac12 \ell} \cdot 2^m \cdot 2^{-m} \cdot \ell \cdot m^3 \varepsilon^3 \cdot 2^{\ell} \lesssim m^4 \varepsilon^3.
 \end{aligned}
\end{equation*}
The bound for the third term $\calI_{2,m; \ell \leq \ell_1, k_3 > -\ell}^{\pvdots, 1, 3, (c)}(t,\xi)$ is similar.

Finally, in the configuration $|\xi_4| \gtrsim 2^{-10}$ we can proceed similarly and integrate by parts in $\xi_3$. The estimates are less tight since $|\pxithree \Phi_2| \simeq 1$ for $|\xi_4| \gtrsim 1$. We omit the straightforward details.

\medskip 
\noindent \underline{\it Case 4.3: $\ell = 0$.}
We assume without loss of generality that $\xi\ge0$.
In this  case $|\xi-\sqrt{3}|\gtrsim 2^{-100}$.  If $||\xi_1|-\sqrt{3}|\gtrsim 1$, i.e., $\ell_1\le 10$ say, then by an $L^2_x \times L^\infty_x \times L^\infty_x \times L^\infty_x$ estimate
\[
 \bigl\| \varphi_0^{(n)}(\xi) \calI_{2,m;  0\le \ell_1\le 10}^{\pvdots, 1}(t,\xi) \bigr\|_{L^2_\xi}\les m^3\varepsilon^3.
\]
So it suffices to consider $\ell_1\ge 10$.

\medskip
\noindent \underline{\it Subcase 4.3.1: $\ell=0$, $\ell_1\ge10$, $|\xi_3|\ll1$.}
Here we can proceed analogously to Subcase 4.2.1 and integrate by parts in $\xi_2$. Then we need to distinguish several subcases depending on the relative sizes of $|\xi|$ and $|\xi_2|$. We leave the details to the reader.

\medskip
\noindent \underline{\it Subcase 4.3.2: $\ell=0$, $\ell_1\ge10$, $k_3\ge -150$.}
This is the analogue of Subcase~4.2.2. As in that analysis, we distinguish relative sizes of $\xi_2$ and $\xi_4$.

\medskip
\noindent \underline{\it Subcase 4.3.2.1: $\ell=0$, $\ell_1\ge10$, $k_3\ge -150$, $k_4\le -300$, $k_2\le-200$.}  The analysis of Subcase~4.2.2.1 applies verbatim with $\ell=0$, and we leave the details to the reader.

\medskip
\noindent \underline{\it Subcase 4.3.2.2: $\ell=0$, $\ell_1\ge10$, $k_3\ge -150$, $k_4\le -300$, $k_2\ge-200$.} We proceed exactly as in Subcase~4.2.2.2 with one difference. In~\eqref{eq:pv12a}, we place $L^2_\xi$ onto $\hat{g}(s,\xi_4)$ and not onto the $\xi$-cutoff. For the estimate~\eqref{eq:Ipv12c} we first substitute $\xi_2\to \xi-\xi_2$ and $\xi_4\to \xi_2-\xi_1-\xi_3$. As a result, in the third line of~\eqref{eq:Ipv12c} we then have $L^2_{\xi_2}$ and $L^1_{\xi_4}$ leading to the same bound of $m^4\varepsilon^3$. Finally, the analogue of Subcase~4.2.2.3 is modified in the same fashion and we skip the details.

\medskip 
\noindent {\bf Step 5: Weighted energy estimate for the term $\calI_{2,m}^{\pvdots, 2}(t,\xi)$.}
For the term $\calI_{2,m}^{\pvdots, 2}(t,\xi)$ we seek to show for all $1 \leq m \leq n+5$ that
\begin{equation*}
 \sup_{0 \leq \ell \leq n} \, 2^{-\frac12 \ell} \bigl\| \varphi_\ell^{(n)}(\xi) \calI_{2,m}^{\pvdots, 2}(t,\xi) \bigr\|_{L^2_\xi} \lesssim m^5 \varepsilon^3.
\end{equation*}
As usual, we distinguish the cases $\ell = n$, $1 \leq \ell \leq n-1$, and $\ell = 0$.

\medskip 
\noindent \underline{\it Case 5.1: $\ell = n$.}
Analogously to Case 4.1 of the treatment of the term $\calI_{2,m}^{\pvdots, 1}(t,\xi)$ in the preceding Step~4, here we can just use H\"older's inequality in the output frequency variable $\xi$ and the bounds from Lemma~\ref{lem:g_bounds_repeated} to obtain the desired bound.

\medskip 
\noindent \underline{\it Case 5.2: $1 \leq \ell \leq n-1$.}
We again insert a smooth partition of unity to distinguish how close the frequency variable $\xi_2$ is to the problematic frequencies $\pm \sqrt{3}$. Without loss of generality, we assume that the output frequency is positive, i.e., $|\xi-\sqrt{3}| \simeq 2^{-\ell-100}$.
Then the tighter case is when $\xi_2 \approx - \sqrt{3}$. We correspondingly only provide the details for the case $\xi_2 \approx - \sqrt{3}$, and as usual, we right away consider the configuration $|\xi_2+\sqrt{3}| \ll |\xi-\sqrt{3}|$. We may also assume that $\ell+10 \leq m$.

Thus, we now carry out the energy estimate for the term
\begin{equation*}
 \begin{aligned}
  \calI_{2, m; \ell \leq \ell_2}^{\pvdots, 2}(t,\xi) &:= - \int_0^t \tau_m(s) \iiint e^{is\Phi_2} \varphi_\ell^{(n),+}(\xi) \varphi_{\geq \ell+10}^{(m),-}(\xi_2) \,  (\jxione^{-1} \xi_1) \hatg(s,\xi_1) \jxitwo (\pxitwo \hat{\barg})(s,\xi_2) \\
  &\qquad \qquad \qquad \qquad \qquad \qquad \qquad \qquad \qquad \qquad \quad \times \hatg(s,\xi_4) \, \pvdots \frac{\hatq(\xi_3)}{\xi_3} \, \ud \xi_1 \, \ud \xi_2 \, \ud \xi_3 \, \ud s.
 \end{aligned}
\end{equation*}
We will have to keep track of how close the frequency variable $\xi_2$ is to the problematic frequency $-\sqrt{3}$ in terms of
\begin{equation*}
 |\xi_2+\sqrt{3}| \simeq 2^{-\ell_2-100}, \quad 0 \leq \ell_2 \leq m.
\end{equation*}
Sometimes, we will also need to distinguish the absolute sizes of the frequencies
\begin{equation*}
 |\xi_1| \simeq 2^{k_1}, \quad |\xi_3| \simeq 2^{k_3}, \quad |\xi_4| \simeq 2^{k_4} \quad \text{for } k_1, k_3, k_4 \in \bbZ.
\end{equation*}
The two main subcases that we distinguish are (1) $|\xi_3| \lesssim 2^{-\ell-1000}$ and (2) $|\xi_3| \gtrsim 2^{-\ell-1000}$.
In the former subcase, the input and output variables are still approximately correlated, while in the latter subcase they are decorrelated.

\medskip

\noindent \underline{\it Subcase 5.2.1: $1 \leq \ell \leq n-1$, $\ell+10 \leq \ell_2 \leq m$, $k_3 \leq -\ell-1000$.}
We can proceed similarly to Subcase 4.2.2 in the proof of Proposition~\ref{prop:weighted_energy_est_delta_T2}, because the input and the output frequency variables are still approximately correlated owing to the assumption $k_3 \leq -\ell-1000$.
Since the second input $\hat{\barg}(s,\xi_2)$ is already differentiated, it is natural to try to integrate by parts in $\xi_1$. But $\pxione \Phi_2$ can vanish in this frequency configuration since
\begin{equation*}
 \pxione \Phi_2 = \frac{\xi_1}{\jxione} - \frac{\xi_4}{\jxifour} = 0 \quad \Leftrightarrow \quad \xi_1 - \xi_4 = 0 \quad \Leftrightarrow \quad \xi_1 = \frac12 (\xi-\xi_2-\xi_3),
\end{equation*}
which is possible here.
We therefore further distinguish relative to the size of
\begin{equation*}
 |\xi_1-\xi_4| \simeq 2^{-\ell_5}, \quad \ell_5 \in \bbZ,
\end{equation*}
and we separately treat the subcases (1) $\ell_5 \geq \ell+1000$ and (2) $\ell_5 < \ell+1000$.
In the latter subcase we can still integrate by parts in $\xi_1$.


\medskip

\noindent \underline{\it Subcase 5.2.1.1: $1 \leq \ell \leq n-1$, $\ell+10 \leq \ell_2 \leq m$, $k_3 \leq -\ell-1000$, $\ell_5 \geq \ell+1000$.}
Here we should be able to proceed similarly as in Subcase~4.2.1.1 in the proof of Proposition~\ref{prop:weighted_energy_est_delta_T2}, but to this end we need to further distinguish the subcases (1) $-2\ell-1000 \leq k_3 \leq -\ell-1000$ and (2) $k_3 \leq -2\ell-1000$.
This distinction arises naturally in \eqref{equ:calIpv2m2_subcase5211_phase_size} below.

Observe that under the assumptions $|\xi_2 + \sqrt{3}| \ll |\xi-\sqrt{3}| \simeq 2^{-\ell-100}$ and $|\xi_3| \lesssim 2^{-\ell-1000}$, the relation $\xi_1 + \xi_4 = \xi - \xi_2 -\xi_3$ gives
\begin{equation*}
  (\xi_1-\sqrt{3}) + (\xi_4-\sqrt{3}) = (\xi-\sqrt{3}) - (\xi_2+\sqrt{3}) - \xi_3 \simeq 2^{-\ell-100}.
\end{equation*}
Since $|\xi_1-\xi_4| \lesssim 2^{-\ell-1000}$, we must have
\begin{equation} \label{equ:calIpv2m2_subcase5211_xione_xi_four_close_to_xi}
 \Bigl| (\xi_j - \sqrt{3}) - \frac12 (\xi-\sqrt{3}) \Bigr| \lesssim 2^{-\ell-500}, \quad j = 1, 4.
\end{equation}
In particular, this means that $|\xi_j-\sqrt{3}| \simeq 2^{-\ell-100}$ for $j = 1, 4$.

\medskip
\noindent \underline{\it Subcase 5.2.1.1.1: $1 \leq \ell \leq n-1$, $\ell+10 \leq \ell_2 \leq m$, $-2\ell-1000 \leq k_3 \leq -\ell-1000$, $\ell_5 \geq \ell+1000$.}
Here we consider the energy estimate for the term
\begin{equation*}
 \begin{aligned}
  &\calI_{2,m;\ell\leq\ell_2,-2\ell \leq k_3 \leq -\ell}^{\pvdots, 2, \ell \leq \ell_5}(t,\xi) \\
  &:= - \int_0^t \tau_m(s) \iiint e^{is\Phi_2} \frakn(\xi,\xi_1,\xi_2,\xi_3) \, \hatg(s,\xi_1) \jxitwo (\pxitwo \hat{\barg})(s,\xi_2)  \hatg(s,\xi_4) \, \pvdots \frac{\hatq(\xi_3)}{\xi_3} \, \ud \xi_1 \, \ud \xi_2 \, \ud \xi_3 \, \ud s
 \end{aligned}
\end{equation*}
with
\begin{equation*}
 \frakn(\xi,\xi_1,\xi_2,\xi_3) := \varphi_\ell^{(n),+}(\xi) \varphi_{\geq \ell+10}^{(m),-}(\xi_2) \varphi_{[-2\ell-1000, -\ell-1000]}(\xi_3) \varphi_{\leq -\ell-1000}(\xi_1-\xi_4) \, (\jxione^{-1} \xi_1).
\end{equation*}
Note that we included $\jxione^{-1} \xi_1$ into the symbol $\frakn(\xi,\xi_1,\xi_2,\xi_3)$.
Since $|\xi_3| \gtrsim 2^{-2\ell-1000}$ we do not need to treat the integration with respect to $\xi_3$ in the $\pvdots$ sense. We can therefore integrate by parts in $\xi_3$, observing that $|\pxithree \Phi_2| \gtrsim 1$ in view of $\pxithree \Phi_2 = - \jxifour^{-1} \xi_4$ and $\xi_4 = \sqrt{3} + \calO(2^{-\ell-100})$ in this configuration. We find upon integrating by parts in $\xi_3$,
\begin{equation*}
 \begin{aligned}
  &\calI_{2,m;\ell\leq\ell_2,-2\ell \leq k_3 \leq -\ell}^{\pvdots, 2, \ell \leq \ell_5}(t,\xi) \\
  &= -i \int_0^t \tau_m(s) \frac{1}{s} \iiint e^{is\Phi_2} \pxithree \biggl( \frac{1}{\pxithree \Phi_2} \frakn(\xi,\xi_1,\xi_2,\xi_3) \biggr) \, \hatg(s,\xi_1) \jxitwo (\pxitwo \hat{\barg})(s,\xi_2) \\
  &\qquad \qquad \qquad \qquad \qquad \qquad \qquad \qquad \qquad \qquad \qquad \qquad \quad \times \hatg(s,\xi_4) \, \frac{\hatq(\xi_3)}{\xi_3} \, \ud \xi_1 \, \ud \xi_2 \, \ud \xi_3 \, \ud s \\
  &\quad +i \int_0^t \tau_m(s) \frac{1}{s} \iiint e^{is\Phi_2} \frac{1}{\pxithree \Phi_2} \frakn(\xi,\xi_1,\xi_2,\xi_3) \, \hatg(s,\xi_1) \jxitwo (\pxitwo \hat{\barg})(s,\xi_2) \\
  &\qquad \qquad \qquad \qquad \qquad \qquad \qquad \qquad \qquad \qquad \qquad \qquad \quad \times (\pxifour \hatg)(s,\xi_4) \, \frac{\hatq(\xi_3)}{\xi_3} \, \ud \xi_1 \, \ud \xi_2 \, \ud \xi_3 \, \ud s \\
  &\quad -i \int_0^t \tau_m(s) \frac{1}{s} \iiint e^{is\Phi_2} \frac{1}{\pxithree \Phi_2} \frakn(\xi,\xi_1,\xi_2,\xi_3) \, \hatg(s,\xi_1) \jxitwo (\pxitwo \hat{\barg})(s,\xi_2) \\
  &\qquad \qquad \qquad \qquad \qquad \qquad \qquad \qquad \qquad \qquad \qquad \qquad \quad \times \hatg(s,\xi_4) \, \pxithree \biggl( \frac{\hatq(\xi_3)}{\xi_3} \biggr) \, \ud \xi_1 \, \ud \xi_2 \, \ud \xi_3 \, \ud s \\
  &=: \calI_{2,m;\ell\leq\ell_2,-2\ell \leq k_3 \leq -\ell}^{\pvdots, 2, \ell \leq \ell_5, (a)}(t,\xi) +  \calI_{2,m;\ell\leq\ell_2,-2\ell \leq k_3 \leq -\ell}^{\pvdots, 2, \ell \leq \ell_5, (b)}(t,\xi) +  \calI_{2,m;\ell\leq\ell_2,-2\ell \leq k_3 \leq -\ell}^{\pvdots, 2, \ell \leq \ell_5, (c)}(t,\xi).
 \end{aligned}
\end{equation*}
For the first term on the right-hand side, using H\"older's inequality in the frequency variables and the bounds \eqref{equ:g_bound_pxiL1}, \eqref{equ:g_bound_Linftyxi}, we obtain
\begin{equation*}
 \begin{aligned}
  &2^{-\frac12 \ell} \bigl\| \calI_{2,m;\ell\leq\ell_2,-2\ell \leq k_3 \leq -\ell}^{\pvdots, 2, \ell \leq \ell_5, (a)}(t,\xi) \bigr\|_{L^2_\xi} \\
  &\lesssim 2^{-\frac12 \ell} \bigl\| \widetilde{\varphi}_\ell^{(n)}(\xi) \bigr\|_{L^2_\xi} \cdot 2^{2\ell} \cdot \sup_{s \, \simeq \, 2^m} \, \bigl\| \varphi_{\leq -\ell}(\xi_1-\sqrt{3}) \hatg(s,\xi_1) \bigr\|_{L^1_{\xi_1}} \bigl\| \jxitwo (\pxitwo \hat{\barg})(s,\xi_2) \bigr\|_{L^1_{\xi_2}} \\
  &\qquad \qquad \qquad \qquad \qquad \qquad \qquad \qquad \qquad \times \bigl\|\hatg(s,\xi_4)\bigr\|_{L^\infty_{\xi_4}} \biggl( \int_{|\xi_3| \gtrsim 2^{-2\ell-1000}} \frac{|\hatq(\xi_3)|}{|\xi_3|} \, \ud \xi_3 \biggr) \\
  &\lesssim 2^{-\frac12 \ell} \cdot 2^{-\frac12 \ell} \cdot 2^{2\ell} \cdot 2^{-\ell} \cdot \sup_{s \, \simeq \, 2^m} \, \bigl\|\hatg(s,\xi_1)\bigr\|_{L^\infty_{\xi_1}} \bigl\| \jxitwo (\pxitwo \hat{\barg})(s,\xi_2) \bigr\|_{L^1_{\xi_2}} \bigl\|\hatg(s,\xi_4)\bigr\|_{L^\infty_{\xi_4}} \cdot 2\ell \lesssim m^4 \varepsilon^3.
 \end{aligned}
\end{equation*}
The second term is simpler. For the third term we analogously find by H\"older's inequality in the frequency variables and by the bounds \eqref{equ:g_bound_pxiL1}, \eqref{equ:g_bound_Linftyxi} that
\begin{equation*}
 \begin{aligned}
  &2^{-\frac12 \ell} \bigl\| \calI_{2,m;\ell\leq\ell_2,-2\ell \leq k_3 \leq -\ell}^{\pvdots, 2, \ell \leq \ell_5, (c)}(t,\xi) \bigr\|_{L^2_\xi} \\
  &\lesssim 2^{-\frac12 \ell} \bigl\| \widetilde{\varphi}_\ell^{(n)}(\xi) \bigr\|_{L^2_\xi} \cdot \sup_{s \, \simeq \, 2^m} \, \bigl\| \varphi_{\leq -\ell}(\xi_1-\sqrt{3}) \hatg(s,\xi_1) \bigr\|_{L^1_{\xi_1}} \bigl\| \jxitwo (\pxitwo \hat{\barg})(s,\xi_2) \bigr\|_{L^1_{\xi_2}} \\
  &\qquad \qquad \qquad \qquad \qquad \qquad \qquad \qquad \qquad \times \bigl\|\hatg(s,\xi_4)\bigr\|_{L^\infty_{\xi_4}} \biggl( \int_{|\xi_3| \gtrsim 2^{-2\ell-1000}} \frac{1}{|\xi_3|^2} \, \ud \xi_3 \biggr) \\
  &\lesssim 2^{-\frac12 \ell} \cdot 2^{-\frac12 \ell} \cdot 2^{-\ell} \cdot \sup_{s \, \simeq \, 2^m} \, \bigl\|\hatg(s,\xi_1)\bigr\|_{L^\infty_{\xi_1}} \bigl\| \jxitwo (\pxitwo \hat{\barg})(s,\xi_2) \bigr\|_{L^1_{\xi_2}} \bigl\|\hatg(s,\xi_4)\bigr\|_{L^\infty_{\xi_4}} \cdot 2^{2\ell} \lesssim m^3 \varepsilon^3,
 \end{aligned}
\end{equation*}
which is acceptable.

\medskip
\noindent \underline{\it Subcase 5.2.1.1.2: $1 \leq \ell \leq n-1$, $\ell+10 \leq \ell_2 \leq m$, $k_3 \leq -2\ell-1000$, $\ell_5 \geq \ell+1000$.}
Now we turn to the energy estimate for the term
\begin{equation*}
 \begin{aligned}
  &\calI_{2,m;\ell\leq\ell_2, k_3 \leq -2\ell}^{\pvdots, 2, \ell \leq \ell_5}(t,\xi) \\
  &:= - \int_0^t \tau_m(s) \iiint e^{is\Phi_2} \frakn(\xi,\xi_1,\xi_2,\xi_3) \, \hatg(s,\xi_1) (\pxitwo \hat{\barg})(s,\xi_2) \hatg(s,\xi_4) \, \pvdots \frac{\hatq(\xi_3)}{\xi_3} \, \ud \xi_1 \, \ud \xi_2 \, \ud \xi_3 \, \ud s \\
 \end{aligned}
\end{equation*}
with
\begin{equation*}
 \frakn(\xi,\xi_1,\xi_2,\xi_3) := \varphi_\ell^{(n),+}(\xi) \varphi_{\geq \ell+10}^{(m),-}(\xi_2) \varphi_{\leq -2\ell-1000}(\xi_3) \varphi_{\leq -\ell-1000}(\xi_1-\xi_4) \, (\jxione^{-1} \xi_1) \, \jxitwo.
\end{equation*}
Note that we included $\jxione^{-1} \xi_1$ and $\jxitwo$ into the symbol $\frakn(\xi,\xi_1,\xi_2,\xi_3)$.
Here our only resort is to integrate by parts in time, similarly to Subcase 4.2.2.1 from the proof of
Proposition~\ref{prop:weighted_energy_est_delta_T2}. It turns out that it is preferable to first integrate by parts in $\xi_2$, because this leads to a better balance between all inputs when we later insert the equation for $\partial_s \hatg(s)$.
Integrating by parts in $\xi_2$, we get
\begin{equation*}
 \begin{aligned}
  &\calI_{2,m;\ell\leq\ell_2, k_3 \leq -2\ell}^{\pvdots, 2, \ell \leq \ell_5}(t,\xi) \\
  &= - \int_0^t \tau_m(s) \iiint e^{is\Phi_2} \frakn(\xi,\xi_1,\xi_2,\xi_3) \, \hatg(s,\xi_1) \hat{\barg}(s,\xi_2) (\pxifour \hatg)(s,\xi_4) \, \pvdots \frac{\hatq(\xi_3)}{\xi_3} \, \ud \xi_1 \, \ud \xi_2 \, \ud \xi_3 \, \ud s \\
  &\quad \, \, + \int_0^t \tau_m(s) \iiint e^{is\Phi_2} \pxitwo \bigl( \frakn(\xi,\xi_1,\xi_2,\xi_3) \bigr) \, \hatg(s,\xi_1) \hat{\barg}(s,\xi_2) \hatg(s,\xi_4) \, \pvdots \frac{\hatq(\xi_3)}{\xi_3} \, \ud \xi_1 \, \ud \xi_2 \, \ud \xi_3 \, \ud s \\
  &\quad \, \, + i \int_0^t \tau_m(s) \cdot s \iiint e^{is\Phi_2} (\pxitwo \Phi_2) \frakn(\xi,\xi_1,\xi_2,\xi_3) \, \hatg(s,\xi_1) \hat{\barg}(s,\xi_2) \hatg(s,\xi_4) \, \pvdots \frac{\hatq(\xi_3)}{\xi_3} \, \ud \xi_1 \, \ud \xi_2 \, \ud \xi_3 \, \ud s \\
  &=: \calI_{2,m;\ell\leq\ell_2, k_3 \leq -2\ell}^{\pvdots, 2, \ell \leq \ell_5, (a)}(t,\xi) + \calI_{2,m;\ell\leq\ell_2, k_3 \leq -2\ell}^{\pvdots, 2, \ell \leq \ell_5, (b)}(t,\xi) + \calI_{2,m;\ell\leq\ell_2, k_3 \leq -2\ell}^{\pvdots, 2, \ell \leq \ell_5, (c)}(t,\xi).
 \end{aligned}
\end{equation*}
To estimate the first term $\calI_{2,m;\ell\leq\ell_2, k_3 \leq -2\ell}^{\pvdots, 2, \ell \leq \ell_5, (a)}(t,\xi)$ we integrate by parts using the identity
\begin{equation}
 e^{is\Phi_2} = \frac{1}{is} \frac{1}{(\pxione-\pxitwo)\Phi_2} (\pxione-\pxitwo)\bigl(e^{is\Phi_2}\bigr).
\end{equation}
Crucially relying on the simple fact that $(\pxione - \pxitwo) \pxifour \hatg(s,\xi_4) = 0$, we have
\begin{equation*}
 \begin{aligned}
  &\calI_{2,m;\ell\leq\ell_2, k_3 \leq -2\ell}^{\pvdots, 2, \ell \leq \ell_5, (a)}(t,\xi) \\
  &= -i \int_0^t \tau_m(s) \cdot \frac{1}{s} \iiint e^{is\Phi_2} \frac{\frakn(\xi,\xi_1,\xi_2,\xi_3)}{(\pxione-\pxitwo)\Phi_2} \, (\pxione \hatg)(s,\xi_1) \hat{\barg}(s,\xi_2) (\pxifour \hatg)(s,\xi_4) \, \pvdots \frac{\hatq(\xi_3)}{\xi_3} \, \ud \xi_1 \, \ud \xi_2 \, \ud \xi_3 \, \ud s \\
  &\quad +i \int_0^t \tau_m(s) \cdot \frac{1}{s} \iiint e^{is\Phi_2} \frac{\frakn(\xi,\xi_1,\xi_2,\xi_3)}{(\pxione-\pxitwo)\Phi_2} \, \hatg(s,\xi_1) (\pxitwo \hat{\barg})(s,\xi_2) (\pxifour \hatg)(s,\xi_4) \, \pvdots \frac{\hatq(\xi_3)}{\xi_3} \, \ud \xi_1 \, \ud \xi_2 \, \ud \xi_3 \, \ud s \\
  &\quad -i \int_0^t \tau_m(s) \cdot \frac{1}{s} \iiint e^{is\Phi_2} \bigl(\pxione - \pxitwo) \biggl( \frac{\frakn(\xi,\xi_1,\xi_2,\xi_3)}{(\pxione-\pxitwo)\Phi_2} \biggr) \, \hatg(s,\xi_1) \hat{\barg}(s,\xi_2) (\pxifour \hatg)(s,\xi_4) \\
  &\qquad \qquad \qquad \qquad \qquad \qquad \qquad \qquad \qquad \qquad \qquad \qquad \qquad \qquad \qquad \qquad \times \pvdots \frac{\hatq(\xi_3)}{\xi_3} \, \ud \xi_1 \, \ud \xi_2 \, \ud \xi_3 \, \ud s \\
  &=: \calI_{2,m;\ell\leq\ell_2, k_3 \leq -2\ell}^{\pvdots, 2, \ell \leq \ell_5, (a), 1}(t,\xi) + \calI_{2,m;\ell\leq\ell_2, k_3 \leq -2\ell}^{\pvdots, 2, \ell \leq \ell_5, (a), 2}(t,\xi) + \calI_{2,m;\ell\leq\ell_2, k_3 \leq -2\ell}^{\pvdots, 2, \ell \leq \ell_5, (a), 3}(t,\xi).
 \end{aligned}
\end{equation*}
By Taylor expansion and \eqref{equ:calIpv2m2_subcase5211_xione_xi_four_close_to_xi},
\begin{equation*}
 (\pxione-\pxitwo) \Phi_2 = \frac{\xi_1}{\jxione} + \frac{\xi_2}{\jxitwo} = \frac{1}{16}(\xi-\sqrt{3}) + \frac18 (\xi_2+\sqrt{3}) + \calO\bigl(2^{-2(\ell+100)}\bigr) + \calO\bigl(2^{-\ell-500}\bigr) \simeq 2^{-\ell-100}.
\end{equation*}
We conclude
\begin{equation*}
 \biggl| \pxi^{\kappa} \pxione^{\kappa_1} \pxitwo^{\kappa_2} \pxithree^{\kappa_3} \biggl( \frac{\frakn(\xi,\xi_1,\xi_2,\xi_3)}{(\pxione-\pxitwo)\Phi_2} \biggr) \biggr| \lesssim 2^\ell \cdot 2^{(\kappa+\kappa_1+\kappa_2)\ell} 2^{\kappa_3 2\ell},
\end{equation*}
as well as
\begin{equation*}
 \biggl| \pxi^{\kappa} \pxione^{\kappa_1} \pxitwo^{\kappa_2} \pxithree^{\kappa_3} \cdot \bigl( \pxione - \pxitwo \bigr) \biggl( \frac{\frakn(\xi,\xi_1,\xi_2,\xi_3)}{(\pxione-\pxitwo)\Phi_2} \biggr) \biggr| \lesssim 2^{2\ell} \cdot 2^{(\kappa+\kappa_1+\kappa_2)\ell} 2^{\kappa_3 2\ell}.
\end{equation*}
It follows that
\begin{equation} \label{equ:calIpv2m2_subcase5211_frakn_bound1}
 \biggl\| \calF^{-1} \biggl[ \frac{\frakn(\xi,\xi_1,\xi_2,\xi_3)}{(\pxione-\pxitwo)\Phi_2} \biggr] \biggr\|_{L^1(\bbR^4)} \lesssim 2^\ell,
\end{equation}
and
\begin{equation} \label{equ:calIpv2m2_subcase5211_frakn_bound2}
 \biggl\| \calF^{-1} \biggl[ \bigl( \pxione - \pxitwo \bigr) \biggl( \frac{\frakn(\xi,\xi_1,\xi_2,\xi_3)}{(\pxione-\pxitwo)\Phi_2} \biggr) \biggr] \biggr\|_{L^1(\bbR^4)} \lesssim 2^{2\ell}.
\end{equation}
Thus, using Lemma~\ref{lem:frakn_for_pv} with \eqref{equ:calIpv2m2_subcase5211_frakn_bound1} and the bounds \eqref{equ:g_bound_pxiL1} and \eqref{equ:g_bound_Linftyxi}, we obtain that
\begin{equation*}
 \begin{aligned}
  &2^{-\frac12 \ell} \bigl\| \calI_{2,m;\ell\leq\ell_2, k_3 \leq -2\ell}^{\pvdots, 2, \ell \leq \ell_5, (a), 1}(t,\xi) \bigr\|_{L^2_\xi} \\
  &\lesssim 2^{-\frac12 \ell} \cdot 2^m \cdot 2^{-m} \cdot \biggl\| \calF^{-1} \biggl[ \frac{\frakn(\xi,\xi_1,\xi_2,\xi_3)}{(\pxione-\pxitwo)\Phi_2} \biggr] \biggr\|_{L^1(\bbR^4)} \\
  &\qquad \qquad \times \sup_{s \, \simeq \, 2^m} \, \bigl\| e^{is\jD} x g(s) \bigr\|_{L^\infty_x} \bigl\| \varphi_{\leq -\ell}(D+\sqrt{3}) g(s) \bigr\|_{L^2_x} \bigl\| e^{is\jD} x g(s) \bigr\|_{L^\infty_x} \\
  &\lesssim 2^{-\frac12 \ell} \cdot 2^m \cdot 2^{-m} \cdot 2^\ell \cdot \bigl\| \pxione \hatg(s,\xi_1)\bigr\|_{L^1_{\xi_1}} \cdot 2^{-\frac12\ell} \bigl\| \hat{\barg}(s,\xi_2) \bigr\|_{L^\infty_{\xi_2}} \bigl\| \pxifour \hatg(s,\xi_4)\bigr\|_{L^1_{\xi_4}} \lesssim m^3 \varepsilon^3.
 \end{aligned}
\end{equation*}
The energy estimates for the terms $\calI_{2,m;\ell\leq\ell_2, k_3 \leq -2\ell}^{\pvdots, 2, \ell \leq \ell_5, (a), 2}(t,\xi)$ and $\calI_{2,m;\ell\leq\ell_2, k_3 \leq -2\ell}^{\pvdots, 2, \ell \leq \ell_5, (a), 3}(t,\xi)$ are analogous. This finishes the discussion fo the first term $\calI_{2,m;\ell\leq\ell_2, k_3 \leq -2\ell}^{\pvdots, 2, \ell \leq \ell_5, (a)}(t,\xi)$, and the estimates for the second term $\calI_{2,m;\ell\leq\ell_2, k_3 \leq -2\ell}^{\pvdots, 2, \ell \leq \ell_5, (b)}(t,\xi)$ are analogous.

We can therefore now turn to the weighted energy estimate for the third term $\calI_{2,m;\ell\leq\ell_2, k_3 \leq -2\ell}^{\pvdots, 2, \ell \leq \ell_5, (c)}(t,\xi)$, which is the most delicate one. Here we integrate by parts in time,
\begin{equation*}
 \begin{aligned}
  &\calI_{2,m;\ell\leq\ell_2, k_3 \leq -2\ell}^{\pvdots, 2, \ell \leq \ell_5, (c)}(t,\xi) \\
  &= - \int_0^t \tau_m(s) \cdot s \iiint e^{is\Phi_2} \frac{(\pxitwo \Phi_2)}{\Phi_2} \frakn(\xi,\xi_1,\xi_2,\xi_3) \, \ps \bigl( \hatg(s,\xi_1) \hat{\barg}(s,\xi_2) \hatg(s,\xi_4) \bigr) \, \pvdots \frac{\hatq(\xi_3)}{\xi_3} \, \ud \xi_1 \, \ud \xi_2 \, \ud \xi_3 \, \ud s \\
  &\quad - \int_0^t \ps \bigl( \tau_m(s) \cdot s \bigr) \iiint e^{is\Phi_2} \frac{(\pxitwo \Phi_2)}{\Phi_2} \frakn(\xi,\xi_1,\xi_2,\xi_3) \, \hatg(s,\xi_1) \hat{\barg}(s,\xi_2) \hatg(s,\xi_4) \, \pvdots \frac{\hatq(\xi_3)}{\xi_3} \, \ud \xi_1 \, \ud \xi_2 \, \ud \xi_3 \, \ud s \\
  &\quad + \tau_m(s) \cdot s \iiint e^{is\Phi_2} \frac{(\pxitwo \Phi_2)}{\Phi_2} \frakn(\xi,\xi_1,\xi_2,\xi_3) \, \hatg(s,\xi_1) \hat{\barg}(s,\xi_2) \hatg(s,\xi_4) \, \pvdots \frac{\hatq(\xi_3)}{\xi_3} \, \ud \xi_1 \, \ud \xi_2 \, \ud \xi_3 \biggr|_{s=0}^{s=t} \\
  &=: \calI_{2,m;\ell\leq\ell_2, k_3 \leq -2\ell}^{\pvdots, 2, \ell \leq \ell_5, (c), 1}(t,\xi) + \calI_{2,m;\ell\leq\ell_2, k_3 \leq -2\ell}^{\pvdots, 2, \ell \leq \ell_5, (c), 2}(t,\xi) + \calI_{2,m;\ell\leq\ell_2, k_3 \leq -2\ell}^{\pvdots, 2, \ell \leq \ell_5, (c), 3}(t,\xi).
 \end{aligned}
\end{equation*}
We only discuss the weighted energy estimate for the first term $\calI_{2,m;\ell\leq\ell_2, k_3 \leq -2\ell}^{\pvdots, 2, \ell \leq \ell_5, (c), 1}(t,\xi)$ on the right-hand side, the treatment of the other terms being analogous and simpler.
To obtain acceptable bounds, we need precise control of the size of the phase $\Phi_2(\xi,\xi_1,\xi_2,\xi_3) = -\jxi + \jxione - \jxitwo + \jxifour$ with $\xi_4 := \xi-\xi_1-\xi_2-\xi_3$ in this frequency configuration. Recall that here
\begin{equation*}
 |\xi-\sqrt{3}| \simeq 2^{-\ell-100}, \quad |\xi_2+\sqrt{3}| \ll 2^{-\ell-100}, \quad |\xi_3| \lesssim 2^{-2\ell-1000},
\end{equation*}
as well as
\begin{equation*}
 \Bigl| (\xi_j-\sqrt{3}) - \frac12 (\xi-\sqrt{3}) \Bigr| \lesssim 2^{-\ell-500}, \quad |\xi_j - \sqrt{3}| \simeq 2^{-\ell-100},  \quad j = 1, 4.
\end{equation*}
By Taylor expansion around $\xi \approx \sqrt{3}$, respectively around $\xi_4 \approx \sqrt{3}$, we have
\begin{equation*}
 \pxi \Phi_2 = - \frac{\xi}{\jxi} + \frac{\xi_4}{\jxifour} = - \frac18 (\xi-\sqrt{3}) + \frac18 (\xi_4-\sqrt{3}) + \calO\bigl( 2^{-2(\ell+100)} \bigr),
\end{equation*}
and thus obtain that $|\pxi \Phi_2| \lesssim 2^{-\ell}$ in this frequency configuration. We obtain the same bounds on $|\pxione \Phi_2|$ and $|\pxitwo \Phi_2|$, while
\begin{equation*}
 \pxithree \Phi_2 = -\frac{\xi_4}{\jxifour} = -\frac{\sqrt{3}}{2} + \calO\bigl( 2^{-\ell-100} \bigr).
\end{equation*}
Thus, on the support of $\frakn(\xi,\xi_1,\xi_2,\xi_3)$ we have
\begin{equation} \label{equ:calIpv2m2_subcase5211_pxiderivatives_phase_bounds}
 |\pxi \Phi_2| + |\pxione \Phi_2| + |\pxitwo \Phi_2| \lesssim 2^{-\ell}, \quad |\pxithree \Phi_2| \lesssim 1.
\end{equation}
Next, we compute
\begin{equation*}
 \begin{aligned}
  \bigl[\mathrm{Hess} \, (\Phi_2)\bigr](\sqrt{3},\sqrt{3},-\sqrt{3}, 0) = \frac18 \begin{bmatrix} 0 & -1 & -1 & -1 \\ -1 & 2 & 1 & 1 \\ -1 & 1 & 0 & 1 \\ -1 & 1 & 1 & 1 \end{bmatrix}.
 \end{aligned}
\end{equation*}
Since $\Phi_2(\sqrt{3},\sqrt{3},-\sqrt{3},0) = 0$ and
\begin{equation*}
 \pxi\Phi_2(\sqrt{3},\sqrt{3},-\sqrt{3},0) = \pxione \Phi_2(\sqrt{3},\sqrt{3},-\sqrt{3},0) = \pxitwo \Phi_2(\sqrt{3},\sqrt{3},-\sqrt{3},0) = 0,
\end{equation*}
while
\begin{equation*}
 \pxithree \Phi_2(\sqrt{3},\sqrt{3},-\sqrt{3},0) = -\frac{\sqrt{3}}{2},
\end{equation*}
proceeding analogously as in the derivation of \eqref{equ:phi2_delta_subcase4221_est2},
we conclude by Taylor expansion of the phase $\Phi_2(\xi,\xi_1,\xi_2,\xi_3)$ around $(\xi,\xi_1,\xi_2,\xi_3) = (\sqrt{3},\sqrt{3},-\sqrt{3},0)$ that
\begin{equation} \label{equ:calIpv2m2_subcase5211_phase_size}
 \Phi_2(\xi,\xi_1,\xi_2,\xi_3) = -\frac{\sqrt{3}}{2} \xi_3 -\frac{1}{32} (\xi-\sqrt{3})^2 + \calO \bigl( 2^{-3(\ell+100)} \bigr) \simeq 2^{-2(\ell+100)}.
\end{equation}
The restriction to $|\xi_3| \lesssim 2^{-2\ell-1000}$ naturally entered here.
From \eqref{equ:calIpv2m2_subcase5211_pxiderivatives_phase_bounds} and \eqref{equ:calIpv2m2_subcase5211_phase_size} we conclude that
\begin{equation*}
 \biggl| \pxi^\kappa \pxione^{\kappa_1} \pxitwo^{\kappa_2} \pxithree^{\kappa_3} \biggl( \frac{\pxitwo \Phi_2}{\Phi_2} \frakn(\xi,\xi_1,\xi_2,\xi_3) \biggr) \biggr| \lesssim 2^\ell 2^{(\kappa + \kappa_1 + \kappa_2)\ell} 2^{\kappa_3 2\ell},
\end{equation*}
and thus,
\begin{equation} \label{equ:calIpv2m2_subcase5211_frakn_bound3}
 \biggl\| \calF^{-1} \biggl[  \frac{\pxitwo \Phi_2}{\Phi_2} \frakn(\xi,\xi_1,\xi_2,\xi_3) \biggr] \biggr\|_{L^1(\bbR^4)} \lesssim 2^\ell.
\end{equation}
Now we are prepared to carry out the weighted energy estimate for the term
$\calI_{2,m;\ell\leq\ell_2, k_3 \leq -2\ell}^{\pvdots, 2, \ell \leq \ell_5, (c), 1}(t,\xi)$.
We only write out the details when the time derivative falls onto the first input $\hatg(s,\xi_1)$, the other cases being identical.
Inserting the equation \eqref{equ:decomposition_FT_pt_hatg} for $\ps\hatg(s,\xi_1)$, using Lemma~\ref{lem:frakn_for_pv} with \eqref{equ:calIpv2m2_subcase5211_frakn_bound3}, along with the bounds
\eqref{equ:g_bound_Linftyxi}, \eqref{equ:g_bound_dispersive_est}, \eqref{equ:decomposition_FT_pt_hatg_Ncdecay} we find that
\begin{equation*}
 \begin{aligned}
  &2^{-\frac12 \ell} \biggl\| \int_0^t \tau_m(s) \cdot s \iiint e^{is\Phi_2} \frac{(\pxitwo \Phi_2)}{\Phi_2} \frakn(\xi,\xi_1,\xi_2,\xi_3) \, (2i\jxione)^{-1} e^{-is\jxione} \\
  &\qquad \qquad \qquad \times \Bigl( \widehat{\beta}(\xi_1) \bigl( v(s,0) + \barv(s,0) \bigr)^2 + \widehat{\calN}_c(s,\xi_1) \Bigr) \hat{\barg}(s,\xi_2) \hatg(s,\xi_4) \, \pvdots \frac{\hatq(\xi_3)}{\xi_3} \, \ud \xi_1 \, \ud \xi_2 \, \ud \xi_3 \, \ud s \biggr\|_{L^2_\xi} \\
  &\lesssim 2^{-\frac12 \ell} \cdot 2^{2m} \cdot \biggl\| \calF^{-1} \biggl[  \frac{\pxitwo \Phi_2}{\Phi_2} \frakn(\xi,\xi_1,\xi_2,\xi_3) \biggr] \biggr\|_{L^1(\bbR^4)} \\
  &\qquad \qquad \qquad \qquad \qquad \times \sup_{s \, \simeq \, 2^m} \, \biggl( \bigl\| \varphi_{\leq -\ell+100}(D-\sqrt{3}) (2i\jD)^{-1} \beta \bigr\|_{L^2_x} |v(s,0)|^2 \bigl\| e^{is\jD} g \bigr\|_{L^\infty_x}^2 \\
  &\qquad \qquad \qquad \qquad \qquad \qquad \qquad  \qquad + \bigl\| \calN_c(s)\bigr\|_{L^\infty_x} \bigl\| \varphi_{\leq -\ell+100}(D+\sqrt{3}) g(s) \bigr\|_{L^2_x} \bigl\| e^{is\jD} g(s)\bigr\|_{L^\infty_x} \biggr) \\
  &\lesssim 2^{-\frac12 \ell} \cdot 2^{2m} \cdot 2^\ell \cdot \Bigl( 2^{-\frac12 \ell} \cdot 2^{-2m} \cdot m^4 \varepsilon^4 + 2^{-\frac32 m} \cdot m^3 \varepsilon^2 \cdot 2^{-\frac12 \ell} \cdot m \varepsilon \cdot 2^{-\frac12 m} \cdot m \varepsilon \Bigr) \lesssim m^5 \varepsilon^4,
 \end{aligned}
\end{equation*}
which is acceptable.

\medskip
\noindent \underline{\it Subcase 5.2.1.2: $1 \leq \ell \leq n-1$, $\ell+10 \leq \ell_2 \leq m$, $k_3 \leq -\ell-1000$, $\ell_5 < \ell+1000$.}
In this configuration $\pxione \Phi_2$ cannot vanish, so we integrate by parts in $\xi_1$.
In the process we need to further distinguish the subcases (1) $\ell_1 \geq \ell-10$, (2) $1 \leq \ell_1 \leq \ell-10$, (3) $\ell_1 = 0$ with $k_1 \leq 10$, and (4) $\ell_1 = 0$ with $k_1 > 10$.
We may also assume that $\ell > 10$, since the scenario $1 \leq \ell \leq 10$ can be subsumed into Case~5.3 below.
For $\ast \in \{1, 2, 3, 4\}$, we now consider
\begin{equation*}
\begin{aligned}
 &\calI_{2,m;\ell\leq\ell_2,k_3\leq -\ell}^{\pvdots, 2, \ell \geq \ell_5, \ast}(t,\xi) \\
 &:=- \int_0^t \tau_m(s) \iiint e^{is\Phi_2} \frakn_\ast(\xi,\xi_1,\xi_2,\xi_3) \,  (\jxione^{-1} \xi_1) \hatg(s,\xi_1) \jxitwo (\pxitwo \hat{\barg})(s,\xi_2) \\
 &\qquad \qquad \qquad \qquad \qquad \qquad \qquad \qquad \qquad \qquad \quad \times \hatg(s,\xi_4) \, \pvdots \frac{\hatq(\xi_3)}{\xi_3} \, \ud \xi_1 \, \ud \xi_2 \, \ud \xi_3 \, \ud s \\
 &= - \int_0^t \tau_m(s) \frac{1}{s} \iiint e^{is\Phi_2} \frac{1}{\pxione \Phi_2} \frakn_\ast(\xi,\xi_1,\xi_2,\xi_3) \,  (\jxione^{-1} \xi_1) (\pxione \hatg)(s,\xi_1) \jxitwo (\pxitwo \hat{\barg})(s,\xi_2) \\
 &\qquad \qquad \qquad \qquad \qquad \qquad \qquad \qquad \qquad \qquad \quad \times \hatg(s,\xi_4) \, \pvdots \frac{\hatq(\xi_3)}{\xi_3} \, \ud \xi_1 \, \ud \xi_2 \, \ud \xi_3 \, \ud s \\
 &\quad + \int_0^t \tau_m(s) \frac{1}{s} \iiint e^{is\Phi_2} \frac{1}{\pxione \Phi_2} \frakn_\ast(\xi,\xi_1,\xi_2,\xi_3) \,  (\jxione^{-1} \xi_1) \hatg(s,\xi_1) \jxitwo (\pxitwo \hat{\barg})(s,\xi_2) \\
 &\qquad \qquad \qquad \qquad \qquad \qquad \qquad \qquad \qquad \qquad \quad \times (\pxifour \hatg)(s,\xi_4) \, \pvdots \frac{\hatq(\xi_3)}{\xi_3} \, \ud \xi_1 \, \ud \xi_2 \, \ud \xi_3 \, \ud s \\
 &\quad - \int_0^t \tau_m(s) \frac{1}{s} \iiint e^{is\Phi_2} \pxione \biggl( \frac{1}{\pxione \Phi_2} \frakn_\ast(\xi,\xi_1,\xi_2,\xi_3) \biggr) \, (\jxione^{-1} \xi_1) \hatg(s,\xi_1) \jxitwo (\pxitwo \hat{\barg})(s,\xi_2) \\
 &\qquad \qquad \qquad \qquad \qquad \qquad \qquad \qquad \qquad \qquad \quad \times \hatg(s,\xi_4) \, \pvdots \frac{\hatq(\xi_3)}{\xi_3} \, \ud \xi_1 \, \ud \xi_2 \, \ud \xi_3 \, \ud s \\
 &\quad + \bigl\{ \text{lower order terms} \bigr\} \\
 &=: \calI_{2,m;\ell\leq\ell_2,k_3\leq -\ell}^{\pvdots, 2, \ell \geq \ell_5, \ast, (a)}(t,\xi) + \calI_{2,m;\ell\leq\ell_2,k_3\leq -\ell}^{\pvdots, 2, \ell \geq \ell_5, \ast, (b)}(t,\xi) + \calI_{2,m;\ell\leq\ell_2,k_3\leq -\ell}^{\pvdots, 2, \ell \geq \ell_5, \ast, (c)}(t,\xi) + \{ \text{lower order terms} \}.
\end{aligned}
\end{equation*}
The lower order terms arise when the derivative $\pxione$ falls onto $\jxione^{-1} \xi_1$, which we ignore in the following.
Note that the first and the second term on the right-hand side are symmetric, so it suffices to carry out the details of the weighted energy estimates for the first and the third term.
We record that
\begin{equation} \label{equ:calIpv2m2_subcase5212_pxi_1Phi2_inverse}
 \frac{1}{\pxione \Phi_2} = \frac{1}{\xi_1-\xi_4} \frac{\jxione \jxifour (\xi_1 \jxifour + \xi_4 \jxione)}{\xi_1+\xi_4}
\end{equation}
and that
\begin{equation} \label{equ:calIpv2m2_subcase5212_two_pxi_1Phi2}
 \pxione^2 \Phi_2 = \frac{1}{\jxione^3} + \frac{1}{\jxifour^3}.
\end{equation}
Moreover, we observe that in the current frequency configuration
\begin{equation} \label{equ:calIpv2m2_subcase5212_xione_plus_xifour}
 \xi_1 + \xi_4 = \xi - \xi_2 - \xi_3 = 2\sqrt{3} + (\xi-\sqrt{3}) - (\xi_2+\sqrt{3}) - \xi_3 = 2\sqrt{3} + \calO\bigl(2^{-\ell-100}\bigr).
\end{equation}

\noindent \underline{\it Subcase 5.2.1.2.1: $1 \leq \ell \leq n-1$, $\ell+10 \leq \ell_2 \leq m$, $k_3 \leq -\ell-1000$, $\ell_5 < \ell+1000$, $\ell_1 > \ell-10$.}
Say $\xi_1 \approx \sqrt{3}$, i.e., $|\xi_1 - \sqrt{3}| \lesssim 2^{-\ell-90}$, the other case $\xi_1 \approx -\sqrt{3}$ being analogous.
So we consider
\begin{equation*}
 \frakn_1(\xi,\xi_1,\xi_2,\xi_3) := \varphi_\ell^{(n),+}(\xi) \varphi_{\geq \ell-10}^{(m),+}(\xi_1) \varphi_{\geq \ell+10}^{(m),-}(\xi_2) \varphi_{\leq -\ell-1000}(\xi_3) \varphi_{>-\ell-1000}(\xi_1-\xi_4).
\end{equation*}
Since $|\xi_1-\xi_4| \gtrsim 2^{-\ell-1000}$ by assumption and in view of \eqref{equ:calIpv2m2_subcase5212_xione_plus_xifour}, we infer from  \eqref{equ:calIpv2m2_subcase5212_pxi_1Phi2_inverse} that $|\pxione \Phi_2| \gtrsim 2^{-\ell}$ in this configuration, and thus
\begin{equation*}
 \biggl| \pxi^\kappa \pxione^{\kappa_1} \pxitwo^{\kappa_2} \pxithree^{\kappa_3} \biggl( \frac{1}{\pxione \Phi_2} \frakn_1(\xi, \xi_1, \xi_2, \xi_3) \biggr) \biggr| \lesssim 2^\ell 2^{(\kappa+\kappa_1+\kappa_2+\kappa_3)\ell}.
\end{equation*}
In view of \eqref{equ:calIpv2m2_subcase5212_two_pxi_1Phi2}, we also have
\begin{equation*}
 \biggl| \pxi^\kappa \pxione^{\kappa_1} \pxitwo^{\kappa_2} \pxithree^{\kappa_3} \cdot \pxione \biggl( \frac{1}{\pxione \Phi_2} \frakn_1(\xi, \xi_1, \xi_2, \xi_3) \biggr) \biggr| \lesssim 2^{2\ell} 2^{(\kappa+\kappa_1+\kappa_2+\kappa_3)\ell}.
\end{equation*}
It follows that
\begin{equation} \label{equ:calIpv2m2_subcase52121_frakn_bound1}
 \biggl\| \calF^{-1} \biggl[ \frac{1}{\pxione \Phi_2} \frakn_1(\xi, \xi_1, \xi_2, \xi_3) \biggr] \biggr\|_{L^1(\bbR^4)} \lesssim 2^\ell,
\end{equation}
and
\begin{equation} \label{equ:calIpv2m2_subcase52121_frakn_bound2}
 \biggl\| \calF^{-1} \biggl[ \pxione \biggl( \frac{1}{\pxione \Phi_2} \frakn_1(\xi, \xi_1, \xi_2, \xi_3) \biggr) \biggr] \biggr\|_{L^1(\bbR^4)} \lesssim 2^{2\ell}.
\end{equation}
From $\xi_4 = \sqrt{3} + (\xi-\sqrt{3}) - (\xi_1-\sqrt{3}) - (\xi_2 +\sqrt{3}) - \xi_3$ we infer under the frequency restrictions in this configuration that $|\xi_4 - \sqrt{3}| \lesssim 2^{-\ell-90}$.
Exploiting this additional frequency localization and using Lemma~\ref{lem:frakn_for_pv} with \eqref{equ:calIpv2m2_subcase52121_frakn_bound1} and the bounds \eqref{equ:g_bound_pxiL1}, \eqref{equ:g_bound_Linftyxi}, we thus obtain for $1 \leq m \leq n+5$,
\begin{equation*}
 \begin{aligned}
  &2^{-\frac12 \ell} \bigl\| \calI_{2,m;\ell\leq\ell_2,k_3\leq -\ell}^{\pvdots, 2, \ell \geq \ell_5, 1, (a)}(t,\xi) \bigr\|_{L^2_\xi} \\
  &\lesssim 2^{-\frac12 \ell} \cdot 2^m \cdot 2^{-m} \cdot \biggl\| \calF^{-1} \biggl[ \frac{1}{\pxione \Phi_2} \frakn_1(\xi, \xi_1, \xi_2, \xi_3) \biggr] \biggr\|_{L^1(\bbR^4)} \\
  &\quad \quad \times \sup_{s \simeq 2^m} \, \bigl\| e^{is\jD} \jD^{-1} D x g(s) \bigr\|_{L^\infty_x} \bigl\| e^{is\jD} \jD x g(s) \bigr\|_{L^\infty_x} \bigl\| \varphi_{\leq -\ell}(D-\sqrt{3}) g(s) \bigr\|_{L^2_x} \\
  &\lesssim 2^{-\frac12 \ell} \cdot 2^m \cdot 2^{-m} \cdot 2^\ell \cdot \sup_{s \simeq 2^m} \, \bigl\| \jxione^{-1} \xi_1 \pxione \hatg(s,\xi_1) \bigr\|_{L^1_{\xi_1}} \bigl\| \jxitwo \pxitwo \hat{\barg}(s,\xi_2)\bigr\|_{L^1_{\xi_2}} \bigl\| \varphi_{\leq -\ell}(\xi_4-\sqrt{3}) \hatg(s,\xi_4) \bigr\|_{L^2_{\xi_4}} \\
  &\lesssim 2^{-\frac12 \ell} \cdot 2^m \cdot 2^{-m} \cdot 2^\ell \cdot (m\varepsilon)^2 \cdot 2^{-\frac12 \ell} \cdot \bigl\| \hatg(s,\xi_4) \bigr\|_{L^\infty_{\xi_4}} \lesssim m^3 \varepsilon^3,
 \end{aligned}
\end{equation*}
where we could freely insert a fattened cut-off to $|\xi_4-\sqrt{3}| \lesssim 2^{-\ell}$ on the third input.
Similarly, by Lemma~\ref{lem:frakn_for_pv} with \eqref{equ:calIpv2m2_subcase52121_frakn_bound2} and the bounds \eqref{equ:g_bound_pxiL1}, \eqref{equ:g_bound_Linftyxi}, we find for $1 \leq m \leq n+5$,
\begin{equation*}
 \begin{aligned}
  &2^{-\frac12 \ell} \bigl\| \calI_{2,m;\ell\leq\ell_2,k_3\leq -\ell}^{\pvdots, 2, \ell \geq \ell_5, 1, (c)}(t,\xi) \bigr\|_{L^2_\xi} \\
  &\lesssim 2^{-\frac12 \ell} \cdot 2^m \cdot 2^{-m} \cdot \biggl\| \calF^{-1} \biggl[ \pxione \biggl( \frac{1}{\pxione \Phi_2} \frakn_1(\xi, \xi_1, \xi_2, \xi_3) \biggr) \biggr] \biggr\|_{L^1(\bbR^4)} \\
  &\quad \quad \times \sup_{s \simeq 2^m} \, \bigl\| e^{is\jD} \varphi_{\leq -\ell}(D-\sqrt{3}) \jD^{-1} D g(s) \bigr\|_{L^\infty_x} \bigl\| e^{is\jD} \jD x g(s) \bigr\|_{L^\infty_x} \bigl\| \varphi_{\leq -\ell}(D-\sqrt{3}) g(s) \bigr\|_{L^2_x} \\
  &\lesssim 2^{-\frac12 \ell} \cdot 2^m \cdot 2^{-m} \cdot 2^{2\ell} \\
  &\quad \quad \times \sup_{s \simeq 2^m} \, \bigl\|  \varphi_{\leq -\ell}(\xi_1-\sqrt{3}) \hatg(s,\xi_1) \bigr\|_{L^1_{\xi_1}} \bigl\| \jxitwo \pxitwo \hat{\barg}(s,\xi_2)\bigr\|_{L^1_{\xi_2}} \bigl\| \varphi_{\leq -\ell}(\xi_4-\sqrt{3}) \hatg(s,\xi_4) \bigr\|_{L^2_{\xi_4}} \\
  &\lesssim 2^{-\frac12 \ell} \cdot 2^m \cdot 2^{-m} \cdot 2^{2\ell} \cdot 2^{-\ell} \cdot \bigl\| \hatg(s,\xi_1) \bigr\|_{L^\infty_{\xi_1}} \bigl\| \jxitwo \pxitwo \hat{\barg}(s,\xi_2)\bigr\|_{L^1_{\xi_2}} \cdot 2^{-\frac12 \ell} \cdot \bigl\| \hatg(s,\xi_4) \bigr\|_{L^\infty_{\xi_4}} \lesssim m^3 \varepsilon^3,
 \end{aligned}
\end{equation*}
where we could freely insert fattened cut-offs to $|\xi_1 - \sqrt{3}| \lesssim 2^{-\ell}$ on the first input and to $|\xi_4-\sqrt{3}| \lesssim 2^{-\ell}$ on the third input.

\medskip

\noindent \underline{\it Subcase 5.2.1.2.2: $1 \leq \ell \leq n-1$, $\ell+10 \leq \ell_2 \leq m$, $k_3 \leq -\ell-1000$, $\ell_5 < \ell+1000$, $1 \leq \ell_1 \leq \ell-10$.}
We again consider the scenario $\xi_1 \approx \sqrt{3}$, i.e., $|\xi_1 - \sqrt{3}| \simeq 2^{-\ell_1-100}$, the other case $\xi_1 \approx -\sqrt{3}$ being analogous.
Writing $\xi_1-\xi_4 = 2 (\xi_1-\sqrt{3}) - (\xi-\sqrt{3}) + (\xi_2+\sqrt{3}) + \xi_3$, we infer from the frequency restrictions that here $|\xi_1-\xi_4| \simeq 2^{-\ell_1-100} \gg 2^{-\ell-1000}$.
We can therefore take
\begin{equation*}
\begin{aligned}
 \frakn_2(\xi,\xi_1,\xi_2,\xi_3) &:= \sum_{1 \leq \ell_1 \leq \ell-10} \varphi_\ell^{(n),+}(\xi) \varphi_{\ell_1}^{(m),+}(\xi_1) \varphi_{\geq \ell+10}^{(m),-}(\xi_2) \varphi_{\leq -\ell-1000}(\xi_3) \\
 &=: \sum_{1 \leq \ell_1 \leq \ell-10} \frakn_{2,\ell_1}(\xi,\xi_1,\xi_2,\xi_3).
\end{aligned}
\end{equation*}
Thus, using also \eqref{equ:calIpv2m2_subcase5212_xione_plus_xifour}, we obtain from \eqref{equ:calIpv2m2_subcase5212_pxi_1Phi2_inverse} that $|\pxione \Phi_2| \gtrsim 2^{-\ell_1}$.
It follows that
\begin{equation*}
 \biggl| \pxi^\kappa \pxione^{\kappa_1} \pxitwo^{\kappa_2} \pxithree^{\kappa_3} \biggl( \frac{1}{\pxione \Phi_2} \frakn_{2,\ell_1}(\xi, \xi_1, \xi_2, \xi_3) \biggr) \biggr| \lesssim 2^{\ell_1} 2^{(\kappa+\kappa_2+\kappa_3)\ell} 2^{\kappa_1 \ell_1},
\end{equation*}
as well as
\begin{equation*}
 \biggl| \pxi^\kappa \pxione^{\kappa_1} \pxitwo^{\kappa_2} \pxithree^{\kappa_3} \cdot \pxione \biggl( \frac{1}{\pxione \Phi_2} \frakn_{2,\ell_1}(\xi, \xi_1, \xi_2, \xi_3) \biggr) \biggr| \lesssim 2^{2\ell_1} 2^{(\kappa+\kappa_2+\kappa_3)\ell} 2^{\kappa_1 \ell_1}.
\end{equation*}
Hence,
\begin{equation} \label{equ:calIpv2m2_subcase52122_frakn_bound1}
 \biggl\| \calF^{-1} \biggl[ \frac{1}{\pxione \Phi_2} \frakn_{2,\ell_1}(\xi, \xi_1, \xi_2, \xi_3) \biggr] \biggr\|_{L^1(\bbR^4)} \lesssim 2^{\ell_1},
\end{equation}
and
\begin{equation} \label{equ:calIpv2m2_subcase52122_frakn_bound2}
 \biggl\| \calF^{-1} \biggl[ \pxione \biggl( \frac{1}{\pxione \Phi_2} \frakn_{2,\ell_1}(\xi, \xi_1, \xi_2, \xi_3) \biggr) \biggr] \biggr\|_{L^1(\bbR^4)} \lesssim 2^{2\ell_1}.
\end{equation}
Correspondingly, by Lemma~\ref{lem:frakn_for_pv} with \eqref{equ:calIpv2m2_subcase52122_frakn_bound1}, and by the bounds \eqref{equ:g_bound_pxiL1}, \eqref{equ:g_bound_Linftyxi}, we obtain
\begin{equation*}
 \begin{aligned}
  &2^{-\frac12 \ell} \bigl\| \calI_{2,m;\ell\leq\ell_2,k_3\leq -\ell}^{\pvdots, 2, \ell \geq \ell_5, 2, (a)}(t,\xi) \bigr\|_{L^2_\xi} \\
  &\lesssim \sum_{1 \leq \ell_1 \leq \ell-10} 2^{-\frac12 \ell} \cdot 2^m \cdot 2^{-m} \cdot \biggl\| \calF^{-1} \biggl[ \frac{1}{\pxione \Phi_2} \frakn_{2,\ell_1}(\xi, \xi_1, \xi_2, \xi_3) \biggr] \biggr\|_{L^1(\bbR^4)} \\
  &\quad \quad \times \sup_{s \simeq 2^m} \, \bigl\| e^{is\jD} \jD^{-1} D x g(s) \bigr\|_{L^\infty_x} \bigl\| e^{is\jD} \jD x g(s) \bigr\|_{L^\infty_x} \bigl\| \varphi_{\leq -\ell_1}(D-\sqrt{3}) g(s) \bigr\|_{L^2_x} \\
  &\lesssim \sum_{1 \leq \ell_1 \leq \ell-10} 2^{-\frac12 \ell} \cdot 2^m \cdot 2^{-m} \cdot 2^{\ell_1} \\
  &\quad \quad \times \sup_{s \simeq 2^m} \, \bigl\| \jxione^{-1} \xi_1 \pxione \hatg(s,\xi_1) \bigr\|_{L^1_{\xi_1}} \bigl\| \jxitwo \pxitwo \hat{\barg}(s,\xi_2)\bigr\|_{L^1_{\xi_2}} \bigl\| \varphi_{\leq -\ell_1}(\xi_4-\sqrt{3}) \hatg(s,\xi_4) \bigr\|_{L^2_{\xi_4}} \\
  &\lesssim \sum_{1 \leq \ell_1 \leq \ell-10} 2^{-\frac12 \ell} \cdot 2^m \cdot 2^{-m} \cdot 2^{\ell_1} \cdot (m\varepsilon)^2 \cdot 2^{-\frac12 \ell_1} \cdot \bigl\| \hatg(s,\xi_4) \bigr\|_{L^\infty_{\xi_4}} \lesssim m^3 \varepsilon^3,
 \end{aligned}
\end{equation*}
where we could freely insert a fattened cut-off to $|\xi_4-\sqrt{3}| \lesssim 2^{-\ell_1}$ on the third input.
The bound for $\calI_{2,m;\ell\leq\ell_2,k_3\leq -\ell}^{\pvdots, 2, \ell \geq \ell_5, 2, (c)}(t,\xi)$ follows using Lemma~\ref{lem:frakn_for_pv} with \eqref{equ:calIpv2m2_subcase52122_frakn_bound2}, and proceeding analogously to the estimate for $\calI_{2,m;\ell\leq\ell_2,k_3\leq -\ell}^{\pvdots, 2, 1, (c)}$ in Subcase~5.2.1.2.1

\medskip

\noindent \underline{\it Subcase 5.2.1.2.3: $1 \leq \ell \leq n-1$, $\ell+10 \leq \ell_2 \leq m$, $k_3 \leq -\ell-1000$, $\ell_5 < \ell+1000$, $\ell_1 = 0$, $k_1 \leq 10$.}
Writing $\xi_1-\xi_4 = 2 (\xi_1-\sqrt{3}) - (\xi-\sqrt{3}) + (\xi_2+\sqrt{3}) + \xi_3$, we infer that $|\xi_1-\xi_4| \gtrsim 2^{-100}$ in this configuration.
So we consider
\begin{equation*}
 \frakn_3(\xi,\xi_1,\xi_2,\xi_3) := \varphi_\ell^{(n),+}(\xi) \varphi_{0}^{(m)}(\xi_1) \varphi_{\leq 10}(\xi_1) \varphi_{\geq \ell+10}^{(m),-}(\xi_2) \varphi_{\leq -\ell-1000}(\xi_3).
\end{equation*}
In view of \eqref{equ:calIpv2m2_subcase5212_xione_plus_xifour}, we must have $|\xi_1| \geq 2^{-10}$ or $|\xi_4| \geq 2^{-10}$ here.
Hence, from $\pxione \Phi_2 = \xi_1 \jxione^{-1} - \xi_4 \jxifour^{-1}$ we infer that if $\xi_1 \xi_4 < 0$, then $|\pxione \Phi_2|^{-1} \lesssim 2^{10}$. If instead $\xi_1 \xi_4 > 0$, then in view of \eqref{equ:calIpv2m2_subcase5212_xione_plus_xifour} we cannot have a high-high interaction, whence $|\xi_1| + |\xi_4| \leq 100$, and then \eqref{equ:calIpv2m2_subcase5212_pxi_1Phi2_inverse} together with \eqref{equ:calIpv2m2_subcase5212_xione_plus_xifour} and $|\xi_1-\xi_4| \gtrsim 2^{-100}$ imply $|\pxione \Phi_2|^{-1} \lesssim 2^{100}$. It follows that
\begin{equation*}
 \biggl| \pxi^\kappa \pxione^{\kappa_1} \pxitwo^{\kappa_2} \pxithree^{\kappa_3} \biggl( \frac{1}{\pxione \Phi_2} \frakn_3(\xi, \xi_1, \xi_2, \xi_3) \biggr) \biggr| \lesssim 2^{(\kappa+\kappa_2+\kappa_3)\ell},
\end{equation*}
as well as
\begin{equation*}
 \biggl| \pxi^\kappa \pxione^{\kappa_1} \pxitwo^{\kappa_2} \pxithree^{\kappa_3} \cdot \pxione \biggl( \frac{1}{\pxione \Phi_2} \frakn_3(\xi, \xi_1, \xi_2, \xi_3) \biggr) \biggr| \lesssim 2^{(\kappa+\kappa_2+\kappa_3)\ell}.
\end{equation*}
Hence,
\begin{equation} \label{equ:calIpv2m2_subcase52123_frakn_bound1}
 \biggl\| \calF^{-1} \biggl[ \frac{1}{\pxione \Phi_2} \frakn_3(\xi, \xi_1, \xi_2, \xi_3) \biggr] \biggr\|_{L^1(\bbR^4)} \lesssim 1,
\end{equation}
and
\begin{equation} \label{equ:calIpv2m2_subcase52123_frakn_bound2}
 \biggl\| \calF^{-1} \biggl[ \pxione \biggl( \frac{1}{\pxione \Phi_2} \frakn_3(\xi, \xi_1, \xi_2, \xi_3) \biggr) \biggr] \biggr\|_{L^1(\bbR^4)} \lesssim 1.
\end{equation}
Then the weighted energy estimates for the terms $\calI_{2,m;\ell\leq\ell_2,k_3\leq -\ell}^{\pvdots, 2, \ell \geq \ell_5, 3, (a)}(t,\xi)$ and $\calI_{2,m;\ell\leq\ell_2,k_3\leq -\ell}^{\pvdots, 2, \ell \geq \ell_5, 3, (c)}(t,\xi)$ can be obtained analogously to Subcase~5.2.1.2.1, but less care has to be taken to compensate for growing factors coming from application of Lemma~\ref{lem:frakn_for_pv}.

\medskip
\noindent \underline{\it Subcase 5.2.1.2.4: $1 \leq \ell \leq n-1$, $\ell+10 \leq \ell_2 \leq m$, $k_3 \leq -\ell-1000$, $\ell_5 < \ell+1000$, $\ell_1 = 0$, $k_1 > 10$.}
Finally, writing $\xi_1-\xi_4 = 2 (\xi_1-\sqrt{3}) - (\xi-\sqrt{3}) + (\xi_2+\sqrt{3}) + \xi_3$, we infer that $|\xi_1-\xi_4| \gtrsim 2^{5}$ in this configuration.
So we take
\begin{equation*}
\begin{aligned}
 \frakn_4(\xi,\xi_1,\xi_2,\xi_3) &:= \sum_{k_1 > 10} \varphi_\ell^{(n),+}(\xi) \varphi_{k_1}(\xi_1) \varphi_{\geq \ell+10}^{(m),-}(\xi_2) \varphi_{\leq -\ell-1000}(\xi_3) =: \sum_{k_1 > 10} \frakn_{4, k_1}(\xi,\xi_1,\xi_2,\xi_3).
\end{aligned}
\end{equation*}
In view of \eqref{equ:calIpv2m2_subcase5212_xione_plus_xifour}, we must then have a high-high interaction, so that $\xi_1 \xi_4 < 0$ and $\pxione \Phi_2 = \xi_1 \jxione^{-1} - \xi_4 \jxifour^{-1}$ give $|\pxione \Phi_2| \gtrsim 1$.
It follows that
\begin{equation*}
 \biggl| \pxi^\kappa \pxione^{\kappa_1} \pxitwo^{\kappa_2} \pxithree^{\kappa_3} \biggl( \frac{1}{\pxione \Phi_2} \frakn_{4,k_1}(\xi, \xi_1, \xi_2, \xi_3) \biggr) \biggr| \lesssim 2^{(\kappa+\kappa_2+\kappa_3)\ell} 2^{- \kappa_1 k_1},
\end{equation*}
as well as
\begin{equation*}
 \biggl| \pxi^\kappa \pxione^{\kappa_1} \pxitwo^{\kappa_2} \pxithree^{\kappa_3} \cdot \pxione \biggl( \frac{1}{\pxione \Phi_2} \frakn_{4,k_1}(\xi, \xi_1, \xi_2, \xi_3) \biggr) \biggr| \lesssim 2^{(\kappa+\kappa_2+\kappa_3)\ell} 2^{-\kappa_1 k_1}.
\end{equation*}
Hence,
\begin{equation} \label{equ:calIpv2m2_subcase52124_frakn_bound1}
 \biggl\| \calF^{-1} \biggl[ \frac{1}{\pxione \Phi_2} \frakn_{4,k_1}(\xi, \xi_1, \xi_2, \xi_3) \biggr] \biggr\|_{L^1(\bbR^4)} \lesssim 1,
\end{equation}
and
\begin{equation} \label{equ:calIpv2m2_subcase52124_frakn_bound2}
 \biggl\| \calF^{-1} \biggl[ \pxione \biggl( \frac{1}{\pxione \Phi_2} \frakn_{4,k_1}(\xi, \xi_1, \xi_2, \xi_3) \biggr) \biggr] \biggr\|_{L^1(\bbR^4)} \lesssim 1.
\end{equation}
Correspondingly, by Lemma~\ref{lem:frakn_for_pv} with \eqref{equ:calIpv2m2_subcase52124_frakn_bound1}, and by the bounds \eqref{equ:g_bound_pxiL1}, \eqref{equ:g_bound_Linftyxi}, we obtain
\begin{equation*}
 \begin{aligned}
  &2^{-\frac12 \ell} \bigl\| \calI_{2,m;\ell\leq\ell_2,k_3\leq -\ell}^{\pvdots, 2, \ell \geq \ell_5, 4, (a)}(t,\xi) \bigr\|_{L^2_\xi} \\
  &\lesssim \sum_{k_1 > 10} 2^{-\frac12 \ell} \cdot 2^m \cdot 2^{-m} \cdot \biggl\| \calF^{-1} \biggl[ \frac{1}{\pxione \Phi_2} \frakn_{4, k_1}(\xi, \xi_1, \xi_2, \xi_3) \biggr] \biggr\|_{L^1(\bbR^4)} \\
  &\quad \quad \times \sup_{s \simeq 2^m} \, \bigl\| e^{is\jD} \widetilde{\varphi}_{k_1}(D) \jD^{-1} D x g(s) \bigr\|_{L^\infty_x} \bigl\| e^{is\jD} \jD x g(s) \bigr\|_{L^\infty_x} \bigl\| g(s) \bigr\|_{L^2_x} \\
  &\lesssim \sum_{k_1 > 10} 2^{-\frac12 \ell} \cdot 2^m \cdot 2^{-m} \cdot \sup_{s \simeq 2^m} \, \bigl\| \widetilde{\varphi}_{k_1}(\xi_1) \pxione \hatg(s,\xi_1) \bigr\|_{L^1_{\xi_1}} \bigl\| \jxitwo \pxitwo \hat{\barg}(s,\xi_2)\bigr\|_{L^1_{\xi_2}} \bigl\| \hatg(s,\xi_4) \bigr\|_{L^2_{\xi_4}} \\
  &\lesssim \sum_{k_1 > 10} 2^{-\frac12 \ell} \cdot 2^m \cdot 2^{-m} \cdot 2^{-k_1} \cdot (m\varepsilon)^3 \lesssim m^3 \varepsilon^3.
 \end{aligned}
\end{equation*}
The bound for the term $\calI_{2,m;\ell\leq\ell_2,k_3\leq -\ell}^{\pvdots, 2, \ell \geq \ell_5, 4, (c)}(t,\xi)$ is analogous.

\medskip
\noindent \underline{\it Subcase 5.2.2: $1 \leq \ell \leq n-1$, $\ell+10 \leq \ell_2 \leq m$, $k_3 > -\ell-1000$.}
We now turn to the scenario when the input and the output frequencies are decorrelated, i.e., when $|\xi_3| \gtrsim 2^{-\ell-1000}$. Note that in this regime the integration with respect to $\xi_3$ does not have to be understood in a $\pvdots$ sense. We further distinguish the subcases (1) $|\xi_4| \gtrsim 2^{-1000}$ and (2) $|\xi_4| \lesssim 2^{-1000}$.

\medskip
\noindent \underline{\it Subcase 5.2.2.1: $1 \leq \ell \leq n-1$, $\ell+10 \leq \ell_2 \leq m$, $k_3 > -\ell-1000$, $k_4 > -1000$.}
Here we consider the energy estimate for the term

\begin{equation*}
 \begin{aligned}
  &\calI_{2,m;\ell\leq\ell_2, k_3 > -\ell}^{\pvdots, 2,1}(t,\xi) \\
  &:= - \int_0^t \tau_m(s) \iiint e^{is\Phi_2} \frakn(\xi,\xi_1,\xi_2,\xi_3) \, (\jxione^{-1} \xi_1) \hatg(s,\xi_1) \jxitwo (\pxitwo \hat{\barg})(s,\xi_2)  \\
  &\qquad \qquad \qquad \qquad \qquad \qquad \qquad \qquad \qquad \qquad \qquad \qquad \qquad \times \hatg(s,\xi_4) \frac{\hatq(\xi_3)}{\xi_3} \, \ud \xi_1 \, \ud \xi_2 \, \ud \xi_3 \, \ud s
 \end{aligned}
\end{equation*}
with
\begin{equation*}
 \frakn(\xi,\xi_1,\xi_2,\xi_3) := \varphi_\ell^{(n),+}(\xi) \varphi_{\geq \ell+10}^{(m),-}(\xi_2) \varphi_{\geq -\ell-1000}(\xi_3) \varphi_{\geq -1000}(\xi_4).
\end{equation*}
Since $|\pxithree \Phi_2| = |\xi_4 \jxifour^{-1}| \gtrsim 1$ in this configuration, we can just integrate by parts in $\xi_3$,
\begin{equation*}
 \begin{aligned}
  &\calI_{2,m;\ell\leq\ell_2, k_3 > -\ell}^{\pvdots, 2,1}(t,\xi) \\
  &= i \int_0^t \tau_m(s) \cdot \frac{1}{s} \iiint e^{is\Phi_2} \frac{1}{\pxithree \Phi_2} \frakn(\xi,\xi_1,\xi_2,\xi_3) \, (\jxione^{-1} \xi_1) \hatg(s,\xi_1) \jxitwo (\pxitwo \hat{\barg})(s,\xi_2) \\
  &\qquad \qquad \qquad \qquad \qquad \qquad \qquad \qquad \qquad \qquad \qquad \qquad \qquad  \times (\pxifour \hatg)(s,\xi_4) \frac{\hatq(\xi_3)}{\xi_3} \, \ud \xi_1 \, \ud \xi_2 \, \ud \xi_3 \, \ud s \\
  &\quad - i \int_0^t \tau_m(s) \cdot \frac{1}{s} \iiint e^{is\Phi_2} \frac{1}{\pxithree \Phi_2} \frakn(\xi,\xi_1,\xi_2,\xi_3) \, (\jxione^{-1} \xi_1) \hatg(s,\xi_1) \jxitwo (\pxitwo \hat{\barg})(s,\xi_2)  \\
  &\qquad \qquad \qquad \qquad \qquad \qquad \qquad \qquad \qquad \qquad \qquad \qquad \qquad  \times \hatg(s,\xi_4) \pxithree \biggl( \frac{\hatq(\xi_3)}{\xi_3} \biggr) \, \ud \xi_1 \, \ud \xi_2 \, \ud \xi_3 \, \ud s \\
  &\quad - i \int_0^t \tau_m(s) \cdot \frac{1}{s} \iiint e^{is\Phi_2} \pxithree \biggl( \frac{1}{\pxithree \Phi_2} \frakn(\xi,\xi_1,\xi_2,\xi_3) \biggr) \, (\jxione^{-1} \xi_1) \hatg(s,\xi_1) \jxitwo (\pxitwo \hat{\barg})(s,\xi_2) \\
  &\qquad \qquad \qquad \qquad \qquad \qquad \qquad \qquad \qquad \qquad \qquad \qquad \qquad  \times \hatg(s,\xi_4) \frac{\hatq(\xi_3)}{\xi_3} \, \ud \xi_1 \, \ud \xi_2 \, \ud \xi_3 \, \ud s \\
  &=: \calI_{2,m;\ell\leq\ell_2, k_3 > -\ell}^{\pvdots, 2,1, (a)}(t,\xi) + \calI_{2,m;\ell\leq\ell_2, k_3 > -\ell}^{\pvdots, 2,1, (b)}(t,\xi) + \calI_{2,m;\ell\leq\ell_2, k_3 > -\ell}^{\pvdots, 2,1, (c)}(t,\xi).
 \end{aligned}
\end{equation*}
Using H\"older's inequality in the frequency variables and the bounds \eqref{equ:g_bound_pxiL1}, \eqref{equ:g_bound_Linftyxi}, we obtain for the first term $\calI_{2,m;\ell\leq\ell_2, k_3 > -\ell}^{\pvdots, 2,1, (a)}(t,\xi)$ on the right-hand side
\begin{equation*}
 \begin{aligned}
  &2^{-\frac12\ell} \bigl\| \calI_{2,m;\ell\leq\ell_2, k_3 > -\ell}^{\pvdots, 2,1, (a)}(t,\xi)\bigr\|_{L^2_\xi} \\
  &\lesssim 2^{-\frac12 \ell} \cdot \bigl\| \widetilde{\varphi}_\ell^{(n),+}(\xi) \bigr\|_{L^2_\xi} \cdot 2^{m} \cdot 2^{-m} \\
  &\qquad \qquad \times \sup_{s \, \simeq \, 2^m} \, \bigl\| \hatg(s,\xi_1)\bigr\|_{L^\infty_{\xi_1}} \bigl\| \jxitwo \pxitwo \hat{\barg}(s,\xi_2) \bigr\|_{L^1_{\xi_2}} \bigl\| \pxifour \hat{\barg}(s,\xi_4) \bigr\|_{L^1_{\xi_4}} \biggl( \int_{|\xi_3| \gtrsim 2^{-\ell-1000}} \frac{|\hatq(\xi_3)|}{|\xi_3|} \, \ud \xi_3 \biggr) \\
  &\lesssim 2^{-\frac12 \ell} \cdot 2^{-\frac12 \ell} \cdot 2^m \cdot 2^{-m} \cdot m^3 \varepsilon^3 \cdot \ell \lesssim m^3 \varepsilon^3.
 \end{aligned}
\end{equation*}
For the second term $\calI_{2,m;\ell\leq\ell_2, k_3 > -\ell}^{\pvdots, 2,1, (b)}(t,\xi)$ we obtain similarly, using the bounds \eqref{equ:g_bound_pxiL1}, \eqref{equ:g_bound_jD2_L2},
\begin{equation*}
 \begin{aligned}
  &2^{-\frac12\ell} \bigl\| \calI_{2,m;\ell\leq\ell_2, k_3 > -\ell}^{\pvdots, 2,1, (b)}(t,\xi)\bigr\|_{L^2_\xi} \\
  &\lesssim 2^{-\frac12 \ell} \cdot \bigl\| \widetilde{\varphi}_\ell^{(n),+}(\xi) \bigr\|_{L^2_\xi} \cdot 2^{m} \cdot 2^{-m} \\
  &\qquad \qquad \times \sup_{s \, \simeq \, 2^m} \, \bigl\| \hatg(s,\xi_1)\bigr\|_{L^2_{\xi_1}} \bigl\| \jxitwo \pxitwo \hat{\barg}(s,\xi_2) \bigr\|_{L^1_{\xi_2}} \bigl\| \hat{\barg}(s,\xi_4) \bigr\|_{L^2_{\xi_4}} \biggl( \int_{|\xi_3| \gtrsim 2^{-\ell-1000}} \frac{1}{|\xi_3|^2} \, \ud \xi_3 \biggr) \\
  &\lesssim 2^{-\frac12 \ell} \cdot 2^{-\frac12 \ell} \cdot 2^m \cdot 2^{-m} \cdot m \varepsilon^3 \cdot 2^\ell \lesssim m \varepsilon^3.
 \end{aligned}
\end{equation*}
The bound for the third term $\calI_{2,m;\ell\leq\ell_2, k_3 > -\ell}^{\pvdots, 2,1, (c)}(t,\xi)$ is analogous.

\medskip
\noindent \underline{\it Subcase 5.2.2.2: $1 \leq \ell \leq n-1$, $\ell+10 \leq \ell_2 \leq m$, $k_3 > -\ell-1000$, $k_4 \leq -1000$.}
In this regime we further distinguish the subcases (1) $|\xi_1| \gtrsim 2^{-100}$ and (2) $|\xi_1| \ll 2^{-100}$. In the former case we can integrate by parts in $\xi_1$, because in that configuration $|\pxione \Phi_2| = |\xi_1 \jxione^{-1} - \xi_4 \jxifour^{-1}| \gtrsim 2^{-100}$. Instead, in the latter case the size of the phase $|\Phi_2| \simeq 1$ is of order one, and we can integrate by parts in time.

\medskip 
\noindent \underline{\it Case 5.3: $\ell = 0$.}
We can adapt the arguments from the preceding Case 5.2 to the setting when the output frequency $||\xi|-\sqrt{3}| \gtrsim 2^{-100}$ is far away from the problematic frequencies, compare with the preceding Case~4.3 and with Cases~3.3 and 4.3 from the proof of Proposition~\ref{prop:weighted_energy_est_delta_T2}. We leave the details to the reader.
\end{proof}

\section{Conclusion of the Proof of Theorem~\ref{thm:main}} \label{sec:conclusion_of_proof}

We begin by establishing the main bootstrap estimates asserted in Proposition~\ref{prop:main_bootstrap}.

\begin{proof}[Proof of Proposition~\ref{prop:main_bootstrap} (Main bootstrap bounds)]
In Proposition~\ref{prop:H2_energy_estimate} we obtained the following $H^2_x$ energy estimate for the profile
\begin{equation*}
 \begin{aligned}
  \sup_{0 \leq t \leq T} \, \bigl\| \jD^2 f(t) \bigr\|_{L^2_x} \lesssim \varepsilon + \bigl( \log(2+T) \bigr)^3 \varepsilon^2.
 \end{aligned}
\end{equation*}
From the integral formulation of the evolution equation~\eqref{equ:f_equ_refer_to} for the profile and from the weighted energy estimates established in Proposition~\ref{prop:weighted_energy_localized_cubic_type}, Proposition~\ref{prop:weighted_energy_est_main_quadratic}, Proposition~\ref{prop:weighted_energy_est_delta_cubic}, and Proposition~\ref{prop:weighted_energy_est_pv_cubic}, we conclude
\begin{equation*}
 \begin{aligned}
  &\sup_{0 \leq t \leq T} \, \sup_{n \geq 1} \, \sup_{0 \leq \ell \leq n} \, 2^{-\frac12 \ell} \tau_n(t) \bigl\| \varphi_\ell^{(n)}(\xi) \jxi^2 \pxi \hatf(t,\xi) \bigr\|_{L^2_\xi} \\
  &\quad \lesssim \|\jx v_0\|_{H^2_x} + \bigl( \log(2+T) \bigr)^3 \varepsilon^2 + \bigl( \log(2+T) \bigr)^{\frac72} \varepsilon^{\frac52} + \bigl( \log(2+T) \bigr)^6 \varepsilon^3 \\
  &\quad \lesssim \|\jx (\varphi_0, \varphi_1)\|_{H^4_x \times H^3_x} + \bigl( \log(2+T) \bigr)^3 \varepsilon^2 + \bigl( \log(2+T) \bigr)^{\frac72} \varepsilon^{\frac52} + \bigl( \log(2+T) \bigr)^6 \varepsilon^3  \\
  &\quad \lesssim \varepsilon + \bigl( \log(2+T) \bigr)^3 \varepsilon^2 + \bigl( \log(2+T) \bigr)^{\frac72} \varepsilon^{\frac52} + \bigl( \log(2+T) \bigr)^6 \varepsilon^3.
 \end{aligned}
\end{equation*}
Hence,
\begin{equation} \label{equ:proof_of_boostrap_bounds_NT_bound}
 \begin{aligned}
  \sup_{0 \leq t \leq T} \, \|f\|_{N_T} \lesssim \varepsilon + \bigl( \log(2+T) \bigr)^3 \varepsilon^2 + \bigl( \log(2+T) \bigr)^{\frac72} \varepsilon^{\frac52} + \bigl( \log(2+T) \bigr)^6 \varepsilon^3.
 \end{aligned}
\end{equation}
Moreover, in Lemma~\ref{lem:stable_coefficient} we proved for the stable coefficient that 
\begin{equation}
 \sup_{0 \leq t \leq T} \, \jt \bigl( \log(2+t) \bigr)^{-2} |a_-(t)| \lesssim \varepsilon + \varepsilon^2.
\end{equation}
Hence, up to times $0 \leq T \leq \exp(c\varepsilon^{-\frac13})$ with $0 < c \ll 1$ the stronger estimates \eqref{equ:stronger_bootstrap1} and \eqref{equ:stronger_bootstrap2} asserted in the statement of Proposition~\ref{prop:main_bootstrap} follow by bootstrap.
\end{proof}

We use a topological shooting argument as in \cite[Lemma 6]{CMM11} to conclude the proof of Theorem~\ref{thm:main}.

\begin{proof}[Proof of Theorem~\ref{thm:main}]
 Let $0 < \varepsilon_1 \ll 1$ be as in the statement of Proposition~\ref{prop:main_bootstrap}. Our goal is to show that there exists $0 < \varepsilon_0 \leq \varepsilon_1$ such that for any even initial conditions $(\varphi_0, \varphi_1) \in H^4_x \times H^3_x$ satisfying
 \begin{equation} \label{equ:proof_of_thm_orthogonality_condition}
  \langle Y_0, \nu \varphi_0 + \varphi_1 \rangle = 0,
 \end{equation}
 and 
 \begin{equation} \label{equ:proof_of_thm_smallness_data}
  \varepsilon := \|\jx (\varphi_0, \varphi_1)\|_{H^4_x \times H^3_x} \leq \varepsilon_0, 
 \end{equation} 
 there exists at least one choice $d \in \bbR$ with $|d| \leq (\log(2))^{-2} \varepsilon^{\frac32}$ so that the solution $(\varphi, \pt \varphi)$ to \eqref{equ:pert_equ_varphi_refer_to} with data
 \begin{equation} \label{equ:proof_of_thm_initial_data}
  (\varphi, \pt \varphi)|_{t=0} = (\varphi_0, \varphi_1) + d (Y_0, \nu Y_0)
 \end{equation} 
 exists at least on the time interval $[0, \exp(c \varepsilon^{-\frac13})]$, and satisfies \eqref{equ:bootstrap1}, \eqref{equ:bootstrap2}, and \eqref{equ:trapping} with $T = \exp(c \varepsilon^{-\frac13})$.
 For such a choice of $d$, recalling the decomposition
 \begin{equation} \label{equ:proof_of_thm_varphi_decomp}
  \varphi = P_c \calJ[v+\bar{v}] + (a_+ + a_-) Y_0,
 \end{equation} 
 we then obtain the asserted decay estimate~\eqref{equ:thm_asserted_decay} using Lemma~\ref{lem:I1_and_J_bounds} and Lemma~\ref{lem:Linfty_decay_v}, 
 \begin{equation*}
 \begin{aligned}
  \|\phi(t)-Q\|_{L^\infty_x} = \|\varphi(t)\|_{L^\infty_x} &\leq \bigl\| P_c \calJ[v(t) + \barv(t)] \bigr\|_{L^\infty_x} + (|a_+(t)| + |a_-(t)|) \|Y_0\|_{L^\infty_x} \\
  &\lesssim \|v(t)\|_{L^\infty_x} + |a_+(t)| + |a_-(t)| \\
  &\lesssim \jt^{-\hf} \log(2+t) \varepsilon + \jt^{-1} \bigl( \log(2+t) \bigr)^2 \varepsilon^\thf + \jt^{-1} \bigl( \log(2+t) \bigr)^2 \varepsilon \\
  &\lesssim \jt^{-\hf} \log(2+t) \varepsilon.
 \end{aligned}
 \end{equation*}
 
 We fix even initial conditions $(\varphi_0, \varphi_1) \in H^4_x \times H^3_x$ satisfying \eqref{equ:proof_of_thm_orthogonality_condition} and \eqref{equ:proof_of_thm_smallness_data}.
 Using a standard fixed-point argument, for any $d \in \bbR$ we can construct a unique solution $(\varphi, \pt \varphi) \in C([0,T^\ast(d)); H^4_x \times H^3_x)$ to \eqref{equ:pert_equ_varphi_refer_to} with data \eqref{equ:proof_of_thm_initial_data} defined on a maximal interval of existence $[0,T^\ast(d))$. Moreover,  
 \begin{equation} \label{equ:proof_of_thm_maximal_time_interval_characterization}
  T^\ast(d) < \infty \quad \Rightarrow \quad \limsup_{t \nearrow T^\ast} \, \bigl\| \bigl( \varphi(t), \pt \varphi(t) \bigr) \bigr\|_{H^4_x \times H^3_x} = \infty.
 \end{equation}
 Finally, recall that the initial data for $v(t)$ and for the coefficients $a_-(t)$ and $a_+(t)$ are given by
 \begin{align*}
  v(0) = \frac12 \bigl( \calD_1 \calD_2 \varphi_0 - i \jD^{-1} \calD_1 \calD_2 \varphi_1 \bigr), \quad a_-(0) = \frac12 \langle Y_0, \varphi_0 - \nu^{-1} \varphi_1 \rangle, \quad a_+(0) = d.
 \end{align*}

 For every $d \in \bbR$ with $|d| \leq (\log(2))^{-2} \varepsilon^{\frac32}$, denote by $T(d) \geq 0$ the maximal time such that the bounds \eqref{equ:bootstrap1}, \eqref{equ:bootstrap2}, and the trapping condition \eqref{equ:trapping} are satisfied with $T=T(d)$.
 Observe that $T(d) < T^\ast(d)$ since in view of \eqref{equ:proof_of_thm_maximal_time_interval_characterization}, the solution $(\varphi(t), \pt \varphi(t))$ can be continued beyond the time $T(d)$ if the bounds \eqref{equ:bootstrap1}, \eqref{equ:bootstrap2}, and \eqref{equ:trapping} hold with $T = T(d)$. 
 Moreover, note that by standard local well-posedness $T(d) > 0$ is strictly greater than zero for $|d| < (\log(2))^{-2} \varepsilon^{\frac32}$ and $T(d) = 0$ for $d = \pm (\log(2))^{-2} \varepsilon^{\frac32}$.

 We seek to show that there exists a choice of $d \in \bbR$ with $|d| \leq (\log(2))^{-2} \varepsilon^{\frac32}$ such that $T(d) \geq \exp(c \varepsilon^{-\frac13})$. Suppose instead that $T(d) < \exp(c \varepsilon^{-\frac13})$ for every $|d| \leq (\log(2))^{-2} \varepsilon^{\frac32}$.
 Then the trapping condition \eqref{equ:trapping} must be saturated at time $t=T(d)$, i.e.,
 \begin{equation} \label{equ:proof_of_thm_saturated_bound}
  \jap{T(d)} \bigl( \log(2+T(d)) \bigr)^{-2} |a_+(T(d))| = \bigl(\log(2)\bigr)^{-2} \varepsilon^{\frac32}.
 \end{equation}
 Indeed, if \eqref{equ:trapping} was a strict inequality at $t = T(d)$, then by Proposition~\ref{prop:main_bootstrap} the bounds \eqref{equ:bootstrap1} and \eqref{equ:bootstrap2} would also have to be strict. But that would be a contradiction to the maximality of $T(d)$.

 As a consequence of \eqref{equ:proof_of_thm_saturated_bound}, we now conclude the following outgoing property
 \begin{equation} \label{equ:outgoing_property}
  \pt \bigl( a_+(t)^2 \bigr) \big|_{t=T(d)} \geq \nu a_+(t)^2 \big|_{t=T(d)} > 0.
 \end{equation}
 To this end we infer from the differential equation \eqref{equ:aplus_equ_refer_to} for the unstable coefficient $a_+(t)$ that
\begin{equation*}
 \begin{aligned}
  \pt \bigl( a_+^2 \bigr) = 2 a_+ (\pt a_+) = 2 \nu a_+^2 + \nu^{-1} a_+ \langle Y_0, (3Q \varphi^2 + \varphi^3) \rangle
 \end{aligned}
\end{equation*}
with $\varphi$ as in \eqref{equ:proof_of_thm_varphi_decomp}.
Thus, for \eqref{equ:outgoing_property} to hold at $t = T(d)$, we must have that
\begin{equation} \label{equ:proof_of_thm_for_outgoing_to_hold}
 \bigl| \nu^{-1} \bigl\langle Y_0, (3Q \varphi(T(d))^2 + \varphi(T(d))^3) \bigr\rangle\bigr| \leq \nu |a_+(T(d))| = \nu \jap{T(d)}^{-1} \bigl( \log(2+T(d)) \bigr)^2 \bigl(\log(2)\bigr)^{-2} \varepsilon^{\frac32}.
\end{equation}
To see why this is the case, note that since \eqref{equ:bootstrap1}, \eqref{equ:bootstrap2} and \eqref{equ:trapping} hold with $T = T(d)$, using Lemma~\ref{lem:I1_and_J_bounds} and Lemma~\ref{lem:Linfty_decay_v}, we have for all $0 \leq t \leq T(d)$ that
\begin{equation*}
 \begin{aligned}
  \|\varphi(t)\|_{L^\infty_x} &\lesssim \|v(t)\|_{L^\infty_x} + |a_+(t)| + |a_-(t)| \\
  &\lesssim \jt^{-\hf} \log(2+t) \varepsilon + \jt^{-1} \bigl( \log(2+t) \bigr)^2 \varepsilon^\thf + \jt^{-1} \bigl( \log(2+t) \bigr)^2 \varepsilon \lesssim \jt^{-\hf} \log(2+t) \varepsilon.
 \end{aligned}
\end{equation*}
Hence, we have
\begin{equation*}
 \begin{aligned}
  \bigl| \nu^{-1} \bigl\langle Y_0, (3Q \varphi(T(d))^2 + \varphi(T(d))^3) \bigr\rangle\bigr| &\lesssim \jap{T(d)}^{-1} \bigl( \log(2+T(d)) \bigr)^2 \varepsilon^2,
 \end{aligned}
\end{equation*}
and \eqref{equ:proof_of_thm_for_outgoing_to_hold} holds for $0 < \varepsilon \leq \varepsilon_0 \leq \varepsilon_1 \ll 1$ sufficiently small.

Next, we infer from the outgoing property~\eqref{equ:outgoing_property} that the map
\begin{equation*}
 \bigl[ -(\log(2))^{-2}\varepsilon^{\frac32}, (\log(2))^{-2} \varepsilon^{\frac32} \bigr] \to [0,\infty), \quad d \mapsto T(d),
\end{equation*}
must be continuous.
Fix $d \in [- (\log(2))^{-2} \varepsilon^{\frac32}, (\log(2))^{-2} \varepsilon^{\frac32}]$. For all sufficiently small $\eta > 0$ with $T(d) + 2\eta < T^\ast(d)$, by \eqref{equ:outgoing_property} there exists $\delta > 0$ such that for all $T(d)-\eta \leq t \leq T(d)+\eta$ we have
\begin{equation*}
 (\log(2))^{-2}\varepsilon^{\frac32}-\delta \leq \jt \bigl( \log(2+t) \bigr)^{-2} |a_+(t)| \leq (\log(2))^{-2} \varepsilon^{\frac32}+\delta,
\end{equation*}
and such that 
\begin{equation*}
\begin{aligned}
 \jap{T(d)-\eta} \bigl( \log(2 + T(d) - \eta) \bigr)^{-2} |a_+(T(d)-\eta)| &\leq (\log(2))^{-2} \varepsilon^{\frac32}-\delta, \\
 \jap{T(d)+\eta} \bigl( \log(2 + T(d) + \eta) \bigr)^{-2}  |a_+(T(d)+\eta)| &\geq (\log(2))^{-2} \varepsilon^{\frac32} + \delta.
\end{aligned}
\end{equation*}
By continuity of the flow for~\eqref{equ:pert_equ_varphi_refer_to}, there exists $\mu > 0$ so that for any $|\tild| \leq (\log(2))^{-2} \varepsilon^{\frac32}$ with $|\tild - d| \leq \mu$, denoting by $\tila_+(t)$ the corresponding evolution with initial condition $a_+(0) = \tild$, we have $T^\ast(\tild) \geq T(d) + \eta$ and for all $t \in [0, T(d) + \eta]$ that
\begin{equation*}
 \bigl| \jt \bigl( \log(2+t) \bigr)^{-2} \, a_+(t) - \jt \bigl( \log(2+t) \bigr)^{-2} \, \tila_+(t) \bigr| \leq \frac{\delta}{2}.
\end{equation*}
Thus, we must have 
$T(d) - \eta \leq T(\tild) \leq T(d) + \eta$,
whence the map $d \mapsto T(d)$ is continuous.

Finally, we consider the map 
\begin{equation*}
\begin{aligned}
 \calA \colon \bigl[ -(\log(2))^{-2} \varepsilon^{\frac32}, (\log(2))^{-2} \varepsilon^{\frac32} \bigr] \to \bigl\{ \pm (\log(2))^{-2} \varepsilon^{\frac32} \bigr\}, \quad d \mapsto \jap{T(d)} \bigl( \log(2+T(d)) \bigr)^{-2} \, a_+(T(d)),
\end{aligned}
\end{equation*}
which is clearly continuous by the continuity of the flow for \eqref{equ:pert_equ_varphi_refer_to} and by the continuity of the map $d \mapsto T(d)$.
Observe that $\calA(\pm (\log(2))^{-2} \varepsilon^{\frac32}) = \pm (\log(2))^{-2} \varepsilon^{\frac32}$.
But this is a contradiction to the continuity of the map~$\calA$ and the intermediate value theorem, which finishes the proof of Theorem~\ref{thm:main}.
\end{proof}

\bibliographystyle{amsplain}
\bibliography{references}

\end{document}